\documentclass[12pt,letterpaper]{book}
\usepackage{amsmath,amstext,amsfonts,amssymb,amsbsy,amscd}
\usepackage[spanish, english]{babel}
\usepackage{amsthm}
\usepackage{fancyhdr}
\usepackage[Lenny]{fncychap}
\usepackage[all]{xy}
\usepackage{graphicx}
\usepackage{MnSymbol}
\usepackage{hyperref}
\usepackage{makeidx}
\makeindex
\def\cleardoublepage{\clearpage\if@twosize\if@odd\c@page\else
\vspace*{\fill}
\thispagestyle{empty}
\newpage
\if@twocolumn\hbox{}\newpage\fi\fi\fi}
\makeatother
\newtheorem{theorem}{Theorem}[section]

\newtheorem{corollary}[theorem]{Corollary}
\newtheorem*{corolario}{Corolario}
\newtheorem*{teorema}{Teorema}
\newtheorem*{corolario1}{Corollary}
\newtheorem*{teorema1}{Theorem}

\newtheorem{definition}[theorem]{Definition}
\newtheorem{example}[theorem]{Example}
\newtheorem{examples}[theorem]{Examples}

\newtheorem{lemma}[theorem]{Lema}
\newtheorem{question}[theorem]{Question}
\newtheorem{notation}[theorem]{Notation}

\newtheorem{proposition}[theorem]{Proposition}

\newtheorem{remark}[theorem]{Remark}

\newtheorem{TG}[theorem]{Terminology of Grothendieck}
\newtheorem{remarks}[theorem]{Remarks}

\newcommand{\pullbackcorner}[1][dr]{\save*!/#1+1.2pc/#1:(1,-1)@^{*}\restore}
\newcommand{\pushoutcorner}[1][dr]{\save*!/#1-1.2pc/#1:(-1,1)@^{*}\restore}

\newcommand{\Z}{$\mathbb{Z}$}
\newcommand{\p}{\mathbf{p}}
\newcommand{\C}{\mathcal{C}}
\newcommand{\D}{\mathcal{D}}
\newcommand{\Hom}{\text{Hom}}
\newcommand{\Ext}{\text{Ext}}
\newcommand{\MinSpec}{\text{MinSpec}}
\newcommand{\End}{\text{End}}

\newcommand{\Pres}{\text{Pres}}
\newcommand{\Copres}{\text{Copres}}
\newcommand{\Supp}{\text{Supp}}
\newcommand{\Add}{\text{Add}}
\newcommand{\add}{\text{add}}
\newcommand{\ann}{\text{ann}}
\newcommand{\Rej}{\text{Rej}}

\newcommand{\Spec}{\text{Spec}}
\newcommand{\Prod}{\text{Prod}}
\newcommand{\te}{\mathbf{t}}
\newcommand{\flecha}{\xymatrix{ \ar[r] &}}
\newcommand{\Ker}{\text{Ker}}
\newcommand{\Gen}{\text{Gen}}
\newcommand{\Cogen}{\text{Cogen}}
\newcommand{\Cone}{\text{Cone}}
\newcommand{\Coker}{\text{Coker}}
\newcommand{\Imagen}{\text{Im}}
\newcommand{\Mod}{\text{Mod-}}
\newcommand{\Mode}{\text{-Mod}}
\newcommand{\limite}{\varinjlim_{\mathcal{H}_{\mathbf{t}}}}
\newcommand{\Ht}{\mathcal{H}_{\mathbf{t}}}
\newcommand{\Hp}{\mathcal{H}_{\phi}}
\newcommand{\G}{\mathcal{G}}
\newcommand{\progen}{\xymatrix{G:=\cdots \ar[r] & 0 \ar[r] & X \ar[r]^{j} & Q \ar[r]^{d} & P\ar[r] & 0 \ar[r] & \cdots}}
\newcommand{\T}{\mathcal{T}}
\newcommand{\F}{\mathcal{F}}
\newcommand{\iso}{\xymatrix{\ar[r]^{\sim} &}}
\newcommand{\epic}{\xymatrix{\ar@{>>}[r] &}}
\newcommand{\monic}{\xymatrix{\ar@{^(->}[r] &}}
\begin{document}
\thispagestyle{empty}
\begin{titlepage}
\begin{center}
 \includegraphics[width=50mm]{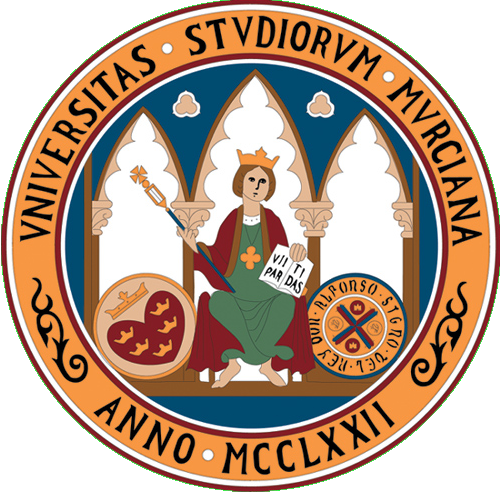}
\end{center}
\vspace{1mm}
\begin{center}
{\LARGE{\bf UNIVERSIDAD DE MURCIA}}\\[8mm]
{\Large \sffamily FACULTAD DE MATEM\'ATICAS}\\[3mm]

\vspace{2cm}
\begin{large}{ \slshape HEARTS OF T-STRUCTURES WHICH ARE GROTHENDIECK OR MODULE CATEGORIES} 
\end{large}


\vspace{0.5 cm}
\begin{large}{\slshape CORAZONES DE T-ESTRUCTURAS QUE SON CATEGOR\'IAS DE GROTHENDIECK O DE M\'ODULOS}
\end{large} 

\vspace{4cm}
\begin{Large}
D. Carlos Eduardo Parra Molina\\ 
\end{Large}
\vspace{0.8 cm}
\begin{Large}
2014
\end{Large}
\end{center}
\end{titlepage}


\newpage




\textbf{Acknowledgements}
\bigskip

I am grateful to my supervisor, Manuel Saor\'in Casta\~no, whose expertise, understanding, generous guidance and support made it possible for me to work on a topic that was of great interest to me. It was a pleasure to work with him.\\

I also thank the Universidad de Los Andes, Venezuela, for making this
study possible by providing me with a grant.\\

Finally, I thank the members of the Department of Informatics of the University of Verona, and specially professor Lidia Angeleri-H$\ddot{u}$gel, for their help and support during my visit while preparing this thesis.

\newpage

\tableofcontents

\fancyhead{}
\fancyhead[RO,LE]{\Large Introducci\'on}
\chapter*{Introducci\'{o}n}
\addcontentsline{toc}{chapter}{Introducci\'on}
\renewcommand{\thepage}{\roman{page}}
\setcounter{page}{1}

La noci\'on de \emph{t-estructura}\index{t-estructura} fue introducida por Beilinson, Bernstein y Deligne \cite{BBD}, en su estudio de los haces perversos sobre una variedad anal\'itica o algebraica estratificada por un subconjunto cerrado. Permite asociar a un objeto de una categor\'ia triangulada arbitraria sus correspondientes ``objetos de homolog\'ia'', que viven en una cierta subcategor\'ia abeliana de dicha categor\'ia triangulada. Tal subcategor\'ia es llamada \emph{coraz\'on}\index{coraz\'on} de la t-estructura. Las t-estructuras han dado lugar a una cantidad considerable de resultados que son importantes en muchas ramas de las Matem\'aticas. No obstante, la definici\'on de t-estructura da origen a un problema dif\'icil que ha llamado la atenci\'on de varios investigadores, como Colpi, Gregorio, Mantese, Tonolo, etc. De manera natural surge la cuesti\'on de cu\'ando dicho coraz\'on es una categor\'ia abeliana lo m\'as ``agradable'' posible. Lo que en el argot categ\'orico significar\'ia cu\'ando la categor\'ia en cuesti\'on es una categor\'ia de m\'odulos o, en su defecto, una categor\'ia de Grothendieck. Planteado de esa manera tan general, el problema es inabordable. Por tanto, es natural imponer restricciones sobre la categor\'ia triangulada ambiente y, tambi\'en sobre la t-estructura que se considera en la misma. De hecho, todos los trabajos que conocemos en esta direcci\'on, est\'an concentrados en la llamada t-estructura de Happel-Reiten-Smal\o.\\

En 1996, Happel, Reiten y Smal\o \ \cite{HRS} asocian a cada par de torsi\'on en una categor\'ia abeliana $\mathcal{A}$, una t-estructura en la categor\'ia derivada acotada $\D^{b}(\mathcal{A})$. En el caso en que $\D(\mathcal{A})$ sea una categor\'ia bien definida, esta t-estructura es la restricci\'on de un t-estructura en $\D(\mathcal{A})$. Diversos autores han estudiado su coraz\'on, con el objetivo de dar respuesta a las siguientes cuestiones: \\

{\bf Cuesti\'on 1}: Dado un par de torsi\'on $\te=(\T,\F)$ en una categor\'ia abeliana $\mathcal{A}$, ?`cu\'ando es el coraz\'on $\Ht$ de la t-estructura asociada en $\D(\mathcal{A})$ una categor\'ia de Grothendieck?. \\

Para el caso m\'as ambicioso:\\

{\bf Cuesti\'on 2}: Dado un anillo asociativo con unidad $R$ y un par de torsi\'on $\te=(\T,\F)$ en $R$-Mod, ?`cu\'ando es $\Ht$ una categor\'ia de m\'odulos?.\\

La cuesti\'on 1, fue estudiada por Colpi, Gregorio y Mantese en \cite{CGM}. En dicho trabajo, utilizan la teor\'ia de \emph{inclinaci\'on}\index{inclinaci\'on} de m\'odulos y la relaci\'on de esta con las categor\'ias derivadas, descubierta por Happel \cite{Ha}. Los autores mencionados se plantean el problema en el contexto particular de la teor\'ia de inclinaci\'on. Para ello toman la noci\'on de objeto inclinante en una categor\'ia abeliana arbitraria $\mathcal{A}$ introducida en \cite{CF} y consideran el par de torsi\'on que dicho objeto, denotado $V$ en lo que sigue, define en $\mathcal{A}$. Por el llamado ``teorema inclinante'' (ver principio de la secci\'on \ref{remark tilting theorem}), este par de torsi\'on define a su vez un par de torsi\'on \emph{fiel}, $\te^{'}=(\mathcal{X,Y})$ en Mod-$R$, donde $R=\text{End}_{\mathcal{A}}(V)$ es el anillo de endomorfismos de $V$. En el curso de la aproximaci\'on al problema, los autores muestran que una categor\'ia abeliana $\mathcal{A}$ tiene un objeto inclinante si, y s\'olo si, es equivalente a $\mathcal{H}_{\mathbf{t}^{'}}$, donde $\te^{'}$ es un par de torsi\'on fiel en una categor\'ia de m\'odulos. A la vista de este hecho, los autores estudiaron un caso particular de la cuesti\'on 1, sacando provecho del objeto inclinante que aparece de manera impl\'icita:\\

{\bf Cuestion 1}$^{'}$: Dado un anillo $R$ y un par de torsi\'on fiel $\te^{'}=(\mathcal{X,Y})$ en $R$-Mod, ?`cu\'ando es $\mathcal{H}_{\te^{'}}$ una categor\'ia de Grothendieck? \\

Para tal cuesti\'on, Colpi, Gregorio y Mantese dan una condici\'on necesaria (ver \cite[Proposition 3.8]{CGM}), que posteriormente Colpi y Gregorio en \cite{CG}, muestran que es tambi\'en suficiente y, por tanto, es una caracterizaci\'on. En concreto muestran, que $\mathcal{H}_{\te^{'}}$ es una categor\'ia de Grothendieck si, y s\'olo si, $\te^{'}$ es una teor\'ia de torsi\'on coinclinante. Durante la elaboraci\'on de esta memoria, nosotros decidimos liberarnos de esta aproximaci\'on por la v\'ia de la teor\'ia de la inclinaci\'on, decidiendo abordar la cuesti\'on 1 sin precondiciones sobre $\te$, pero exigiendo que $\mathcal{A}=\G$ sea una categor\'ia de Grothendieck. \\

En 1964, Barry Mitchell caracteriza las categor\'ias de m\'odulos como las categor\'ias abelianas con coproductos que tienen un generador projectivo compacto (=peque\~no). Tal objeto es llamado \emph{progenerador}\index{progenerador}. Es por ello, que estudiar la cuesti\'on 2, se reduce a construir un progenerador de $\Ht$. Un ejemplo en el que $\Ht$ es una categor\'ia de m\'odulos es debido a Hoshino, Kato y Miyachi \cite{HKM}. Estos autores asocian a cada complejo de dos t\'erminos finitamente generados proyectivos, un par de subcategor\'ias de $R$-Mod y estudian las condiciones bajo las cuales dichas subcategor\'ias forman un par de torsi\'on. En tal caso, el correspondiente coraz\'on es una categor\'ia de m\'odulos. En esta memoria llamamos \emph{complejo HKM} \index{complejo HKM}a cualquier complejo como el indicado que define un par de torsi\'on. Llamamos \emph{par de torsi\'on HKM} a un tal par. Surge entonces la siguiente cuesti\'on:\\

{\bf Cuestion 3:} Si $\Ht$ es una categor\'ia de m\'odulos, ?`es $\te$ un par de torsi\'on HKM?\\

En \cite{CGM}, los autores dan respuesta definitiva a la cuesti\'on 2, para pares de torsi\'on fieles en una categor\'ia de m\'odulos $R$-Mod. Concretamente, el coraz\'on asociado a un par de torsi\'on fiel, es una categor\'ia de m\'odulos si, y s\'olo si, tal coraz\'on es equivalente a el coraz\'on de una t-estructura generada por un complejo inclinante. Sin embargo, esta \'ultima condici\'on es dif\'icil de verificar, quiz\'as tanto como verificar la existencia del progenerador. Esta raz\'on motiv\'o a Colpi, Mantese y Tonolo \cite{CMT} a enfocar el problema con la existencia del progenerador, obteniendo una caracterizaci\'on en funci\'on del mismo, m\'as f\'acil de verificar. No obstante, qued\'o el problema abierto para pares de torsi\'on que no son fieles. En los trabajos \cite{CGM} y \cite{CMT}, no hay un abordaje explic\'ito de la cuesti\'on 3. Ello es llevado a cabo por Mantese y Tonolo \cite{MT}. En este trabajo relacionan los resultados anteriores, de Colpi, Gregorio, Mantese y Tonolo con el trabajo \cite{HKM}. Prueban que si el anillo $R$ es semiperfecto ``poised'' \'o $\te$ es un par de torsi\'on fiel, entonces la respuesta a la pregunta 3 es afirmativa. Esto hizo m\'as significativa la pregunta, por la posibilidad de que la respuesta fuese siempre afirmativa.  \\ 



Uno de nuestros objetivos, desde el principio de este trabajo, fue tratar de abstraer algunas propiedades (ideas y/o argumentos) que surgen al estudiar la t-estructura de Happel-Reiten-Smal\o,  para aplicarlas a t-estructuras arbitrarias. Debido a un trabajo reciente de Alonso,  Jerem\'ias y Saor\'in \cite{AJS}, donde clasifican todas las t-estructuras compactamente generada en la categor\'ia derivada de un anillo Noetheriano conmutativo, el pr\'oximo objetivo a estudiar fue el coraz\'on de tales t-estructuras, plante\'andonos tambi\'en en este contexto la cuesti\'on de cu\'ando el coraz\'on es una categor\'ia de Grothendieck o una categor\'ia de m\'odulos. \\


Pasamos a describir el contenido de esta memoria, que est\'a dividido en 6 cap\'itulos. \\

{\bf Cap\'itulo 1}\\

En este cap\'itulo se introduce las definiciones y resultados m\'as relevantes. Hemos tratado de mantener una equidistancia entre preparar una memoria totalmente autocontenida, que la habr\'ia alargada en exceso, y el abordar los problemas planteados directamente, sin una m\'inima dosis de preliminares. Esperamos que el resumen de conceptos y resultados cl\'asicos m\'as relevantes para la memoria sea suficiente para la \'agil lectura de la misma.\\


{\bf Cap\'itulo 2}\\

El segundo cap\'itulo de esta tesis, est\'a dedicado a presentar los resultados previos a los que hemos hecho menci\'on anteriormente. Est\'an organizados en tres secciones de la siguiente manera:

\begin{enumerate}
\item[-] En la secci\'on 1, presentamos los resultados previos a las cuestiones 1 y 2, que fueron estudiadas en los trabajos \cite{CGM}, \cite{CMT}, \cite{HKM}, \cite{CG} y \cite{MT}. Todo esto con el fin, de aclarar los casos a\'un pendientes de estudio.

\item[-] En la secci\'on 2, introducimos la notaci\'on y terminolog\'ia de \cite{AJS}, as\'i como tambi\'en algunos resultados de dicho trabajo, que fueron \'utiles para nuestros fines.

\item[-] En la secci\'on 3, planteamos una lista de preguntas que surgen de manera natural, tanto para la t-estructura de Happel-Reiten-Smal\o, como para t-estructuras m\'as generales.

\end{enumerate}

{\bf Cap\'itulo 3}\\

En el tercer cap\'itulo, se empieza estudiando con detenimiento el coraz\'on de cualquier t-estructura, con fines de obtener resultados que nos caracterizen cu\'ando dicho coraz\'on es AB5. Como es natural, lo primero que deb\'iamos chequear es la condici\'on AB3, la cual queda resuelta con la proposici\'on \ref{AB3 t-structure}, donde s\'olo exigimos condiciones sobre la categor\'ia triangulada ambiente. El siguiente paso natural a seguir es estudiar la condici\'on AB4. Esta condici\'on es resuelta en la proposici\'on \ref{sufficient AB4}, exigiendo un m\'inimo de condiciones tanto a la categor\'ia triangulada como a la t-estructura. Motivados por el trabajo \cite{CGM}, se logra dar una condici\'on suficiente para que el coraz\'on de una t-estructura en la categor\'ia derivada de una categor\'ia de Grothendieck, sea AB5.

\begin{corolario}[\ref{sufficient AB5 general}]
Sea $\mathcal{G}$ una categor\'ia de Grothendieck, sea $(\mathcal{U},\mathcal{U}^{\perp}[1])$ una t-estructura en $\D(\G)$ y denotemos $\mathcal{H}:=\mathcal{U}\cap \mathcal{U}[1]$, su coraz\'on. \\

Si los funtores de homolog\'ia $\xymatrix{H^{m}:\mathcal{H} \ar[r] & \mathcal{G}}$ conmutan con l\'imites directos, para todo entero $m$, entonces $\mathcal{H}$ es una categor\'ia abeliana AB5.
\end{corolario}

De dicho corolario se desprende un resultado m\'as general, que nos trae como consecuencia el siguiente teorema.

\begin{teorema}[\ref{teo. H compactly general}]
Sea $\D$ una categor\'ia triangulada con coproductos, tal que coproductos arbitrarios de tri\'angulos son tri\'angulos. Para cualquier t-estructura compactamente generada en $\D$, los l\'imites directos numerables son exactos en el coraz\'on.   
\end{teorema}

En el resto del cap\'itulo, nos concentramos en estudiar la t-estructura de Happel-Reiten-Smal\o, con miras a resolver la cuesti\'on 1, donde la categor\'ia abeliana ambiente es una categor\'ia de Grothendieck. En tal caso, todo objeto del coraz\'on viene determinado por una sucesi\'on exacta corta con tallos en los extremos, concentrados en grados -1 y 0 (sus homolog\'ias en las respectivas posiciones). Es por ello, que prestamos atenci\'on a los tallos del coraz\'on, obteniendo as\'i el siguiente resultado, donde el rec\'iproco del corolario anterior se cumple.

\begin{teorema}[\ref{caracterizacion AB5}]
Sea $\G$ una categor\'ia de Grothendieck y $\te=(\T,\F)$ un par de torsi\'on en $\G$. Las siguientes afirmaciones son equivalentes:
\begin{enumerate}
\item[1)] $\Ht$ es una categor\'ia abeliana AB5; 
\item[2)] $\F$ es una clase cerrada bajo l\'imites directos y para cualquier sistema directo $(M_i)$ de $\Ht$, el morfismo can\'onico $\varinjlim{H^{-1}(M_i)} \flecha H^{-1}(\varinjlim_{\Ht}M_i) $ es un monomorfismo;
\item[3)] Los funtores de homolog\'ia $H^{m}:\Ht \flecha \G$ conmutan con l\'imites directos, para todo entero $m$.
\end{enumerate} 
\end{teorema}

Quedando solventado el problema ``AB5'', procedimos entonces a estudiar cu\'ando el coraz\'on tiene un generador. Garantizar la existencia de un generador en el coraz\'on no es una tarea f\'acil. Por ello nos hemos visto obligados a estudiar unos casos particulares de pares de torsi\'on. Nuestros resultados prueban que, para los pares de torsi\'on m\'as usuales (por ejemplo, inclinantes, coinclinantes o hereditarios), el coraz\'on de la t-estructura asociada tiene un generador. M\'as a\'un, se tiene:

\begin{teorema}[\ref{Grothendieck characterization}]
Sea $\G$ una categor\'ia de Grothendieck y $\te=(\T,\F)$ un par de torsi\'on en $\G$ satisfaciendo al menos una, de las siguientes condiciones:
\begin{enumerate}
\item[a)] $\te$ es hereditario.

\item[b)] Cada objeto de $\Ht$ es isomorfo en $\D(\G)$ a un complejo $F$, tal que $F^{k}=0$, si $k\neq -1,0$, y $F^{k}\in \F$, si $k=-1,0$.

\item[c)] Cada objeto de $\Ht$ es isomorfo en $\D(\G)$ a un complejo $T$, tal que $T^{k}=0$, si $k\neq -1,0$, y $T^{k}\in \T$, si $k=-1,0$.
\end{enumerate}
Las siguientes afirmaciones son equivalentes:
\begin{enumerate}
\item[1)] $\Ht$ es una categor\'ia de Grothendieck; 
\item[2)] $\Ht$ es una categor\'ia abeliana AB5;
\item[3)] $\F$ es una clase cerrada bajo l\'imites directos.
\end{enumerate} 
\end{teorema}

Utilizando el teorema anterior, logramos dar resultados m\'as all\'a de la condici\'on Grothendieck que se buscaba para los pares de torsi\'on inclinante y coinclinante, generalizando algunos resultados de \cite{CG}, \cite{CGM} y \cite{BK}. \\

{\bf Cap\'itulo 4}\\

En este cap\'itulo damos respuesta definitiva a la cuesti\'on 2. Al igual que Colpi, Mantese y Tonolo, nuestro objetivo se bas\'o en ver qu\'e propiedades debe satisfacer un complejo para ser un progenerador del coraz\'on. Aqu\'i, el teorema \ref{caracterizacion AB5} juega un papel importante, pues el hecho de que la clase $\F$ sea cerrada para l\'imites directos arroja cierta informaci\'on sobre la homolog\'ia en cero del progenerador. Ello nos ha permitido dar la caracterizaci\'on que busc\'abamos.

\begin{teorema}[\ref{teo. Ht is a module category}]
El coraz\'on $\Ht$ es una categor\'ia de m\'odulos si, y s\'olo si, existe un complejo de $R$-m\'odulos
$$\xymatrix{G:=\cdots \ar[r] & 0 \ar[r] & X \ar[r]^{j} & Q \ar[r]^{d} & P \ar[r] & 0 \ar[r] & \cdots}$$
con $P$ en grado cero, satisfaciendo las siguientes propiedades, donde $V=H^{0}(G)$:
\begin{enumerate}
\item[1)] $\T=\Pres(V)\subseteq \Ker(\Ext^{1}_{R}(V,?))$;

\item[2)] $Q$ y $P$ son $R$-m\'odulos proyectivos finitamente generados y $j$ es un monomorfismo tal que $H^{-1}(G)\in \F$;

\item[3)] $H^{-1}(G)\subseteq \Rej_{\T}(\frac{Q}{X})$;

\item[4)] $\Ext^{1}_{R}(\Coker(j),?)$ se anula en $\F$;

\item[5)] Existe un morfismo $h:(\frac{Q}{X})^{(I)}\flecha \frac{R}{t(R)}$, para alg\'un conjunto $I$, tal que el con\'ucleo de su restricci\'on a $(H^{-1}(G))^{(I)}$ pertenece a $\overline{Gen}$(V)
\end{enumerate} 
(En este enunciado $\Rej_\T(M)$ es el ``\emph{reject}'' de $\T$ en $M$, se decir, la intersecci\'on de todos los n\'ucleos de homomorfismos $M\flecha T$, con $T\in \T$. Por otro lado, $\overline{\Gen}(V)$ es la clase de los $R$-m\'odulos subgenerados por $V$).
\end{teorema}

Aprovechando este resultado, nos concentramos en la cuesti\'on 3. M\'as concretamente, si $\te$ es un par de torsi\'on HKM, ?`qu\'e relaci\'on hay entre el complejo HKM que define el par de torsi\'on $\te$ y el progenerador de $\Ht$?. Como consecuencia, hemos obtenido dos ejemplos con cierta particularidad. Uno de estos ejemplos muestra un complejo HKM que no pertenece a $\Ht$ (ver ejemplo \ref{exam. HKM torsion}), mientras que el otro ejemplo da respuesta negativa a la cuesti\'on 3 (ver corolario \ref{cor. Ht module not HKM}).\\

En gran parte de este cap\'itulo, nos concentramos en el caso que $\te$ es un par de torsi\'on hereditario. En tal caso las condiciones del teorema \ref{teo. Ht is a module category}, en cierta forma, se simplifican. Sin embargo, surge de manera natural una divisi\'on del problema. Por un lado est\'a el caso cuando $\T$ es cerrada para productos. Por otro lado est\'a el caso cuando $\te$ es la parte derecha de una terna TTF en $R$-Mod. \\

Otra pregunta natural que surge es la de encontrar un progenerador de $\Ht$ que sea lo m\'as sencillo posible. En la memoria estudiamos cu\'ando dicho progenerador puede ser elegido de manera que sea una suma directa de tallos. En el caso de un solo tallo, tenemos:

\begin{corolario}[\ref{cor. progenerator casi tilting}]
Fijemos un $R$-m\'odulo $V$ y consideremos las siguientes condiciones:
\begin{enumerate}
\item[1)] $V$ es un m\'odulo 1-inclinante cl\'asico;

\item[2)] $\te=(\Pres(V),\Ker(\Hom_R(V,?)))$ es un par de torsi\'on en $R$-Mod y $V[0]$ es un progenerador de $\Ht$;

\item[3)] $V$ es $R$-m\'odulo finitamente presentado y satisface las siguientes condiciones:
\begin{enumerate}
\item[a)] $\T:=\Pres(V)=\Gen(V)\subseteq \Ker(\Ext^{1}_{R}(V,?))$;

\item[b)] $\Ext^{2}_{R}(V,?)$ se anula en m\'odulos de $\F:=\Ker(\Hom_{R}(V,?))$;

\item[c)] Todo m\'odulo de $\F$ es isomorfo a un subm\'odulo de alg\'un m\'odulo de $\T$.
\end{enumerate}
\end{enumerate}
Entonces se verifican las implicaciones 1) $\Longrightarrow$ 2) $\Longleftrightarrow$ 3). 
\end{corolario}

En tal caso, por inspecci\'on uno pudiese esperar que tal par de torsi\'on es inclinante. Sin embargo, el siguiente resultado muestra lo contrario.

\begin{teorema}[\ref{teo. V[0] prig. which is not tilting}]
Sea $A$ un algebra finito dimensional sobre un cuerpo algebraicamente cerrado $K$. Fijemos $V$ un $A$-m\'odulo por la izquierda 1-inclinante cl\'asico tal que $\Hom_{A}(V,A)=0$ y un $A$-m\'odulo simple por la derecha $X$ tal que $X \otimes_A V=0$. Consideremos la extenci\'on trivial $R=A \rtimes M$, donde $M=V \otimes_K X$. Si vemos $V$ como un $R$-m\'odulo por la izquierda anulado por $0 \rtimes M$, entonces el par $\te=(\Gen(V),\Ker(\Hom_R(V,?)))$ es un par de torsi\'on no inclinante de $R$-Mod, y sin embargo, $V[0]$ es un progenerador de $\Ht$.
\end{teorema}

En el caso en que suponemos que $\te$ es la parte derecha de una terna TTF en $R$-Mod, en unas condiciones suficientemente generales, la caracterizaci\'on queda reducida.

\begin{teorema}[\ref{teo. progenerator Imd=aP}]
Sea $(\mathcal{C,T,F})$ una terna TTF en $R$-Mod definida por el ideal idempotente $\mathbf{a}$ y sea $\te=(\T,\F)$ su parte derecha. Supongamos que el morfimo de monoides $V(R) \flecha V(R/\mathbf{a})$ es suprayectivo. Las siguientes afirmaciones son equivalentes:
\begin{enumerate}
\item[1)] $\Ht$ es una categor\'ia de m\'odulos;

\item[2)] $\mathbf{a}$ es un ideal finitamente generado por la izquierda y existe un $R$-m\'odulo proyectivo finitamente generado $P$ tal que:

\begin{enumerate}
\item[a)] $P/\mathbf{a}P$ es un progenerador de $R/\mathbf{a}$-Mod;

\item[b)] Existe una sucesi\'on exacta de la forma
$\xymatrix{0 \ar[r] & F \ar[r] & C \ar[r] & \mathbf{a}P \ar[r] & 0}$,
con $F\in \F$ y $C$ un $R$-m\'odulo finitamente generado que pertenece a $\mathcal{C}$, tal que $\Ext^{1}_{R}(C,?)_{|\F}=0$ y $C$ genera $\mathcal{C}\cap \F$.
\end{enumerate}
\end{enumerate}
\end{teorema}

Finalmente, para los anillos conmutativo, semi hereditarios, locales, perfectos y noetherianos, damos respuesta definitiva a la cuesti\'on 2 para los pares de torsi\'on hereditarios. \\

{\bf Cap\'itulo 5}\\

En este cap\'itulo, dados un anillo conmutativo Noetheriano $R$ y una t-estructura compactamente generada en $\D(R)$, estudiamos cu\'ando el coraz\'on de \'esta es una categor\'ia de Grothendieck o una categor\'ia de m\'odulos. Para ello usamos la clasificaci\'on de dichas t-estructuras dada por Alonso, Jerem\'ias y Saor\'in (ver \cite{AJS}), en t\'erminos de filtraciones por soportes del espectro. \\

La experiencia desarrollada en el estudio de la t-estructura de Happel-Reiten-Smal\o \ nos motiv\'o a estudiar en primera instancia el comportamiento de los tallos con respecto a los l\'imites directos. Aunque hay aspectos similares aprovechables, hay una gran diferencia en el sentido de que hay complejos en $\Hp$ cuyos tallos asociados (es decir, las homolog\'ias en sus grados respectivos) no pertenecen ninguno a $\Hp$. Aqu\'i $\phi$ denota una filtraci\'on por soportes del espectro y $\Hp$ es el coraz\'on de la t-estructura asociada. Lo que s\'i es una herramienta fundamental para el estudio de la condici\'on AB5 de $\Hp$ es la localizaci\'on. Los dos siguientes resultados son una muestra de ello. En ellos $\phi_\p$ denota la filtraci\'on por soportes de $\Spec(R_\p)$ dada por $\phi_\p(i)=\{\mathbf{q}R_\p \ | \ \mathbf{q} \subseteq \mathbf{p}$ y $\mathbf{q} \in \phi(i)\}$ para todo $i\in \mathbb{Z}$ y $\mathcal{H}_{\phi_\p}$ es el coraz\'on de la t-estructura asociada en $\D(R_\p)$. 



\begin{corolario}[\ref{cor. localization AB5}]
Las siguientes afirmaciones son equivalentes:
\begin{enumerate}
\item[1)] $\Hp$ es una categor\'ia abeliana AB5;
\item[2)] $\mathcal{H}_{\phi_{\mathbf{p}}}$ es una categor\'ia abeliana AB5, para todo $\mathbf{p}\in \Spec(R)$.
\end{enumerate}
\end{corolario}

\begin{corolario}[\ref{cor. localizando (H) categorías de modulos}]
Si $\Hp$ es una categor\'ia de m\'odulos, entonces $\mathcal{H}_{\phi_{\mathbf{p}}}$ es una categor\'ia de m\'odulos, para cada $\mathbf{p}\in \Spec(R)$.
\end{corolario}

El problema de la existencia de un generador en $\Hp$ es m\'as asequible. La proposici\'on \ref{prop. H has a generator}, prueba que el coraz\'on de cualquier t-estructura compactamente generada en $\D(R)$ tiene un generador. En consecuencia, el problema de decidir cu\'ando dicho coraz\'on es una categor\'ia de Grothendieck se reduce a determinar cu\'ando es una categor\'ia abeliana AB5.\\

El estudio de la condici\'on AB5, lo dividimos en dos casos. El primero cubre el caso en que en la filtraci\'on por soportes s\'olo interviene una cantidad finita de subconjuntos de $\Spec(R)$ estables por especializaci\'on. El segundo cubre el caso en que la filtraci\'on por soportes es acotada por la izquierda. En el primer caso, la prueba la realizamos por inducci\'on, mientras que en el segundo se utiliz\'o fuertemente el resultado del caso primero. Concretamente, mostramos que:


\begin{teorema}[\ref{teo. first main chapter V}]
Sea $R$ un anillo conmutativo Noetheriano y $\phi$ una filtraci\'on por soportes acotada por la izquierda de $\Spec(R)$, entonces los funtores de homolog\'ia \linebreak $H^{m}:\Hp \flecha R$-Mod conmutan con l\'imites directos. En particular, $\Hp$ es una categor\'ia de Grothendieck. 
\end{teorema}

Para la cuesti\'on de cu\'ando $\Hp$ es de m\'odulos, empezamos estudiando las filtraciones por soportes tales que los complejos de $\Hp$, est\'an acotados por la derecha. Ello lo hemos hecho as\'i porque en el caso de Happel-Reiten-Smal\o, la homolog\'ia de la derecha del progenerador, en cierto sentido, determinaba el par de torsi\'on as\'i como el progenerador. Para obtener tales filtraciones por soportes, basta con hacer $\phi(m)=\emptyset$, para alg\'un entero $m$. En tal caso, el resultado es el siguiente:

\begin{teorema}[\ref{prop. H mod. cat. R. con.}]
Supongamos que $R$ es un anillo conmutativo Noetheriano conexo y sea $\phi$ una filtraci\'on por soportes de $\Spec(R)$ que es eventualmente trivial, tal que $\phi(i)\neq \emptyset$, para alg\'un entero $i$. Las siguientes afirmaciones son equivalentes:
\begin{enumerate}
\item[1)] $\Hp$ es una categor\'ia de m\'odulos;

\item[2)] Existe un entero $m$, tal que $\phi(m)=\Spec(R)$ y $\phi(m+1)=\emptyset$;

\item[3)] Existe un entero $m$, tal que la t-estructura asociada a $\phi$ coincide con $(\D^{\leq m}(R),\D^{\geq m}(R))$. 
\end{enumerate}
En este caso, $\Hp$ es equivalente a $R$-Mod.
\end{teorema}

Aprovechando el corolario \ref{cor. localizando (H) categorías de modulos}, junto con el teorema anterior, hemos dado respuesta definitiva a la cuesti\'on de cu\'ando $\Hp$ es de m\'odulos, donde la categor\'ia cociente de $R$-Mod por una clase de torsi\'on hereditaria $\T$ juega un papel importante


\begin{teorema}[\ref{teo. second main}]
Supongamos que $R$ es un anillo conmutativo Noetheriano, sea $(\mathcal{U,U}^{\perp}[1])$ una t-estructura compactamente generada en $\D(R)$, tal que $\mathcal{U}\neq \mathcal{U}[-1]$, y sea $\mathcal{H}$ su coraz\'on. Las siguientes afirmaciones son equivalentes:
\begin{enumerate}
\item[1)] $\mathcal{H}$ es una categor\'ia de m\'odulos;

\item[2)] Existe un subconjunto $Z$ de $\Spec(R)$ estable bajo especializaci\'on, posiblemente vac\'io, junto con una familia de idempotentes ortogonales no nulos $\{e_1,\dots,e_t\}$ del anillo de cocientes $R_Z$, y una familia de enteros $m_1 < m_2 < \dots < m_t$ satisfaciendo las siguientes propiedades:
\begin{enumerate}
\item[a)] Si $\mu_{*}:R_Z \Mode \flecha R\Mode$ es el functor restricci\'on de escalares y \linebreak $q:\D(R) \flecha \D(\frac{R\Mode}{\T_Z})$ es el funtor cociente, entonces $\mathcal{U}$ esta dado por los complejos $U\in \D(R)$ tal que $q(U)$ pertence a $\underset{1 \leq k \leq t}{\oplus}\D^{\leq m_k}(\frac{R_Ze_k\Mode}{\mu_{*}^{-1}(\T_Z)\cap R_Ze_k\Mode})$;
\item[b)] La categor\'ia cociente $\frac{R_Ze_k\Mode}{\mu_{*}^{-1}(\T_Z)\cap R_Ze_k\Mode}$ es una categor\'ia de m\'odulos, para todo $k=1,\dots,t$.
\end{enumerate}
En tal caso, $\mathcal{H}$ es equivalente a $\frac{R_Ze_1\Mode}{\mu_{*}^{-1}(\T_Z)\cap R_Ze_1\Mode} \times \dots \times \frac{R_Ze_t\Mode}{\mu_{*}^{-1}(\T_Z)\cap R_Ze_t\Mode}$.
\end{enumerate}
\end{teorema}

{\bf Cap\'itulo 6}\\

En este cap\'itulo, dejamos una lista de problemas abiertos, a los que nos hubiese gustado dar respuesta. Todos estos problemas surgen de manera natural de los resultados obtenidos en este trabajo.


\fancyhead[RO,LE]{\Large Introduction}
\chapter*{Introduction}
\addcontentsline{toc}{chapter}{Introduction}

The notion of t-structure was introduced by Beilinson, Bernstein and Deligne \cite{BBD}, in their study of the perverse sheaves over an analytic or algebraic variety stratified by some closed subset. This notion allows us to associate to an object of an arbitrary triangulated category its corresponding ``objects of homology'', which live in some abelian subcategory of such triangulated category. Such subcategory is called the \emph{heart}\index{heart} of the t-structure. T-structures have made space for a big amount of results which are very important in several branches of Mathematics. However, the definition of t-structure gives rise to a difficult problem that has interested several researchers, as Colpi, Gregorio, Mantese, Tonolo etc. Naturally the following question arises: when is that heart the ``nicest" possible as an abelian category?. In the categorical jargon this means: when is such a category a module category?, or in lack of that, when is it a Grothendieck category?. Asked in such a general way, this question is unapproachable. Therefore, it is natural to put some conditions or restrictions on the triangulated category, and also on the t-structure that is considered in it. In fact, all the works that we know of in this respect are focused on the so-called t-structure of Happel-Reiten-Smal\o.\\

In 1996, Happel, Reiten and Smal\o \ \cite{HRS} associated to each torsion pair in an abelian category $\mathcal{A}$, a t-structure in the bounded derived category $\D^{b}(\mathcal{A})$. In the case when $\D(\mathcal{A})$ is a well-defined category, this t-structure is actually the restriction of a t-structure in the unbounded derived category $\D(\mathcal{A})$. Several authors have studied its heart, with the goal of giving an answer to the questions:\\

{\bf Question 1:} Given a torsion pair $\te=(\T,\F)$ in an abelian category $\mathcal{A}$, when is the heart $\Ht$ of the associated t-structure in $\D(\mathcal{A})$ a Grothendieck category? \\

More ambitiously in the context of modules categories:\\

{\bf Question 2}: Given an associative ring with unit $R$ and a torsion pair $\te=(\T,\F)$ in $R$-Mod, when is $\Ht$ a module category?\\

The first question was studied by Colpi, Gregorio and Mantese in [CGM]. In that paper they use the \emph{tilting theory}\index{tilting theory} of modules and its relationship with derived categories, discovered by Happel \cite{Ha}. For this reason the mentioned authors consider the problem in the particular context of tilting theory. In order to do it, they take the notion of tilting object in an arbitrary abelian category $\mathcal{A}$ introduced in \cite{CF} and they consider the torsion pair defined in $\mathcal{A}$ by that object, which we denote by $V$. From the ``tilting theorem'' (see the begining of the section \ref{remark tilting theorem}), this torsion pair defines, in turn, a faithful torsion pair $\te^{'}= (\mathcal{X,Y})$ in Mod-$R$, where $R = \text{End}_{\mathcal{A}}(V)$ is the endomorphism ring of $V$. In the course of the approximation to the problem, the authors prove that an abelian category $\mathcal{A}$ has a tilting object if, and only if, it is equivalent to $\mathcal{H}_{\te^{'}}$, where $\te^{'}$ is a faithful torsion pair in a module category. In view of this fact, the authors studied a particular case of the question 1 by taking advantage of the tilting object which appears in an implicit way.\\

{\bf Question 1$^{'}$}: Given $R$ a ring and $\te^{'}=(\mathcal{X,Y})$ a faihtful torsion pair in $R$-Mod, when is $\mathcal{H}_{\te^{'}}$ a Grothendieck category?\\

For this question, Colpi, Gregorio and Mantese give a necessary condition (see \cite[Proposition 3.8]{CGM}). Later, Colpi and Gregorio in \cite{CG} prove that such condition is a characterization. In particular, they prove that $\mathcal{H}_{\te^{'}}$ is a Grothendieck category if, and only if, $\te^{'}$ is a cotilting torsion pair. During the research for this thesis, we decided to release from tilting theory, and we chose to deal with question 1 without preconditions over $\te$, although we require that $\mathcal{A}=\G$ is a Grothendieck category.\\

In 1964, Barry Mitchell characterized the module categories as those abelian categories with arbitrary coproducts which have a compact (=small) and projective generator. Such object is called \emph{progenerator}\index{object! progenerator}. That is why studying question 2 is reduced to constructing a progenerator of $\Ht$. An example where $\Ht$ is a module category is due to Hoshino, Kato y Miyachi \cite{HKM}. These authors associate to each complex of two finitely generated projective terms, a pair of subcategories of $R$-Mod and they study the conditions for such subcategories to form a torsion pair. In such case, its associated heart is a module category. In this thesis we call \emph{HKM complex} to any such kind of complex and we call \emph{HKM torsion pair} to its associated torsion pair. So the next question arises:\\

{\bf Question 3}: Suppose that $\Ht$ is a module category. Is $\te$ an HKM torsion pair?\\

In \cite{CGM}, the authors give a definitive answer to question 2 for faithful torsion pairs in a module category $R$-Mod. Specifically, the heart associated with a faithful torsion pair is equivalent to a module category if, and only if, it is the heart of a t-structure generated by a tilting complex. However, this condition is not easily verifiable and is at least as difficult as finding the progenerator. This fact, motivated Colpi, Matese and Tonolo \cite{CMT}, to focus the problem on the existence of the progenerator and they obtained a characterization depending on itself, which is easier to check. However, it remained the open problem for non-faithful torsion pairs. In the papers \cite{CGM} and \cite{CMT}, there is not an explicit result which deals with question 3. This is done by Mantese and Tonolo \cite{MT}. In this work they related  the previous results by Colpi, Gregorio, Mantese and Tonolo with the paper \cite{HKM}. They prove that if $R$ is a left poised and semiperfect ring or $\te$ is a faithful torsion pair, then question 3 has an affirmative answer. This fact makes this question more significant for the possibility that it had always an affirmative answer. \\

One of our goals, from the beginning of this work, was to abstract some properties (ideas and/or arguments) which arise when we study the t-structure of Happel-Reiten-Smal\o,  in order to apply them to arbitrary t-structures. Recently Alonso, Jerem\'ias and Saor\'in \cite{AJS}, classify all the compactly generated t-structures in the derived category of a commutative Noetherian ring. So our next goal was to study the heart of such t-structures, by considering, in this context, the question of when the heart of such a t-structure is a Grothendieck category or a module category.\\


Now, we describe the contents of this thesis, which is divided into six chapters.\\

{\bf Chapter 1}:\\

In this chapter we introduce the more relevant definitions and results. We have tried to keep an equidistance between a self-contained presentation, which would have lengthen excessively the manuscript, and going directly to the problems without minimal preliminary concepts. We hope that the summary of classical definitions and results which are more relevant for this thesis will be enough to make the reading attractive.\\

{\bf Chapter 2}:\\

This chapter is devoted to present the results which are previous to those mentioned before. They are organized in three sections in the following way:
\begin{enumerate}
\item[-] In section 1, we present the results which are previous to questions 1 and 2, which were studied in the works \cite{CGM}, \cite{CMT}, \cite{HKM}, \cite{CG} and \cite{MT}. This is done with the intention of clarify the cases which are pending of study. 
\item[-] In section 2, we introduce the notation and terminology from \cite{AJS}, and we also give some results of that paper which are useful for our goals.
\item[-] In section 3, we present a list of questions that arise in a natural way, for the t-structure of Happel-Reiten-Smal\o \ and for more general t-structures.
\end{enumerate}

{\bf Chapter 3}:\\

We begin the chapter by studying meticulously the heart of any t-structure, in order to obtain results that characterize when that heart is AB5. As natural, the first thing we should check is the condition AB3. It follows from proposition \ref{AB3 t-structure}, where we only require conditions over the triangulated category. Next step is to study the condition AB4. We show that this condition holds (see proposition \ref{sufficient AB4}), by requiring minimal conditions to the triangulated category and the t-structure. Motivated by the paper \cite{CGM}, we give a sufficient condition in order for the heart of a t-structure in the derived category of a Grothendieck category to be AB5.

\begin{corolario1}[\ref{sufficient AB5 general}]
Let $\mathcal{G}$ be a Grothendieck category and let $(\mathcal{U},\mathcal{U}^{\perp}[1])$ be a t-structure on $\D(\mathcal{G})$ and let $\mathcal{H}=\mathcal{U}\cap \mathcal{U}^{\perp}[1]$ be its heart. If the classical homological functors $\xymatrix{H^{m}:\mathcal{H} \ar[r] & \mathcal{G}}$ preserve direct limits, for all $m\in $ \Z, then $\mathcal{H}$ is AB5.
\end{corolario1}

This corollary comes from a more general result, which also gives next theorem.

\begin{teorema1}[\ref{teo. H compactly general}]
Let $\mathcal{D}$ be a triangulated category with coproducts and let $(\mathcal{U}, \mathcal{U}^{\perp}[1])$ be a compactly generated t-structure in $\mathcal{D}.$ Then countable direct limits are exact in $\mathcal{H}=\mathcal{U}\cap \mathcal{U}^{\perp}[1]$.
\end{teorema1}

In the rest of the chapter, we focus on the study of the t-structure of Happel-Reiten-Smal\o, in order to solve question 1, when the abelian category is a Grothendieck category. In this case, each object in the heart is determined by a short exact sequence with outer terms which are stalks concentrated in degrees -1 and 0 (their homologies in their respective position). For this reason, we put attention on the stalks of the heart and we obtain the next result, which shows that the converse of the previous corollary holds.

\begin{teorema1}[\ref{caracterizacion AB5}]
Let $\G$ be a Grothendieck category, let $\mathbf{t}=(\mathcal{T,F})$ be a torsion pair in $\G$, let $(\mathcal{U}_{\mathbf{t}},\mathcal{U}_{\mathbf{t}}^{\perp}[1])$ be its associated t-structure in $\D(\G)$ and let $\Ht=\mathcal{U}_{\mathbf{t}}\cap \mathcal{U}_{\mathbf{t}}^{\perp}[1]$ be the heart. The following assertions are equivalent:
\begin{enumerate}
\item[1)] $\Ht$ is an AB5 abelian category;
\item[2)] $\mathcal{F}$ is closed under taking direct limits in $\G$ and, for each direct system $(M_i)_{i\in I},$ the canonical morphism $\xymatrix{\varinjlim{H^{-1}(M_i)} \ar[r] & H^{-1}(\limite {M_i})}$ is a monomorphism;
\item[3)] The classical homological functors $\xymatrix{H^{m}:\Ht \ar[r] & \G}$ preserve direct limits, for all $m\in $ \Z.
\end{enumerate}
\end{teorema1}

When the AB5 problem was solved, we proceeded to study when the heart has a generator. To guarantee the existence of a generator in the heart is not an easy task, so we had to study some particular cases of torsion pairs, in order to answer the question. Our results show that, for the most usual torsion pairs (e.g. tilting, cotilting or hereditary) the heart of the associated t-structure has a generator. Concretely, we have:

\begin{teorema1}[\ref{Grothendieck characterization}]
Let $\G$ be a Grothendieck category and let $\mathbf{t}=(\T,\F)$ be a torsion pair in $\G$ satisfying at least one of the following conditions:
\begin{enumerate}
\item[a)] $\mathbf{t}$ is hereditary.
\item[b)] Each object of $\Ht$ is isomorphic in $\D(\G)$ to a complex $F$ such that $F^{k}=0$, for $k\neq -1,0$, and $F^{k}\in \F$, for $k=-1,0.$
\item[c)] Each object of $\Ht$ is isomorphic in $\D(\G)$ to a complex $T$ such that $T^{k}=0$, for $k\neq -1,0$ and $T^{k}\in \T$, for $k=-1,0$.
\end{enumerate}
The following assertions are equivalent: 
\begin{enumerate}
\item[1)] The heart $\Ht$ is a Grothendieck category;
\item[2)] $\Ht$ is an AB5 abelian category;
\item[3)] $\F$ is closed under taking direct limits in $\G$.
\end{enumerate}
\end{teorema1}

From the previous theorem we derive results beyond the Grothendieck condition for tilting and cotilting torsion pairs. They extend some results from \cite{CG}, \cite{CGM} and \cite{BK}. \\

{\bf Chapter 4}\\

In this chapter, we focus on giving a definitive answer to question 2. As Colpi, Mantese and Tonolo, our goal has been the identification of which properties have to be hold by a complex in order for it to be a progenerator of the heart. Here, theorem \ref{caracterizacion AB5} plays an important role because of the fact that $\F$ being closed under taking direct limits brings some information about the homology in degree 0 of the progenerator. Our next result gives the characterization which we were looking for.

\begin{teorema1}[\ref{teo. Ht is a module category}]
The heart $\Ht$ is a module category if, and only if, there is a chain complex of $R$-modules 

$$\xymatrix{G:=\cdots \ar[r] & 0 \ar[r] & X \ar[r]^{j} & Q \ar[r]^{d} & P \ar[r] & 0 \ar[r] & \cdots }$$

with $P$ in degree 0, satisfying the following properties, where $V:=H^{0}(G)$:

\begin{enumerate}
\item[1)] $\T=\Pres(V)\subseteq \Ker(\Ext^{1}_{R}(V,?))$;

\item[2)] $Q$ and $P$ are finitely generated projective $R$-modules and $j$ is a monomorphism such that $H^{-1}(G)\in \F$;

\item[3)] $H^{-1}(G)\subseteq \Rej_{\T}(\frac{Q}{X})$;

\item[4)] $\Ext^{1}_{R}(\Coker(j),?)$ vanishes on $\F$;

\item[5)] there is a morphism $h:(\frac{Q}{X})^{(I)} \flecha \frac{R}{t(R)}$, for some set $I$, such that the cokernel of its restriction to $(H^{-1}(G))^{(I)}$ is in $\overline{\Gen}(V)$.
\end{enumerate}
(Here, for each $R$-module $M$, the term $\Rej_{\T}(M)$ stands for the reject of $\T$ in $M$, i.e., the intersection of all the kernels of morphisms $f:M\flecha T$, with $T\in \T$. On the other hand, $\overline{\Gen}(V)$ is the class of modules subgenerated by $V$.)
\end{teorema1}

Once we have this result we can face question 3. Suppose that $\te$ is an HKM torsion pair, what is the relation between the HKM complex which defines $\te$ and the progenerator of $\Ht$?. As a consequence, we have obtained two special examples. One of these examples shows an HKM complex which does not belong to $\Ht$ (see example \ref{exam. HKM torsion}), while the other one gives a negative answer to question 3 (see corollary \ref{cor. Ht module not HKM}). \\

In a large part of this chapter, we focus on the case that $\te$ is a hereditary torsion pair. In this case, the conditions of the theorem \ref{teo. Ht is a module category}, are simplified in some way. However, this induces a division of the problem. First we have the case where $\T$ is closed under taking products. On the other hand, we have the case where $\te$ is the right constituent pair of a TTF triple in $R$-Mod.\\

There is another question, which arises in a natural way, namely, how to find a progenerator which is the simplest possible. We study when such a progenerator may be chosen to be a sum of stalk complexes. If we consider only one stalk we have:

\begin{corolario1}[\ref{cor. progenerator casi tilting}]
Let $V$ be an $R$-module and consider the following conditions
\begin{enumerate}
\item[1)] $V$ is a classical 1-tilting module;
\item[2)] $\te=(\Pres(V),\Ker(\Hom_{R}(V,?)))$ is a torsion pair in $R$-Mod and $V[0]$ is a progenerator of $\Ht$;

\item[3)] $V$ is a finitely presented and satisfies the following conditions:

\begin{enumerate}
\item[a)] $\T:=\Pres(V)=\Gen(V)\subseteq \Ker(\Ext^{1}_{R}(V,?))$;
\item[b)] $\Ext^{2}_{R}(V,?)$ vanishes on $\F:=\Ker(\Hom_{R}(V,?))$;
\item[c)] Each module of $\F$ embeds into a module of $T$.
\end{enumerate}
\end{enumerate}
Then the implications 1) $\Longrightarrow$ 2) $\Longleftrightarrow$ 3) hold true. 
\end{corolario1}

By inspection one could expect that a torsion pair satisfying either one of conditions 2 or 3 is tilting. The next result contradicts it.

\begin{teorema1}[\ref{teo. V[0] prig. which is not tilting}]
Let $A$ be a finite dimensional algebra over an algebraically closed field $K$, let $V$ be a classical 1-tilting left $A$-module such that $\Hom_{\mathcal{A}}(V,A)=0$, let $X$ be a simple right $A$-module such that $X \otimes_{A}V=0$ and let us consider the trivial extension $R=A \rtimes M$, where $M=V \otimes_{K}X$. Viewing $V$ as a left $R$-module annihilated by $0 \rtimes M$, the pair $\mathbf{t}=(\Gen(V),\Ker(\Hom_{R}(V,?)))$ is a non-tilting torsion pair in $R$-Mod such that $V[0]$ is a progenerator of $\Ht$.
\end{teorema1}

If we suppose that $\te$ is the right constituent pair of a TTF triple in $R$-Mod, under some general enough conditions, the characterization of when $\Ht$ is a module category gets simplified.

\begin{teorema1}[\ref{teo. progenerator Imd=aP}]
Let $\mathbf{a}$ be an idempotent ideal of the ring $R$, and let $(\mathcal{C,T,F})$ be the associated TTF triple in $R$-Mod. Let $\te=(\T,\F)$ its right constituent pair. Suppose that the monoid morphism $V(R) \flecha V(R/\mathbf{a})$ is surjective. The following assertions are equivalent:

\begin{enumerate}
\item[1] $\Ht$ is a module category;

\item[2)] $\mathbf{a}$ is a finitely generated ideal on the left and there is a finitely generated projective $R$-module $P$ satisfying the following conditions:
\begin{enumerate}
\item[a)] $P/\mathbf{a}P$ is a (pro)generator of $R/\mathbf{a}-Mod$;
\item[b)] There is an exact sequence $\xymatrix{0 \ar[r] & F \ar[r] & C \ar[r]^{q \hspace{0.1 cm}} & \mathbf{a}P \ar[r] & 0}$ in $R$-Mod, where $F\in \F$ and $C$ is a finitely generated module which is in $\C \cap \Ker(\Ext^{1}_{R}(?,\F))$ and generates $\C \cap \F$.
\end{enumerate}
\end{enumerate}
\end{teorema1}

Finally, for the following classes of rings, we give a definitive answer to question 2 for hereditary torsion pairs: commutative, semihereditary, local, perfect and Noetherian.\\

{\bf Chapter 5}\\

In this chapter, given a commutative Noetherian ring $R$ and a compactly generated t-structure in $\D(R)$, we study when its heart is a Grothendieck category or a module category. To do it we use the classification of those t-structures given by Alonso, Jerem\'ias and Saor\'in (see \cite{AJS}), in terms of filtrations by supports of $\Spec(R)$. \\

The experience obtained in the study of the t-structure of Happel-Reiten-Smal\o \ motived us to investigate, at first, the behavior of the stalks with respect to direct limits. Although there are similar aspects that can be used, there is a big difference because there are complexes in $\Hp$ whose associated stalks (i.e. the homologies in their respective degrees) do not belong to $\Hp$. Here, $\phi$ denotes a filtration by supports of $\Spec(R)$ and $\Hp$ is the heart of the associated t-structure. What is really a fundamental tool for the study of the condition AB5 of $\Hp$ is localization. The two following results confirm this fact. In their statement $\phi_\p$ is the filtration by supports of $\Spec(R_\p)$ such that $\phi_\p(i)=\{\mathbf{q}R_\p \ | \  \mathbf{q} \subseteq \p \text{ and } \mathbf{q} \in \phi(i)\},$ for all $i\in \mathbb{Z}$ and $\mathcal{H}_{\phi_\p}$ is the heart of the associated t-structure in $\D(R_\p)$.

\begin{corolario1}[\ref{cor. localization AB5}]
The following assertions are equivalent:
\begin{enumerate}
\item[1)] $\Hp$ is an AB5 abelian category; 
\item[2)] $\mathcal{H}_{\phi_{\mathbf{p}}}$ is an AB5 abelian category, for all $\mathbf{p}\in \Spec(R)$.
\end{enumerate}
\end{corolario1}

\begin{corolario1}[\ref{cor. localizando (H) categorías de modulos}]
If $\Hp$ is a module category, then $\mathcal{H}_{\phi_{\mathbf{p}}}$ is a module category, for each $\mathbf{p}\in \Spec(R)$.
\end{corolario1}

The problem of the existence of a generator in $\Hp$ is more attainable. Proposition \ref{prop. H has a generator} shows that the heart of any compactly generated t-structure in $\D(R)$ has a generator. Consequently, the problem of deciding when such a heart is a Grothendieck category is reduced to determining when it is an AB5 abelian category. \\

The study of the condition AB5, is divided into two cases. The first one covers the case in which a finite number of stable under specialization subsets of $\Spec(R)$ take part  in the filtration. The second one, covers the case in which the filtration is left bounded. In the first case, we use induction in order to prove the result while in the second one we use the result of the first case. Specifically, we prove:

\begin{teorema1}[\ref{teo. first main chapter V}]
Let $R$ be a commutative Noetherian ring and let $\phi$ be any left bounded filtration by supports of $\Spec(R)$. Then, the classical homology functors $H^{m}:\Hp \flecha R$-Mod, preserve direct limits, for all $m\in \mathbb{Z}$. In particular, $\Hp$ is a Grothendieck category. 
\end{teorema1}

For the question of when $\Hp$ is a module category, we start by studying the filtrations by supports such that the complexes of $\Hp$ are right bounded. The reason for this strategy is that in the Happel-Reiten-Smal\o \ cases  the homology of the right side of the progenerator, in some sense, determines the torsion pair and the progenerator. In order to obtain such filtrations by support, it is enough to put $\phi(m)=\emptyset$, for some integer $m$. In such case, we have the next result:

\begin{teorema1}[\ref{prop. H mod. cat. R. con.}]
Let us assume that $R$ is connected and let $\phi$ be an eventually trivial sp-filtration of $\Spec(R)$ such that $\phi(i)\neq \emptyset$, for some $i\in \mathbb{Z}$. The following assertions are equivalent:
\begin{enumerate}
\item[1)] $\Hp$ is a module category; 

\item[2)] There is an $m\in \mathbb{Z}$ such that $\phi(m)=\Spec(R)$ and $\phi(m+1)=\emptyset$;

\item[3)] There is an $m\in \mathbb{Z}$ such that the t-structure associated to $\phi$ is $(\D^{\leq m}(R),\D^{\geq m}(R))$.
\end{enumerate}

In that case $\Hp$ is equivalent to $R$-Mod.
\end{teorema1}

By taking advantage of corollary \ref{cor. localizando (H) categorías de modulos} together with the previous theorem, we have given a definitive answer to the question: when is $\Hp$ a module category?. The quotient category of $R\Mode$ by a hereditary torsion class $\T$ plays an important role.

\begin{teorema1}[\ref{teo. second main}]
Let $R$ be a commutative Noetherian ring, let $(\mathcal{U},\mathcal{U}^{\perp}[1])$ be a compactly generated t-structure in $\D(R)$ such that $\mathcal{U}\neq \mathcal{U}[-1]$ and let $\mathcal{H}$ be its heart. The following assertions are equivalent:
\begin{enumerate}
\item[1)] $\mathcal{H}$ is a module category;

\item[2)] There are a possibly empty stable under specialization subset $Z$ of $\Spec(R)$, a family $\{e_1,\dots,e_t\}$ of nonzero orthogonal idempotents of the ring quotients $R_Z$ and integers $m_1 < m_2 < \dots < m_t$ satisfying the following properties:
\begin{enumerate}
\item[a)] If $\mu_*: R_Z$-Mod$\flecha R$-Mod is the restriction of scalars functor and \linebreak $q:\D(R) \flecha \D(\frac{R\text{-Mod}}{\T_Z})$ is the canonical functor, then $\mathcal{U}$ consists of the complexes $U\in \D(R)$ such that $q(U)$ is in $\underset{1 \leq k \leq t}{\oplus} \D^{\leq m_t}(\frac{R_Ze_k \text{-Mod}}{\mu_{*}^{-1}(\T_Z)\cap R_Ze_k\text{-Mod}})$;

\item[b)] $\frac{R_Ze_k\text{-Mod}}{\mu_{*}^{-1}(\T_Z) \cap R_Ze_k\text{-Mod}}$ is a module category, for $k=1,\dots,t$.
\end{enumerate}
\end{enumerate}
In that case, $\mathcal{H}$ is equivalent to $\frac{R_Ze_1\text{-Mod}}{\mu_{*}^{-1}(\T_Z)\cap R_Ze_1\text{-Mod}} \times \dots \times \frac{R_Ze_t\text{-Mod}}{\mu_{*}^{-1}(\T_Z)\cap R_Ze_t\text{-Mod}}$.
\end{teorema1}

{\bf Chapter 6 }\\

In this chapter, we present a list of open problems that we would like to tackle in the future. All these problems arise in a natural way from our results in this thesis.

\chapter{Preliminares}
\renewcommand{\thepage}{\arabic{page}}
\setcounter{page}{1}

\section{Abelian categories}
Roughly speaking, an \emph{abelian category} is a category that, in what concerns constructions from homological algebra, behaves very similar to a module category. To go from module to abelian categories, it is necessary  to abstract definitions from homological algebra and module theory. In this section we follow this idea in order to emphasize the properties of categories which will be more useful for us.


\subsection{Additive categories}

\begin{definition} \rm{
A category $\mathcal{A}$ is called \emph{preadditive}\index{category! preadditive}, if it has a zero object and satisfies the two following properties:
\begin{enumerate}
\item[i)] All the sets $\text{Hom}_{\mathcal{A}}(X,Y)$ have a structure of abelian group; 
\item[ii)] The compositions maps
$$\xymatrix{\text{Hom}_{\mathcal{A}}(Y,Z) \times \text{Hom}_{\mathcal{A}}(X,Y) \longrightarrow \text{Hom}_{\mathcal{A}}(X,Z) \\ (g,f) \rightsquigarrow g \circ f }$$
are (\Z -)bilinear maps;
\newline
\newline
We will say that such an $\mathcal{A}$ is \emph{additive}\index{category! additive}, if it further satisfies
\item[iii)] Each finite family of objects of $\mathcal{A}$ has a coproduct.
\end{enumerate} }
\end{definition}

\begin{proposition}
Let $\mathcal{A}$ be a preadditive category and $X_1, X_2, X\in \text{Ob}(\mathcal{A})$. The following assertions are equivalent:
\begin{enumerate}
\item[1)] $X$ is isomorphic to the coproduct of $X_1$ and $X_2$;
\item[2)] there are morphisms $\xymatrix{X_{1} \ar@<-0.5ex>[r]_{u_1} & X \ar@<-1 ex>[l]_{p_1} \ar@<-0.5ex>[r]_{p_2}& X_2 \ar@<-1 ex>[l]_{u_2}}$, such that $p_i \circ u_j=\delta_{ij}1_{X_j}$ and $u_1 \circ p_1 + u_2 \circ p_2 =1_X$; where $\delta_{ij}=0$ if $i \neq j$ and $\delta_{ij}=1$ if $i=j$;
\item[3)] $X$ is isomorphic to the product of $X_1$ and $X_2$.
\end{enumerate} 
\end{proposition}

\begin{proof}
See \cite[Proposition II.9.1]{HS}.
\end{proof}

\begin{remark}\rm{
The notion of additive category is self-dual.}
\end{remark}

\begin{notation}\rm{
Given a family of objects $\{X_1,\dots,X_n \}$, we will use $X=X_1 \times \dots  \times X_n$, $X=X_1 \coprod \dots \coprod X_n$ or $X=X_1 \oplus \dots \oplus X_n$, to denote an object which is both product and coproduct of this family of objects. In this case, we will say that each $X_i$ is a \emph{direct summand} of $X$.
\newline
\newline
Now, if $X=X_1 \oplus \dots \oplus X_n$ and $Y=Y_1 \oplus \dots \oplus Y_m$, then a morphism $f:X\rightarrow Y$ will be denoted by:

$$\begin{pmatrix}
 f_{11} & \dots & f_{1n}\\ 
 \vdots & \ddots & \vdots \\
 f_{m1} & \dots & f_{mn}
\end{pmatrix}$$

This indicates that $f$ is the unique morphism obtained by using the universal property of the product $Y$ and of the coproduct $X$, where $f_{ij}$ is the composition:

$$\xymatrix{X_j \ar[r]^{u_j} & X \ar[r]^{f} & Y \ar[r]^{p_i} & Y_i }$$
for all $i,j$. Note that the composition of maps between objects described above, adopts the form of a matrix multiplication between their corresponding matrices.   }
\end{notation}

\begin{definition}\rm{
An \emph{abelian category}\index{category! abelian} is an additive category $\mathcal{A}$ in which every morphism has a kernel and a cokernel and each morphism has a unique factorization as a cokernel followed by a kernel.}
\end{definition}

A consequence of this definition is the following corollary (see \cite[Chapter IV.4]{S} and \cite[Chapter IX.2]{M}).

\begin{corollary}
For any abelian category $\mathcal{A}$, the following assertions hold :
\begin{enumerate}
\item[$1)$] Every monomorphism is the kernel of its cokernel; 
\item[$1)^{op}$] Every epimorphism is the cokernel of its kernel;
\item[$2)$] A morphism $f$ is a monomorphism if, and only if, its kernel is zero;
\item[$2)^{op}$] A morphism $f$ is a epimorphism if, and only if, its cokernel is zero;
\item[$3)$] Every morphism $f:A \rightarrow B$ in $\mathcal{A}$ has a unique factorization of the form:
$$\xymatrix{A \ar[rr]^{f} \ar@{>>}[rd]^{p} && B \\ & Y \ar@{^(->}[ru]^{j} & }$$  
where $p$ is an epimorphism and $j$ is a monomorphism. 
\end{enumerate}
\end{corollary}

\begin{remark}\rm{
The object $Y$ (or the subobject $Y\xymatrix{\ar@{^(->}[r]^{j} &} B$) of condition 3 above is called the \emph{image}\index{image} of $f$, denoted $\Imagen(f)$}.
\end{remark}

\begin{definition}\rm{
A sequence in $\mathcal{A}$ of the form $\xymatrix{A \ar[r]^{f} & B \ar[r]^{g} & C}$ is said to be \emph{exact} at $B$ when $\Ker(g)=\Imagen(f)$. An arbitrary sequence of morphisms  
$$\xymatrix{A_0 \ar[r]^{f_1} & A_1 \ar[r] & \cdots \ar[r]^{f_n} & A_n & (n \geq 2)}$$
is said to be \emph{exact at $A_t$} if the sequence $\xymatrix{ A_{t-1} \ar[r]^{f_t} & A_{t} \ar[r]^{f_{t+1} \hspace{0.3cm}} & A_{t+1}}$ is exact (in $A_t$). The sequence is called exact when it is exact in $A_i$, for $i=1, \dots, n-1$.}
\end{definition}

A consequence of this definition is the following fact.

\begin{proposition}
Let $\xymatrix{0 \ar[r] & A \ar[r]^{f} & B \ar[r]^{g} & C \ar[r] & 0}$ be a sequence in $\mathcal{A}$. The following assertions are equivalent:
\begin{enumerate}
\item[1)] the sequence is exact;
 
\item[2)] $f$ is a monomorphism, $g$ is a epimorphism and $\Imagen(f)=\Ker(g)$;

\item[3)] $f$ is the kernel of $g$ and $g$ is the cokernel of $f$.

\end{enumerate}
\end{proposition}

\begin{definition}\rm{
A sequence $\xymatrix{0 \ar[r] & A \ar[r]^{f} & B \ar[r]^{g} & C \ar[r] & 0}$  satisfying any of the equivalent conditions of the previous proposition (and hence all of them),
is said to be a \emph{short exact sequence}\index{exact sequence! short}.}
\end{definition}

The following short exact sequences are considered as the trivial ones, from certain point of view.

\begin{proposition}
For a short exact sequence $\xymatrix{0 \ar[r] & A \ar[r]^{f} & B \ar[r]^{g} & C \ar[r] & 0}$ the following assertions are equivalents:
\begin{enumerate}
\item[1)] It is isomorphic to the sequence: 
$$\xymatrix{0 \ar[r] & A \ar[rr]^{\begin{pmatrix}
 1\\ 
 0
\end{pmatrix} \hspace{0.6 cm}} & & A \oplus C \ar[rr]^{\hspace{0.3 cm}\begin{pmatrix}
 0 & 1
\end{pmatrix}} & & C \ar[r] & 0}$$

\item[2)] $f$ is a \emph{section} in $\mathcal{A}$ (i.e. there is a morphism  $p:B\rightarrow A$ such that $p \circ f = 1_{A}$);

\item[3)] $g$ is a \emph{retraction} in $\mathcal{A}$ (i.e. there is a morphism  $u:C\rightarrow B$ such that $g\hspace{0.03cm} \circ \hspace{0.03cm} u = 1_{C}$).

\end{enumerate}
\end{proposition}

\begin{definition}\rm{
We say that an exact sequence $\xymatrix{0 \ar[r] & A \ar[r]^{f} & B \ar[r]^{g} & C \ar[r] & 0}$ is \emph{split}\index{exact sequence! split} if it satisfies any of the conditions of the previous proposition.}
\end{definition}

\vspace{0.3 cm}

We finish this subsection with the next result. We will often use it.

\begin{proposition}{\bf (Snake Lemma)} \index{Snake lemma} \newline
Let $\mathcal{A}$ be an abelian category. Given a commutative diagram with exact rows
$$\xymatrix{& A \ar[r]^{f} \ar[d]^{u} & B \ar[r]^{g} \ar[d]^{v} & C \ar[r] \ar[d]^{w} & 0\\ 0 \ar[r] & A^{'} \ar[r]^{f^{'}} & B^{'} \ar[r]^{g^{'}} & C^{'} }$$
there is an exact sequence in $\mathcal{A}$
$$\xymatrix{Ker(u) \ar[r]^{\widetilde{f}} & Ker(v) \ar[r] & Ker(w) \ar[r]^{\delta \hspace{0.2cm}} & Coker(u) \ar[r] & Coker(v) \ar[r]^{\overline{g^{'}}} & Coker(w)}$$

Moreover, if $f$ is a monomorphism, then $\widetilde{f}$ is a monomorphism and if $g^{'}$ is an epimorphism, then $\overline{g^{'}}$ is an epimorphism.
\end{proposition}

\begin{proof}
See \cite[Theorem 6.6.5]{R}.
\end{proof}

\subsection{Additive functors and their exactness}
In this subsection, we emphasize some properties that make additive functors special, in certain sense. In what follows, all rings will be assumed to be associative unital, and all modules will be assumed to be unitary. We shall use the following notation, with $R$ (or $S$) being a ring, 
\begin{center}
$R$-Mod = category of left $R$-modules,

Mod-$R$ = category of right $R$-modules,

$_{R}$Mod$_{S}$ = category of $R$-$S$ bimodules. 
\end{center}

\begin{definition}\rm{
A functor $F: \mathcal{A} \flecha \mathcal{B}$ between additives categories is called \emph{additive functor}\index{functor! additive} if, for each couple of objects $A$ and $A^{'}$ in $\mathcal{A}$, the map induced
$$\xymatrix{Hom_{\mathcal{A}}(A,A^{'}) \ar[r] & Hom_{\mathcal{B}}(F(A), F(A^{'})), & f \mapsto F(f)}$$
is a morphism of abelian groups.}
\end{definition}

\begin{remark}\rm{
An additive functor between additive categories always preserves finite (co)products. When the categories are abelian, it preserves split short exact sequences but, in general, it does not preserve arbitrary short exact sequences. }
\end{remark}

\begin{definition}\rm{
An additive functor $F: \mathcal{A} \flecha \mathcal{B}$ is called \emph{left exact}\index{functor! left exact} if for each exact sequence in $\mathcal{A}$
$$\xymatrix{0 \ar[r] & A^{'} \ar[r]^{f} & A \ar[r]^{g} & A^{''} \ar[r] & 0}$$
the sequence $\xymatrix{0 \ar[r] & F(A^{'}) \ar[r]^{F(f)} & F(A) \ar[r]^{F(g)} & F(A^{'})}$ is exact in $\mathcal{B}$. Dually, we can define a \emph{right exact}\index{functor! right exact} functor. }
\end{definition}

\begin{example}\rm{
\begin{enumerate} \vspace{0.1 cm}
\item[1)] For each abelian category $\mathcal{A}$ and $A\in Ob(\mathcal{A})$, the functor \newline $Hom_{\mathcal{A}}(A,-):\mathcal{A} \flecha \text{Ab}$ (resp. $Hom_{\mathcal{A}}(-,A):\mathcal{A}^{\text{op}} \flecha \text{Ab}$) is left exact;
\item[2)] If $\mathcal{A}=\text{Mod-}R$ and $N$ is a left $R$-module, then the tensor product functor \linebreak  $-\otimes_{R} N: \text{Mod-}R \flecha \text{Ab}$ is right exact.
\end{enumerate}}
\end{example}

\subsection{(Co)limits.}
In this subsection we introduce the notions of limit and colimit, which are fundamental for us. Recall that a category $I$ is called \emph{small}\index{category! small} when its objects form a set.

\begin{definition}\rm{
If $\mathcal{C}$ and $I$ are an arbitrary and a small category, respectively, then a functor $I \rightarrow \mathcal{C} $ will be called an $I$-\emph{diagram}\index{$I$-diagram} on $\mathcal{C}$, or simply a diagram when $I$ is understood. The $I$-diagrams on $\mathcal{C}$ together with their natural transformations form a category, which we will denote by $[I, \mathcal{C}]$. In this situation, we always have a \emph{constant diagram functor} $\kappa: \mathcal{C} \rightarrow [I,\mathcal{C}]$ which assigns to each object $X$ of $\mathcal{C}$ the constant diagram $\kappa(X):I \rightarrow \mathcal{C}$, which takes $i \rightsquigarrow X$ and $\alpha \rightsquigarrow 1_{X}$, for each object $i$ and morphism $\alpha$ of $I$. }
\end{definition}

\begin{definition}\rm{
Let $M\in [I,\mathcal{C}]$ be an $I$- diagram, we will say that the pair $(C,\tau)$ is the $I$-\emph{colimit}\index{$I$-colimit} of $M$, where $C\in Ob(\mathcal{C})$ and $\tau$ is a natural transformation $\xymatrix{M \ar[r]^{\tau \hspace{0.7 cm}} &\kappa(C)=C}$, if for each $\xymatrix{M \ar[r]^{\sigma \hspace{0.7 cm}} &\kappa(A)=A}$ natural transformation, there is an unique morphism \linebreak $f:C \flecha A$ such that the diagram
$$\xymatrix{& \kappa(C)=C \ar@{-->}[dd]^{f}\\ M \ar[ur]^{\tau} \ar[dr]^{\sigma}& \\ & \kappa(A)=A }$$
is commutative. Dually, we can define the $I$-\emph{limit}\index{$I$-limit} of $M$ which will be a pair of the form $(L, \kappa(L) \flecha M)$.} 
\end{definition}

\begin{notation}\rm{
If $(C, \xymatrix{M \ar[r]^{\tau \hspace{0.2 cm}} & \kappa(C)})$ is the $I$-colim of $M$, we denote by  $C=colim(M)$ or $C=colim(M_i)$. Dually, if $(L, \xymatrix{ \kappa (L) \ar[r]^{\hspace{0.15 cm}\tau} & M})$ is the $I$-limit of $M$, we denote by \linebreak $L=lim(M)$ or $L=lim(M_i)$. When each $I$-diagram has a colimit (resp. limit), we say that $\mathcal{C}$ has $I$-\emph{colimits} (resp. $I$-\emph{limits}). In such case, the functor $\kappa$ has a left (resp. right) adjoint $colim: [I,\mathcal{C}] \flecha \mathcal{C}$ (resp. $lim: [I,\mathcal{C}] \flecha \mathcal{C}$) called the $(I$-)\emph{colimit}\index{functor! colimit} (resp. $(I$-)\emph{limit}) \index{functor! limit}functor.   }
\end{notation}

This applies to two particular situations, which are of great interest to us.

\begin{example}\rm{
\begin{enumerate}
\item[1)] Let $I$ be an any set, viewed as a small category where the identities are their unique morphisms. In such a case $I$-limits (resp. $I$-colimits) are the same as $I$-products (resp. $I$-coproducts). We denote by $\prod:[I,\mathcal{C}] \flecha \mathcal{C}$ (resp. $\coprod: [I,\mathcal{C}] \flecha \mathcal{C}$) the corresponding functor, called the ($I$-)\emph{product} (resp. $I$-\emph{coproduct}) functor. 

\item[2)] Let $I$ be an upward directed (resp. downward directed) set. That is, $I$ is an ordered set $I=(I,\leq)$ such that, for each $i,j \in I,$ there is a $k\in I$ such that $i\leq k$ and $j\leq k$ (resp. $k\leq i$ and $k\leq j$). In this case, we view $I$ as a small category on which there is a unique morphism $i \flecha j$ exactly when $i\leq j$. By historical reasons, the corresponding colimit (resp. limit) functor is denoted by $\varinjlim: [I,\mathcal{C}] \flecha \mathcal{C}$ (resp. $\varprojlim: [I, \mathcal{C}] \flecha \mathcal{C}$) and is called the $(I)$-\emph{direct limit} \index{direct limit}(resp. \emph{inverse limit})\index{inverse limit} functor. When $I$ is upward directed, the $I$-diagrams on $\mathcal{C}$ are called $I$-\emph{directed systems}\index{$I$-direct systems} in $\mathcal{C}$. If $X:I \flecha \mathcal{C}$ is such a direct system, frequently, we will write $X=[(X_i)_{i\in I},(u_{ij})_{i \leq j}]$ or simply $X=(X_i)_{i \in I}$, where $u_{ij}$ is the image by $X$ of the unique morphism $i \flecha j$ in $I$.
\end{enumerate}}
\end{example}

\vspace{0.3 cm}

A key morphism that we shall frequently use is the following.

\begin{definition}\label{def. colimit-defining morphism}\rm{
Let $I$ be a directed set and $\mathcal{A}$ be any additive category with arbitrary coproducts. Given an $I$-directed system $[(X_i)_{i\in I},(u_{ij})_{i\leq j}]$, we put $X_{ij}=X_i$, for all $i\leq j$. The \emph{colimit-defining morphism}\index{morphism! colimit-defining} associated to the direct system is the unique morphism $f:\underset{i \leq j }{\coprod}X_{ij} \flecha \underset{i\in I}{\coprod}X_i$ such that if $\lambda_{kl}:X_{kl}\flecha \underset{i \leq j}{\coprod}X_{ij}$ and $\lambda_j:X_j \flecha \underset{i\in I}{\coprod}X_i$ are the canonical morphisms into the coproducts, then $f\circ\lambda_{ij}=\lambda_i-\lambda_j\circ u_{ij}$ for all $i\leq j$.}
\end{definition}

\begin{remark}\label{rem. cokerl=limite directo}\rm{
In the situation of last definition, the $I$-direct system $(X_i)$ has a colimit (i.e. a direct limit) if, and only if, $f$ has a cokernel in $\mathcal{A}$. In such case we have $L:=\varinjlim X_i=\text{Coker}(f)$ (see e.g. \cite[Proposition II.6.2]{P}).}
\end{remark}

\begin{remark}\label{rem. preserve colimits functor}\rm{ $M\in [I,\mathcal{A}]$ be an $I$-diagram in $\mathcal{A}$ and $F:\mathcal{A} \flecha \mathcal{A}^{'}$ be an additive functor. We will say that $F$ \emph{preserves $I$-colimits}\index{functor! preserves $I$-colimits}, when $F \circ M: I \flecha  \mathcal{A}^{'}$ has a colimit in $\mathcal{A}^{'}$, for each $I$-diagram $M:I \flecha \mathcal{A}$ which has a colimit in $\mathcal{A}$, and, moreover, the canonical morphism $\psi_M:colim(F \circ M) \flecha F (colim (M))$ is an isomorphism. In particular, we get a commutative diagram 
$$\xymatrix{colim(F \circ M) \ar[r]^{\psi_M} \ar[d]_{colim(F(f))} & F(colim(M)) \ar[d]^{F(colim(f))} \\ colim(F \circ N) \ar[r]^{\psi_N} & F(colim(N))}$$ 
for each $I$-diagram $N:I \flecha \mathcal{A}$ which has a colimit in $\mathcal{A}$ and each morphism $f:M \flecha N$ of $I$-diagrams.  }
\end{remark}

The following result is very useful and will be frequently used in this work.

\begin{proposition}
Let $I$ be a small category. Every left adjoint functor between abelian categories preserves $I$-colimits. 
\end{proposition}
\begin{proof}
See \cite[Proposition II.12.1]{Mi}.
\end{proof}

\begin{definition}\rm{
The category $\mathcal{C}$ is called \emph{complete}\index{category! complete} (resp. \emph{cocomplete}\index{category! cocomplete}) when it has $I$-limits (resp. $I$-colimits) in $\mathcal{C}$, for every small category $I$.}
\end{definition}

\vspace{0.3 cm}

A much easier criterion for (co)completeness is the following:

\begin{proposition}
Let $\mathcal{A}$ be an additive category. It is complete (resp. cocomplete) if, and only if, it has products (resp. coproducts) and kernels (resp. cokernels).
\end{proposition}
\begin{proof}
See \cite[Proposition II.5.2]{P}.
\end{proof}
\subsection{Grothendieck categories}
\vspace{0.1cm}

After the sixties of last century, Algebraic Geometry adopted a categorical point of view, being the category of quasi-coherent sheaves $Qcoh(X)$ on an algebraic scheme $X$ his cornerstone (see \cite{Gr2} and \cite{Har} for relevant definitions). The behaviour of this category is, in many aspects, similar to that of modules. Nevertheless, this is a category of modules just when the scheme $X$ is affine, that is, when, $X\cong \Spec(R)$ for a commutative ring $R$. Grothendieck saw the need of abstracting the properties of $Qcoh(X)$ and defining a new class of categories including $Qcoh(X)$ and $R$-Mod as particular cases. In this section we shall introduce such categories, known since then as Gothendieck categories. They play an important role in this work. In the sequel $\mathcal{A}$ is an abelian category and $I$ is a small category.

\begin{definition}\rm{
Given a small category $I$, we will say that $I$-\emph{colimits are exact} in $\mathcal{A}$, when $\mathcal{A}$ has $I$-colimits and the functor $colim_I: [I,\mathcal{A}] \flecha \mathcal{A}$ is exact, which is equivalent to saying that it preserves kernels.}
\end{definition}

\begin{TG}\rm{
Grothendieck (see \cite{Gr}) introduced the following terminology for an abelian category $\mathcal{A}$:

\begin{enumerate}
\item[-] $\mathcal{A}$ is AB3 (resp. AB3*) when it has coproducts (resp. products);

\item[-] $\mathcal{A}$ is AB4 (resp. AB4*) when it is AB3 (resp. AB3*) and the coproduct functor $\coprod: [I,\mathcal{A}] \flecha \mathcal{A}$ (resp. product functor $\prod: [I,\mathcal{A}] \flecha \mathcal{A}$) is exact, for each set $I$;

\item[-] $\mathcal{A}$ is AB5 (resp. AB5*) when it is AB3 (resp. AB3*) and the direct limit functor $\varinjlim:[I,\mathcal{A}] \flecha \mathcal{A}$ (resp. inverse limit functor $\varprojlim: [I,\mathcal{A}] \flecha \mathcal{A}$) is exact, for each upward (resp. downward) directed set $I$.
\end{enumerate}}
\end{TG}

Note that $\mathcal{A}$ is AB3 (resp. AB3*) if, and only if, it is cocomplete (resp. complete). On the other hand, for each upward (resp. downward) directed set $I$, the functor $\varinjlim: [I,\mathcal{A}] \flecha \mathcal{A}$ (resp. $\varprojlim:[I,\mathcal{A}] \flecha \mathcal{A}$) is always right (resp. left) exact since it is a left (resp. right) adjoint functor. 

\begin{definition}\label{def. locally small}\rm{
An object $G$ of $\mathcal{A}$ is called a \emph{generator}\index{object! generator} when the functor \linebreak $Hom_{\mathcal{A}}(G,?):\mathcal{G} \flecha Ab$ is \emph{faithful}\index{functor! faithful} i.e. $\Hom_{\mathcal{A}}(G,f)\neq 0 $ for each $f:A \flecha B$ non-zero morphism in $\mathcal{A}$. Dually, an object $E$ of $\mathcal{A}$ is called $\emph{cogenerator}\index{object! cogenerator}$ when it is a generator of the opposite category $\mathcal{A}^{op}$. When such $G$ (or $E$) exists, $\mathcal{A}$ is \emph{locally small}\index{category! locally small}, i.e., the class of subobjects of any given object is actually a set (see \cite[Proposition IV.6.6]{S}). }
\end{definition}

\vspace{0.3 cm}

The following result gives us a characterization of generator in AB3 abelian categories.

\begin{proposition}
Let $\mathcal{A}$ be a cocomplete (=AB3) abelian category and $G\in Ob(\mathcal{A})$. The following assertions are equivalent:
\begin{enumerate}
\item[1)] $G$ is a generator of $\mathcal{A}$;
\item[2)] Every object of $\mathcal{A}$ is a quotient of a coproduct of copies of $G$.
\end{enumerate} 
\end{proposition}
\begin{proof}
See \cite[Lemma II.11.2]{P}.
\end{proof}

\vspace{0.3 cm}

We can now define the desire new type of categories.

\begin{definition}\rm{
An AB5 abelian category $\mathcal{G}$ having a generator is called a \emph{Grothendieck category}\index{category! Grothendieck}.}
\end{definition}

We finish this section with some important properties of Grothendieck categories which we will use in this work.

\begin{definition}\rm{
Let $\mathcal{A}$ be a cocomplete (resp. complete) abelian category. Let $A$ be an object of $\mathcal{A}$ and let $(A_i)_{i\in I}$ be a family of subobjects of $A$, then the image (resp. kernel) of the induced morphism $\coprod A_{i} \flecha A$ (resp. $A \flecha \coprod A_{i}$) is called \emph{the sum of the $A_i$}\index{subobject! sum} (resp. \emph{intersection of the $A_i$}\index{subobject! intersection}), which we will denote by $\sum A_i$ (resp. $\cap A_i$). }
\end{definition}

\begin{definition}\rm{
Let $\mathcal{A}$ be an abelian category and let $A$ be an object of $\mathcal{A}$. A subobject $B$ of $A$ is called \emph{essential}\index{subobject! essential} if $B \cap B^{'}\neq 0$, for every nonzero subobject $B^{'}$ of $A$. More generally, we say that a monomorphism $\alpha:B \monic C$ is \emph{essential} \index{essential monomorphism} if $\Imagen(\alpha)$ is an essential subobject of $C$.    }
\end{definition}

\begin{definition}\rm{
An \emph{injective envelope}\index{injective envelope} of an object $A$ of $\mathcal{A}$, is an essential monomorphism $A \monic E$, where $E$ is an injective object of $\mathcal{A}$. Such, an injective envelope is uniquely determined by $A$, up to isomorphism.}
\end{definition}

The next result highlights some properties which make Grothendieck categories special.

\begin{theorem}
If $\mathcal{G}$ is a Grothendieck category, then $\G$ is complete and each object of $\G$ has an injective envelope.  
\end{theorem}
\begin{proof}
See \cite[Example V.2.1, Corollaries X.4.3 and X.4.4]{S}.
\end{proof}



\begin{definition}\rm{
An object $A$ of an additive category $\mathcal{A}$ with coproducts is called \emph{compact} \index{object! compact}, when the functor $\Hom_{\mathcal{A}}(A,?): \mathcal{A} \flecha$ Ab preserves coproducts. If $\mathcal{A}$ also has direct limits, then the object $A$ is said to be \emph{finitely presented} \index{object! finitely presented} when that functor preserves direct limits.}
\end{definition}

\begin{definition}\label{def. locally fp G}\rm{
A Grothendieck category $\G$ is called \emph{locally finitely presented}\index{category! locally finitely presented} if there is a set $\mathcal{S}$ of finitely presented objects such that each object of $\G$ is a direct limit of objects of $\mathcal{S}$. }
\end{definition}

\begin{examples}\rm{
The following are examples of locally finitely presented Grothendieck categories:
\begin{enumerate}
\item[1)] Each category of additive functors $\mathcal{A}\flecha \text{Ab}$, for every skeletally small preadditive category $\mathcal{A}$. Equivalently (see \cite[Proposition II.2]{G}), each category $R$-Mod of unitary modules over a ring $R$ with enough idempotents.

\item[2)] The category $Qcoh(X)$ of quasi-coherent sheaves over any quasi-compact and quasi-separated algebraic scheme (\cite[I.6.9.12]{Gr2}). 

\item[3)] Each quotient category $\G/\T$ (see subsection \ref{sec. quotients category}), where $\G$ is locally finitely presented Grothendieck category and $\T$ is a hereditary torsion class in $\G$ generated by finitely presented objects (\cite[Proposition 2.4]{ES}).
\end{enumerate}}
\end{examples}

 
\begin{definition}\rm{
Let $\mathcal{A}$ be an AB3 abelian category. An object $P$ of $\mathcal{A}$ will be called a \emph{progenerator}\index{object! progenerator} of $\mathcal{A}$ if it is projective, compact and a generator of $\mathcal{A}$.}
\end{definition}

\begin{theorem}[Gabriel-Mitchell]\label{teo. Gabriel-Michell}
An AB3 abelian category $\mathcal{A}$ is a module category if, and only if, it has a progenerator. 
\end{theorem}
\begin{proof}
see \cite[Corollary 3.6.4]{Po}.
\end{proof}

\section{Torsion pairs}\label{section:torsion pair}

The notion of torsion pair was introduced in the sixties by Dickson (see \cite{D}) in the setting of abelian categories, generalizing the classical notion for abelian groups. In this section, we introduce the terminology concerning torsion pairs and give the torsion pairs which will be frequently used in this work. The reader is referred to \cite{S} and \cite{Go} for the terminology concerning torsion pairs (also called torsion theories) that we use in this manuscript. 

\begin{definition}\rm{
A \emph{torsion pair}\index{torsion pair} in an abelian category $\mathcal{A}$ is a pair $\mathbf{t} = (\mathcal{T} , \mathcal{F})$ of full subcategories satisfying the following two conditions:

\begin{enumerate}
\item[1)] $\text{Hom}_{\mathcal{A}}(T,F)=0$, for all $T\in \mathcal{T}$ and $F \in \mathcal{F}$;

\item[2)] For each object $X$ of $\mathcal{A}$, there is an exact sequence 
$$0 \flecha T_{X} \flecha X \flecha F_{X} \flecha 0 $$
where $T_{X} \in \mathcal{T}$ and $F_{X} \in \mathcal{F}$.
\end{enumerate}
}

\end{definition}

\begin{remark}\rm{
\begin{enumerate}

\item[1)] In such case the objects $T_X$ and $F_X$ are uniquely determined, up to isomorphism, and the assignment $X \rightsquigarrow T_{X}$ (resp. $X \rightsquigarrow F_X$) underlies a functor $t: \mathcal{A} \flecha \mathcal{T}$ (resp. $(1:t): \mathcal{A} \flecha \mathcal{F}$), which is right (resp. left) adjoint to the inclusion functor $\mathcal{T} \monic \mathcal{A}$ (resp. $\mathcal{F} \monic \mathcal{A}$). The composition $\mathcal{A} \xymatrix{\ar[r]^{t}&} \T \monic \mathcal{A}$ (resp. $\mathcal{A} \xymatrix{\ar[r]^{(1:t)}&} \F \monic \mathcal{A}$), which we will still denote by $t$ (resp. (1:t)), is called the \emph{torsion radical} \index{functor! torsion radical}(resp. torsion coradical)\index{functor! torsion coradical}.

\item[2)] The class $\mathcal{T}$ is called the \emph{torsion class}\index{class! torsion} and its objects are the \emph{torsion objects}\index{object! torsion}, while $\mathcal{F}$ is the \emph{torsion-free class}\index{class! torsion-free} consisting of the \emph{torsion-free objects}\index{object! torsion-free}. All these notions are, of course, referred to the torsion pair $(\T,\F)$.

\item[3)] For each class $\mathcal{X}$ of objects of $\mathcal{A}$, we will put $\mathcal{X}^{\perp}=\{M\in \mathcal{A} \ | \ \Hom_{\mathcal{A}}(X,M)=0, \text{ for all }X\in \mathcal{X}\}$ and $^{\perp}\mathcal{X}=\{M\in \mathcal{A} \ | \  \Hom_{\mathcal{A}}(M,X)=0, \text{ for all }X\in \mathcal{X}\}$. It is easy to see that $\T=^{\perp}\F$ and $\F=\T^{\perp}$.

\item[4)] We will frequently write $X/t(X)$ to denote $(1:t)(X)$.

\item[5)] If $\mathcal{A}$ is a cocomplete and locally small category and $\mathcal{T}\subseteq \mathcal{A}$ is a full subcategory, then $\mathcal{T}$ is a torsion class if and only if $\mathcal{T}$ is closed under taking quotients (=epimorphic images), coproducts and extensions.

\item[5)$^{op}$] If $\mathcal{A}$ is a complete and locally small category and $\mathcal{F}\subseteq \mathcal{A}$ is a full subcategory, then $\mathcal{F}$ is a torsion-free class if and only if $\mathcal{F}$ is closed under taking subobjects, products and extensions.
\end{enumerate}}
\end{remark}

\begin{proposition}
Let $\mathcal{A}$ be an abelian category and $\mathbf{t}=(\mathcal{T},\mathcal{F})$ be a torsion pair in $\mathcal{A}$. The following assertions are equivalent:
\begin{enumerate}
\item[1)] $\mathcal{T}$ is closed under subobjects;
\item[2)] The functor $\xymatrix{\mathcal{A} \ar[r]^{t} & \mathcal{A}}$ is left exact;
\end{enumerate}
If $\mathcal{A}$ is a Grothendieck category, then the previous conditions are equivalent to:
\begin{enumerate} 
\item[3)] $\mathcal{F}$ is closed taking injective envelopes.
\end{enumerate}
\end{proposition}
\begin{proof}
See \cite[Proposition VI.1.7]{S}.
\end{proof}

\begin{definition}\rm{
Let $\mathbf{t}=(\mathcal{T},\mathcal{F})$ be a torsion pair in $\mathcal{A}$. It is called \emph{hereditary} \index{torsion pair! hereditary}if it satisfies one of (and hence all) the conditions of the previous proposition. It is called \emph{split}\index{torsion pair! split} when the canonical sequence $\xymatrix{0 \ar[r] & t(X) \ar[r] & X \ar[r] & (1:t)(X) \ar[r] & 0}$ splits, for each object $X$ of $\mathcal{A}$. If $R$ is a ring and $\mathcal{A}=R$-Mod, we will say that $\mathbf{t}$ is \emph{faithful}\index{torsion pair! faithful} when $R\in \F$.   }
\end{definition}


\vspace{0.3 cm}

Among the hereditary torsion pairs we are interested in a particular case of them, which plays an important role in this work for the case of a module category.

\begin{definition}\rm{
A class $\mathcal{T}\subseteq R\text{-Mod}$ is a \emph{TTF class}\index{TTF} when it is both a torsion and a torsionfree class in $R$-Mod. Each triple of the form $(\mathcal{C,T,F})=(^{\perp}\mathcal{T},\mathcal{T},\mathcal{T}^{\perp})$, for some TTF class $\mathcal{T}$, will be called a \emph{TTF triple} and the two torsion pairs $(\mathcal{C},\mathcal{T})$ and $(\mathcal{T,F})$ will be called the \emph{left constituent pair}\index{TTF! left constituent pair} and \emph{right constituent pair}\index{TTF! right constituent pair} of the TTF triple.}
\end{definition}

\begin{remark}\rm{
\begin{enumerate}
\item[1)] It is well-known (see \cite[Chapter VI]{S}) that if $\mathcal{T}$ is a TTF class in $R$-Mod, then there is an unique idempotent two-sided ideal $\mathbf{a}$ of $R$ such that $\mathcal{T}$ consists of the $R$-modules $T$ such that $\mathbf{a}T=0;$
\item[2)] The torsion radical associated to the left constituent pair $(\mathcal{C,T})$ of the TTF triple assigns to each module $M$ the submodule $c(M)=\mathbf{a}M$. In particular, we have $\mathcal{C}=Gen(\mathbf{a})=\{C \in R\text{-Mod}| \mathbf{a}C=C\}$. 
\end{enumerate}}
\end{remark}

\begin{example}\label{exam. TTF projective module}\rm{
For each projective $R$-module $P$, the class $\T:=\Ker (\Hom_{R}(P,?))$ is a TTF class and $\Gen(P)=^{\perp}\T$. Moreover, the corresponding idempotent ideal is the trace of $P$ in $R$ (see  \cite[Proposition VI.9.4 and Corollary VI.9.5]{S}).}
\end{example}

\begin{definition}\rm{
For each TTF triple $(\mathcal{C},\T,\F)$ in $R$-Mod, we say that the TTF triple is \emph{left split} \index{TTF! left split} (resp. \emph{right split}\index{TTF!right split}) when its left (resp. right) constituent pair is split. Moreover, we say that the TTF triple is \emph{centrally split}\index{TTF! centrally split} when both its left and right constituent pairs are split. In such case, there is a central idempotent $e$ such that $t(M)=eM$ for every left $R$-module $M$ (see \cite[Proposition VI.8.5]{S}).}
\end{definition}

\vspace{0.3 cm}

We now look at a particular class of torsion pairs, introduced by Hoshino, Kato and Miyachi in \cite{HKM}, which will be of interest to us.

\begin{definition}\label{def. HKM torsion pair}\rm{
Let $P^{\bullet}:=\cdots \flecha 0 \flecha Q \xymatrix{\ar[r]^{d} & }P \flecha  0 \flecha  \cdots$ be a complex of finitely generated projective $R$-modules concentrated in degrees -1 and 0. We will define a pair $(\mathcal{X}(P^{\bullet}),\mathcal{Y}(P^{\bullet}))$ of full subcategories of $\Mod R$, where $M$ is an $R$-module, as follows:
\begin{center}
$M\in \mathcal{X}(P^\bullet) \Longleftrightarrow \Hom_{D(R)}(P^\bullet,M[1])=0 \linebreak M\in \mathcal{Y}(P^\bullet) \Longleftrightarrow \Hom_{D(R)}(P^\bullet,M[0])=0$.
\end{center}
We say that $P^{\bullet}$ is an \emph{HKM complex}\index{complex! HKM} when the pair $(\mathcal{X}(P^{\bullet}),\mathcal{Y}(P^{\bullet}))$ is a torsion pair. In such case, $\mathbf{t}=(\mathcal{X}(P^{\bullet}),\mathcal{Y}(P^{\bullet}))$ is called the \emph{HKM torsion pair associated}\index{torsion pair! HKM} to $P^{\bullet}$.
}
\end{definition}

\subsection{Tilting and Cotilting}

\emph{Tilting Theory}\index{Tilting Theory} started with a fundamental paper of Brenner and Butler \cite{BB}. They studied the relations between two module categories when one of the rings is the endomorphism ring of a tilting module over the other ring. In such case, we have an equivalence between the derived categories of the module categories involved (see \cite{Ha}, \cite{Ri} and also \cite{K2}).
In this section, we deal with torsion pairs associated to tilting and cotilting objects and generalizations of them. In the sequel, $\G$ is a Grothendieck category.

\begin{definition}\label{def. gen and cogen}\rm{
Let $X$ and $V$ be objects of $\G.$ We say that $X$ is $V$-generated \index{$V$-generated}(resp. $V$-presented\index{$V$-presented}) when there is an epimorphism $V^{(I)} \epic X$ (resp. an exact sequence \linebreak $V^{(J)} \flecha V^{(I)} \flecha X \flecha 0$, for some sets $I$ and $J$). We will denote by $\Gen(V)$ and $\Pres(V)$ the classes of $V$-generated and $V$-presented objects, respectively. Dually, we can define a $\Cogen(Q)$\index{$Q$-cogenerated} and $\Copres(Q)$\index{$Q$-copresented}, for each $Q$ object of $\G$.}
\end{definition}

\begin{remark}\rm{
For all objects $X$ and $V$ in $\G$, we have that:

\begin{enumerate}
\item[1)] $X$ always contains a largest $V$-generated subobject called the \emph{trace of $V$ in $X$}\index{trace}, given by $tr_{V}(X):=\underset{f\in \Hom_{\G}(V,X)}{\sum}\Imagen(f);$

\item[2)] As a sort of dual concept, given a class $\mathcal{S}$ of objects of $\G$, the \emph{reject of $\mathcal{S}$ in $X$}\index{reject} is $\Rej_{\mathcal{S}}(X):=\underset{f\in \Hom_{\G}(X,S)}{\cap} \Ker(f)$;

\item[3)] We say that $X$ is \emph{$V$-subgenerated}\index{$V$-subgenerated} when it is isomorphic to a subobject of a $V$-generated object. The class of $V$-subgenerated objects will be denoted by $\overline{\Gen}(V)$. This subcategory is itself a Grothendieck category, and the inclusion $\overline{\Gen}(V)\monic \G$ is an exact functor (see the arguments of \cite[Chapter 3]{Wi} although in this book $\overline{\Gen}(V)$ is denoted by $\sigma[V]$); 

\item[4)] We will denote by $\Add(V)$ (resp. $\add(V)$) the class of objects $X$ of $\G$ which are isomorphic to direct summands of coproducts (resp. finite coproducts) of copies of $V$. Dually, we can define a $\Prod(Q)$ and $\text{prod}(Q)$, for each $Q$ object of $\G$.   
\end{enumerate}}
\end{remark}

\begin{definition}\rm{
Slightly diverting from the terminology of \cite{CDT1} and \cite{CDT2}, Let $\mathcal{A}$ be an AB3 abelian category. An object $V$ of $\mathcal{A}$ will be called \emph{quasi-tilting}\index{object! quasi-tilting} when $\Gen(V)=\overline{\Gen}(V)\cap \Ker (\Ext^{1}_{\mathcal{A}}(V,?)).$ When, in addition, we have that $\overline{\Gen}(V)=\mathcal{A},$ we will say that $V$ is a \emph{1-tilting object}\index{object! 1-tilting}. That is, $V$ is 1-tilting if, and only if, $\Gen(V)=\Ker(\Ext^{1}_{\mathcal{A}}(V,?)).$}
\end{definition}

\begin{definition}\rm{
Let $\mathcal{A}$ be an AB3* abelian category. An object $Q$ of $\mathcal{A}$ is called \emph{1-cotilting}\index{object! 1-cotilting} when $\Cogen(Q)=\Ker (\Ext^{1}_{\mathcal{A}}(?,Q)).$} 
\end{definition}

With some additional hypotheses, one obtains the following more familiar characterization of 1-tilting objects. The dual result characterizes 1-cotilting objects in AB4* categories with an injective cogenerator.

\begin{proposition}\label{description tilting object}
Let $\mathcal{A}$ be an AB4 abelian category with a projective generator and let $V$ be any object of $\mathcal{A}$. Consider the following assertions:
\begin{enumerate}
\item[1)] $V$ is a 1-tilting object.
\item[2)] The following conditions hold:
\begin{enumerate}
\item[a)] There exists an exact sequence $0 \flecha P^{-1} \flecha P^{0} \flecha V \flecha 0$ in $\mathcal{A}$, where the $P^{k}$ are projective;
\item[b)] $\Ext^{1}_{\mathcal{A}}(V,V^{(I)})=0$, for all sets $I$;
\item[c)] for some (resp. every) projective generator $P$ of $\mathcal{A}$, there is an exact sequence
$$\xymatrix{0 \ar[r] & P \ar[r] & V^{0} \ar[r] & V^{1} \ar[r] & 0}$$
where $V^{i}\in \Add(V)$, for $i=0,1$.
\end{enumerate}
\end{enumerate}
The implications 2) $\Longrightarrow$ 1), 1) $\Longrightarrow$ 2.b) and 1) $\Longrightarrow$ 2.c) hold. When $\mathcal{A}$ has enough injectives, assertions 1 and 2 are equivalent. 
\end{proposition}
\begin{proof}
It essentially follows as in module categories (see \cite[Proposition 1.3]{CT}, and also \cite[Section 2]{C}), but, for the sake of completeness, we give a proof. \\

2) $\Longrightarrow $ 1) It goes as in modules.\\

1) $\Longrightarrow $ 2) Condition 2.b) is automatic. On the other hand, if $P$ is any projective generator, then we look at $E:=\Ext^{1}_{\mathcal{A}}(V,P)$ as a set. The canonical morphism
$$\Psi:\Ext^{1}_{\mathcal{A}}(V^{(E)},P) \flecha \Ext^{1}_{\mathcal{A}}(V,P)^{E}$$
is an isomorphism due to the AB4 condition of $\mathcal{A}$ and the exactness of products in Ab. Let $\delta \in \Ext^{1}_{\mathcal{A}}(V^{(E)},P)$ be such that $\Psi(\delta)=(\delta_{e})_{e\in E}$, where $\delta_e=e$ for each $e\in E$. In particular, if $\iota_{e}:V\flecha V^{(E)}$ is the injection into the coproduct, we have that $\iota_{e}^{*}(\delta)=e$, for each $e\in E$, where $\iota_{e}^{*}=\Ext^{1}_{\mathcal{A}}(\iota_e,P):\Ext^{1}_{\mathcal{A}}(V^{(E)},P) \flecha \Ext^{1}_{\mathcal{A}}(V,P).$ Fix now an exact sequence in $\mathcal{A}$.
$$\xymatrix{0 \ar[r] & P \ar[r] & W \ar[r] & V^{(E)} \ar[r] &0}$$
representing $\delta$. When applying the long exact sequence of $\Ext(V,?)$, we then get that the connecting morphism $w:\Hom_{\mathcal{A}}(V,V^{(E)}) \flecha \Ext^{1}_{\mathcal{A}}(V,P)$ is surjective since $w(\iota_e)=e,$ for each $e\in E$. We then have a monomorphism $\Ext^{1}_{\mathcal{A}}(V,W) \monic  \Ext^{1}_{\mathcal{A}}(V,V^{(E)})=0$, which shows that $W\in \T:=\Ker(\Ext^{1}_{\mathcal{A}}(V,?))=\Gen(V)$. Note that we get an exact sequence of functors
$$\xymatrix{\cdots \ar[r] & \Ext^{1}_{\mathcal{A}}(V^{(E)},?) \ar[r] & \Ext^{1}_{\mathcal{A}}(W,?) \ar[r] & \Ext^{1}_{\mathcal{A}}(P,?)=0}$$

It follows that $\Ext^{1}_{\mathcal{A}}(W,?)$ vanishes on $\T$. On the other hand, bearing in mind that $\Gen(V)=\Pres(V),$ we can fix an exact sequence

$$\xymatrix{0 \ar[r] & W^{'} \ar[r] & V^{(I)} \ar[r] & W \ar[r] & 0}$$
for some set $I$, where $W^{'}\in \T$. Then this sequence splits, which shows that $W\in \Add(V)$ and that conditions 2.c) holds. \\

Finally, assuming also that $\mathcal{A}$ has enough injectives, to prove that condition 2.a) holds one uses the argument in \cite[Proposition 2.2]{C}.
\end{proof}

\vspace{0.3 cm}

The next lemma is proved for modules in \cite[Section 2]{CDT1}.

\begin{lemma}\label{lem. quasi-tilting}
Let $V$ be an object of $\G$. If $\Gen(V)\subseteq \Ker(\Ext^{1}_{\G}(V,?))$, then $\Gen(V)$ is closed under taking extensions in $\G$, and hence it is a torsion class. Consider now the following assertions, for an object $V$:
\begin{enumerate}
\item[1)] $\Gen(V)=\Pres(V)\subseteq \Ker(\Ext_{\G}^{1}(V,?))$;
\item[2)] $V$ is a quasi tilting object of $\G$;
\item[3)] $V$ is a 1-tilting object of $\overline{\Gen}(V)$. 
\end{enumerate}
Then the implications 1) $\Longleftrightarrow$ 2) $\Longrightarrow $ 3) hold true.
 \end{lemma}
\begin{proof}
Let $0 \flecha T \flecha M \flecha T^{'} \flecha 0$ be an exact sequence in $\G$, with $T,T^{'}\in \Gen(V) \subseteq (\Ker(\Ext^{1}_{\G}(V,?)))$. We want to prove that $M\in\Gen(V)$. By pulling back the exact sequence along an epimorphism $p:V^{(I)} \epic T^{'}$, for some set $I$, we can assume without loss of generality that $T^{'}=V^{(I)}$ since $\Gen(V)$ is closed under taking quotients. But in this case the sequence splits since $\Ext^{1}_{\G}(V^{(I)},T)=\Hom_{\D(G)}(V^{(I)}[0],T[1])\cong \Hom_{\D(G)}((V[0])^{(I)},T[1]) \cong \Hom_{\D(\G)}(V[0],T[1])^{I}=\Ext^{1}_{\G}(V,T)^{I}=0.$ Hence $M\cong T \oplus T^{'}$ which is in $\Gen(V)$.\\ 

1) $\Longrightarrow$ 2) We just need to verify that $\Ker(\Ext^{1}_{\G}(V,?))\cap \overline{\Gen}(V)\subseteq \Gen(V)$. Let $X \in \Ker(\Ext^{1}_{\G}(V,?))\cap \overline{\Gen}(V)$ and we consider an exact sequence in $\G$ of the form $\xymatrix{0 \ar[r] & X \ar[r] & T \ar[r] & T^{'} \ar[r] & 0}$, where $T,T^{'}$ are in $\Gen(V)$. We take an epimorphism $\pi:V^{(I)} \epic T^{'}$, for some set $I$ such that $\Ker(\pi)\in \Gen(V)$ (since $\Gen(V)=\Pres(V)$). By pulling back the exact sequence along the epimorphism $\pi$, we obtain the following commutative diagram with exact rows and columns:

$$\xymatrix{\\ & & \Ker(\pi) \ar@{=}[r] \ar@{^(->}[d] & \Ker(\pi) \ar@{^(->}[d] & \\ 0 \ar[r] & X \ar[r] \ar@{=}[d]& \tilde{T} \ar[r] \ar@{>>}[d] & V^{(I)} \ar[r] \ar@{>>}[d]^(0.41){\pi} & 0 \\ 0 \ar[r] & X \ar[r] & T \ar[r] & T ^{'}\ar[r] & 0}$$ 

From the first paragraph of this proof together with the central column, we obtain that $\tilde{T}\in \Gen(V)$ and the central row splits since $\Ext^{1}_{\G}(V^{(I)},X)=0$. Then $X$ is a direct summand of an object in $\Gen(V)$, so that $X\in \Gen(V)$.\\

2) $\Longrightarrow$ 3) First note that $\Gen_{\overline{\Gen}(V)}(V)=\Gen(V)$. On the other hand, let us take $X\in \overline{\Gen}(V)$ and consider an exact sequence 
$$\xymatrix{0 \ar[r] & X \ar[r] & M \ar[r] & V \ar[r] & 0 & (\ast)}$$
Taking now an monomoprhism $X \monic T$, with $T\in \Gen(V)$, we get the following commutative diagram, where the left square is cocartesian 
$$\xymatrix{0 \ar[r] & X \ar[r] \ar@{^(->}[d] & M \ar[r] \ar[d] & V \ar@{=}[d] \ar[r] & 0 \\ 0 \ar[r] & T \ar[r] & N \pullbackcorner  \ar[r]  & V \ar[r] & 0}$$ 
By hypothesis, the lower row splits and, hence $M\in \overline{\Gen}(V)$. This shows that the sequence $(\ast)$ lives in $\overline{\Gen}(V)$. Therefore, $\Ext^{1}_{\G}(V,X)\cong \Ext^{1}_{\overline{\Gen}(V)}(V,X)$, for all $X\in \overline{\Gen(V)}$. Therefore, $\Gen_{\overline{\Gen}(V)}(V)=\Gen(V)=\Ker(\Ext^{1}_{\G}(V,?))\cap \overline{\Gen}(V)=\Ker(\Ext^{1}_{\overline{\Gen}(V)}(V,?))$. \\

2) $\Longrightarrow$ 1) We have $\Gen(V)=\Gen_{\overline{\Gen}(V)}(V)$ and $\Pres(V)=\Pres_{\overline{\Gen}(V)}(V)$. Since we have already proved 2) $\Longrightarrow$ 3), we know that $V$ is a 1-tilting object of $\overline{\Gen}(V)$. But when $V$ is a tilting object in a Grothendieck category (e.g. $\overline{\Gen}(V)$), the classes of $V$-generated and $V$-presented objects coincide (see \cite[Theorem 2.4]{CF}). The inclusion $\Gen(V)\subseteq \Ker(\Ext^{1}_{\G}(V,?))$ is clear from the definition of quasi-tilting object.

\end{proof}

\begin{remark}\label{quasi-tilting torsion pair}\rm{
For each quasi-tilting object $V$ of $\G$, from the previous lemma we have that $\Gen(V)$ is a torsion class. 
The pair $(\Gen(V),\Ker (\Hom_{\G}(V,?)))$ is called the \emph{quasi-tilting torsion pair} \index{torsion pair! quasi-tilting}associated to $V$. In the particular case when $V$ is 1-tilting, this pair is called the \emph{tilting torsion pair} \index{torsion pair! tilting}associated to $V$. Similarly, if $Q$ is 1-cotilting, then $\Cogen(Q)=\Copres(Q)$ and $(\Ker(\Hom_{\G}(?,Q)), \Cogen(Q))$ is a torsion pair in $\G$, that is called the \emph{cotilting torsion pair}\index{torsion pair! cotilting} associated to $Q$.}
\end{remark}

\begin{definition}\rm{
A \emph{classical quasi-tilting}\index{object! classical quasi-tilting} (resp. \emph{classical 1-tilting}\index{object! classical 1-tilting}) object is a quasi-tilting (resp. 1-tilting) object $V$ that is \emph{self-small}\index{object! self-small}, i.e., when the canonical morphism \linebreak $\Hom_{\G}(V,V)^{(I)} \flecha \Hom_{G}(V,V^{(I)})$ is an isomorphism, for all sets $I$.}
\end{definition}

\begin{remark}\rm{
By \cite[Proposition 2.1]{CDT1}, we know that if $\G=R$-Mod, then a classical quasi-tilting $R$-module is just a finitely generated quasi-tilting module. Even more (see \cite[Proposition 1.3]{CT}), a classical 1-tilting $R$-module is just a finitely presented 1-tilting $R$-module. }
\end{remark}

\section{Triangulated Categories}
The notion of a derived category was introduced by Verdier in \cite{V}, based on the ideas of Grothendieck. Verdier also defined the notion of a triangulated category, based on the observation that a derived category has some special ``triangles", which in some way reflect the behavior of the short exact sequences in the category of chain complexes over it. 


\begin{definition}\rm{
A pair of composable morphisms $(i,p)$
$$\xymatrix{X \ar[r]^{i} & Y \ar[r]^{p} & Z}$$
in an additive category $\mathcal{A}$ is \emph{exact} when $i$ is the kernel of $p$ and $p$ is the cokernel of $i$.}
\end{definition}

\begin{definition}\rm{
Let $\mathcal{A}$ be an additive category. An exact pair $(i,p)$ in $\mathcal{A}$ is called \emph{split} if it is isomorphic to a short exact sequence of the form
$$\xymatrix{X \ar[rr]^{\begin{pmatrix} 1 \\ 0\end{pmatrix} \hspace{0.5cm}} & & X\oplus Z \ar[rr]^{\hspace{0.3 cm}\begin{pmatrix} 0  & 1 \end{pmatrix}} & & Z}$$
This is equivalent to require that $i$ is a section and $p$ is a retraction.}
\end{definition}

In \cite{K1} Keller gives a refinement of the classical definition of exact category given by Quillen in \cite{Q}.

\begin{definition}\rm{
An \emph{exact category}\index{category! exact} is an additive category $\mathcal{A}$ endowed with a set $\mathcal{E}$ of exact pairs closed under isomorphism and satisfying the following axioms \textbf{Ex0)-Ex2)$^{o}$}. The \emph{deflations} \index{deflations}(resp. \emph{inflations}\index{inflations}) mentioned in the axioms are by definition the morphisms $p$ (resp. $i$) occurring in pairs $(i,p)$ of $\mathcal{E}$. We shall refer to such a pair as a \emph{conflation}\index{conflation}.
\begin{enumerate}
\item[\bf{Ex0)}] The identity morphism of the zero object, $1_0$, is a deflation;
\item[\bf{Ex1)}] A composition of two deflations is a deflation;
\item[\bf{Ex1)$^{o}$}] A composition of two inflations is an inflation;
\item[\bf{Ex2)}] Each diagram 
$$\xymatrix{&Z^{'} \ar[d]^{c}\\ Y \ar[r]^{p} & Z}$$
where $p$ is a deflation, may be completed to a cartesian square
$$\xymatrix{Y^{'} \ar[r]^{p^{'}} \ar[d]^{b} & Z^{'} \ar[d]^{c}\\ Y \ar[r]^{p} & Z}$$
where $p^{'}$ is a deflation.
\item[\bf{Ex2)$^{o}$}] Each diagram
$$\xymatrix{X \ar[r]^{i} \ar[d]^{a} & Y \\ X^{'}}$$
where $i$ is an inflation, may be completed to a cartesian square
$$\xymatrix{X \ar[r]^{i} \ar[d]^{a} & Y  \ar[d]^{b}\\ X^{'} \ar[r]^{i^{'}} & Y^{'}}$$
where $i^{'}$ is a deflation.
\end{enumerate}}
\end{definition}

\begin{remark}\rm{
\begin{enumerate}
\item[1)]The terms \emph{inflation, deflation and conflation} were invented by Gabriel. They are a substitution of Quillen original terms \emph{admissible monomorphism, admissible epimorphism} and \emph{admissible short exact sequence}, respectively;
\item[2)] Observe that the definition of exact category is self-dual, i.e., if $(\mathcal{A},\mathcal{E})$ is an exact category, then $(\mathcal{A}^{o},\mathcal{E}^{o})$ so is, with $\mathcal{E}^{o}:=\{ (p^{o},i^{o}) | (i,p) \in \mathcal{E} \}$.
\end{enumerate}}
\end{remark}

\begin{definition}\rm{
Let $\mathcal{A}$ be an additive category. Denote by $\mathcal{C(A)}$ the \emph{category of cochain complexes or simply complexes}\index{category! of cochain complex}
$$\xymatrix{\cdots \ar[r] & X^{n} \ar[rr]^{d^{n}_X} & & X^{n+1} \ar[rr]^{\hspace{0.1cm}d^{n+1}_{X}} & & \cdots }$$
over $\mathcal{A}$. We define an automorphism of $\mathcal{C(A)}$, called the \emph{shifting functor}\index{functor! shifting}, as follows:
$$\xymatrix{?[1]:\mathcal{C(A)} \ar[r] & \mathcal{C(A)}}$$
\begin{enumerate}
\item[]Action on objects: For each object $(X,d_X)$ we have $X[1]^{n}:=X^{n+1}$ and $d^{n}_{X[1]}:=-d^{n+1}_X$, for all $n\in$\Z;
\item[]Action on morphisms: For each morphism $f$ we have $f[1]^n:=f^{n+1}$, for all $n\in$\Z.
\end{enumerate}}
\end{definition}

\begin{definition}\rm{
Let $\xymatrix{f:X \ar[r] & Y}$ be a morphism in $\mathcal{C(A)}$. The \emph{cone}\index{cone} of $f$ is the complex $C:=Cone(f)$ given by:
\begin{enumerate}
\item[-] $C^n:=X^{n+1}\oplus Y^n $, for all $n\in $\Z;
\item[-] The differential $\xymatrix{d^{n}_C:C^{n} \ar[r] & C^{n+1}}$ is defined in the  following way 
$$d^{n}_{C}=:\begin{pmatrix} -d^{n+1}_X & 0 \\ f^{n+1} & d^{n}_Y\end{pmatrix}: \xymatrix{X^{n+1}\oplus Y^{n} \ar[r] & X^{n+2}\oplus Y^{n+1}}, \text{ for all } n\in \text{\Z}. $$
\end{enumerate}}

\end{definition}

\begin{examples}\rm{
\begin{enumerate}
\item[1)] Let us assume that $\mathcal{A}$ is an abelian category (e.g. $\mathcal{A}$=$R$-Mod, for the some ring $R$) and $\mathcal{E}$ denote the class of short exact sequences in $\mathcal{A}$, then $(\mathcal{A},\mathcal{E})$ is an exact category;

\item[2)] Let $\mathcal{A}$ be an additive category. Endowed with all split short exact sequences, $\mathcal{A}$ becomes an exact category;

\item[3)] Let $\mathcal{A}$ be an additive category. Endow $\mathcal{C(\mathcal{A})}$ with the set of all pairs $(i,p)$ such that $(i^{n},p^{n})$ is a split short exact sequence for each $n \in \text{\Z}$. Then $\mathcal{C(\mathcal{A})}$ is an exact category. We will say that $\mathcal{C}(\mathcal{A})$ is \emph{endowed with its semi-split exact structure}. We also have that the \emph{right bounded full subcategory}

$$\mathcal{C}^{-}(\mathcal{A}):=\{ X \in obj(\mathcal{C(\mathcal{A})}) \hspace{0.1 cm} | \hspace{0.1 cm}X^{n}=0, \forall \hspace{0.1 cm} n >> 0\},$$
the \emph{left bounded full subcategory}
$$\mathcal{C}^{+}(\mathcal{A}):=\{ X \in obj(\mathcal{C(\mathcal{A})}) \hspace{0.1 cm} | \hspace{0.1 cm}X^{n}=0, \forall \hspace{0.1 cm} n << 0\},$$
and the \emph{bounded full subcategory}
$$\mathcal{C}^{b}(\mathcal{A}):=\{ X \in obj(\mathcal{C(\mathcal{A})}) \hspace{0.1 cm} | \hspace{0.1 cm}X^{n}=0, \forall \hspace{0.1 cm}  |n| >>0 \},$$
are fully exact subcategories of $\mathcal{C(A)}$. If $R$ is a ring sometimes we will frequently write $\mathcal{C}(R)$ instead of $\mathcal{C}(R\text{-Mod})$, $\mathcal{C}^{-}(R)$ instead of $\mathcal{C}^{-}(R\text{-Mod})$, etc.
\end{enumerate}}
\end{examples}

\begin{definition}\rm{
Let $\mathcal{A}$ be an exact category. An object $I\in \mathcal{A}$ is \emph{injective}\index{injective} (with respect to the set of conflations) if the sequence
$$\xymatrix{\mathcal{A}(Y,I) \ar[r]^{i^*} & \mathcal{A}(X,I) \ar[r] & 0}$$
is exact (in Ab) for each inflation $i:X\flecha Y$. \\

Dually, an object $P\in \mathcal{A}$ is \emph{projective}\index{projective} (with respect to the set of conflations) if it is injective when viewed as an object in $\mathcal{A}^{op}$. We will say that $\mathcal{A}$ has \emph{enough injectives} \index{enough injectives}if for each object $X\in \mathcal{A}$ there exists a conflation 
$$\xymatrix{X \ar[r]^{i_X} & IX \ar[r]^{p_X} & SX}$$
with injective $IX$. Dually, we will say that $\mathcal{A}$ has \emph{enough projectives}\index{enough projectives} if $\mathcal{A}^{\text{op}}$ has enough injectives. }
\end{definition}

\begin{definition}\rm{
Let $\mathcal{A}$ be an exact category with enough injectives and projectives. We will say that $\mathcal{A}$ is a \emph{Frobenius category}\index{category! Frobenius}, if the class of injective objects coincides with the class of projective ones.}
\end{definition}

\begin{examples}\rm{
\begin{enumerate}
\item[1)] The module category $R$-Mod, where $R$ is a quasi-Frobenius (QF) ring (e.g. $R=$\Z$_{p^{n}}$, for some prime $p$, or the group algebra $KG$, where $K$ is a field and $G$ is a finite group) is a Frobenius category;
\item[2)] Let $\mathcal{A}$ be an additive category. A complex of $ \mathcal{C(A)}$ is called \emph{contractible}\index{complex! contractible} if it is isomorphic to a complex
\begin{small}
$$\xymatrix{\cdots \ar[r] & Z^{n-1}\oplus Z^{n} \ar[rr]^{\begin{pmatrix} 0 & 1\\ 0 & 0 \end{pmatrix}}  && Z^{n}\oplus Z^{n+1} \ar[rr]^{\begin{pmatrix} 0 & 1\\ 0 & 0 \end{pmatrix}} & &Z^{n+1}\oplus Z^{n+2} \ar[rr]^{\hspace{0.7 cm}\begin{pmatrix} 0 & 1\\ 0 & 0 \end{pmatrix}} & & \cdots }$$
\end{small}
The category of chain complexes $\mathcal{C(A)}$, with the semi-split exact structure is a Frobenius category, where both the projectives and the injectives are the contractible complexes (see \cite[Section I.3]{H}).
\end{enumerate}}
\end{examples}

\begin{definition}\rm{
Let $\mathcal{A}$ be a Frobenius category. The \emph{stable category} \index{category! stable}$\underline{\mathcal{A}}$ associated to $\mathcal{A}$, is defined in the following way
\begin{enumerate}
\item[1)] $Obj(\underline{\mathcal{A}}):=Obj(\mathcal{A})$; 
\item[2)] $\Hom_{\underline{\mathcal{A}}}(X,Y):=\Hom_{\mathcal{A}}(X,Y)/\mathcal{P}(X,Y)$, where $\mathcal{P}(X,Y)$ is the subgroup of $\Hom_{\mathcal{A}}(X,Y)$ given by the morphisms that factor through a projective (=injective) object of $\mathcal{A}$;
\item[3)] If $\overline{f}\in \Hom_{\underline{\mathcal{A}}}(X,Y)$ and $\overline{g}\in \Hom_{\underline{\mathcal{A}}}(Y,Z)$, then the composition is given by $\overline{g} \circ \overline{f}:= \overline{g \circ f}.$ 
\end{enumerate}}
\end{definition}

\begin{proposition}[{\bf Universal property of the stable category}] 
Let $\mathcal{A}$ be a Frobenius category, let $\mathcal{B}$ be an additive category and let $\xymatrix{q:\mathcal{A} \ar[r] & \mathcal{B}}$ be a functor which preserves the zero object. The following assertions are equivalent:
\begin{enumerate}
\item[1)] $q$ vanishes on the projective objects of $\mathcal{A}$ and, for each additive functor $\xymatrix{F:\mathcal{A} \ar[r] & \mathcal{C}}$ that vanishes on the projective objects, where $\mathcal{C}$ is an additive category, there is an unique functor $\xymatrix{\overline{F}: \mathcal{B} \ar[r] & \mathcal{C}}$ such that $\overline{F} \circ q =F$;
\item[2)] There is an equivalence of categories $\xymatrix{H:\underline{A} \ar[r]^{\cong} & \mathcal{B} }$ such that $H \circ p=q$, where $\xymatrix{p:\mathcal{A} \ar[r] & \underline{\mathcal{A}}}$ is the canonical projection. 
\end{enumerate}
\end{proposition}

\begin{definition}\rm{
Let $\mathcal{A}$ be an additive category and $\xymatrix{f,g:X \ar[r] & Y}$ be morphisms in $\mathcal{C(A)}$. A \emph{homotopy}\index{homotopy} between $f$ and $g$ is a family of morphisms $$s:=(\xymatrix{s^n:X^n \ar[r] & Y^{n-1}})_{n \in \text{\Z}}$$ in $\mathcal{A}$ such that $g-f=d_Y \circ s + s \circ d_X$ (i.e. $g^n-f^n=d^{n-1}_Y \circ s^n +s^{n+1} \circ d^{n}_X$, $\forall n\in $\Z).}
\end{definition}

\begin{remark}\rm{
We will define an equivalence relation in $\mathcal{C(A)}$ given by: two chain maps are \emph{homotopic}\index{homotopic} if there is a homotopy between them. A chain map is called \emph{null-homotopic}\index{null-homotopic} when it is homotophic to the zero map. }
\end{remark}

\begin{definition}\rm{
The \emph{homotopy category}\index{category! homotopy} $\mathcal{H(A)}$ of an additive category $\mathcal{A}$ is the category whose objects are the chain complexes over $\mathcal{A}$ and its morphism are homotopy classes of morphisms in $\C(\mathcal{A})$. In the case $\mathcal{A}=R$-Mod, we write $\mathcal{H}(R)$ instead of $\mathcal{H}(R\text{-Mod})$.}
\end{definition}

\begin{proposition}
$\mathcal{H(A)}$ is the stable category of the category $\mathcal{C(A)}$.
\end{proposition}
\begin{proof}
We only need to prove that each chain map is null-homotopic if, and only if, it factors through contractible complex. We suppose that there exist applications $\xymatrix{X \ar[r]^{g} & P \ar[r]^{h} & Y}$, with $P$ contractible. We show that $f:=h \circ g$ is null-homotopic. Without loss of generality we can assume that $P$ is of the form, $P^{n}=Z^{n}\oplus Z^{n+1}$ in $\mathcal{A}$, for all $n\in \mathbb{Z}$, and $d^{n}_{P}:P^{n}\flecha P^{n+1}$ is given by:
$$d^{n}_{P}=\begin{pmatrix} 0 & 1_{Z^{n+1}} \\ 0 & 0 \end{pmatrix}: Z^{n} \oplus Z^{n+1} \flecha Z^{n+1} \oplus Z^{n+2}.$$

Putting $g^{n}=\begin{pmatrix}u^{n} \\ v^{n}\end{pmatrix}:X^{n} \flecha Z^{n}\oplus Z^{n+1}$ and $h^{n}=\begin{pmatrix}w^{n} & t^{n}\end{pmatrix}:Z^{n} \oplus Z^{n+1} \flecha Y^{n}$, for all $n\in \mathbb{Z}$, we obtain that $u^{n+1} \circ d^{n}_{X}=v^{n}$ and $w^{n+1}=d^{n}_{Y} \circ t^{n}$, for all $n\in \mathbb{Z}$. It follows that $s:=(t^{n-1} \circ u^{n}: X^{n} \flecha Y^{n-1})_{n\in \mathbb{Z}}$ is a homotopy between 0 and $f$.\\

Conversely, let $f:X \flecha Y$ be a cochain map which is null-homotopic, and let us fix a homotopy $s$ between $0$ and $f$. We define $P^{n}:=Y^{n-1} \oplus Y^{n}$ and $d^{n}_{P}=\begin{pmatrix} 0 & 1 \\ 0 & 0\end{pmatrix}:P^{n}=Y^{n-1}\oplus Y^{n} \flecha Y^{n} \oplus Y^{n+1}=P^{n+1}$, for all $n\in \mathbb{Z}$. Is clear that, $P$ is a contractible complex, furthermore, we get a deflation $h:P \epic Y$, given by $h^{n}:=\begin{pmatrix} d^{n-1}_Y & 1_{Y^{n}}\end{pmatrix}:P^{n}=Y^{n-1}\oplus Y^{n} \flecha Y^{n}$, for all $n\in \mathbb{Z}$. Finally, we define a chain map $g:X \flecha P$, given by
$$g^{n}:=\begin{pmatrix}s^{n} \\ s^{n+1}d^{n}_X\end{pmatrix}:X^{n} \flecha Y^{n-1} \oplus Y^{n}=P^{n}$$
It is easy to see that $f=h \circ g$.
\end{proof}

\begin{definition}\rm{
If $\mathcal{C}$ is an arbitrary category endowed with an endofunctor $\xymatrix{S:\mathcal{C} \ar[r] & \mathcal{C}}$, an $S$-\emph{sequence}\index{$S$-sequence} is a sequence $(u,v,w)$
$$\xymatrix{X \ar[r]^{u} & Y \ar[r]^{v} & Z \ar[r]^{w \hspace{0.2 cm}} & SX}$$
of morphisms of $\mathcal{C}$, and a \emph{morphism of $S$-sequences} from $(u,v,w)$ to $(u^{'},v^{'},w^{'})$ is a commutative diagram of the form
$$\xymatrix{X \ar[r]^{u} \ar[d]^{x} & Y \ar[r]^{v} \ar[d] & Z \ar[r]^{w \hspace{0.2 cm}} \ar[d]& SX \ar[d]^{Sx} \\ X^{'} \ar[r]^{u^{'}} & Y^{'} \ar[r]^{v^{'}} & Z^{'} \ar[r]^{w^{'} \hspace{0.2 cm}} & SX^{'}}$$
In this way we define \emph{the category of the $S$-sequences} of $\mathcal{C}.$ }
\end{definition}

\begin{remark}\rm{
Let $\mathcal{A}$ be a Frobenius category and $A\in Obj(\mathcal{A})$. For each choice of a inflation $\xymatrix{i:A \hspace{0.03 cm}\ar@{^(->}[r] & I}$, with $I$ injective (=projective) object, the $coker(i)$ is denoted by $\Omega^{-1}(A)$ is called \emph{the cosyzygy} \index{cosyzygy}of $A$. It is uniquely determine up to isomorphism in the stable category $\underline{\mathcal{A}}$ and does not depend of the choice of $i$. Moreover, the assignment $A \rightsquigarrow \Omega^{-1}(A)$ define a self-equivalence of categories $\Omega^{-1}: \underline{\mathcal{A}} \iso \underline{\mathcal{A}}$, whose inverse functor is called \emph{the syzygy functor}\index{syzygy}.}
\end{remark}

\begin{example}\label{triangles in stable Frobenius category}\rm{
We associate now with each conflation $e=(i,p)$ of the Frobenius category $\mathcal{A}$ an $S$-sequence in $\underline{\mathcal{A}}$
$$\xymatrix{X \ar[r]^{\overline{i}} & Y \ar[r]^{\overline{p}} & Z \ar[r]^{\delta(e)\hspace{0.2 cm}} & SX}$$
called \emph{standard triangle},\index{standard triangle} defined by requiring the existence of a commutative diagram in $\mathcal{A}$:
$$\xymatrix{X \ar[r]^{i} \ar@{=}[d] & Y \ar[r]^{p} \ar[d]^{g} & Z \ar[d]^{e_{1}} \\ X \ar[r]^{i_X} & IX \ar[r]^{p_X} & SX }$$
and taking $\delta(e)=\overline{e_{1}}$ (again $g$ exists by the injective condition of $IX$, $e_1$ exists by the universal property of cokernels and is uniquely determined by $g$, and $\overline{e_{1}}$ does not depend on the choice of $g$).}
\end{example}

\begin{definition}\rm{
A \emph{triangulated category}\index{category! triangulated} is a pair $(\mathcal{D},?[1])$ where $\mathcal{D}$ is an additive category and $?[1]$ is an additive self-equivalence of $\mathcal{D}$, together with a class of $?[1]$-sequences closed under isomorphism, and whose elements are called \emph{distinguished triangles}\index{distinguished triangles} or simply \emph{triangles}\index{triangles}, satisfying the following axioms:
\begin{enumerate}
\item[\bf{TR1})] For every object $X$ the $?[1]$-sequence 
$$\xymatrix{X \ar[r]^{1_X} & X \ar[r] & 0 \ar[r] & X[1]}$$
is a triangle. Each morphism $\xymatrix{u:X \ar[r] & Y}$ can be embedded into a triangle $(u,v,w)$;
\item[\bf{TR2})] $(u,v,w)$ is a triangle if and only if $(v,w,-u[1])$ is a triangle;

\item[\bf{TR3})] If $(u,v,w)$ and $(u^{'},v^{'},w^{'})$ are triangles and $x,y$ morphism such that $y\circ u= u^{'} \circ x $, then there exists a morphism $z$ such that $zv=v^{'}y$ and $x[1] \circ w= w^{'} \circ z$
$$\xymatrix{X \ar[r]^{u} \ar[d]^{x} & Y \ar[r]^{v}  \ar[d]^{y} & Z \ar[r]^{w} \ar@{-->}[d]^{z} & X[1] \ar[d]^{x[1]} \\ X^{'} \ar[r]^{u^{'}} & Y^{'} \ar[r]^{v^{'}} & Z^{'} \ar[r]^{w^{'}} & X^{'}[1]}$$ 

\item[\bf{TR4})] (Octahedral axiom)\index{Octahedral axiom} For each pair of composable morphisms
$$\xymatrix{X \ar[r]^{u} & Y \ar[r]^{v} & Z}$$
and for an arbitrary choice of triangles
$$\xymatrix{X \ar[r]^{u}  & Y \ar[r]^{x}  & Z^{'} \ar[r]  & X[1]  \\Y \ar[r]^{v}  & Z \ar[r]  & X^{'} \ar[r]^{r}  & Y[1]   \\ X \ar[r]^{u v}  & Z \ar[r]^{y}  & Y^{'} \ar[r]^{s}  & X[1]  }$$  
starting in $u,v$ and $uv$ respectively, there is a commutative diagram
$$\xymatrix{X \ar[r]^{u} \ar@{=}[d]  & Y \ar[r]^{x} \ar[d]^{v}  & Z^{'} \ar[r] \ar[d]^{w}  & X[1] \ar@{=}[d]  \\ X \ar[r]^{uv}  & Z \ar[r]^{y} \ar[d]  & Y^{'} \ar[r]^{s} \ar[d] & X[1] \ar[d]^{u[1]} \\  & X^{'} \ar@{=}[r] \ar[d]^{r}  & X^{'} \ar[r]^{r} \ar[d]  & Y[1]  \\   & Y[1] \ar[r]^{x[1]}  & Z^{'}[1]   &   }$$
where the first two rows and the two central columns are triangles.
\end{enumerate}
The self-equivalence $?[1]$ is called the \emph{shifting functor}\index{functor! shifting} or \emph{suspension functor}\index{functor! suspension}. We will then put $?[0] = 1_{\mathcal{D}}$ and $?[k]$ will denote the $k-th$ power of $?[1]$, for each integer $k$.}
\end{definition}

\begin{remark}\rm{
The notion of triangulated category is self-dual.}
\end{remark}

\begin{proposition}
Let $\mathcal{A}$ be a Frobenius category. Then $(\mathcal{A},\Omega^{-1})$ is a triangulated category with the class of distinguished triangles described in the example \ref{triangles in stable Frobenius category}.
\end{proposition}
\begin{proof}
See \cite[Theorem I.2.6]{Ha}.
\end{proof}

\begin{corollary}
For each additive category $\mathcal{A}$, the homotopy category $\mathcal{H(A)}$ admits a structure of triangulated category where the triangles are isomorphic to those described by each of the following two equivalent conditions:
\begin{enumerate}
\item[1)] The image of conflations in $\mathcal{C(A)}$ by the projection $\xymatrix{p:\mathcal{C(A)} \ar[r] & \mathcal{H(A)}}$, as described in the example \ref{triangles in stable Frobenius category};
\item[2)] Sequences of the form
$$\xymatrix{X \ar[r]^{\overline{f}} & Y \ar[r]^{\begin{pmatrix} 0 \\ 1\end{pmatrix} \hspace{0.6 cm}} & Cone(f) \ar[r]^{\hspace{0.3 cm}\begin{pmatrix} 1 & 0\end{pmatrix}} & X[1]}$$
where $f$ is an arbitrary chain map.
\end{enumerate}
\end{corollary}

\begin{definition}\rm{
Let $(\mathcal{D},?[1])$ be a triangulated category and let $\mathcal{A}$ be an abelian category. A functor $\xymatrix{H:\mathcal{D} \ar[r] & \mathcal{A}}$ is called \emph{cohomological}\index{functor! cohomological} if for every triangle
$$\xymatrix{X \ar[r]^{u} & Y \ar[r]^{v} & Z \ar[r]^{w} & X[1]}$$
the sequence
$$\xymatrix{H(X) \ar[r]^{H(u)} & H(Y) \ar[r]^{H(v)} & H(Z) }$$
is exact in the abelian category $\mathcal{A}$.}
\end{definition}

\begin{remark}\rm{
The cohomological functors are additive (see \cite[Remarque II.1.2.7]{V}).}
\end{remark}

\begin{example}\rm{
Let $(\mathcal{D},?[1])$ be a triangulated category and $X$ be an object of $\mathcal{D}$. Then the functor $$\xymatrix{Hom_{\mathcal{D}}(X,?): \mathcal{D} \ar[r] &\text{Ab}}$$
is cohomological. In particular, if 
$$\xymatrix{Y \ar[r] & Z \ar[r] & W \ar[r] & Y[1]}$$
is a triangle in $\mathcal{D}$, then we get the following long exact sequence in Ab, with $n\in $ \Z
$$\xymatrix{\cdots \ar[r] & Hom_{\mathcal{D}}(X,Y[n]) \ar[r] & Hom_{\mathcal{D}}(X,Z[n]) \ar[r] & Hom_{\mathcal{D}}(X,W[n]) \ar@(dr,u)[dll] \\ & Hom_{\mathcal{D}}(X,Y[n+1]) \ar[r] & Hom_{\mathcal{D}}(X,Z[n+1]) \ar[r] & \cdots}$$
Dually, the functor $\xymatrix{Hom_{\mathcal{D}}(?,X): \mathcal{D}^{\text{op}} \ar[r] & \text{Ab}}$, is cohomological.}
\end{example}

\begin{example}\rm{
Let $\mathcal{A}$ be an abelian category, the functor ``0-th'' object of homology \index{homology}
$\xymatrix{H^{0}:\mathcal{C(A)} \ar[r] & \mathcal{A}}$ vanishes on contractible complex, therefore there is a unique functor $\xymatrix{H^{0}:\mathcal{H(A)} \ar[r] & \mathcal{A}}$ which we shall still denote by $H^{0}$, such that $H^{0} \circ p = H^{0}$, where $\xymatrix{p: \mathcal{C(A)} \ar[r] & \mathcal{H(A)}}$ is the canonical functor. The functor $\xymatrix{H^{0}:\mathcal{H(A)} \ar[r] & \mathcal{A}}$ is cohomological and the equality $H^{0} \circ (?[n])=H^{n}$ holds for all $n\in $ \Z. Thus, for each triangle
$$\xymatrix{X \ar[r] & Y \ar[r] & Z \ar[r] & X[1] & (\ast)}$$
in $\mathcal{H(A)}$, we get the following long exact sequence in $\mathcal{A}$, called \emph{the long exact sequence of homology associated to the triangle}\index{exact sequence! long} ($\ast$)
$$\xymatrix{\cdots \ar[r] & H^{n}(X) \ar[r] & H^{n}(Y) \ar[r] & H^{n}(Z) \ar[r] & H^{n+1}(X) \ar[r] & H^{n+1}(Y) \ar[r] & \cdots}$$}
\end{example}

\begin{definition}\rm{
Let $(\mathcal{D},?[1])$ and $(\mathcal{D}^{'},?^{'}[1])$ be two triangulated categories. A \emph{triangulated functor}\index{functor! triangulated} from $\mathcal{D}$ to $\mathcal{D}^{'}$ is a pair $(F,\eta),$ where $\xymatrix{F:\mathcal{D} \ar[r] & \mathcal{D}^{'}}$ is a functor and $\xymatrix{\eta:F\circ ?[1] \ar[r] & ?^{'}[1] \circ F}$ is a natural transformation such that for each triangle in $\mathcal{D}$
$$\xymatrix{X \ar[r]^{u} & Y \ar[r]^{v} & Z \ar[r]^{w} & X[1]}$$
we have a triangle in $\mathcal{D}^{'}$
$$\xymatrix{F(X) \ar[r]^{F(u)} & F(Y) \ar[r]^{F(v)} & F(Z) \ar[rr]^{(\eta_X)(F(w)) \hspace{1.8 cm}} & & (?^{'}[1] \circ F)(X) = F(X)[1]}$$}
\end{definition}


\begin{definition}\rm{
Let $(\mathcal{D},?[1])$ be a triangulated category such that $\mathcal{D}$ has coproducts. We will say that it is \emph{compactly generated}\index{category!triangulated!compactly generated} when there is a set $\mathcal{S}$ of compact objects in $\mathcal{D}$ such that an object $X$ of $\mathcal{D}$ is zero whenever $Hom_{\mathcal{D}}(S[k], X)=0$, for all $S\in \mathcal{S}$ and $k\in$ \Z. In that case $\mathcal{S}$ is called a \emph{set of compact generators of} \index{set! of compact generators}$\mathcal{D}$.   }
\end{definition}

\begin{definition}\rm{
Let $(\mathcal{D},?[1])$ be a triangulated category. Given a sequence 
$$\xymatrix{X_0 \ar[r]^{f_1} & X_{1} \ar[r]^{f_2} &  \cdots \ar[r]^{f_n}& X_n \ar[r]^{f_{n+1}} & \cdots }$$ of morphism in $\mathcal{D}$, we will call \emph{Milnor colimit}\index{Milnor colimit} of the sequence, denoted $Mcolim X_n$, to the object given, up to non-canonical isomorphism, by the triangle
$$\xymatrix{\overset{\infty}{\underset{i=0}{\coprod}}X_i \ar[rr]^{1-shift} & & \overset{\infty}{\underset{i=0}{\coprod}}X_i \ar[rr] &&McolimX_n \ar[rr]&&(\overset{\infty}{\underset{i=0}{\coprod}}X_i)[1]  }$$
where the shift map $\xymatrix{\overset{\infty}{\underset{i=0}{\coprod}}X_i \ar[r]^{shift}&\overset{\infty}{\underset{i=0}{\coprod}}X_i }$ is represented by the infinite matrix of morphism
$$\begin{pmatrix} 0 & 0 & 0 & 0 & \dots\\ f_1 & 0 & 0 & 0 & \dots \\ 0 & f_2 & 0 & 0 & \dots \\ 0 & 0 & f_3 & 0 & \dots & \\ \vdots & \vdots & \vdots & \vdots & 
\end{pmatrix}$$}
\end{definition}

\section{Localization of categories}
In this section, we introduce of notion the \emph{localization of categories} that, roughly speaking, consists of inverting some fixed class of morphisms in a category. We do this, in order to introduce the derived category of an abelian category together with the quotient category.


\begin{definition}\rm{
Let $\mathcal{C}$ be a category, and $\Sigma$ a class of morphism in $\mathcal{C}$. The \emph{localization of $\mathcal{C}$ with respect to $\Sigma$} is a pair $(\mathcal{D},q)$, where $\mathcal{D}$ is a category and $\xymatrix{q:\mathcal{C} \ar[r] &  \mathcal{D}}$ is a functor such that the following assertions hold.

\begin{enumerate}
\item[1)] $q$ takes morphisms in $\Sigma$ onto isomorphisms;

\item[2)] if $\xymatrix{F:\mathcal{C} \ar[r] & \mathcal{C^{'}}}$ is a functor taking the morphisms in $\Sigma$ onto isomorphisms, then there exists a unique $\xymatrix{\overline{F}:\mathcal{D} \ar[r] & \mathcal{C}^{'}}$ making commutative the following diagram 
$$\xymatrix{\mathcal{C} \ar[r]^{q} \ar[d]_{F \vspace{1.2 cm}}& \mathcal{D} \ar@{-->}[dl]^{\overline{F}}\\ \mathcal{C}^{'}}$$  
\end{enumerate} 
the pair $(\mathcal{D},q)$ is obviously unique up to equivalence. The functor $q$ is called the \emph{localization functor}\index{functor! localization}, and $\mathcal{D}$ is the \emph{localized category}\index{category! localized}.}
\end{definition}

\begin{remark}\rm{
If such category $\mathcal{D}$ exists it is called the localization of $\mathcal{C}$ with respect to $\Sigma$ and it is usually denoted by $\mathcal{C}[\Sigma^{-1}]$. Gabriel and Zisman (see \cite{GZ}) proved that, if $\mathcal{D}$ exists, such a category only has a possible way of definition up to equivalence:

\begin{enumerate}
\item[-] $Ob(\mathcal{C}[\Sigma^{-1}])=Ob(\mathcal{C})$.

\item[-] If $X,Y\in Ob(\mathcal{C}[\Sigma^{-1}])$ then the morphisms $X \flecha Y$ in $\mathcal{C}[\Sigma^{-1}]$ are finite sequences

$$\xymatrix{X:=X_0 \ar[r] & X_{1} \ar[r] \ar[l] & \cdots \ar[r] \ar[l] & X_{n-1} \ar[r] \ar[l] & X_n=Y \ar[l] } $$ 

where each arrow $\xymatrix{X_{i-1} \ar[r] & X_{i} \ar[l]}$ is either a morphism $X_{i-1} \flecha X_i$ in $\mathcal{C}$ or a morphism $\xymatrix{X_{i-1} & X_{i} \ar[l]_{\hspace{0.4 cm}s} }$ in the class $\Sigma$.

\item[-] The composition of morphisms is the concatenation of sequences.
\end{enumerate}
}
\end{remark}

\vspace{0.3 cm}

In order to verify that the category $\mathcal{C}[\Sigma^{-1}]$ is well-defined (i.e. it is actually a category) it is necessary to prove that $\Hom_{\mathcal{C}[\Sigma^{-1}]}(X,Y)$ is a set, for any $X,Y\in Ob(\mathcal{C}[\Sigma^{-1}])$. But this is a very hard task because of the definition of the morphisms above. So, it is natural to impose specific conditions on the class $\Sigma$ in order to obtain more tractable morphisms in $\mathcal{C}[\Sigma^{-1}]$.

\begin{definition}\rm{
Let $\mathcal{C}$ be a category and $\Sigma$ is a class of morphisms in $\mathcal{C}$. We will say that $\Sigma$ \emph{admits a calculus of right fractions}\index{right fractions} if the following assertions holds:
\begin{enumerate}
\item[{\bf F1)}] For each $X \in \mathcal{C}$, $1_X\in \Sigma$ and $\Sigma$ is closed under compositions;

\item[{\bf F2)}] Each diagram $\xymatrix{X^{'} \ar[r]^{f^{'}} & Y^{'} & \ar[l]_{\hspace{0.3cm}s^{'}} Y}$ with $s^{'}\in \Sigma$ can be completed to a commutative square
$$\xymatrix{X \ar[r]^{f} \ar[d]^{s} & Y \ar[d]^{s^{'}} \\ X^{'} \ar[r]^{f^{'}} & Y^{'}}$$
with $s\in \Sigma$;

\item[{\bf F3)}] If $f$ and $g$ are morphisms and there exists $t\in \Sigma$ such that $t \circ f = t \circ g$, then there exists $s\in \Sigma$ such that $f \circ s = g \circ s$. 
\end{enumerate}

Dually we define the notion of \emph{admitting a calculus of left fractions}\index{left fractions}, by imposing conditions {\bf F1), F2)$^{op}$,} {\bf F3)$^{op}$}. A \emph{multiplicative system}\index{multiplicative system} is a class of morphisms such that it admits a calculus of left and right fractions, i.e., it satisfies axioms {\bf S1):=F1), S2):= F2)+F2)$^{op}$, S3:=F3)+F3)$^{op}$}.}
\end{definition}

\vspace{0.3 cm}

In the sequel $\mathcal{C}$ will denote a category and $\Sigma$ a class of morphisms of $\mathcal{C}$ admitting a calculus of right fractions.

\begin{definition}\rm{
A \emph{right roof}\index{right roof} of $\mathcal{C}$ with respect to $\Sigma$, starting in $X$ and ending in $Y$, is a pair $(s,f)$ of the form 
$$\xymatrix{&Y^{'} \ar[dl]_{s} \ar[dr]^{f}&\\ X & & Y}$$
where $s\in \Sigma$ and $f$ is an arbitrary morphism of $\mathcal{C}$. The class of all the right roofs starting in $X$ and ending in $Y$ is denoted by $\alpha(X,Y)$. Let $R(X,Y)$ be a binary relation on $\alpha(X,Y)$ defined in the following way: $((s,f),(t,g)) \in R(X,Y)$ if and only if there exists a commutative diagram of the form
$$\xymatrix{&Y^{'} \ar[dl]_{s} \ar[dr]^{f}&\\ X &Y^{'''} \ar[l] \ar[r] \ar[u]_{a} \ar[d]^(0.45){b} & Y \\&Y^{''} \ar[ul]^{t} \ar[ur]_{g}&}$$
with $as\in \Sigma$. Dually, we define a \emph{left roof.}\index{left roof}}
\end{definition}

\begin{lemma}
For each pair $(X,Y)$ of objects of $\mathcal{C}$, $R(X,Y)$ is an equivalence relation on $\alpha(X,Y)$.
\end{lemma}
\begin{proof}
See \cite[Section I.2]{GZ}.
\end{proof}

\begin{proposition}
Let $\mathcal{C}$ be a category, and $\Sigma$ be a multiplicative system of morphisms in $\mathcal{C}$. The category $\mathcal{C}[\Sigma^{-1}]$ is equivalent to a category defined as follows:
\begin{enumerate}
\item[1)] Its objects are those of $\mathcal{C}$;

\item[2)] The morphisms $\xymatrix{X \ar[r] & Y}$ are the elements of the quotient set 
$\alpha(X,Y)/R(X,Y)$. The equivalence class of a right roof $(s,f)$ is denoted by $fs^{-1}$;

\item[3)] Let $fs^{-1}\in \alpha(X,Y)/R(X,Y)$ and $gt^{-1}\in \alpha(Y,Z)/R(Y,Z)$ two morphisms. Their composition is given by $(g \circ f^{'})(s \circ t^{'})^{-1}\in \alpha(X,Z)/R(X,Z)$ where $f^{'}$ and $t^{'}$ fit into the following commutative diagram which exists by {\bf F2)}
$$\xymatrix{&&Z^{''} \ar@{-->}[dl]_{t^{'}} \ar@{-->}[dr]^{f^{'}} & & \\ & Y^{'} \ar[dl]_{s} \ar[dr]^{f} & & Z^{'} \ar[dl]_{t} \ar[dr]^{g} \\ X && Y && Z}$$
\end{enumerate}
The composition does not depend on the roofs representing the fractions. 
\end{proposition}

\begin{remark}\rm{
Gabriel and Zisman proved that if $\Sigma$ is a multiplicative system in $\mathcal{C}$, then the category constructed as in the previous proposition but taking left roofs, is also equivalent to $\mathcal{C}[\Sigma^{-1}].$}
\end{remark}

\vspace{0.3 cm}

The classic example of the localization with respect to a class of morphisms is the following.

\begin{example}\rm{
Let $R$ be a commutative ring. We can see $R$ as a preadditive category with only one object. If $\Sigma$ is a multiplicative subset of $R$, then $\Sigma$ is a multiplicative system and the category $R[\Sigma^{-1}]$ is the classical ring of fractions of $R$ with respect to $\Sigma$.}
\end{example}

\subsection{Quotients of Grothendieck categories}\label{sec. quotients category}
In this subsection, we introduce the definition of the quotient category of a category of Grothendieck and recall some of its properties, where possibly the most important result is the Gabriel-Popescu's theorem.

\begin{proposition}
Let $(\T,\F)$ be a hereditary torsion pair in the Grothendieck category $\G$. The class of morphisms $s$ in $\G$ such that $\Ker(s),\Coker(s) \in \T$ form a multiplicative systems $\Sigma_\T$ in $\G$ such that $\G[\Sigma_{\T}^{-1}]$ is a Grothendieck category and the localization functor $q:\G \flecha \G[\Sigma^{-1}_{\T}]$ is exact and has a fully faithful right adjoint $j: \G[\Sigma^{-1}_{\T}]\flecha \G$.
\end{proposition}
\begin{proof}
See \cite{G}.
\end{proof}

\begin{remark}\rm{
 In the situation above, the following assertions hold.
\begin{enumerate}
\item[1)] the category $\G[\Sigma^{-1}_{\T}]$ is denoted $\frac{\G}{\T}$ and is called the \emph{quotient category} of $\G$ by $\T$.

\item[2)] the functor $j:\frac{\G}{\T} \flecha \G$ is usually called the \emph{section functor} and its essential image consists of the objects $Y$ such that $\Hom_{\G}(T,Y)=0=\Ext^{1}_{\G}(T,Y)$, for all $T\in \T$. We will denote by $\G_\T$ this full subcategory and will call it the \emph{Giraud subcategory}\index{subcategory! Giraud} associated to the hereditary torsion pair $(\T,\F)$.

\item[3)] Gabriel-Popescu's theorem (\cite[Theorem X.4.1]{S}) says that each Grothendieck category is equivalent to $\frac{R\Mode}{\T}$, for some ring $R$ and some hereditary torsion class $\T$ in $R$-Mod.

\item[4)] The injective objects of $\G_\T$ are precisely the torsionfree injective objects. Moreover, a sequence $\xymatrix{Y^{'} \ar[r]^{f} & Y \ar[r]^{g} & Y^{''}}$ in $\G_\T$ is exact in this category if, and only if, the object $\Ker(g)/\Imagen(f)$ is in $\T$.  

\item[5)] The counit of the adjoint pair $(q,j)$ is then an isomorphism and we will denote by $\mu:1_{\G} \flecha j \circ q$ its unit. Note that $\Ker(\mu_M)$ and $\Coker(\mu_M)$ are in $\T$, for each object $M$.
\end{enumerate}}
\end{remark}

\vspace{0.3 cm}

If $\T$ is a hereditary torsion class in $R$-Mod, then the set of left ideals $\mathbf{a}$ of $R$ such that $R/\mathbf{a}$ is in $\mathcal{T}$ form a \emph{Gabriel topology} \index{Gabriel topology} $\mathbf{F}=\mathbf{F}_\mathcal{T}$ (see \cite[Chapter VI]{S} for the definition), which is then a downward directed set. The $R$-module $R_\mathbf{F}:=
\varinjlim_{\mathbf{a}\in\mathbf{F}}\Hom_R(\mathbf{a},\frac{R}{t(R)})$ is then an object of $\mathcal{G}_\mathcal{T}$ and, moreover, admits a unique ring structure such that the composition of canonical maps

$$R\flecha \frac{R}{t(R)} \iso \Hom_R(R,\frac{R}{t(R)}) \flecha  R_\mathbf{F}$$
is a ring homomorphism. We call $R_\mathbf{F}$ the \emph{rings of quotients of $R$ with respect to the hereditary torsion pair $(\mathcal{T},\mathcal{T}^\perp )$}.  Viewed as a morphism of $R$-modules this last ring  homomorphism gets identified with the unit map $\mu_R:R \flecha (j\circ q)(R)$ of the adjunction $(R\Mode \xymatrix{\ar[r]^{q} & }\frac{R\Mode}{\T}, \frac{R\Mode}{\T} \xymatrix{ \ar[r]^{j} & } R\Mode)$. We will simply put $\mu =\mu_R$ and it will become clear from the context when we refer to the ring homomorphism or the unit natural transformation. 
It turns out that each module in $\mathcal{G}_\mathcal{T}$ has a canonical structure of $R_\mathbf{F}$-module. Furthermore, the section functor factors in the form

$$\xymatrix{j:\frac{R\Mode}{\mathcal{ T}} \ar[r]^{\hspace{0.01 cm}j'} & R_\mathbf{F}\Mode \ar[r]^{\hspace{0.15 cm}\mu_*} & R\Mode} $$
where $j'$ is fully faithful and $\mu_*$ is the restriction of scalars functor. Moreover, the functor $j'$ has an exact left adjoint $q':R_\mathbf{F}\Mode \flecha \frac{R\Mode}{\mathcal{T}}$ whose kernel is $\mu_*^{-1}(\mathcal{T})$, i.e., the class of $R_\mathbf{F}$-modules $\tilde{T}$ such that $\mu_*(\tilde{T})\in\mathcal{T}$. This latter class is clearly a hereditary torsion class in $R_\mathbf{F}\Mode$, and then $\mathcal{G}_\mathcal{T}$ may
be also viewed as the  Giraud subcategory of $R_\mathbf{F}\Mode$ associated to $(\mu_*^{-1}(\mathcal{T}),\mu_*^{-1}(\mathcal{T})^\perp )$. Note that then $\mu_*$ induces an equivalence of categories $\frac{R_\mathbf{F}\Mode}{\mu_*^{-1}(\mathcal{T})} \iso \frac{R\Mode}{\mathcal{T}}$. Although the terminology is slightly changed, we refer the reader to \cite[Chapters VI, IX, X]{S} for details on the concepts and their properties introduced and mentioned in this subsection. \\





\subsection{Localization of triangulated categories}
In this subsection we will see that, with the appropriate hypotheses on the class of morphisms, the localization of a triangulated category produces another triangulated category.

\begin{definition}\rm{
Let $(\mathcal{D},?[1])$ be a triangulated category and $\Sigma$ a multiplicative system in $\mathcal{D}$, we say that $\Sigma$ is \emph{compatible with the triangulation} if it satisfies the axioms:
\begin{enumerate}
\item[{\bf S4)}] $\Sigma$ is closed under taking powers of the automorphism $?[1]$, i.e, if $s\in \Sigma$, then $s[n]\in \Sigma$ for all $n\in $ \Z;
\item [{\bf S5)}] For every pair of triangles $(u,v,w),(u^{'},v^{'},w^{'})$ and every commutative square 
$$\xymatrix{X \ar[r]^{u} \ar[d]^(0.45){s} & Y \ar[d]^(0.45){s^{'}} \\ X^{'} \ar[r]^{u^{'}} & Y^{'}}$$ 
with $s,s^{'}\in \Sigma$, there exists $s^{''}\in \Sigma$ such that $(s,s^{'},s^{''})$ is a morphism of triangles.
\end{enumerate}}
\end{definition}

\begin{definition}\rm{
Let $(\mathcal{D},?[1])$ be a triangulated category and $\mathcal{N}$ be a subcategory of $\mathcal{D}$. We say that $\mathcal{N}$ is a \emph{full triangulated subcategory}\index{subcategory! triangulated} of $\mathcal{D}$ if the following conditions are satisfied:
\begin{enumerate}
\item[1)] $\mathcal{N}$ is a full subcategory of $\mathcal{D}$ containing the zero object;
\item[2)] $\mathcal{N}$ is closed under $?[1]$ and $?[-1]$;
\item[3)] $\mathcal{N}$ is closed under extensions, i.e., if the end terms of a triangle of $\mathcal{D}$ belong to $\mathcal{N}$, then so does the middle term. 
\end{enumerate}} 
\end{definition}

\begin{example}\rm{
If $(\mathcal{D},?[1])$ is a triangulated category and $\mathcal{N}$ is a full triangulated subcategory of $\mathcal{D}$, then the class $\Sigma=Mor_{\mathcal{N}}$ given by the morphisms $\xymatrix{f:X \ar[r] & Y}$ appearing in a triangle of the form
$$\xymatrix{X \ar[r]^{f} & Y \ar[r] & N \ar[r] & X[1]}$$
with $N\in \mathcal{N}$, gives a multiplicative system in $\mathcal{D}$ compatible with the triangulation (see details in \cite{V})}.
\end{example}

\begin{theorem}
Let $(\mathcal{D},?[1])$ be a triangulated category and $\Sigma$ be a multiplicative system in $\mathcal{D}$ compatible with the triangulation. The following assertions hold:
\begin{enumerate}
\item[1)] There exists on $\mathcal{D}[\Sigma^{-1}]$ a unique structure of additive category making the localization functor $\xymatrix{q:\mathcal{D} \ar[r] & \mathcal{D}[\Sigma^{-1}]}$ additive;

\item[2)] There exists on $\mathcal{D}[\Sigma^{-1}]$ a unique structure of triangulated category making the localization functor $\xymatrix{q:\mathcal{D} \ar[r] & \mathcal{D}[\Sigma^{-1}]}$ a triangulated functor. The distinguished triangles in $\mathcal{D}[\Sigma^{-1}]$ are the $?[1]$-sequences isomorphic to the image by $q$ of distinguished triangles in $\mathcal{D}$;

\item[3)] Every triangulated functor $\xymatrix{F:\mathcal{D} \ar[r] & \mathcal{D}^{\hspace{0.1cm}'}}$ taking elements of $\Sigma$ to isomorphisms factors uniquely through the localization functor in the form
$$\xymatrix{\mathcal{T} \ar[r]^{q \hspace{0.3 cm}} \ar[d]_{F}&  \mathcal{D}[\Sigma^{-1}] \ar@{-->}[dl]^{\tilde{F}} \\ \mathcal{D}^{\hspace{0.1 cm}'} }$$
where $\tilde{F}$ is a triangulated functor. Similarly for cohomological functors.
\end{enumerate}
\end{theorem}
\begin{proof}
See \cite[Th\'eor\`eme II.2.2.6]{V} or \cite[Chapter 2]{N}.
\end{proof}

\section{The derived category of an abelian category}
At the beginnig, the use of Homology in order to deal with mathematical problems focused on the use of homology groups. In broad terms, the process consisted in assigning to each object (topological space, differentiable manifold, etc.) a complex of abelian groups, whose homology groups were used to study classification problems. In the process, the chain complexes were just tools to pass from the object to its homology groups. In the sixties, Grothendieck defended the idea that it was necessary to pay more attention to the complexes themselves. His idea was based on the possibility of finding, for instance, topological spaces which were very differents in their structure but whose homological groups were isomorphic. For him, this fact was a pathology which could be corrected by taking into account also the complexes. Grothendieck's idea, proposed to his student Verdier, was to associate to each ``good'' abelian category $\mathcal{A}$ a category whose objects were the objects of $\mathcal{C(A)}$ and in which the quasi-isomorphisms were invertible. Keeping in mind the parallel work of Gabriel and Zisman, the answer was clear; but the wealth of it was undiscovered.

\begin{definition}\rm{
Let $\mathcal{A}$ be an abelian category. The localization of $\mathcal{C(A)}$ with respect to the class $\Sigma$ of the quasi-isomorphisms is called \emph{derived category of $\mathcal{A}$}\index{category! derived} and it is denoted by $\D(\mathcal{A})$. If $R$ is a ring and $\mathcal{A}=R$-Mod, most times we write $\D(R)$ instead of $\D(R\text{-Mod})$. }
\end{definition}

Through this section $\mathcal{A}$ will be an abelian category and $\Sigma$ will be the class of the quasi-isomorphisms in $\mathcal{C(A)}$. \newline

The problem of the definition of $\D(\mathcal{A})$ is the same as in the general definition of localization of a category, the initial unmanageability of their morphisms. This is so because the quasi-isomorphisms do not admit a calculus of fractions in $\mathcal{C(A)}$. The first obstacle is subtly solved.

\begin{proposition}
The canonical functor $\xymatrix{q:\mathcal{C(A)} \ar[r] & \D(\mathcal{A})}$ vanishes on contractible complexes, and hence there is a unique $\xymatrix{q:\mathcal{H(A)} \ar[r] & \D(\mathcal{A})}$, still denoted by $q$, such that $q \circ p = q$ where $\xymatrix{p:\mathcal{C(A)} \ar[r] & \mathcal{H(A)}}$ is the canonical functor. Furthermore $\overline{\Sigma}:=p(\Sigma)$ is a multiplicative system in $\mathcal{H(A)}$ that is compatible with the triangulation such that $\mathcal{H(A)}[\overline{\Sigma}^{-1}]\cong \D(\mathcal{A}).$
\end{proposition}
\begin{proof}
See \cite{V}.
\end{proof}

\begin{corollary}
If $\Hom_{\D(\mathcal{A})}(X,Y)$ is a set for each pair of objects $X,Y$ of $\mathcal{A}$, then $\D(\mathcal{A})$ is a triangulated category and the quotient functor $\xymatrix{q:\mathcal{H(A)} \ar[r] & \D(\mathcal{A})}$ is a triangulated functor which satisfies the following universal property: \newline

If $\mathcal{D}$ is a triangulated category and $\xymatrix{F:\mathcal{H(A)} \ar[r] & \mathcal{D}}$ is a triangulated functor such that $F(\overline{s})$ is an isomorphism for each $s$ quasi-isomorphism, then there exist a unique triangulated functor $\xymatrix{\tilde{F}:\D(\mathcal{A}) \ar[r] & \mathcal{T}}$ such that $\tilde{F} \circ q = F.$ \end{corollary}

\begin{remark}\rm{
Note that the functor ``n-th'' object of homology\index{homology} $\xymatrix{H^{n}:\mathcal{H(A)} \ar[r] & \mathcal{A}}$ factors uniquely through the quotient functor $\xymatrix{q:\mathcal{H(A)} \ar[r] & \D(\mathcal{A})}$ by a cohomological functor which will be still denoted by $H^{n}$:

$$\xymatrix{\mathcal{H(A)} \ar[r]^{q} \ar[d]_{H^{n}} & \D(\mathcal{A}) \ar@{-->}[dl]^{H^{n}}\\ \mathcal{A}}$$}
\end{remark}

\begin{corollary}
The cohomological functor $\xymatrix{H^{0}:\mathcal{D(A)} \ar[r] & \mathcal{A}}$ satisfies the property of that $H^{0} \circ (?[n])=H^{n}$, for each $n$ integer. Therefore, given a triangle in $\D(\mathcal{A})$:
$$\xymatrix{X \ar[r] & Y \ar[r] & Z \ar[r] & X[1]},$$ 
we get the following exact sequence in $\mathcal{A}$
$$\xymatrix{\cdots \ar[r] & H^{n}(X) \ar[r] & H^{n}(Y) \ar[r] & H^{n}(Z) \ar[r] & H^{n+1}(X) \ar[r] & H^{n+1}(Y) \ar[r] & \cdots}$$
\end{corollary}

\vspace{0.3 cm}

It is still a difficult problem to know if the category $\D(\mathcal{A})$
 really exists, since the fact of that a class of morphisms in a category is a multiplicative system does not guarantee that the corresponding localization is a proper category (i.e. if $\Hom_{\D(\mathcal{A})}(X,Y)$ is a set, for all $X,Y\in \D(\mathcal{A})$). The idea is to try to replace each complex by a quasi-isomorphic one which works well passing from $\mathcal{H(A)}$ to $\D(\mathcal{A})$. Recall that a complex $X$ is said to be \emph{acyclic}\index{complex! acyclic} when $H^{n}(X)=0$ for all $n\in $ \Z.

\begin{proposition}
For each complex $I$ over $\mathcal{A}$ the following assertions are equivalent:
\begin{enumerate}
\item[1)] $\Hom_{\mathcal{H(A)}}(X,I)=0$, for each acyclic complex $X$;
\item[2)] For each quasi-isomorphism $\xymatrix{s:X \ar[r] & Y}$, the map 
$$\xymatrix{\Hom_{\mathcal{H(A)}}(s,I):\Hom_{\mathcal{H(A)}}(Y,I) \ar[r] & \Hom_{\mathcal{H(A)}}(X,I)}$$
is bijective;
\item[3)] The assignment 
$$\xymatrix{\Hom_{\mathcal{H(A)}}(Y,I) \ar[r] & \Hom_{D(\mathcal{A})}(Y,I) \hspace{0.2 cm} (f \leadsto q(f))} $$
is one-to-one for each complex $Y$ over $\mathcal{A}$
\end{enumerate}
\end{proposition}
\begin{proof}
See \cite{V}. 
\end{proof}

\begin{definition}\rm{
A complex $I$ satisfying some of (and hence all) the equivalent conditions of the previous proposition is said to be  \emph{homotopically injective}\index{complex! homotopically injective}. Given a complex $X\in \mathcal{C(A)}$, a quasi-isomorphism $\xymatrix{u:X \ar[r] & I}$, where $I$ is a homotopically injective complex is called \emph{homotopically injective resolution}\index{resolution! homotopically injective} of $X$. Dually, one defines \emph{homotopically projective complexes}\index{complex! homotopically projective} and \emph{homotopically projective resolution}\index{resolution! homotopically projective}.  }
\end{definition}

\begin{corollary}
Let $\mathcal{A}$ be an abelian category with enough injectives and $Y$ is a complex over $\mathcal{A}$ that admits a homotopically injective resolution. Then $\Hom_{\D(\mathcal{A})}(X,Y)$ is a set, for each $X\in Obj(\D(\mathcal{A}))$.
\end{corollary}

\begin{remark}\rm{
Nothing guarantees that a given complex admits a homotopically injective or projective resolution. However some partial results are known since the dawn of the use of triangulated categories. Recall that for each pair $a,b\in$ \Z$\cup \{  \pm \infty \}$ with $a\leq b$, a chain complex $X$ is said to be \emph{concentrated in the range} $[a,b]$ if $X^{n}=0$ for all $n\notin [a,b]$. Similarly, we say that $X$ has \emph{homology concentrated in the range} $[a,b]$ when $H^{n}(X)=0$, for all $n\notin [a,b]$.}
\end{remark}

\begin{proposition}
Suppose that $\mathcal{A}$ has enough injectives. If $- \infty < a \leq b \leq +\infty $ and $X\in Obj(\mathcal{C(A)})$, then the following assertions are equivalents:
\begin{enumerate}
\item[1)] $X$ has homology concentrated in $[a,b]$; 
\item[2)] $X$ is quasi-isomorphic in $\D(\mathcal{A})$ to a chain complex concentrated in $[a,b]$;
\item[3)] There exists a homotopically injective resolution $\xymatrix{u: X \ar[r] & I}$, where $I$ is a complex of injectives concentrated in $[a,+\infty]$ and its cohomology concentrated in $[a,b]$.
\end{enumerate} 
\end{proposition}

\begin{lemma}
Let $\mathcal{A}$ be an abelian category with enough injectives. If $I$ is a \emph{left bounded complex} \index{complex! left bounded} (i.e., $I^{n}=0$ for $n<<0$) of objects injectives, then $I$ is a complex homotopically injective. 
\end{lemma}

\begin{example}\rm{
The result of lat lemma is not true if $I$ is not left bounded. Indeed, the following chain complex
$$\xymatrix{\cdots \ar[r] & \mathbb{Z}_4 \ar[r]^{2.} &  \mathbb{Z}_4  \ar[r]^{2.} &  \mathbb{Z}_4  \ar[r] & \cdots}$$
is an acyclic complex of injective $ \mathbb{Z}_4$-modules which is not zero in $\mathcal{H}( \mathbb{Z}_4 )$, i.e., it is not contractible.}
\end{example}
A sample of the homological importance of the derived category is the following result, proved by Verdier in \cite{V}.

\begin{corollary}
Suppose that $\mathcal{A}$ has enough injectives. If $A$ and $B$ are objects of $\mathcal{A}$ and $n$ is an integer, then we have an isomorphism of abelian groups:
\begin{enumerate}
\item[1)] $\Hom_{\D(\mathcal{A})}(A,B[n])=0$, if $n<0$;
\item[2)] $\Hom_{\D(\mathcal{A})}(A,B[n])\cong \Ext^{\hspace{0.04cm}n}_{\mathcal{A}}(A,B)$, if $n\geq 0$.
\end{enumerate}   
\end{corollary}

The next result, which is essential and was inspired by Spalstenstein (see \cite{Sp}), says that in general enough situations $\D(\mathcal{A})$ is a category strict sense.

\begin{theorem}
If $\mathcal{A}$ is a Grothendieck category, then every chain complex over $\mathcal{A}$ admits a homotopically injective resolution. Therefore $\Hom_{\D(\mathcal{A})}(X,Y)$ is a set, for each pair $X,Y\in \D(\mathcal{A})$. 
\end{theorem}
\begin{proof}
See \cite[Theorem 5.4]{AJSo}.
\end{proof}

\begin{remarks}\rm{
\begin{enumerate}
\item[1)] The results for left bounded complexes
in the case that $\mathcal{A}$ has enough injectives can be immediately dualized to right bounded complexes in the case that $\mathcal{A}$ has enough projectives;
\item[2)] In the case that $\mathcal{A}$ is a Grothendieck category that also has enough projectives (e.g. $\mathcal{A}=R$-Mod), Spalstenstein proved that every chain complex over $\mathcal{A}$ admit a homotopically projective resolution. This means that if $X,Y\in \mathcal{C(A)}$, then we can use a homotopically projective resolution of $X$ or a homotopically injective resolution of $Y$ for calculate the morphisms in $\D(\mathcal{A})$. In fact, if $\xymatrix{p:P \ar[r] & X}$ and $\xymatrix{u:Y \ar[r] & I}$ are  homotopically projective and injective resolution respectively, then we have an isomorphism of abelian groups:
$$\xymatrix{\Hom_{\mathcal{H(A)}}(P,Y) \ar[r]^{\sim \hspace{0.1cm}} & \Hom_{\D(\mathcal{A})}(X,Y) &  \ar[l]_{\hspace{0.1cm}\sim} \Hom_{\mathcal{H(A)}}(X,I)}$$ 
\end{enumerate}}
\end{remarks}

In the situation of the previous remark, we can be more precise.

\begin{definition}
Let $\G$ be a Grothendieck category. The canonical functor \linebreak $q:\mathcal{H}(\G) \flecha \D(\G)$ always has a right adjoint triangulated functor, called the \emph{homotopically injective resolution functor} $\mathbf{i}:\D(\G) \flecha \mathcal{H}(\G)$. \\

In the particular case $\G=R\Mode$, the functor $q:\mathcal{H}(R)\flecha \D(R)$ also has a left adjoint triangulated functor, called \emph{homotopically projective resolution functor} $\mathbf{p}:\D(R) \flecha \mathcal{H}(R)$.
\end{definition}

\begin{definition}
Let $\G,\G^{'}$ be Grothendieck categories and let $F:\C(\G) \flecha \C(\G^{'})$ be a functor which takes contractible complexes to contractible complexes and preserves the semi-split exact sequences of complexes. The it induces a triangulated functor \linebreak $\mathcal{H}(\G) \flecha \mathcal{H}(\G^{'})$ which will be still denoted by $F$. The \emph{right derived functor of $F$}\index{functor! right derived} is the composition: 
$$\mathbf{R}F:\D(\G) \xymatrix{\ar[r]^{\mathbf{i} \hspace{0.5 cm}} & \mathcal{H}(\G) \ar[r]^{F} & \mathcal{H}(\G^{'}) \ar[r]^{q} & \D(\G^{'})}$$

In case $\G$ is a module category, one dually defines the \emph{left 
derived functor of F} \index{functor! left derived}as the composition:

$$\mathbf{L}F:\D(\G) \xymatrix{\ar[r]^{\mathbf{p}\hspace{0.5 cm}} & \mathcal{H}(\G) \ar[r]^{F} & \mathcal{H}(\G^{'}) \ar[r]^{q} & \D(\G^{'})}$$
\end{definition}

\begin{examples}\rm{
\begin{enumerate}
\item[1)] Let $\G,\G^{'}$ be Grothendieck categories and let $F:\G \flecha \G^{'}$ be an additive functor. Then one considers its canonical extension to a functor $F:\C(\G) \flecha \C(\G^{'})$. This functor satisfies the requirement of the previous definition and $\mathbf{R}F$ is called the \emph{right derived functor of F}. If $F$ is exact, then one has $F=\mathbf{R}F$, meaning that one does not need to take homotopically injective resolution.

\item[2)] Given $X,Y\in \mathcal{C}(R)$. We have a chain complex of abelian groups $\Hom^{\bullet}_R(X,Y)$ defined as follows:
\begin{enumerate}
\item[a)] $\Hom^{\bullet}_R(X,Y)^{n}=\Hom^{n}(X,Y)=\prod_{i\in \text{ \Z}}\Hom_{R}(X^{i},Y^{i+n})$ for each integer $n$;
\item[b)] The differential $\xymatrix{d^{n}:\Hom^{n}(X,Y) \ar[r] & \Hom^{n+1}(X,Y)}$ is define by the rule $d^{n}(f):=d_{Y} \circ f-(-1)^{n} f \circ d_{X}$.
\end{enumerate}
When $X$ is fixed, the assignment $Y  \rightsquigarrow \Hom^{\bullet}_R(X,Y)$ defines a functor $\Hom^{\bullet}_R(X,?):\C(R) \flecha \C(\text{Ab})$, which satisfies the requirement of the previous definition. Its right derived functor is denoted by $\mathbf{R}\Hom_R(X,?):\D(R) \flecha \D(\text{Ab})$. 

\item[3)] Given a complex of right $R$-modules $X$ and a complex of left $R$-modules, we denote by $X\otimes_R Y$ to complex defined as follows:
\begin{enumerate}
\item[a)] $(X\otimes_RY)^{n}=\underset{i+j=n}{\oplus} X^{i} \otimes_R Y^{j}$ for each integer $n$;

\item[b)] The differential $d^{n}:(X\otimes_RY)^{n} \flecha (X \otimes_R Y)^{n+1}$, $x\otimes y  \rightsquigarrow d_X(x)\otimes y + (-1)^{deg(x)} x \otimes d_Y(y)$, for all homogeneous elements $x\in X$ and $y\in Y$.
\end{enumerate}
When $X$ is fixed, the assignment $Y  \rightsquigarrow X \otimes_R Y$ gives a functor $X\otimes_R ?: \C(R) \flecha \C(\text{Ab})$ which satisfies the requirements of the previous definition. Its left derived functor is denoted by $X\otimes_R^{\mathbf{L}}?:\D(R)\flecha \D(\text{Ab})$.
\end{enumerate}}
\end{examples}

\vspace{0.3 cm}
In this thesis, we will use frequently the following result.

\begin{proposition}\label{prop. adjoint derived}
Let $\G$ be a Grothendieck category and let $\xymatrix{R\Mode \ar@<0.6ex>[r]^{\hspace{0.5 cm}F} & \G \ar@<0.6ex>[l]^{\hspace{0.5 cm}G } }$ a pair of functors such that $F$ is left adjoint of $G$. Then $\mathbf{L}F:\D(R) \flecha \D(\G)$ is left adjoint of $\mathbf{R}G:\D(\G) \flecha \D(R)$.
\end{proposition}
\begin{proof}
See \cite[Proposition 2.28]{NS3}. 
\end{proof}

\section{T-structures in triangulated categories}\label{section t-structure}
T-structures on triangulated categories were introduced in the early eighties by Beillison, Berstein and Deligne in their study of perverse sheaves on an algebraic or an analytic variety (see \cite{BBD}). The main discovery of this concept was the existence of an abelian category, called the heart of the t-structure, which allowed the development of a homology theory that is intrinsic to the triangulated category. Thereafter this concept has been interpreted by several authors (e.g. \cite{BR} and \cite{NS}) as the correspondent of torsion theory in abelian categories.

\begin{definition}\label{definition t-structure}\rm{
Let $(\mathcal{D},?[1])$ be a triangulated category. A \emph{t-structure} \index{t-structure}on $\mathcal{D}$ is a couple of full subcategories closed under direct summands $(\mathcal{U,W}[1])$ such that:
\begin{enumerate}
\item[1)] $\mathcal{U}[1]\subset \mathcal{U};$
\item[2)] $\Hom_{\mathcal{D}}(U,W)=0$, for $U\in \mathcal{U}$ and $W\in \mathcal{W}$;
\item[3)] For each $X \in \mathcal{D}$, there is a distinguished triangle:
$$\xymatrix{U \ar[r] & X \ar[r] & W \ar[r] & U[1]}$$
with $U\in \mathcal{U}$ and $W\in \mathcal{W}$. 
\end{enumerate}
The subcategory $\mathcal{U}$ is called the \emph{aisle}\index{aisle} of the t-structure, and $\mathcal{W}$ is called the \emph{co-aisle}\index{co-aisle}.}
\end{definition}

\begin{remarks}\label{properties t-structure}\rm{
\begin{enumerate}
\item[1)] It is easy to see that in such case $\mathcal{W} = \mathcal{U}^{\perp}[1]$ and \newline $\mathcal{U}=^{\perp}(\mathcal{W}[-1])= ^{\perp}(\mathcal{U}^{\perp})$. For this reason, we will write a t-structure as $(\mathcal{U}, \mathcal{U}^{\perp}[1])$;

\item[2)] The objects $U$ and $W$ in the above triangle are uniquely determined by $X$, up to isomorphism, and define functors $\xymatrix{\tau_{\mathcal{U}}: \mathcal{D} \ar[r] & \mathcal{U}}$ and $\xymatrix{\tau^{\mathcal{U}^{\perp}}:\mathcal{D} \ar[r] & \mathcal{U}^{\perp}}$ which are right and left adjoints to the respective inclusion functors. The composition $\xymatrix{\tau_{\mathcal{U}}: \mathcal{D} \ar[r] & \mathcal{U} \hspace{0.1cm} \ar@{^(->}[r] & \mathcal{D}}$ (resp. $\xymatrix{\tau^{\mathcal{U}^{\perp}}:\mathcal{D} \ar[r] & \mathcal{U}^{\perp} \hspace{0.08 cm}\ar@{^(->}[r] & \mathcal{D}}$), which we will still denote by $\tau_{\mathcal{U}}$ (resp. $\tau^{\mathcal{U}^{\perp}}$) is called the \emph{left truncation}\index{functor! left truncation} (resp. \emph{right truncation}\index{functor! right truncation}) functor with respect to the given t-structure.

\item[3)] $(\mathcal{U}[k], \mathcal{U}^{\perp}[k+1])$ is also a t-structure in $\mathcal{D}$, for each $k\in $ \Z;

\item[4)] For each class $\mathcal{U}\subseteq \mathcal{D}$, we have that $(\mathcal{U},\mathcal{U}^{\perp}[1])$ is a t-structure in $\mathcal{D}$ if and only if $(\mathcal{U}^{\perp}[1], \mathcal{U})$ is a t-structure in $\mathcal{D}^{op}$;

\item[5)] The full subcategory $\mathcal{H}=\mathcal{U} \cap \mathcal{U}^{\perp}[1]$ is called the \emph{heart}\index{heart} of the t-structure and it is an abelian category, where the short exact sequences ``are'' the triangles in $\mathcal{D}$ with its three terms in $\mathcal{D}$. In particular, one has $\Ext^{1}_{\mathcal{H}}(M,N)=\Hom_{\mathcal{D}}(M,N[1])$, for all objects $M$ and $N$ in $\mathcal{H}$ (see \cite{BBD});

\item[6)] The assignments $X \leadsto (\tau_{\mathcal{U}} \hspace{0.05 cm} \circ \hspace{0.05 cm} \tau^{\mathcal{U}^{\perp}[1]})(X)$ and $X \leadsto (\tau^{\mathcal{U}^{\perp}[1]} \hspace{0.05 cm} \circ \hspace{0.05 cm} \tau_{\mathcal{U}})(X)$ define naturally isomorphic functors $\xymatrix{\mathcal{D} \ar[r] & \mathcal{H}}$ which are cohomological (see \cite{BBD}). We will denote by $\xymatrix{\tilde{H}:\mathcal{D} \ar[r] & \mathcal{H}}$ either of these naturally isomorphic functors.

\end{enumerate}}
\end{remarks}

\begin{examples}\label{example t-structure}\rm{
\begin{enumerate}
\item[1)] Let $\mathcal{A}$ be an abelian category such that $\D(\mathcal{A})$ is a real (i.e. with Hom groups which are sets) category (e.g. if $\mathcal{A}$ is a Grothendieck category) and, for each $k\in$ \Z, will denote by $\D^{\leq k}(\mathcal{A})$ (resp. $\D^{\geq k}(\mathcal{A})$) the full subcategory of $\D(\mathcal{A})$ consisting of the chain complexes $X$ such that $H^{j}(X)=0$, for all $j>k$ (resp. $j<k$). The pair $(\D^{\leq k}(\mathcal{A}),\D^{\geq k}(\mathcal{A}))$ is a t-structure in $\D(\mathcal{A})$ whose heart is equivalent to $\mathcal{A}$. Its left and right truncation functors will be denoted by $\xymatrix{\tau^{\leq k}: \D(\mathcal{A}) \ar[r] & \D^{\leq k}(\mathcal{A})}$ and $\xymatrix{\tau^{> k}: \D(\mathcal{A}) \ar[r] & \D^{> k}(\mathcal{A}):=\D^{\geq k}(\mathcal{A})[-1]}.$ For $k=0$, the given t-structure is known as the \emph{canonical t-structure} in $\D(\mathcal{A})$.  

\item[2)] (Happel-Reiten-Smal\o)\index{Happel-Reiten-Smal\o} Let $\mathcal{G}$ be a Grothendieck category and $\mathbf{t}=(\mathcal{T,F})$ be a torsion pair in $\mathcal{G}$. As done in \cite{HRS} in the context of finite dimensional modules over a finite dimensional algebra and its corresponding bounded derived category, one gets a t-structure $(\mathcal{U}_{\mathbf{t}},\mathcal{U}_{\mathbf{t}}^{\perp}[1])=(\mathcal{U}_{\mathbf{t}},\mathcal{W}_{\mathbf{t}})$ in $\D(\mathcal{G}),$ where: 
\begin{center}
$\mathcal{U}_{\mathbf{t}}=\{X \in \D^{\leq 0}(\mathcal{G})| H^{0}(X)\in \mathcal{T}\}$ and $\mathcal{W}_{\mathbf{t}}=\{Y\in \D^{\geq -1}(\mathcal{G})| H^{-1}(Y) \in \mathcal{F}\}$
\end{center}

In this case, the heart $\mathcal{H}_{\mathbf{t}}$ consists of the chain complexes which are isomorphic in $\D(\mathcal{G})$ to complexes 
$$\xymatrix{M:=\cdots \ar[r] & 0 \ar[r] & M^{-1} \ar[r]^{d} & M^{0} \ar[r] & \cdots}$$ 
concentrated in degrees -1 and 0, such that $H^{-1}(M)=\Ker(d)\in \mathcal{F}$ and $H^{0}(M)=\Coker(d)\in \mathcal{T}$.

When $\G=R$-Mod, such a chain complex is always quasi-isomorphic to a chain complex of the form $\xymatrix{\cdots \ar[r] & 0 \ar[r] & X \ar[r]^{j} & Q \ar[r]^{d} & P \ar[r] & 0 \ar[r] & \cdots}$, concentrated in degrees -2,-1, 0, such that $j$ is a monomorphism and $P,Q$ are projective modules. If $M$ is such chain complex, then for each complex $N$ of $\Ht$, we have that the canonical map $\Hom_{\mathcal{H}(R)}(M,N) \flecha \Hom_{\D(R)}(M,N)=\Hom_{\Ht}(M,N)$ is bijective. We will frequently use this fact throughout the work.  
\end{enumerate}}
\end{examples}

\begin{definition}\label{t-structure compactly generated}\rm{
Let $\D$ be a triangulated category with coproducts, $(\mathcal{U},\mathcal{U}^{\perp}[1])$ be a t-structure in $\D$ and $\mathcal{S}\subset \mathcal{U}$ be a set of objects. We will say that $(\mathcal{U,U}^{\perp}[1])$ is the \emph{t-structure generated} by $\mathcal{S}$ if $\mathcal{U}^{\perp}$ consists of the objects $Y\in \D$ such that $\Hom_{\D}(S[k],Y)=0$, for all $k\geq 0$ and all $S\in \mathcal{S}$. In that case, we will write $\mathcal{U}=\text{aisle}<\mathcal{S}>$. \\

We will say that $(\mathcal{U,U}^{\perp}[1])$ is \emph{compactly generated} when there is a set of compact objects $\mathcal{S}\subset \mathcal{U}$ such that $(\mathcal{U,U}^{\perp}[1])$ is the t-structure generated by $\mathcal{S}$.}
\end{definition}

\vspace{0.3 cm}

If $\D=\D(R)$, for some ring $R$, and $\mathcal{S}$ is a set of object, then the smallest full subcategory of $\D$ closed under coproducts, extensions and positive shifts is the aisle of a t-structure (cf. \cite[Proposition 3.2]{AJSo2}). In that case, the corresponding t-structure is the one generated by $\mathcal{S}$.

\begin{example}\label{exam. Ht compactly generated}\rm{
Let $R$ be a commutative Noetherian ring. For each hereditary torsion pair $\te=(\T,\F)$ in $R$-Mod, the pair $(\mathcal{U}_{\mathbf{t}},\mathcal{U}_{\mathbf{t}}^{\perp}[1])$ is a compactly generated t-structure on $\D(R)$ (see details in \cite{AJS}). }
\end{example}

\begin{definition}\rm{
An object $T$ of a triangulated category $\mathcal{D}$ will be called \emph{classical tilting}\index{object! classical tilting on $\D$} when it satisfies the following conditions:
\begin{enumerate}
\item[i)] $T$ is compact in $\mathcal{D}$;
\item[ii)] $\Hom_{\mathcal{D}}(T,T[i])=0$, for all $i\neq 0$;
\item[iii)] If $X\in \mathcal{D}$ is an object such that $\Hom_{\mathcal{D}}(T[i],X)=0$, for all $i \in $ \Z, then $X=0$. 
\end{enumerate}}
\end{definition}

\begin{remark}\rm{
The compact objects of $\D(R)$ are the complexes which are quasi-isomorphic to bounded complexes of finitely generated projective modules (see \cite{Ri}).}
\end{remark}

\begin{example}\rm{
If $T$ is a classical 1-tilting $R$-module, then $T[0]$ is a classical tilting object of $\D(R)$. By a well-know result of Rickard (see \cite{Ri} and \cite{Ri2}), two rings $R$ and $S$ are \emph{derived equivalent}\index{ring! derived equivalent}, i.e., have equivalent derived categories, if and only if there exists a classical tilting object $T$ in $\D(R)$ such that $S\cong \text{End}_{\D(R)}(T)^{op}$ (see also \cite[Theorem 6.6]{AJSo2}).}
\end{example}

\section{Some points of Commutative Algebra}\label{sec. Commutative algebra}
In this section we assume that $R$ is commutative Noetherian ring and we denote by $\Spec (R)$ its spectrum. For each ideal $\mathbf{a} \subseteq R$ we  put $\text{V}(\mathbf{a}):=\{\mathbf{p}\in \Spec(R) \ | \ \mathbf{a} \subseteq \mathbf{p}\}$. These subsets are the closed subsets for the \emph{Zariski topology} in $\Spec (R)$. The so defined topological space is noetherian, i.e. satisfies ACC on open subsets. When one writes $1$ as a sum $1=\underset{1\leq i\leq n}{\sum}e_i$ of primitive orthogonal idempotents, we have a ring decomposition $R\cong Re_1\times \dots \times Re_n$ and can identify $\Spec(Re_i)$ with $V(R(1-e_i))$. These are precisely the connected components of $\Spec (R)$ as a topological space, so that we have a disjoint union $\Spec (R)=\underset{1\leq i\leq n}{\bigcup}\Spec(Re_i)$. We call the rings $Re_i$ the \emph{connected components of $R$}.\\

A subset $Z \subseteq \Spec(R)$ is \emph{stable under specialization} if, for any couple of prime ideals $\mathbf{p} \subseteq \mathbf{q}$ with $\mathbf{p}\in Z$, it holds that $\mathbf{q}\in Z$. Equivalently, when $Z$ is a union of closed subsets of $\Spec(R)$. Such a subset will be called \emph{sp-subset} in the sequel. The typical example is the \emph{support of a $R$-module} $N$, denoted $\Supp (N)$,  which consists of the prime ideals $\mathbf{p}$ such that $N_\mathbf{p}\cong R_{\mathbf{p}} \otimes_{R} N\neq 0$. Here $(?)_\mathbf{p}$ denotes the localization at $\mathbf{p}$, both when applied to $R$-modules or to the ring $R$ itself. Recall that the assignment $N\rightsquigarrow N_\mathbf{p}$ defines an exact functor $R\Mode \flecha R_\mathbf{p}\Mode$. If $f:R\flecha R_\mathbf{p}$ is the canonical ring homomorphism, then we have a natural isomorphism $(?)_\mathbf{p}\cong f^*$, where $f^*$ is the extension of scalars with respect to $f$. In particular, we have an adjoint pair of exact functors $(f^*,f_*)$, where $f_*:R_\mathbf{p}\Mode \flecha R\Mode$ is the restriction of scalars. \\

We have the following precise description for hereditary torsion pairs in $R$-Mod. 

\begin{proposition} \label{prop. sp-subsets versus torsion pairs}
The assignment $Z\rightsquigarrow (\mathcal T_Z,\mathcal T_Z^\perp)$ defines a bijection between the sp-subsets of $\Spec (R)$ and the hereditary torsion pairs in $R\Mode$, where $\mathcal T_Z$ is the class of modules $T$ such that $\Supp (T)\subseteq Z$. Its inverse takes $(\mathcal T,\mathcal T^\perp)$ to the set $Z_{\mathcal T}$ of prime ideals $\mathbf{p}$ such that $R/\mathbf{p}$ is in $\mathcal{T}$. 
\end{proposition}

In the situation of last proposition, we will denote by $\Gamma_Z$ the torsion radical associated to the torsion pair $(\T_Z,\T_Z^{\perp})$. \\

For simplicity, we will put $\mathcal{G}_Z=\mathcal{G}_{\mathcal{T}_Z}$ and $R_\mathbf{F}=R_Z$ to denote the Giraud subcategory of $R$-Mod and the ring of quotients associated to the hereditary torsion pair $(\T_Z,\T_Z^{\perp})$. The $R$-(resp. $R_Z$-)modules in $\G_Z$ will be called \emph{$Z$-closed}. It seems to be folklore that $R_Z$ is a commutative ring in this case, however we give a proof taken from \cite{Ana}.


\begin{lemma}\label{lem. R_F is a commutative ring}
If $R$ is a commutative ring and $Z\subseteq \Spec(R)$ is a sp-subset, then $R_Z$ is a commutative ring.
\end{lemma}
\begin{proof}
We start showing that if $x\in R_Z$ is represented by the morphism $\alpha:\mathbf{a} \flecha \frac{R}{t(R)}$, with $\mathbf{a}$ in the associated Gabriel topology $\mathbf{F}$, then $\mathbf{a}.x=x.\mathbf{a}$ in $R_Z$. They are represented by:
\begin{align*}
\mathbf{a}.x: R \xymatrix{ \ar[r]^{\rho_{\mathbf{a}} \hspace{0.2 cm}} & R \ar[r]^{\alpha\hspace{0.2 cm}} & \frac{R}{t(R)}}\\
x.\mathbf{a}: \mathbf{a} \xymatrix{ \ar[r]^{\alpha \hspace{0.4 cm}} & \frac{R}{t(R)} \ar[r]^{\rho_{\mathbf{a}}} & \frac{R}{t(R)}}
\end{align*}
The first one takes $b \rightsquigarrow \alpha(ba)$, while the second one takes $b \rightsquigarrow \alpha(b)a$. But $\alpha(ba)=\alpha(ab)=a\alpha(b)$, since $\alpha$ is morphism in $R$-Mod. The equality $\alpha(b)\tilde{a}=\tilde{a}\alpha(b)$ holds in $\frac{R}{t(R)}$ because this is a commutative ring. \\

Now, let $x,y \in R_Z$. We then have ideals $\mathbf{a}_x$ and $\mathbf{a}_y$ in the Gabriel topology such that $\mathbf{a}_{x}x\subseteq \frac{R}{t(R)}$ and $\mathbf{a}_{y}y\subseteq \frac{R}{t(R)}$. We now take:
\begin{align*}
\mathbf{b}_x=\{a \in \mathbf{a}_x: ax\in \tilde{ \mathbf{a}}_y:=\frac{ \mathbf{a}_y}{t( \mathbf{a}_y)}\} \\ \mathbf{b}_y=\{a \in \mathbf{a}_y: ay\in \tilde{ \mathbf{a}}_x:=\frac{ \mathbf{a}_x}{t( \mathbf{a}_x)}\} 
\end{align*}
Then $\mathbf{b}_x$ and $\mathbf{b}_y$ are ideals in the Gabriel topology, and hence, $\mathbf{b}:=\mathbf{b}_x \cap \mathbf{b}_y$ is an ideal in the Gabriel topology. Now we have:
$$\mathbf{b}xy=(\mathbf{b}x)y=y(\mathbf{b}x)=(y\mathbf{b})x=\mathbf{b}yx$$ 
Then $\rho_{xy-yx}:\mathbf{b} \flecha \frac{R}{t(R)}$ vanishes.
\end{proof}

\chapter{Previous results}\label{chapter result previous}

Since the heart of a t-structure is an abelian category, it is natural to ask when such a category is the nicest possible; which in category theory means: when is it a Grothendieck category? or, in the most ambitious case,  when is it a module category?. In this work, we will deal with this kind of questions, for the heart of the Hapel-Reiten-Smal\o \ t-structure; as well as for the heart of the compactly generated t-structures in the derived category of a commutative Noetherian ring. Several works have dealt with the first question, but not with the latter one. However, Alonso, Jerem\'ias and Saor\'in in \cite{AJS} characterized the compactly generated t-structures in the derived category of the category of modules over a commutative Noetherian ring. They described such t-structures in terms of filtrations by supports of the spectrum of such ring. This is the reason why this chapter will be divided in three sections. The first one will have the previous results for the Happel-Reiten-Smal\o \ t-structure and in the second one we will recollect some results from \cite{AJS}. In the third one we will collect some questions which will be tackled in this work.


\section{Happel-Reiten-Smal\o}
\subsection{Grothendieck case}\label{remark tilting theorem} 
All throughout this subsection, $\mathcal{A}$ is an abelian category such that there is a classical 1-tilting object $V$ of $\mathcal{A}$ with the property of that any object of $\mathcal{A}$ embeds in an object of $\Gen(V)$. In such case, if $R=\End_{\mathcal{A}}(V)$, $H_V=\Hom_{\mathcal{A}}(V,?):\mathcal{A} \flecha \Mod R^{op}$, $H^{'}_{V}=\Ext^{1}_{\mathcal{A}}(V,?):\mathcal{A} \flecha \Mod R^{op}$, $T_V:\Mod R^{op} \flecha \mathcal{A}$ the left adjoint to $H_V$, and $T^{'}_{V}:\Mod R^{op} \flecha \mathcal{A}$ the first left derived functor of $T_V$. Set $\T=\Ker(H^{'}_{V})$, $\F=\Ker(H_V)$, $\mathcal{X}=\Ker(T_V)$ and $\mathcal{Y}=\Ker(T^{'}_{V})$. Then $\mathbf{t}=(\T,\F)$ is a torsion pair in $\mathcal{A}$ with $\T=\Gen(V)$, and $\mathbf{t}^{'}=(\mathcal{X,Y})$ is a faithful torsion pair in $\Mod R$ (see \cite[Theorem 3.2]{CF}). The next result from \cite[Corollary 2.4 ]{CGM}, tells us that our goal for faithful torsion pairs in a module category is equivalent to study when such an abelian category $\mathcal{A}$ is a Grothendieck category or a module category.

\begin{proposition}
$\mathcal{A}$ is equivalent to $\mathcal{H}_{\mathbf{t}^{'}}$, where $\mathbf{t}^{'}$ is as above.
\end{proposition}

\vspace{0.3 cm}

The importance of the previous result is the appearance of a tilting object $V$. This fact was exploited by Colpi, Gregorio and Mantese in \cite{CGM}. The next theorem is an example of some results (see \cite[Theorem]{CGM}).

\begin{theorem}\label{teo. motived CGM}
The following conditions are equivalent:
\begin{enumerate}
\item[1)] $\mathcal{H}_{\mathbf{t}^{'}}$ is a Grothendieck category;

\item[2)] for any direct system $(X_{\lambda})$ in $\mathcal{H}_{\mathbf{t}^{'}}$ the canonical map
\begin{center}
$\varinjlim{H_{V}(X_{\lambda})} \flecha H_{V}(\varinjlim_{\mathcal{H}_{\mathbf{t}^{'}}}{X_{\lambda}})$
\end{center}
is a monomorphism in $\Mod R$;

\item[3)] the functor $H_V$ preserves direct limits.
\end{enumerate}
If $\mathcal{Y}$ is closed under direct limits, then the previous conditions are equivalent to:
\begin{enumerate}
\item[4)] the functor $T_{V} \circ H_{V}$ preserves direct limits.
\end{enumerate}
\end{theorem}

\vspace{0.3 cm}

The next result gives a necessary condition for our goal (see \cite[Proposition 3.8]{CGM}).

\begin{proposition}
Let $R$ be a ring and let $\te=(\T,\F)$ be a faithful torsion pair in $R\text{-Mod}$. If $\mathcal{H}_{\mathbf{t}}$ is a Grothendieck category then $\mathbf{t}=(\mathcal{T,F})$ is a cotilting torsion pair.
\end{proposition}

\vspace{0.3 cm}

Subsequently, Colpi and Gregory showed that the necessary condition in the previous proposition is really a characterization (see \cite{CG} and also \cite[Theorem 6.2]{Maa}).

\begin{theorem}
Let $R$ be a ring and let $\te=(\T,\F)$ be a faithful torsion pair in $R$-Mod. Then, $\Ht$ is a Grothendieck category if, and only if, $\te=(\T,\F)$ is a cotilting torsion pair.
\end{theorem}

\subsection{Modular case}
All throughout this section, $R$ is a ring and $\te=(\T,\F)$ is a torsion pair in $R\Mode$. We next give the main result of Hoshino, Kato and Miyachi (see \cite[Theorem 3.8]{HKM}, and definition \ref{def. HKM torsion pair} for the terminology used here). 

\begin{theorem}\label{teo. HKM principal}
Let $\xymatrix{P^{\bullet}:=\cdots \ar[r] & 0 \ar[r] & Q \ar[r]^{d} & P \ar[r] & 0 \ar[r] & \cdots}$ be a complex of finitely generated projective $R$-modules concentrated in degrees -1 and 0. If $P^{\bullet}$ is an HKM complex, then the functor
$$\Hom_{\D(R)}(P^{\bullet},?): \Ht \flecha \Mod \End_{\D(R)}(P^{\bullet})^{op}$$
is an equivalence, where $\te=(\mathcal{X}(P^{\bullet}),\mathcal{Y}(P^{\bullet}))$.
\end{theorem}

\begin{remark}\rm{
If $P^{\bullet}$ is an HKM complex, then it need not be in $\Ht$ (see example \ref{exam. HKM torsion}).}
\end{remark}

On the other hand, Colpi, Gregorio and Mantese gave the next sufficient conditions for $\Ht$ to be a module category (see \cite[Corollary 5.3]{CGM}).

\begin{proposition}
If $(\T,\F)$ is a faithful torsion pair generated by a classical 1-tilting module $V\in R$-Mod, then $\Ht$ is equivalent to $\End(V)^{op} \Mode$. In such case, $V[0]$ is a progenerator of $\Ht$.
\end{proposition}

An AB3 abelian category having a progenerator is a module category (see theorem \ref{teo. Gabriel-Michell}). Therefore, it is natural to enquire about the properties that such a progenerator satisfies. Colpi, Mantese and Tonolo give a characterization of this fact for faithful torsion pairs (see \cite[Theorem 6.1]{CMT}). 

\begin{definition}\rm{
Let $\mathcal{A}$ be an abelian category and $V$ an object of $\mathcal{A}$. The class of objects in $\mathcal{A}$ which admit a finite filtration with consecutive factors in $\overline{\Gen}(V)$ will be denote by $\overline{\overline{\Gen}}(V)$. }
\end{definition}

\begin{theorem}
Let $\te=(\T,\F)$ be a faithful torsion pair in $R$-Mod. The heart $\Ht$ is equivalent to a module category if, and only if, the following assertions holds: 
\begin{enumerate}
\item[1)] $\T=\Gen(V)$ , where $V$ is a classical 1-tilting $R/\ann_{R}(V)$-module;
\item[2)] $_RV$ admits a presentation of the form: 
$$\xymatrix{0 \ar[r] & \Omega \ar[r] & R_{1} \ar[r]^{f} & R_{0} \ar[r] & V \ar[r] & 0}$$
with $R_1$ and $R_0$ finitely generated projective $R$-modules such that:
\begin{enumerate}
\item[a)] any map $R_1 \flecha V$ extends to a map $R_{0} \flecha V$;
\item[b)] there exists a map $R_{1}^{(\alpha)} \xymatrix{\ar[r]^{g} &} R$ such that the cokernel of its restriction to $\Omega^{(\alpha)}$ is in $\overline{\overline{\Gen}}(V)$.
\end{enumerate}
\end{enumerate}
In such case, the complex $G:=\xymatrix{\cdots \ar[r] & 0 \ar[r] & R_{1} \ar[r]^{f} & R_0 \ar[r] & 0 \ar[r] & \cdots}$, concentrated in degrees 1- and 0, is a progenerator of the heart $\Ht$. Therefore, $\Ht$ is equivalent to $S\Mode$, where $S=\End_{\Ht}(G)^{op}$.
\end{theorem}

One of the main consequences of the last theorem is the following corollary, that was also proved by Colpi, Mantese and Tonolo in \cite{CMT}(see \cite[Corollary 6.3]{CMT}). 

\begin{corollary}\label{cor. tiltilng implies cotilting}
Let $\te=(\T,\F)$ be a faithful classical tilting torsion pair. Then $\te$ is a cotilting torsion pair.
\end{corollary}

\begin{definition}\rm{
A ring $R$ is called \emph{left poised}\index{ring! left poised} if the compact left $R$-modules are finitely generated.}
\end{definition}

The work of \cite{HKM}, \cite{CGM} and \cite{CMT}, does not cover all the possible cases in which $\Ht$ is a module category, specially if the torsion pair is not faithful. Another angle not completely covered is the relationship between the modular condition of $\Ht$ and the fact that $\te$ be an HKM torsion pair. The work of Mantese and Tonolo (\cite{MT}) gave a definitive answer for semiperfect left poised ring.



\begin{theorem}
Let $R$ be a ring and let $\te=(\T,\F)$ be a torsion pair in $R$-Mod, the following assertions holds:
\begin{enumerate}
\item[1)] If $\te=(\T,\F)$ is faithful, then $\Ht$ is equivalent to a module category if, and only if, $\te$ is an HKM torsion pair.

\item[2)] If $R$ is left poised and semiperfect, then $\Ht$ is equivalent to a module category if, and only if, $\te$ is an HKM torsion pair.  
\end{enumerate}
\end{theorem}

\section{T-structures given by filtrations}\label{sec. commutative Noetherian rings}
In \cite{AJS}, Alonso, Jerem\'ias and Saor\'in classify all the compactly generated t-structures in the derived category of $R$. In this section, we collect the results in \cite{AJS} which are most useful for our purposes. For each sp-subset $Z$ of $\Spec(R)$ (see section \ref{sec. Commutative algebra} for the terminology) and for each $R$-module $N$ let us denote by $\Gamma_{Z}(N)$ the biggest submodule of $N$ whose support is contained in $Z$. The functor $\Gamma_{Z}:R\text{-Mod} \flecha R\text{-Mod}$ is the associated torsion radical, which is left exact. We start with the following theorem which is useful for our purposes (see \cite[Theorem 5.6]{AJSo3} and also \cite[Theorem 1.8]{AJS}).


\begin{theorem}\label{teo. Right derived AJS}
Let $Z \subseteq \Spec(R)$ be a sp-subset and $M\in \D(R)$. The canonical map $\mathbf{R}\Gamma_{Z}(M) \flecha M$ is an isomorphism if, and only if, $\Supp(H^{j}(M))\subseteq Z$, for every $j\in \mathbb{Z}$.
\end{theorem}

\begin{definition}\rm{
A \emph{filtration}\index{filtration} by supports of $\Spec(R)$ is a decreasing map \linebreak $\phi:\mathbb{Z} \flecha \mathcal{P}(\Spec(R))$ such that $\phi(i)\subseteq \Spec(R)$ is a sp-subset for each $i\in \mathbb{Z}$. Similarly, to abbreviate, we will refer to a filtration by supports of $\Spec(R)$ simply by a \emph{sp-filtration}\index{sp-filtration} of $\Spec(R)$. }
\end{definition}

\begin{notation}\rm{
Let $\phi$ be a sp-filtration of $\Spec(R)$. For each integer, we will denote by $(\T_{k},\F_{k})$ the hereditary torsion pair $(\T_{\phi(k)},\F_{\phi(k)})$.} 
\end{notation}

In \cite{AJS}, Alonso, Jerem\'ias and Saor\'in associated to each sp-filtration \linebreak $\phi:\mathbb{Z} \flecha P(\Spec(R))$ the t-structure $(\mathcal{U}_{\phi},\mathcal{U}_{\phi}^{\perp}[1])$, where:

$$\mathcal{U}_{\phi}:=\text{aisle}<R/\mathbf{p}[-i];i\in \mathbb{Z}\text{ and }\mathbf{p}\in \phi(i)>$$

The following shows the compatibility of these aisles with respect to localization in a multiplicative closed subset of $R$ (see \cite[Proposition 2.9]{AJS}).

\begin{proposition}\label{prop. Localization AJS}
Let $S\subseteq R$ be a multiplicative closed subset. Given a sp-filtration $\phi: \text{\Z} \flecha P(\Spec(R))$ let us denote by $\mathcal{W}_{\phi}$ the right orthogonal of $\mathcal{U}_{\phi}$ in $\D(R)$. Then $(S^{-1}\mathcal{U}_{\phi}, S^{-1} \mathcal{W}_{\phi}[1])$ is a t-structure on $\D(S^{-1}R)$, furthermore $S^{-1}\mathcal{U}_{\phi}=\mathcal{U}_{\phi} \cap \D(S^{-1}R)$ and $S^{-1}\mathcal{W}_{\phi}=\mathcal{W}_{\phi} \cap \D(S^{-1}R)$. Moreover $S^{-1}\mathcal{U}_{\phi}$ is the aisle of $\D(S^{-1}R)$ associated to the sp-filtration $\phi_{S}:\mathbb{Z} \flecha P(\Spec(S^{-1}R))$ defined by $\phi_{S}(i):=\phi(i)\cap \Spec(S^{-1}R)$, for each $i\in$ \Z.\\

Here and all throughout the manuscript, for each multiplicative subset $S\subset R$, we identify $\Spec(S^{-1}R)$ with $\{\p\in \Spec(R):\p\cap S= \emptyset\}$.
\end{proposition}

\begin{remark}\rm{
We are interested in the case that $S=R\setminus \mathbf{p}$, for some $\mathbf{p}\in \Spec(R)$. In such case, if $\mathcal{W}$ is any of the classes $\mathcal{U}_{\phi}$ or $\mathcal{W}_{\phi}$, then it follows from the previous proposition that a complex $X\in \D(R)$ is in $\mathcal{W}$ if, and only if, $X_{\mathbf{p}}$ is in $\mathcal{W}$ for any $\mathbf{p} \in \Spec(R)$.}
\end{remark}

The following is a main result in \cite{AJS} (see \cite[Theorem 3.10]{AJS}).

\begin{theorem}\label{teo. main AJS}
Let $R$ be a commutative Noetherian ring and $(\mathcal{U},\mathcal{W}[1])$ be a t-structure on $\D(R)$. The following assertions are equivalent:
\begin{enumerate}
\item[1)] $\mathcal{U}$ is compactly generated;
\item[2)] there exists a sp-filtration $\phi:\mathbb{Z} \flecha \mathcal{P}(\Spec(R))$ such that $\mathcal{U}=\mathcal{U}_{\phi}$.
\end{enumerate}  
In such case, the corresponding t-structure $(\mathcal{U}_{\phi},\mathcal{U}_{\phi}^{\perp}[1])$ is described in terms of the sp-filtration by:
$$\mathcal{U}_{\phi}=\{X \in \D(R): \Supp(H^{j}(X)) \subseteq \phi(j), \text{ for all }j\in \text{ \Z} \}$$ $$\mathcal{U}_{\phi}^{\perp}=\{Y \in \D(R): \mathbf{R}\Gamma_{\phi(j)}(Y) \in \D^{>j}(R), \text{ for all }j\in \text{ \Z} \} $$
\end{theorem} 

\begin{notation}\rm{
For each sp-filtration $\phi$, we will denoted by $\Hp$ the heart associated to $\mathcal{U}_{\phi}$. Furthermore, we will denoted by $\tau^{\leq}_{\phi}$ the left truncation functor associated to the aisle $\mathcal{U}_{\phi}$ and by $\tau^{>}_{\phi}$ the right truncation functor. Moreover, we will denoted by $\phi_{\leq k}$ the sp-filtration given by:
\begin{enumerate}
\item[1)] $\phi_{\leq k}(j):=\phi(j)$, \hspace{0.5cm} if  \hspace{0.4 cm}$j\leq k$;
\item[2)] $\phi_{\leq k}(j):=\emptyset$, \hspace{0.9cm} if \hspace{0.4 cm}$j> k.$  
\end{enumerate} }
\end{notation}

\begin{definition}\label{def. finite sp-filtration}\rm{
Let $\phi:\mathbb{Z} \flecha \mathcal{P}(\Spec(R))$ be a sp-filtration. Let $n \leq s$ be integers. We say that $\phi$ is \emph{determined in the interval} $[n,s]$ if $\phi(j)=\phi(n)$ for all $j \leq n$, $\phi(n) \supsetneq \phi(n+1)$ and $\phi(s) \supsetneq \phi(s+1)=\emptyset$. \\

If the sp-filtration $\phi:\mathbb{Z} \flecha \mathcal{P}(\Spec(R))$ is determined in the interval $[n,s] \subseteq \mathbb{Z}$ we say that $\phi$ is \emph{finite of length}\index{filtration! length} $L(\phi):=n-s+1$. Note that, for any finite sp-filtration $\phi$ we have that $L(\phi)\geq 1$. } 
\end{definition}

We finish the section by recalling the following result of \cite{AJS} which is very useful for induction arguments (see \cite[Corollary 5.6]{AJS}).

\begin{proposition}\label{prop. coro AJS}
Let $\phi:\mathbb{Z} \flecha \mathcal{P}(\Spec(R))$ be a finite sp-filtration determined in the interval $[n,s] \subseteq \mathbb{Z}$. Then for each $X\in \D(R)$ there is a triangle of the form:

$$\xymatrix{H^{s}(\tau^{\leq}_{\phi}(X))[-s] \ar[r] & \tau^{>}_{\phi_{\leq s-1}}(X) \ar[r] & \tau^{>}_{\phi}(X) \ar[r]^{\hspace{0.5 cm}+} & }$$ 
\end{proposition}

\section{Questions tackled in this work}\label{sec. Questions}
In this section we bring together some questions that arise naturally from the previous results and some other questions for the compactly generated t-structures on $\D(R)$, where $R$ is a commutative Noetherian ring. We work on all of them in this thesis.\\

For the t-structure of Happel-Reiten-Smal\o:
\begin{enumerate}
\item[1)] Given a Grothendieck category $\G$ and a torsion pair in $\G$, $\te=(\T,\F)$, when is $\Ht$ a Grothendieck category?;

\item[2)] Given an associative ring with unity $R$ and a torsion pair $\te=(\T,\F)$ in $R$-Mod, when is $\Ht$ a module category?

\item[3)] Suppose that $\Ht$ is a module category. Is $\te$ an HKM torsion pair?
\end{enumerate}

For the first question, we got a partial answer which embraces a big amount of torsion pairs. However, we get, for the second question, an answer which covers all possible cases. On the other hand, we given a negative answer for the third question. \\

We tackle questions correspondents of 1 and 2 above for the compactly generated t-structures in the derived category of a commutative Noetherian ring $R$

\begin{enumerate}
\item[1$^{'}$)] Given a sp-filtration $\phi$ (finite, left bounded or arbitrary) of $\Spec(R)$, when is $\Hp$ a Grothendieck category?\\

Being more ambitious, we also ask:

\item[2$^{'}$)] Given a sp-filtration $\phi$ of $\Spec(R)$, when is $\Hp$ a module category?
\end{enumerate}

For the first question, we get a positive answer for the sp-filtrations which are left bounded. For the second question, we get an answer which covers all possible cases.

\chapter{Direct limits in the heart } 
In a large part of this work, we study when the heart of the t-structure of Happel-Reiten-Smal\o\hspace{0.1cm}  is a Grothendieck category or a module category. For the Grothendieck case, it is clear that the main difficulty lies in understanding when the heart is an AB5 abelian category. For this reason it is necessary to study in detail the direct limits in the heart and try to develop techniques that can be exploited in the hearts of arbitrary t-structures, as discussed in Chapter 5. \newline

All throughout this chapter, $\mathcal{G}$ is a Grothendieck category and $(\mathcal{D},?[1])$ is a triangulated category and we will fix a t-structure $(\mathcal{U},\mathcal{U}^{\perp}[1])$ in $\mathcal{D}$ and $\mathcal{H}:=\mathcal{U} \cap \mathcal{U}^{\perp}[1]$ will be its heart. We will denote by $\xymatrix{\tilde{H}:\mathcal{D} \ar[r] & \mathcal{H}}$ either of the naturally isomorphic functors $\tau_{\mathcal{U}} \hspace{0.05 cm}\circ \hspace{0.05 cm} \tau^{\mathcal{U}^{\perp}[1]}$ or $\tau^{\mathcal{U}^{\perp}[1]} \hspace{0.05 cm} \circ \tau_{\mathcal{U}}$ (see remarks \ref{properties t-structure}). This is in order to avoid confusion, in the case that $\mathcal{D}=\D(\mathcal{G})$, with the classical cohomological functor $\xymatrix{H=H^{0}:\D(\mathcal{G}) \ar[r] & \mathcal{G}}$.

\section{Direct limit properties} 
We start by giving some properties of the heart $\mathcal{H}$.

\begin{lemma}\label{lemma de adjunctions}
The following assertions hold:
\begin{enumerate}
\item[1)] If $X$ is an object of $\mathcal{U}$, the $\tilde{H}(X)\cong \tau^{\mathcal{U}^{\perp}}(X[-1])[1]$ and the assignment $X \rightsquigarrow \tilde{H}(X)$ defines an additive functor $\xymatrix{L:\mathcal{U} \ar[r] & \mathcal{H}}$ which is left adjoint to the inclusion $\xymatrix{j:\mathcal{H} \hspace{0.05cm}\ar@{^(->}[r] & \mathcal{U}}$.
\item[2)] If $Y$ is an object of $\mathcal{U}^{\perp}[1]$, then $\tilde{H}(Y)\cong \tau_{\mathcal{U}}(Y)$ and the assignment $Y \rightsquigarrow \tilde{H}(Y)$ defines an additive functor $\xymatrix{R: \mathcal{U}^{\perp}[1] \ar[r] & \mathcal{H}}$ which is right adjoint to the inclusion $\xymatrix{j:\mathcal{H} \hspace{0.05 cm}\ar@{^(->}[r] & \mathcal{U}^{\perp}[1]}.$
\end{enumerate}
\end{lemma}

\begin{proof}
Note that a sequence of morphisms 
$$\xymatrix{U \ar[r] & X[-1] \ar[r]^{\hspace{0.4 cm}g} & V \ar[r]^{h\hspace{0.2 cm}} & U[1]}$$
is a distinguished triangle if, and only if, the sequence
$$\xymatrix{U[1] \ar[r] & X \ar[r]^{g[1]\hspace{0.2cm}} & V[1] \ar[r]^{h[1]} & U[2]}$$
is so. Taking $X,U\in \mathcal{U}$ and $V\in \mathcal{U}^{\perp}$ we obtain that $\tau^{\mathcal{U}^{\perp}[1]}(X)=\tau^{\mathcal{U}^{\perp}}(X[-1])[1]$, and then the isomorphism $\tilde{H}(X)\cong \tau^{\mathcal{U}^{\perp}}(X[-1])[1]$ follows from the definition of $\tilde{H}$. On the other hand, if $Y\in \mathcal{U}^{\perp}[1]$ then, by definition of $\tilde{H}$, we get $\tilde{H}(Y)\cong \tau_{\mathcal{U}}(Y).$ \newline

The part of assertion 1 relative to the adjunction is dual to that of assertion 2 (see remarks \ref{properties t-structure}). We then prove the adjunction of assertion 1, which follows directly from the following chain of isomorphism, using the fact that 
$\xymatrix{\tau^{\mathcal{U}^{\perp}}:\mathcal{T} \ar[r] & \mathcal{U}^{\perp}}$ is left adjoint to the inclusion $\xymatrix{j_{\mathcal{U^{\perp}}}:\mathcal{U}^{\perp} \hspace{0.1cm}\ar@{^(->}[r]& \mathcal{T}}$: 
\begin{small}
$$\Hom_{\mathcal{U}}(U,j(Z))\cong \Hom_{\mathcal{T}}(U[-1],j(Z)[-1]) \cong \Hom_{\mathcal{U}^{\perp}}(\tau^{\mathcal{U}^{\perp}}(U[-1]),j(Z)[-1]) \cong \Hom_{\mathcal{H}}(L(U),Z)$$
\end{small}
\end{proof}

\vspace{0.3 cm}

The following result tells us that, in general, the heart is not very pathological

\begin{proposition}\label{AB3 t-structure}
Let $\mathcal{D}$ be a triangulated category which has coproducts (resp. products) and let $(\mathcal{U,U}^{\perp}[1])$ be a t-structure in $\mathcal{D}$. Then the heart $\mathcal{H}$ is an AB3 (resp. AB3*) abelian category.
\end{proposition}
\begin{proof}
It is a known fact and very easy to prove that if $\xymatrix{L:\mathcal{C} \ar[r] & \mathcal{C}^{'}}$ is a left adjoint functor and $\mathcal{C}$ has coproducts, then for each family $(C_i)_{i\in I}$ of objects of $\mathcal{C}$, the family $(L(C_i))_{i\in I}$ has a coproduct in $\mathcal{C}^{'}$ and one has an isomorphism $\underset{i\in I}{\coprod} L(C_i) \cong L(\underset{i \in I}{\coprod} C_i)$. \newline

Let now $(Z_i)_{i \in I}$ be a family of objects of $\mathcal{H}$. Since the counit $\xymatrix{L \circ j \ar[r] & 1_{\mathcal{H}}}$ is an isomorphism, we obtain that $(L \circ j)(Z_i)\cong Z_i$ for all $i\in I$, therefore, from the previous paragraph we get that the family $(Z_i)_{i\in I}$ has a coproduct in $\mathcal{H}$. The statement about products is dual to the one for coproducts.
\end{proof}

\begin{definition}\rm{
Let us assume that $\mathcal{D}$ has coproducts (resp. products) and that coproducts (resp. products) of triangles are triangles. The t-structure $(\mathcal{U},\mathcal{U}^{\perp}[1])$ is called \emph{smashing} \index{t-structure ! smashing}(resp. \emph{co-smashing}\index{t-structure ! co-smashing }) when $\mathcal{U}^{\perp}$ (resp. $\mathcal{U}$) is closed under taking coproducts (resp. products) in $\mathcal{D}$. Equivalently, when the left (resp. right) truncation functor $\xymatrix{\tau_{\mathcal{U}}: \mathcal{D} \ar[r] & \mathcal{U}}$ (resp. $\xymatrix{\tau^{\mathcal{U}^{\perp}}:\mathcal{D} \ar[r] & \mathcal{U}^{\perp}}$) preserver coproducts (resp. products).  }
\end{definition}

\begin{examples}\rm{
\begin{enumerate}
\item[1)] Every compactly generated t-structure in a triangulated category with coproducts is smashing.
\item[2)] If $\mathcal{T}\subseteq R$-Mod is a TTF class, then $(\mathcal{U}_{\mathbf{t}},\mathcal{U}_{\mathbf{t}}^{\perp}[1])$ is a co-smashing t-structure in $\D(R)$, where $\mathbf{t}$ is the right constituent pair of the TTF triple $(^{\perp}\mathcal{T}, \mathcal{T} ,\mathcal{T}^{\perp})$.
\end{enumerate}}
\end{examples}

\begin{proposition}\label{sufficient AB4}
Let $\mathcal{D}$ be a triangulated category that has coproducts (resp. products) and suppose that coproducts (resp. products) of triangles are triangles in $\D$. If $\mathcal{H}$ is closed under taking coproducts (resp. products) in $\D$, then $\mathcal{H}$ is an AB4 (resp. AB4*) abelian category. In particular, that happens when $(\mathcal{U},\mathcal{U}^{\perp}[1])$ is a smashing (resp. co-smashing) t-structure.
\end{proposition}

\begin{proof}
The proposition \ref{AB3 t-structure} shows that $\mathcal{H}$ is an AB3 (resp. AB3*) abelian category. Note that $\mathcal{U}$ (resp. $\mathcal{U}^{\perp}$) is closed under taking coproducts (resp. products) in $\mathcal{D}$. Then the final assertion follows automatically from the first part of the proposition and from the definition of smashing (resp. co-smashing) t-structure.  We just do the AB4 case since the AB4* one is dual. Let $(\xymatrix{0 \ar[r] & X_i \ar[r] & Y_i \ar[r] & Z_i \ar[r] & 0})_{i\in I}$ be a family of short exact sequences in $\mathcal{H}$. Accoding to \cite{BBD}, they come from triangles in $\mathcal{D}$. Then the exactness of coproducts in $\mathcal{D}$ gives a triangle
$$\xymatrix{\underset{i\in I}{\coprod}X_i \ar[r] &\underset{i \in I}{\coprod} Y_i \ar[r] &\underset{i \in I}{\coprod} Z_i \ar[r]^{\hspace{0.6 cm}+} &}$$  
where the three terms are in $\mathcal{H}$ since $\mathcal{H}$ is closed under taking coproducts in $\D$. We then get the desired short exact sequence in $\mathcal{H}$: 
$$\xymatrix{0 \ar[r] & \underset{i\in I}{\coprod}X_i \ar[r] &\underset{i \in I}{\coprod} Y_i \ar[r] &\underset{i \in I}{\coprod} Z_i \ar[r] & 0}$$ 
\end{proof}

\begin{example}\label{exam. Ht AB4}\rm{
Let $\te=(\T,\F)$ be a torsion pair in $\G$ and let $(\mathcal{U}_{\te},\mathcal{U}_{\te}^{\perp}[1])$ be the associated Happel-Reiten-Smal\o \ t-structure in $\D(\G)$. The heart $\Ht$ is an AB4 abelian category. }
\end{example}
\begin{proof}
From \cite[Exercises V.1]{S}, we obtain that $\F$ is closed under taking coproducts in $\G$. On the other hand, the classical cohomological functors $H^{m}:\D(\G) \flecha \G$ preserve coproducts since $\G$ is in particular an AB4 abelian category. It follows that if $(M_i)_{i \in I}$ is a family of objects in $\Ht$, then $\underset{i \in I}{\coprod}M_i$ is in $\Ht$ and, hence the example follows from the previous proposition.
\end{proof}

\begin{definition}\label{datum}\rm{
Let $\mathcal{X}$ be any full subcategory of $\mathcal{D}$. A \emph{cohomological datum}\index{cohomological datum} in $\mathcal{D}$ with respect to $\mathcal{X}$ is a pair $(H,r)$ consisting of a cohomological functor $\xymatrix{H:\mathcal{D} \ar[r] & \mathcal{A}}$, where $\mathcal{A}$ is an abelian category, and $r\in $ \Z $\cup \{+ \infty\}$ such that the family of functor \linebreak $(\xymatrix{H_{|_{\mathcal{X}}}^k:\mathcal{X} \hspace{0.05cm} \ar@{^(->}[r] & \mathcal{D} \ar[r]^{H^{k}} & \mathcal{A}})_{k<r}$ is conservative. That is, if $X\in \mathcal{X}$ and $H^{k}(X)=0,$ for all $k<r$, then $X=0$.}
\end{definition}

The following is an useful result inspired by \cite[Theorem 3.7]{CGM} (see also theorem \ref{teo. motived CGM}). It gives a sufficient condition for the heart to be an AB5 abelian category which, as we will see later, is sometimes also necessary.

\begin{proposition}\label{sufficient AB5 general}
Suppose that $\mathcal{D}$ has coproducts and let $(\xymatrix{H:\mathcal{D} \ar[r] & \mathcal{A}},r)$ be a cohomological datum in $\mathcal{D}$ with respect to the heart $\mathcal{H}=\mathcal{U} \cap \mathcal{U}^{\perp}[1]$. Suppose that $I$ is a small category such that $I$-colimits exist and are exact in $\mathcal{A}$. If, for each diagram $\xymatrix{X: I \ar[r] & \mathcal{H}}$ and each integer $k<r$, the canonical morphism 
$$\xymatrix{colim\hspace{0.02 cm} H^{k}(X_i) \ar[r] & H^{k}(colim_{\mathcal{H}} (X_i))}$$
is an isomorphism, then $I$-colimits are exact in $\mathcal{H}$.
\end{proposition}

\begin{proof}
By proposition \ref{AB3 t-structure}, we know that $\mathcal{H}$ is AB3 or equivalently cocomplete. Let us consider an $I$-diagram $\xymatrix{0 \ar[r] & X_i \ar[r] & Y_i \ar[r] & Z_i \ar[r] & 0}$ of short exact sequences in $\mathcal{H}$. For each $i\in I$, we use the cohomological condition of $H$ and get the following exact sequence in $\mathcal{A}$:
$$\xymatrix{\cdots \ar[r] & H^{n}(X_i) \ar[r] & H^{n}(Y_i) \ar[r] & H^{n}(Z_i) \ar[r] & H^{n+1}(X_i) \ar[r] & \cdots }$$

Now using the fact that $I$-colimits are exact in $\mathcal{A}$, we obtain the next exact sequence in $\mathcal{A}$
\begin{small}
$$\xymatrix{\cdots \ar[r] & colim \hspace{0.02 cm} H^{n}(X_i) \ar[r] & colim \hspace{0.02 cm} H^{n}(Y_i) \ar[r] & colim \hspace{0.02 cm} H^{n}(Z_i) \ar[r] & colim \hspace{0.02 cm} H^{n+1}(X_i) \ar[r] & \cdots }$$
\end{small}
On the other hand, by right exactness of colimits, we then get an exact sequence in $\mathcal{H}$:
$$\xymatrix{colim_{\mathcal{H}}(X_i) \ar[r]^{f} & colim_{\mathcal{H}}(Y_i) \ar[r]^{g} & colim_{\mathcal{H}}(Z_i) \ar[r] & 0}$$

We put $L:=Im(f)$ and then consider the two induced short exact sequences in $\mathcal{H}$:

\begin{small}
$$\xymatrix{0 \ar[r] & L \ar[r] & colim_{\mathcal{H}}(Y_i) \ar[r]^{g} & colim_{\mathcal{H}}(Z_i) \ar[r] & 0 \\0 \ar[r] & W \ar[r] & colim_{\mathcal{H}}(X_i) \ar[r]^{\hspace{0.6 cm}p} & L \ar[r] & 0 }$$
\end{small}

We view all given short exact sequences in $\mathcal{H}$ as triangles in $\mathcal{D}$ and using again the cohomological condition of $H$, we get the following commutative diagram in $\mathcal{A}$ with exact columns:
\begin{small}
$$\xymatrix @d {colim \hspace{0.02 cm} H^{k-1}(Y_i) \ar[r] \ar[d]_{\hspace{0.4 cm}\sim} & colim \hspace{0.02 cm} H^{k-1}(Z_i) \ar[r] \ar[d]_{ \hspace{0.4 cm}\sim } & colim \hspace{0.02 cm} H^{k}(X_i) \ar[r] \ar[d] & colim \hspace{0.02 cm} H^{k}(Y_i) \ar[r] \ar[d]_{\hspace{0.4 cm}\sim} & colim \hspace{0.02 cm} H^{k}(Z_i)  \ar[d]_{\hspace{0.4 cm} \sim} \\ H^{k-1}(colim_{\mathcal{H}}(Y_i)) \ar[r] & H^{k-1}(colim_{\mathcal{H}}(Z_i)) \ar[r] & H^{k}(L) \ar[r] & H^{k}(colim_{\mathcal{H}}(Y_i)) \ar[r] & H^{k}(colim_{\mathcal{H}}(Z_i)) &}$$
\end{small}
where the row arrow $\xymatrix{colim \hspace{0.02cm} H^{k}(X_i) \ar[r] & H^{k}(L)}$ is the composition: 
$$\xymatrix{colim \hspace{0.02 cm} H^{k}(X_i) \ar[r] & H^{k}(colim_{\mathcal{H}}(X_i)) \ar[r]^{\hspace{0.8 cm}H^{k}(p)} & H^{k}(L)}$$
for each $k\in \mathbb{Z}$. Further for $k<r$, in principle all the vertical arrows except the central one are isomorphisms. Then also the central one is an isomorphism, which implies that $H^{k}(p)$ is an isomorphisms since, by hypothesis, the canonical morphism 
$$\xymatrix{colim \hspace{0.02 cm} H^{k}(X_i) \ar[r] & H^{k}(colim_{\mathcal{H}}(X_i))}$$
is an isomorphism. We then get that $H^{k}(W)=0$, for all $k<r$, which implies that $W=0$ due to definition \ref{datum}. Therefore $p$ is an isomorphism.
\end{proof}

\vspace{0.3 cm}

Taking into account that we want to understand direct limits in the heart, we will use the following version of last proposition.

\begin{corollary}\label{All functor homology}
Let $\mathcal{G}$ be a Grothendieck category and $(\mathcal{U},\mathcal{U}^{\perp}[1])$ be a t-structure on $\D(\mathcal{G})$. If the classical homological functors $\xymatrix{H^{m}:\mathcal{H} \ar[r] & \mathcal{G}}$ preserve direct limits, for all $m\in $\Z, then $\mathcal{H}$ is AB5.
\end{corollary}
\begin{proof}
It is clear that $\D(\mathcal{G})$ has coproducts and coproducts of triangles are triangles. By proposition \ref{AB3 t-structure} we know that $\mathcal{H}$ is an AB3 category. The result follows from the previous proposition, bearing in mind that $\xymatrix{(H^{0}:\D(\mathcal{G}) \ar[r] & \mathcal{G}}, \infty)$ is a cohomological datum with respect to $\mathcal{H}$.  
\end{proof}

\begin{example}\label{example datum}\rm{
Let $\D$ have coproducts, let $(\mathcal{U},\mathcal{U}^{\perp}[1])$ be a compactly generated t-structure in $\D$ generated by a set $\mathcal{S}$ of compact objects. Then the functor 
$$\xymatrix{H:=\underset{S \in \mathcal{S}}{\coprod} \Hom_{\mathcal{T}}(S,?):\mathcal{D} \ar[r] & Ab}$$
is a cohomological functor. Moreover, the pair $(H,1)$ is a cohomological datum with respect to heart $\mathcal{H}=\mathcal{U} \cap \mathcal{U}^{\perp}[1].$}
\end{example}

\begin{lemma}
Suppose that $\mathcal{D}$ has coproducts and that $(\mathcal{U},\mathcal{U}^{\perp}[1])$ is a compactly generated t-structure in $\mathcal{D}.$ Then $\mathcal{U}^{\perp}$ is closed under taking Milnor colimits.
\end{lemma}
\begin{proof}
Let $\mathcal{D}^{c}$ be the full subcategory of compact objects and take $C\in \mathcal{D}^c$. We claim that, for any diagram in $\mathcal{D}$ of the form:

$$\xymatrix{(\ast) & X_{0} \ar[r]^{f_1} & X_1 \ar[r]^{f_2} & X_2 \ar[r]^{f_3} & \cdots}$$
we have an isomorphism $\Hom_{\mathcal{D}}(C,Mcolim(X_n)) \cong \varinjlim{\Hom_{\mathcal{D}}(C,X_n)}$. To see that, let us consider the triangle 
$$\xymatrix{\underset{n \geq 0}{\coprod} X_n \ar[rr]^{1-shift} & & \underset{n \geq 0}{\coprod} X_n \ar[r] & M colim(X_n) \ar[r]^{\hspace{1 cm}+} &}$$
We have that the following diagram is commutative:
$$\xymatrix{\underset{n \geq 0}{\coprod} \Hom_{\mathcal{D}} (C,X_n) \ar[d]^{\wr}\ar[rr]^{1-(shift)_{\ast}} && \underset{n \geq 0}{\coprod} \Hom_{\mathcal{D}}(C,X_n) \ar[d]^{\wr}\\ \Hom_{\mathcal{D}}(C, \underset{n \geq 0}{\coprod} X_n) \ar[rr]^{(1-shift)_{\ast}} && \Hom_{\mathcal{D}}(C, \underset{n \geq 0}{\coprod} X_n) }$$

It is a known fact and very easy to prove that $1-(shift)_{\ast}$ is a monomorphism in $Ab$. Therefore, we get an exact sequence in $Ab$ of the form:
\begin{small}
$$\xymatrix{\cdots \ar[r]^{0 \hspace{1.3 cm}} & \Hom_{\mathcal{D}}(C,\underset{n \geq 0}{\coprod} X_n) \ar[r] & \Hom_{\mathcal{D}}(C, \underset{n \geq 0}{\coprod} X_n) \ar[r] & \Hom_{\mathcal{D}}(C,M colim(X_n)) \ar@(d,u)[dl]_{0} \\ &&\Hom_{\mathcal{D}}(C[-1], \underset{n \geq 0}{\coprod} X_n) \ar[r] & \Hom_{\mathcal{D}}(C[-1],\underset{n \geq 0}{\coprod} X_n) \ar[r] & \cdots}$$
\end{small}
This proves the claim since $Coker(1-(shift)_{\ast})=\varinjlim{\Hom_{\mathcal{D}}(C,X_n)}$. Now if all the $X_n$ are in $\mathcal{U}^{\perp}$, then for each $C\in \mathcal{U}\cap \mathcal{D}^{c}$ and each $k\geq 0$, we have $\Hom_{\mathcal{D}}(C[k],Mcolim(X_n))\cong \varinjlim \Hom_{\mathcal{D}}(C[k],X_n)=0.$ This shows that $Mcolim (X_n) \in \mathcal{U}^{\perp}$ since the t-structure is compactly generated (see definition \ref{t-structure compactly generated}).
\end{proof}

\vspace{0.3 cm}

In Chapter 5 we will see that if $R$ is a Noetherian commutative ring, then a lot of compactly generated t-structures in $\D(R)$ have an AB5 heart. For an arbitrary compactly generated t-structure, the following is our most general result.


\begin{theorem}\label{teo. H compactly general}
Suppose that $\mathcal{D}$ has coproducts and let $(\mathcal{U}, \mathcal{U}^{\perp}[1])$ be a compactly generated t-structure in $\mathcal{D}.$ Countable direct limits are exact in $\mathcal{H}=\mathcal{U}\cap \mathcal{U}^{\perp}[1]$.
\end{theorem}

\begin{proof}
Let $I$ be a countable directed set. Then there is an ascending chain of finite directed subposets $I_{0} \subset I_1 \subset I_2 \dots $ such that $I=\underset{n \in \mathbb{N}}{\cup} I_n$ (see \cite[Lemma 1.6]{AR}). If we put $i_{n}=max\{I_n\}$, for all $n\in \mathbb{N}$, we clearly have $i_n \leq i_{n+1}$, for all $n$ and $J=\{i_0,i_1,\dots, i_n, \dots\}$ is a cofinal subset of $I$ which is isomorphic to $\mathbb{N}$ as an ordered set. For any category $\mathcal{C}$ with direct limits, the next diagram is commutative
$$\xymatrix{[J,\mathcal{C}] \ar[r]^{colim_{J}} & colim_{J} \mathcal{C}\\ [I,\mathcal{C}] \ar[u]^{restriction} \ar[ur]_{colim_{I}}}$$

We can the assume that $I=\mathbb{N}$. But the previous lemma tells us that if $\xymatrix{X:\mathbb{N} \ar[r] & \mathcal{H}}$ \linebreak $(n \rightsquigarrow X_n)$ is any diagram, then the triangle
$$\xymatrix{\underset{n \geq 0}{\coprod} X_n \ar[rr]^{1-shift} && \underset{n \geq 0}{\coprod}X_n \ar[r] & Mcolim(X_n) \ar[r]^{\hspace{1cm}+} & }$$
lives in $\mathcal{H}$ and, hence, it is an exact sequence in this abelian category. Therefore $Mcolim(X_n)\cong \varinjlim_{\mathcal{H}}{X_n}.$ If now $\mathcal{S}$ is any set of compact generators of the aisle $\mathcal{U}$ and we take the cohomological functor $\xymatrix{H:= \underset{S\in \mathcal{S}}{\coprod}\Hom_{\mathcal{D}}(S,?):\mathcal{D} \ar[r] & Ab}$, then by  the proof of previous lemma, the induced map $\xymatrix{\varinjlim{H^{k}(X_n)} \ar[r] & H^{k}(\varinjlim_{\mathcal{H}}{X_n})}$ is an isomorphism, for each $k\in $\Z, because in Ab we have an isomorphism 
$$\underset{S \in \mathcal{S}}{\coprod} \varinjlim \Hom_{\mathcal{D}}(S,X_n) \cong \varinjlim \underset{S\in \mathcal{S}}{\coprod} \Hom_{\mathcal{D}}(S,X_n)$$
The result now follows from proposition \ref{sufficient AB5 general} and example \ref{example datum}.
\end{proof}

\section{The t-structure of Happel-Reiten-Smal\o}\label{section AB4 Ht}
In this section we study in detail, the direct limits of the heart associated to the t-structure of Happel-Reiten-Smal\o. See Chapter 2 for previous results in this direction. Our strategy of attack is to focus on the stalks of the heart, and this is a posteriori the key of the problem. \newline

All troughout this section $\mathcal{G}$ is a Grothendieck category and $\mathbf{t}=(\mathcal{T,F})$ is a torsion pair in $\mathcal{G}$. From example \ref{exam. Ht AB4}, we know that $\mathcal{H}_{\mathbf{t}}$ is an AB4 category. On the other hand, $\mathcal{T}$ is closed under taking direct limits in $\mathcal{G}$, while $\mathcal{F}$ need not be so. 
In the rest of the chapter unadorned direct limits mean direct limits in $\mathcal{G}$, while we will denote by $\varinjlim_{\mathcal{H}_{\mathbf{t}}}$ the direct limit in $\mathcal{H}_{\mathbf{t}}$.

\begin{lemma}\label{exactness of H}
Let $\mathbf{t}=(\mathcal{T,F})$ be a torsion pair in the Grothendieck category $\mathcal{G}$, let $(\mathcal{U}_{\mathbf{t}},\mathcal{U}_{\mathbf{t}}^{\perp}[1])$ be its associated t-structure and let $\mathcal{H}_{\mathbf{t}}$ be its heart. The functor $\xymatrix{H^{0}:\mathcal{H}_{\mathbf{t}} \ar[r] & \mathcal{G}}$ is right exact while the functor $\xymatrix{H^{-1}:\mathcal{H}_{\mathbf{t}} \ar[r] & \mathcal{G}}$ is left exact. Both of them preserve coproducts.
\end{lemma}
\begin{proof}
The functor $H^{k}$ vanishes on $\mathcal{H}_{\mathbf{t}}$, for each $k\ne -1,0$. By applying now the long exact sequence of homologies to any short exact sequence in $\mathcal{H}_{\mathbf{t}}$ the right (resp. left) exactness of $H^{0}$ (resp. $H^{-1}$) follows immediately. Since coproducts in $\mathcal{H}_{\mathbf{t}}$ are calculated as in $\D(\mathcal{G})$, the fact that $H^{0}$ and $H^{-1}$ preserve coproducts is clear, because in particular $\mathcal{G}$ is an AB4 abelian category.
\end{proof}
\vspace{0.3cm}

The following is the crucial result for our purposes.

\begin{proposition}\label{description of stalk}
Let $\mathbf{t}=(\mathcal{T,F})$ be a torsion pair in the Grothendieck category $\mathcal{G}$ and let $\mathcal{H}_{\mathbf{t}}$ be the heart of the associated t-structure in $\D(\mathcal{G})$. The following assertions hold:
\begin{enumerate}
\item[1)] If $(M_i)_{i\in I}$ is a direct system in $\mathcal{H}_{\mathbf{t}}$, then the induced morphism 
$$\xymatrix{\varinjlim{H^{k}(M_i)} \ar[r] & H^{k}(\varinjlim_{\mathcal{H}_{\mathbf{t}}}{M_i})}$$
is an epimorphism, for $k=-1$, and an isomorphism, for $k \ne -1$.

\item[2)] If $(F_i)_{i\in I}$ is a direct system in $\mathcal{F}$, there is an isomorphism in $\mathcal{H}_{\mathbf{t}}$
$$\xymatrix{(\frac{\varinjlim{F_i}}{t(\varinjlim{F_i})})[1] \cong \limite (F_{i}[1]) }$$

\item[3)] If $(T_i)_{i \in I}$ is a direct system in $\mathcal{T}$, then the following conditions hold true:
\begin{enumerate}
\item[a)] The induced morphism $\xymatrix{\limite (T_i[0]) \ar[r] & (\varinjlim{T_i})[0]}$ is an isomorphism in $\mathcal{H}_{\mathbf{t}}$;
\item[b)] The kernel of the colimit-defining morphism $\xymatrix{f:\underset{i\leq j}{\coprod}T_{ij} \ar[r] & \underset{i\in I}{\coprod}T_i}$ is in $\mathcal{T}$. 
\end{enumerate}
\end{enumerate}
\end{proposition} 
\begin{proof}
1) It essentially follows from \cite[Corollary 3.6]{CGM}, however we give a proof for the sake of completeness. By lemma \ref{exactness of H}, we know that $\xymatrix{H^{0}:\Ht \ar[r] & \mathcal{G}}$ preserves colimits, in particular direct limits. Let $\xymatrix{f:\underset{i\leq j}{\coprod}M_{ij} \ar[r] & \underset{i\in I}{\coprod}M_i}$ be the associated colimit-defining morphism and denote by $W$ the image of $f$ in $\Ht$.  Applying the exact sequence of homology to the exact sequence $\xymatrix{0 \ar[r] & W \ar[r] & \underset{i \in I }{\coprod} M_i \ar[r]^{p\hspace{0.4cm}} & \limite{M_i} \ar[r] & 0}$, we have the next sequence exact in $\mathcal{G}$
$$\xymatrix{0 \ar[r] & H^{-1}(W) \ar[r] & H^{-1}(\underset{i\in I}{\coprod}M_i) \ar[r]^{H^{-1}(p)\hspace{0.3cm}} & H^{-1}(\limite{M_i}) \ar@(d,u)[dlll] \\ H^{0}(W) \ar[r] & H^{0}(\underset{i\in I}{\coprod}M_i) \ar[r] & H^{0}(\limite{M_i}) \ar[r] &0 &}$$
By \cite[Lemma 3.5]{CGM}, we know that $H^{-1}(p)$ is an epimorphism. But this morphism is the composition
$$\xymatrix{\underset{i \in I}{\coprod} H^{-1}(M_i) \ar@{>>}[r] & \varinjlim{H^{-1}(M_i)} \ar[r]^{can \hspace{0.2 cm}} & H^{-1}(\limite{M_i})}.$$
Therefore the second arrow is also an epimorphism. \\

In order to prove the remaining assertions, we first consider any direct system $(M_i)_{i\in I}$ in $\Ht$ and the associated triangle given by the colimit-defining morphism in $\D(\mathcal{G})$:
$$\xymatrix{\underset{i \leq j}{\coprod}M_{ij} \ar[r]^{f} & \underset{i \in I}{\coprod}M_i \ar[r]^{\hspace{0.2 cm}q} & Z \ar[r]^{\hspace{0.1 cm}+} & } $$
We claim that if $\xymatrix{v:Z\ar[r] & L:=\limite(M_i)}$ is a morphism fitting in a triangle \linebreak $\xymatrix{N[1] \ar[r] & Z \ar[r]^{v} & L \ar[r]^{+} & }$ with $N\in \Ht$, then $H^{-1}(v)$ induces an isomorphism \linebreak $\xymatrix{H^{-1}(Z)/t(H^{-1}(Z)) \ar[r]^{\hspace{1cm}\sim} & H^{-1}(L)}$. In fact, by \cite{BBD} we have a diagram:
$$\xymatrix{&&N[1] \ar[d] & \\ \underset{i \leq j}{\coprod}M_{ij} \ar[r]^{f} & \underset{i \in I}{\coprod}M_i \ar[r]^{\hspace{0.2 cm}q} & Z \ar[r]^{\hspace{0.1 cm}+} \ar[d]^{v}& \\ && L \ar[d]^{+}& \\ &&&}$$

where the row and the column are triangles in $\D(\G)$, $N\cong Ker_{\Ht}(f)$, and $Coker_{\Ht}(f)=L=\limite(M_i)$. From the sequence of homologies applied to the column triangle, we get an exact sequence
$$\xymatrix{0 \ar[r] & H^{0}(N) \ar[r] & H^{-1}(Z) \ar[r]^{g} & H^{-1}(L) \ar[r] &0 }$$
It then follows that $H^{0}(N)\cong t(H^{-1}(Z))$ and $H^{-1}(Z)/t(H^{-1}(Z)) \cong H^{-1}(L).$ We now pass to prove the remaining assertions. \\

2) Note that, by assertion 1, we have $H^{0}(\limite{F_i[1]})=0$. On the other hand, when taking $M_i=F_i[1]$ in the last paragraph, the complex $Z$ can be identified with $Cone(f)[1]$, where $\xymatrix{f:\underset{i\leq j}{\coprod}F_{ij} \ar[r] & \underset{i\in I}{\coprod}{F_i}}$ is the colimit-defining morphism. Then we have $H^{-1}(Z)\cong H^{0}(Cone(f))\cong \varinjlim{F_i}$, which, by the last paragraph, implies that $\limite{F_i[1]}\cong \frac{\varinjlim{F_i}}{t(\varinjlim{F_i})}[1]$. \\

3) a) From assertion 1) we get that $\xymatrix{0=\varinjlim{H^{-1}(T_i[0])} \ar[r] & H^{-1}(\limite{(T_i[0])})}$ is an epimorphism. In particular we have an isomorphism $\limite{(T_i[0])}\cong H^{0}(\limite{(T_i[0])})[0]$. But, by lemma \ref{exactness of H}, $H^{0}$ preserves direct limits and then the right term of this isomorphism is $(\varinjlim{T_i})[0].$

b) Let us consider the induced triangle $\xymatrix{\underset{i\leq j}{\coprod}T_{ij} \ar[r]^{f} & \underset{i \in I}{\coprod}T_{i} \ar[r]^{\hspace{0.2cm}q}& Z \ar[r]^{+} &}$ in $\D(\mathcal{G})$. Without loss of generality, we identify $Z$ with the cone of $f$ in $\mathcal{C(G)}$, so that $H^{-1}(Z)=Ker(f)$. But then, by our already proved claim made after proving assertion 1 and by the previous paragraph, we get an isomorphism 
\begin{center}
$\frac{Ker(f)}{t(Ker(f))} \iso H^{-1}(\limite{T_i[0]})=0$.
\end{center}
\end{proof}

\vspace{0.3 cm}

The next lemma is an important tool in order to achieve our goals.

\begin{lemma}\label{clave para AB5}
Let us assume that $\Ht$ is an AB5 abelian category and let
$$\xymatrix{ 0 \ar[r] & F_i \ar[r] & F_{i}^{'} \ar[r]^{w_i} & T_i \ar[r] & T_i^{'} \ar[r] &0 }$$
be a direct system of exact sequences in $\mathcal{G}$, with $F_i,F_{i}^{'}\in \mathcal{F}$ and $T_i,T_{i}^{'}\in \mathcal{T}$ for all $i\in I.$ Then the induced morphism $\xymatrix{w:=\varinjlim{w_i}:\varinjlim{F_{i}^{'}} \ar[r] & \varinjlim{T_i}}$ vanishes on $t(\varinjlim{F_{i}^{'}}).$
\end{lemma}
\begin{proof}
Put $K_i:=Im(w_i)$ and consider the following pullback diagram and pushout diagram, respectively:
$$\xymatrix{0 \ar[r] & F_i \ar[r] \ar@{=}[d] & *+<0.9em>{\tilde{F}_{i}} \ar[r] \ar@{^(->}[d] \pushoutcorner & *+<0.9em>{t(K_i)} \ar@{^(->}[d] \ar[r] & 0 \\ 0 \ar[r] & F_i \ar[r] & F_{i}^{'} \ar[r] & K_i \ar[r] & 0}$$
$$\xymatrix{0 \ar[r] & K_i \ar[r] \ar@{>>}[d] & T_i \ar[r] \ar[d] & T_{i}^{'} \ar[r] \ar@{=}[d]& 0 \\ 0 \ar[r] & \frac{K_i}{t(K_i)} \ar[r] & \tilde{T}_{i} \ar[r] \pullbackcorner & T_{i}^{'} \ar[r]& 0}$$
The top row of the first diagram and the bottom row of the second one give exact sequences in $\Ht$
$$\xymatrix{0 \ar[r] & t(K_i)[0] \ar[r] & F_{i}[1] \ar[r] & \tilde{F}_{i}[1] \ar[r] & 0\\ 0 \ar[r] & \tilde{T}_{i}[0] \ar[r] & T_{i}^{'}[0] \ar[r] & \frac{K_i}{t(K_i)}[1] \ar[r] & 0}$$ 
which give rise to direct systems of short exact sequences in $\Ht$. Using now the AB5 condition of $\Ht$ and proposition \ref{description of stalk}, we get exact sequences in $\Ht$:
$$\xymatrix{0 \ar[r] & (\varinjlim{t(K_i)})[0] \ar[r] & \frac{\varinjlim{F_{i}}}{t(\varinjlim{F_i})}[1] \ar[r] & \frac{\varinjlim{\tilde{F}_{i}}}{t(\varinjlim{\tilde{F}_{i}})}[1] \ar[r] & 0\\ 0 \ar[r] & (\varinjlim{\tilde{T}_{i}})[0] \ar[r] & (\varinjlim{T_{i}^{'}})[0] \ar[r] &  \frac{\varinjlim{K_i/t(K_i)}}{t(\varinjlim{K_i/t(K_i)})}[1] \ar[r] & 0}$$
which necessarily come from exact sequences in $\G$:
$$\xymatrix{0 \ar[r]  & \frac{\varinjlim{F_{i}}}{t(\varinjlim{F_i})} \ar[r] & \frac{\varinjlim{\tilde{F}_{i}}}{t(\varinjlim{\tilde{F}_{i}})}\ar[r] & \varinjlim{t(K_i)} \ar[r] & 0 & (\ast)\\ 0 \ar[r]   &  \frac{\varinjlim{K_i/t(K_i)}}{t(\varinjlim{K_i/t(K_i)})} \ar[r]&  \varinjlim{\tilde{T}_{i}}\ar[r]& \varinjlim{T_{i}^{'}}  \ar[r] & 0 & (\ast \ast)}$$
These sequences are obviously induced from applying the direct limit functor to the row of the first diagram and the bottom row of the second diagram above, respectively. Due to the fact that $\G$ is a Grothendieck category, we then have the following commutative diagrams in $\G$:
$$\xymatrix{0 \ar[r] & \varinjlim{F_i} \ar[r] \ar@{>>}[d] & \varinjlim{\tilde{F}_{i}} \ar[r] \ar@{>>}[d] & \varinjlim{t(K_i)} \ar[r] \ar@{=}[d] &0 \\ 0 \ar[r] & \frac{\varinjlim{F_{i}}}{t(\varinjlim{F_i})} \ar[r] & \frac{\varinjlim{\tilde{F}_{i}}}{t(\varinjlim{\tilde{F}_{i}})}\ar[r] & \varinjlim{t(K_i)} \ar[r] & 0 & (\ast) }$$
$$\xymatrix{0 \ar[r] & \varinjlim{K_i/t(K_i)} \ar@{>>}[d]  \ar[r]&  \varinjlim{\tilde{T}_{i}} \ar[r] \ar@{=}[d]&  \varinjlim{T_{i}^{'}} \ar@{=}[d]\ar[r] & 0 \\ 0 \ar[r] &  \frac{\varinjlim{K_i/t(K_i)}}{t(\varinjlim{K_i/t(K_i)})} \ar[r] &\varinjlim{\tilde{T}_{i}}  \ar[r]& \varinjlim{T_{i}^{'}} \ar[r] & 0 & (\ast \ast)}$$
Thus $t(\varinjlim{F_i})\cong t(\varinjlim{\tilde{F}_{i}})$ and $\varinjlim{K_i/t(K_i)}\in \mathcal{F}$. On the other hand, from the first diagram of this proof, we get the following exact sequence in $\G$
$$\xymatrix{0 \ar[r] &\varinjlim{\tilde{F}_{i}} \ar[r]^{\iota} & \varinjlim{F_{i}^{'} }\ar[r] & \varinjlim{K_i/t(K_i)} \ar[r] & 0}$$
Using snake lemma, we get an exact sequence $\xymatrix{0 \ar[r] & \Ker((1:t)(\iota)) \ar[r] & \Coker(t(\iota)) \ar[r]^{0} & \varinjlim \frac{K_i}{t(K_i)}},$ where the last arrow is zero since $\Coker(t(\iota))\in \T$ and $\varinjlim \frac{K_i}{t(K_i)}\in \F$. But then $\Coker(t(\iota))\in \T \cap \F=0.$ Consequently $t(\iota)$ is an isomorphism, and thus $t(\varinjlim{\tilde{F_{i}}}) \cong t(\varinjlim{F_{i}^{'}}).$  The result follows from the fact that $\Ker(w)=\varinjlim{F_i}$ and $t(\varinjlim{F_i})\cong t(\varinjlim{F_{i}^{'}})$.
\end{proof}

\vspace{0.3cm}

The following is the first main result of the chapter.

\begin{theorem}\label{caracterizacion AB5}
Let $\G$ be a Grothendieck category, let $\mathbf{t}=(\mathcal{T,F})$ be a torsion pair in $\G$, let $(\mathcal{U}_{\mathbf{t}},\mathcal{U}_{\mathbf{t}}^{\perp}[1])$ be its associated t-structure in $\D(\G)$ and let $\Ht=\mathcal{U}_{\mathbf{t}}\cap \mathcal{U}_{\mathbf{t}}^{\perp}[1]$ be the heart. The following assertions are equivalent:
\begin{enumerate}
\item[1)] $\Ht$ is an AB5 abelian category;
\item[2)] $\mathcal{F}$ is closed under taking direct limits in $\G$ and if $(M_i)_{i_\in I}$ is a direct system in $\Ht$, with 
$$\xymatrix{\underset{i \leq j}{\coprod} M_{ij} \ar[r]^{f} & \underset{i \in I}{\coprod}M_i \ar[r] & Z \ar[r]^{+}&}$$
the triangle in $\D(\G)$ afforded by the associated colimit-defining morphism, then the composition $\xymatrix{\varinjlim{H^{-1}(M_i)} \ar[r] & H^{-1}(Z) \ar@{>>}[r]^{can \hspace{1.2cm}} & H^{-1}(Z)/t(H^{-1}(Z))}$ is a monomorphism;
\item[3)] $\mathcal{F}$ is closed under taking direct limits in $\G$ and, for each direct system $(M_i)_{i\in I},$ the canonical morphism $\xymatrix{\varinjlim{H^{-1}(M_i)} \ar[r] & H^{-1}(\limite {M_i})}$ is a monomorphism;
\item[4)] The classical cohomological functors $\xymatrix{H^{m}:\Ht \ar[r] & \G}$ preserve direct limits, for all $m\in \mathbb{Z}$.
\end{enumerate}
\end{theorem} 
\begin{proof}
1) $\Longrightarrow $ 2) Let $(M_i)_{i \in I}$ be a direct system in $\Ht$. We then get an induced direct system of short exact sequences in $\Ht$
$$\xymatrix{0 \ar[r] & H^{-1}(M_i)[1] \ar[r] & M_i \ar[r] & H^{0}(M_i)[0] \ar[r] & 0}$$
From the AB5 condition of $\Ht$ and proposition \ref{description of stalk} we get an exact sequences in $\Ht$
$$\xymatrix{0 \ar[r] & \frac{\varinjlim{H^{-1}(M_i)}}{t(\varinjlim{H^{-1}(M_i)})}[1] \ar[r] & \limite{M_i} \ar[r] & (\varinjlim{H^{0}(M_i)[0]}) \ar[r] & 0}$$

which gives, by taking homologies, an isomorphism $\frac{\varinjlim{H^{-1}(M_i)}}{t(\varinjlim{H^{-1}(M_i)})} \iso H^{-1}(\limite M_i)$. This together with proposition \ref{description of stalk} and its proof, give that the canonical morphism $\xymatrix{\varinjlim{H^{-1}(M_i)} \ar[r] & H^{-1}(\limite{M_i})}$ induces an isomorphism: 
\begin{center}
$\frac{\varinjlim{H^{-1}(M_i)}}{t(\varinjlim{H^{-1}(M_i)})} \iso H^{-1}(\limite M_i)\cong H^{-1}(Z)/t(H^{-1}(Z))$
\end{center}
Therefore the proof of this implication is reduced to prove that $\mathcal{F}$ is closed under taking direct limits. \\

Let $(F_i)_{i\in I}$ be any direct system in $\mathcal{F}$. For each $j\in I$, let us denote by $\gamma_{j}$ the composition $\xymatrix{F_j \ar[r] & \varinjlim{F_i} \ar@{>>}[r]^{p \hspace{0.15cm}} & \frac{\varinjlim{F_i}}{t(\varinjlim{F_i})}}$, where the morphisms are the obvious ones. We consider the following commutative diagram:
$$\xymatrix{0 \ar[r] & \text{Ker}(\gamma_{j}) \ar[r] \ar@{-->}[d] & F_j \ar[r] \ar[d]^{\iota_j}& *+<0.8em> {\text{Im}(\gamma_{j})} \ar[r] \ar@{^(->}[d] & 0 \\ 0 \ar[r] & t(\varinjlim{F_i}) \ar[r] & \varinjlim{F_i} \ar@{>>}[r]^{p} & \frac{\varinjlim{F_i}}{t(\varinjlim{F_i})} \ar[r] & 0  }$$
where the $\iota_j$ are the canonical morphisms into the direct limit. Using now the AB5 condition of $\G$, we get the following commutative diagram:
$$\xymatrix{0 \ar[r] & \varinjlim \text{Ker}(\gamma_{j}) \ar[r] \ar@{^(->}[d] & \varinjlim F_j \ar[r] \ar@{=}[d]& *+<0.8em> { \varinjlim \text{Im}(\gamma_{j})} \ar[r] \ar@{^(->}[d] & 0 \\ 0 \ar[r] & t(\varinjlim{F_i}) \ar[r] & \varinjlim{F_i} \ar@{>>}[r]^{p} & \frac{\varinjlim{F_i}}{t(\varinjlim{F_i})} \ar[r] & 0  }$$
This implies that the vertical morphisms are isomorphisms and therefore, $(\text{Ker}(\gamma_{j}))_{j \in I}$ is a direct system in $\mathcal{F}$ such that $\varinjlim \text{Ker}(\gamma_{j})=t(\varinjlim{F_i})$ is in $\mathcal{T}$. Put $T:=t(\varinjlim{F_i})$ and, for each $i\in I$, consider the canonical map $\xymatrix{u_i:\text{Ker}(\gamma_j) \ar[r] & T}$ into the direct limit. We then get a direct system of exact sequences in $\G$:
$$\xymatrix{0 \ar[r] & \Ker(u_i) \ar[r] & \Ker(\gamma_i) \ar[r]^{ \hspace{0.5cm}u_i} & T \ar[r] & \Coker(u_i) \ar[r] & 0}$$
From lemma \ref{clave para AB5} we then get that the map $\xymatrix{u:=\varinjlim{u_i}:\varinjlim{\Ker(\gamma_i)} \ar[r] & T}$ vanishes on $t(\varinjlim{\Ker(\gamma_i)})$. This implies that $u=0$ since $\varinjlim{\Ker(\gamma_i)}=T$ is in $\mathcal{T}$. But $u$ is an isomorphism by definition of the direct limit. It then follows that $T=0$, so that $\varinjlim{F_i}\in \mathcal{F}$ as desired. \\

2) $\Longleftrightarrow$ 3) $\Longrightarrow $ 4) They follow directly from proposition \ref{description of stalk}, its proof and the fact that all complexes in $\Ht$ have homology concentrated in degrees -1 and 0. \\

4) $\Longrightarrow$ 1) It follows from corollary \ref{All functor homology}.
\end{proof}

\vspace{0.3 cm}

In order to give conditions for $\Ht$ to be a Grothendieck category, we need some auxiliary lemmas.

\begin{lemma}\label{limites directos en C}
Suppose that $\mathcal{F}$ is closed taking direct limits in $\G$. Let $(M_i)_{i\in I}$ be a direct system in $\C(\G)$, where $M_i$ is a complex concentrated in degrees -1,0, for all $i\in I$. If $M_i\in \Ht$, for all $i\in I$, then the canonical morphism
\begin{center}
$\xymatrix{\limite M_i \ar[r] & \varinjlim_{\C(\G)}{M_i}}$
\end{center}
is an isomorphism in $\D(\G).$
\end{lemma}
\begin{proof}
The hypothesis on $\mathcal{F}$ together with the AB5 condition of $\G$, imply that $\varinjlim_{\C(\G)}{M_i} \in \Ht$ and, hence, there is a canonical morphism $\xymatrix{g:\limite{M_i} \ar[r] & \varinjlim_{\C(\G)}{M_i}}$ as indicated in the statement. We then get a composition of morphism in $\G$:
\begin{center}
$\xymatrix{\varinjlim{H^{k}(M_i)} \ar[r] & H^{k}(\limite M_i) \ar[rr]^{H^{k}(g)} & & H^{k}(\varinjlim_{\C(\G)}{M_i})}$
\end{center}
for each $k\in \mathbb{Z}$. But all complexes involved have homology concentrated in degrees -1,0 and, by exactness of $\varinjlim$ in $\G$, we know that the last composition of morphisms is an isomorphism. So the first arrow is a monomorphism for each $k\in \mathbb{Z}$. It follows from proposition \ref{description of stalk} that both arrows in the composition are isomorphisms, for all $k \in \mathbb{Z}$. Therefore $g$ is an isomorphism in $\D(\G)$.
\end{proof} 

\begin{remark}\rm{
Note that if $G$ is a generator of $\Ht$, then for each $T\in \mathcal{T}$ there is a sequence in $\Ht$ of the form
$$\xymatrix{G^{(\beta)} \ar[r] & G^{(\alpha)} \ar[r] & T[0] \ar[r] &0}$$
Using the fact that $\xymatrix{H^{0}:\Ht \ar[r] & \G} $ is a right exact functor that commutes with coproducts, we have $\mathcal{T}=\text{Pres}(H^{0}(G))$. This property appears in other situations, as the following result shows.}
\end{remark}

\begin{lemma}\label{lema presentados de V}
Let $\mathbf{t}=(\T,\F)$ be a torsion pair such that $\F$ is closed under taking direct limits in $\G$. Then there is an object $V$ such that $\T=\text{Pres}(V).$ Moreover, the torsion pairs such that $\F$ is closed under taking direct limits in $\G$ form a set.
\end{lemma}
\begin{proof}
Let us fix a generator $G$ of $\G$. Then each object of $\G$ is a directed union of those of its subobjects which are isomorphic to quotients $G^{(n)}/X$, where $n$ is a natural number and $X$ is a subobject of $G^{(n)}$. Let now take $T\in \T$ and express it as a directed union $T=\underset{i\in I}{\bigcup}U_i$, where $U_i\cong G^{(n_i)}/X_i$, for some integer $n_i>0$ and some subobject $X_i$ of $G^{(n_i)}$. We now get an $I$-directed system of exact sequences 
$$\xymatrix{0 \ar[r] & t(U_i) \ar[r] & U_i \ar[r] & U_i/t(U_i) \ar[r] &0}$$
Due to the AB5 condition of $\G$, after taking direct limits, we get an exact sequence
$$\xymatrix{0 \ar[r] & \varinjlim{t(U_i)} \ar[r] & T \ar[r] & \varinjlim{\frac{U_i}{t(U_i)}} \ar[r] &0}$$  
It follows that $\varinjlim{\frac{U_i}{t(U_i)}}\in \mathcal{T} \cap \mathcal{F}=0$ since $\F$ is closed under taking direct limits in $\F$. Then we have $T=\underset{i \in I}{\bigcup} t(U_i)$. The objects $T^{'}\in \T$ which are isomorphic to subobjects of quotients $G^{(n)}/X$ form a skeletally small subcategory (see definition \ref{def. locally small}). We take a set $\mathcal{S}$ of its representatives, up to isomorphism, and put $V=\underset{S\in \mathcal{S}}{\coprod}S$. The previous paragraph shows that each $T\in \T$ is isomorphic to a direct limit of objects in $\mathcal{S}$, from which we get that $\T \subseteq \Pres(V).$ The other inclusion is clear. \\

For the final statement, note that the last paragraph shows that the assignment $\mathbf{t} \rightsquigarrow S$ gives an injective map from the class of torsion pairs $\mathbf{t}$ such that $\F=\varinjlim{\F}$ to the set of parts of $\underset{n\in \mathbb{N}, X<G^{(n)}}{\bigcup } S(G^{(n)}/X)$, where $S(M)$ denotes the set of subobjects of $M$, for each object $M$.
\end{proof}

\begin{lemma}\label{generador de Ht}
Let $\mathbf{t}=(\T,\F)$ be a torsion pair such that each object of $\Ht$ is isomorphic in $\D(\G)$ to a complex $M$ such that $M^{k}=0$, for $k\neq -1,0$, and $M^{-1}\in \F$ (e.g. if $\mathbf{t}$ is a hereditary torsion pair). If $\F$ is closed under taking direct limits in $\G$, then $\Ht$ has a generator.
\end{lemma}
\begin{proof}
All throughout the proof, we fix a generator $G$ of $\G$ and, using the previous lemma, an object $V$ such that $\T=\Pres(V).$ We also put $W:=G/t(G)$, so that $\F\subseteq Gen(W)$. Let $\mathcal{S}$ denote the class of chain complexes 
$$\xymatrix{\cdots \ar[r] & 0 \ar[r] & S^{-1} \ar[r] & S^{0} \ar[r] & 0 \ar[r] & \cdots}$$
such that $S^{-1}=W^{(m)}$ and $H^{0}(S)=V^{(n)}$, for some $m,n\in \mathbb{N}.$ The class $\mathcal{S}$ is clearly in bijection with that of exact sequences in $\G$
$$\xymatrix{0 \ar[r] & W^{(m)}/K \ar[r] & S^{0} \ar[r] & V^{(n)} \ar[r] & 0}$$
where $m,n \in \mathbb{N}$ and $K$ is a subobject of $W^{(n)}.$ Up to isomorphism, these exact sequences form a set since $\G$ is locally small. Therefore the image of $\mathcal{S}$ by the canonical functor $\xymatrix{q: \C(\G) \ar[r] & \D(\G)}$ is skeletally small in $\Ht$. The goal is to see that $\Ht=\Gen_{\Ht}(q(\mathcal{S}))$, which will end the proof. For simplicity, put $\mathcal{S}=q(\mathcal{S})$. \\

Let now $M\in \Ht$ be any object, which we assume to be a complex of the form $\xymatrix{ \cdots \ar[r] & 0 \ar[r] & M^{-1} \ar[r]^{d} & M^{0} \ar[r] & 0 \ar[r] & \cdots }$ with $M^{-1}\in \mathcal{F}$. We take an epimorphism $\xymatrix{p:W^{(J)} \ar@{>>}[r] & M^{-1}}$ and consider the following commutative diagram 
$$\xymatrix{\tilde{M}:= &\cdots \ar[r] & 0 \ar[r] & W^{(J)} \ar[r]^{d \circ p}  \ar[d]^{p} & M^{0} \ar[r] \ar@{=}[d]  & 0 \ar[r] & \cdots \\ & \cdots \ar[r] & 0 \ar[r] & M^{-1} \ar[r]^{d} & M^{0} \ar[r] & 0 \ar[r] & \cdots}$$
It is clear that $\tilde{M}\in \Ht$. Moreover, we have an obvious exact sequence in the category $\C(\G)$ of chain complexes
$$\xymatrix{0 \ar[r] & \Ker(p)[1] \ar[r] & \tilde{M} \ar[r] & M \ar[r] & 0}$$
This sequence gives a triangle in $\D(\G)$ with its three terms in $\Ht$ and, hence it is an exact sequence in this last category. Our goal is reduced to show that $\tilde{M}\in \Gen_{\Ht}(\mathcal{S}).$ \\

Let then consider any object $\tilde{M}\in \Ht$ represented by a complex 
$$\xymatrix{\cdots \ar[r] & 0 \ar[r] & W^{(J)} \ar[r]^{g} & \tilde{M}^{0} \ar[r] & 0 \ar[r] & \cdots}$$
for some set $J$. For each finite subset $F\subseteq J$, we get the following commutative with exact rows in $\G$, whose bottom left and upper right squares are cartesian:
$$\xymatrix{0 \ar[r] & U_{F} \ar[r] \ar@{=}[d] & W^{(F)} \ar[r] \ar@{=}[d] & \tilde{M}_{F}^{0} \ar[r] \ar[d] \pushoutcorner & *+<1em>{t(X_F)} \ar@{^(->}[d] \ar[r] & 0\\ 0 \ar[r] & *+<1em>{U_F} \ar@{^(->}[d] \ar[r] \pushoutcorner & *+<1em> {W^{(F)}} \ar[r] \ar@{^(->}[d] & \tilde{M}^{0} \ar[r] \ar@{=}[d] & X_F \ar[r] \ar[d] & 0\\ 
0 \ar[r] & H^{-1}(\tilde{M}) \ar[r] & W^{(J)} \ar[r] & \tilde{M}^{0} \ar[r] & H^{0}(\tilde{M}) \ar[r] & 0}$$
This yields a vertical exact sequence in $\C(\G)$, where the first term is an object of $\Ht$ 
$$\xymatrix{\tilde{M}_{F}:=&\cdots \ar[r] & 0 \ar[r] & W^{(F)} \ar[r] \ar@{=}[d] & *+<1.1em>{\tilde{M}_{F}^{0}} \ar[r] \ar@{^(->}[d] &0 \ar[r] & \cdots \\ \tilde{N}_F:=& \cdots \ar[r] & 0 \ar[r] & W^{(F)} \ar[r] \ar[d]& \tilde{M}^{0} \ar[r] \ar@{>>}[d] & 0  \ar[r] & \cdots \\ &\cdots \ar[r] &0 \ar[r] & 0 \ar[r] & (1:t)(X_F) \ar[r] & 0 \ar[r] & \cdots}$$

It is clear that, when $F$ varies on the set of finite subsets of $J$, we obtain a direct system of exact sequences in $\C(\G)$, such that $\varinjlim_{\C(\G)}{\tilde{N}_{F}}=\tilde{M}$. Due to the AB5 condition of $\G$, one has the following exact sequence in $\C(\G)$

$$\xymatrix{0 \ar[r] & \varinjlim_{\C(\G)}{\tilde{M}_F} \ar[r]^{\rho \hspace{0.6 cm}} & \varinjlim_{\C(\G)}\tilde{N}_{F}=\tilde{M} \ar[r] &  \varinjlim{(1:t)(X_F)}[0] \ar[r] & 0}$$

Using the fact that $\varinjlim{\F}=\F$, from lemma \ref{limites directos en C} we get that $\limite{\tilde{M}_{F}}\cong \varinjlim_{\C(\G)}{\tilde{M}_F}$. Therefore from \cite{BBD} we have the next triangle in $\D(\G)$:

$$\xymatrix{ \Ker_{\Ht}(\rho)[1] \ar[r] & \varinjlim{(1:t)(X_F)}[0] \ar[r] & \Coker_{\Ht}(\rho) \ar[r]^{\hspace{0.8cm}+} & }$$

By applying the exact sequence of homologies to this triangle, we readily get that $\Ker_{\Ht}(\rho)=0=\Coker_{\Ht}(\rho)$, so that $\rho$ is a quasi-isomorphism. It follows that, we have an epimorphism $\xymatrix{\underset{F\subseteq J}{\coprod}\tilde{M}_F \ar[r] & \varinjlim_{\Ht}\tilde{M}_{F}=\tilde{M}}$. Hence the proof is further reduced to consider a complex 
$$\xymatrix{N:=\cdots \ar[r] & 0 \ar[r] & W^{(m)} \ar[r]^{\hspace{0.1cm}\delta} & N^{0} \ar[r] & 0 \ar[r] & \cdots}$$ with $m\in \mathbb{N}$ and $H^{0}(N)\in \T$, and to prove that it is an epimorphic image in $\Ht$ of a coproduct of objects of $\mathcal{S}$. We can fix an epimorphism 
$\xymatrix{\pi: V^{(K)} \ar[r] & H^{0}(N)}$ with kernel in $\T$ since $\T=\Pres(V).$ Taking now the pullback of $\pi$ along the projection $\xymatrix{N^{0} \ar@{>>}[r] & H^{0}(N)}$, we can assume without loss of generality that $H^{0}(N)=V^{(K)}$, for some set $K$. Assuming this condition, for each subset $F\subseteq K$, consider the pullback of the inclusion $\xymatrix{V^{(F)} \hspace{0.08cm} \ar@{^(->}[r] & V^{(K)}}$ along the projection $\xymatrix{N^{0} \ar@{>>}[r] & H^{0}(N)}=V^{(K)}$. We then obtain a complex $\xymatrix{N_F:=\cdots \ar[r] & 0 \ar[r] & W^{(m)} \ar[r]& N^{0}_{F} \ar[r] & 0 \ar[r] & \cdots}$ with $H^{0}(N_F)=V^{(F)}$. With an argument used earlier in this proof, we see that, when $F$ varies in the set of finite subsets of $K$, we obtain a direct system $(N_F)$ in $\mathcal{C}(\G)$ such that $\varinjlim_{\C(\G)}N_F=M$ and all $N_F$ are in $\Ht$. Again, from lemma \ref{limites directos en C}, we have that $\limite{N_F} \cong \varinjlim_{\C(\G)}{N_F}=N$. But each $N_F$ is isomorphic to an object in $\mathcal{S}$, and the proof is finished.
\end{proof}

\vspace{0.3cm}

Now we give a wide class of torsion pairs for which the heart is a Grothendieck category.

\begin{theorem}\label{Grothendieck characterization}
Let $\G$ be a Grothendieck category and let $\mathbf{t}=(\T,\F)$ be a torsion pair in $\G$ satisfying, at least, one the following conditions:
\begin{enumerate}
\item[a)] $\mathbf{t}$ is hereditary.
\item[b)] Each object of $\Ht$ is isomorphic in $\D(\G)$ to a complex $F$ such that $F^{k}=0$, for $k\neq -1,0$, and $F^{k}\in \F$, for $k=-1,0.$
\item[c)] Each object of $\Ht$ is isomorphic in $\D(\G)$ to a complex $T$ such that $T^{k}=0$, for $k\neq -1,0$ and $T^{k}\in \T$, for $k=-1,0$.
\end{enumerate}
The following assertions are equivalent: 
\begin{enumerate}
\item[1)] The heart $\Ht$ is a Grothendieck category;
\item[2)] $\Ht$ is an AB5 abelian category;
\item[3)] $\F$ is closed under taking direct limits in $\G$.
\end{enumerate}
\end{theorem}
\begin{proof}
Each object $M\in \Ht$ is isomorphic in $\D(\G)$ to a complex 
$$\xymatrix{\cdots \ar[r] & 0 \ar[r] & E^{-1} \ar[r] & E^{0} \ar[r]^{p} & Y \ar[r] & 0 \ar[r] & \cdots }$$
concentrated in degrees -1,0,1, such that $p$ is an epimorphism and the $E^{k}$ are injective objects. Then, by a standard argument, we can assume without loss of generality that the inclusion $\xymatrix{H^{-1}(M) \hspace{0.08cm}\ar@{^(->}[r] & E^{-1}}$ is the injective envelope. But then $E^{-1}\in \mathcal{F}$ (see \cite[Proposition VI.3.2]{S}) when $\te$ is hereditary. From lemma \ref{generador de Ht}, we get that $\Ht$ has a generator in the cases a) and b), when $\varinjlim \F=\F$. \\


2) $\Longrightarrow$ 3) It follows from theorem \ref{caracterizacion AB5}. \\

 2) $\Longleftrightarrow $ 1) For the cases a) and b) it follows from the first paragraph of this proof. In the other case, the proof of the implication 3) $\Longrightarrow$ 2) for the case c) will show that $\Ht$ has a generator of the form $V[0]$, form some $V\in \T$.  \\

3) $\Longrightarrow$ 2) It is enough to check that assertion 3 of theorem \ref{caracterizacion AB5} holds. Suposse first that $\mathbf{t}$ is a hereditary torsion pair. Let $(M_i)_{i\in I}$ be a direct system in $\Ht.$ With the terminology theorem \ref{caracterizacion AB5} and its proof, we have the following commutative diagram 
$$\xymatrix{0 \ar[r] & \varinjlim{H^{-1}(M_i)} \ar[r] \ar[d]^{\varphi} & H^{-1}(Z) \ar[r] \ar@{>>}[d] & H^{0}(\Ker(f)) \ar[r] \ar[d]& 0\\ 0 \ar[r] & H^{-1}(\limite{M_i}) \ar[r]^{\sim \hspace{0.45 cm}} & H^{-1}(Z)/t(H^{-1}(Z)) \ar[r] & 0 \ar[r] & 0}$$
Hence $\Ker(\varphi)$ is an isomorphic to a subobject of $t(H^{-1}(Z)),$ so $\Ker(\varphi)\in \mathcal{T} \cap \mathcal{F}=0$, since $\varinjlim{\F}=\F.$ \\

Suppose now that condition $b)$ holds. We claim that, in that case, each object of $\Ht$ is isomorphic to a subobject of an object in $\F[1]$. Indeed, if $M$ is isomorphic to the mentioned complex $F$, then its differential $\xymatrix{d:F^{-1} \ar[r] & F^{0}}$ gives a triangle in $\D(\G)$
$$\xymatrix{F^{-1}[0] \ar[r]^{\hspace{0.1cm}d[0]} & F^{0}[0] \ar[r] & M \ar[r]^{\hspace{0.1cm}+} &}$$
hence, we have an exact sequence in $\Ht$
$$\xymatrix{0 \ar[r] & M \ar[r] & F^{-1}[1] \ar[r]^{\hspace{0.1cm}d[1]} & F^{0}[1] \ar[r] & 0}$$

Let $(M_i)_{i\in I}$ be a direct system of $\Ht$. By traditional arguments (see, e.g. \cite[Corollary 1.7]{AR}), it is not restrictive to assume that the directed set $I$ is an ordinal and that the given direct system in $\Ht$ is continuous (smooth in the terminology of \cite{AR}). So we start with a direct system $(M_\alpha)_{\alpha < \lambda}$ in $\Ht$, where $\lambda$ is a limit ordinal and $M_\beta= \varinjlim_{\alpha < \beta}{M_\alpha}$, whenever $\beta$ is a limit ordinal such that $\beta < \lambda$. Now, by transfinite induction, we can define a $\lambda$-direct system of short exact sequences in $\Ht$

$$\xymatrix{0 \ar[r] & M_\alpha \ar[r] & F_{\alpha}[1] \ar[r] & F^{'}_{\alpha}[1] \ar[r] &0}$$
with $F_\alpha,F_{\alpha}^{'}\in \F$ for all $\alpha < \lambda$. Indeed, suppose that $\beta=\alpha +1$ is nonlimit and that the sequence has been defined for $\alpha$. Then the sequence for $\beta$ is the bottom one of the following commutative diagram, which exists since $\F[1]$ is closed under taking quotients in $\Ht$, and where the upper left and lower right squares are bicartesian. In the diagram $F$ and $F^{'}$ denote objects of $\F$ and $u_{\alpha +1}$ any monomorphism into an object of $\F[1]$: 
$$\xymatrix{0 \ar[r] & M_\alpha \pushoutcorner \ar[r] \ar[d] & F_{\alpha}[1] \ar[r] \ar[d] & F_{\alpha}^{'}[1] \ar[r] \ar@{=}[d] & 0\\ 0 \ar[r] & M_{\alpha+1} \ar[r] \ar@{=}[d]& N_\alpha \pullbackcorner \pushoutcorner \ar[r] \ar[d]^{u_{\alpha +1}} & F_{\alpha}^{'}[1] \ar[r] \ar[d] & 0 \\ 0 \ar[r] & M_{\alpha +1} \ar[r] & F_{\alpha+1}[1] \ar[r] & F^{'}_{\alpha +1}[1] \pullbackcorner \ar[r] & 0}$$

Suppose now that $\beta$ is a limit ordinal and that the sequence has been defined for all ordinals $\alpha < \beta$. Using proposition \ref{description of stalk} and the fact that $\F$ is closed under taking direct limits, we get an exact sequence in $\Ht$
$$\xymatrix{\varinjlim_{\alpha < \beta}{M_{\alpha}} \ar[r]^{g \hspace{0.35cm}} & (\varinjlim_{\alpha < \beta}{F_\alpha})[1] \ar[r] & (\varinjlim_{\alpha < \beta}{F_{\alpha}^{'}})[1] \ar[r] & 0}$$

Recall that, by the continuity of the direct system, we have $M_{\beta}=\varinjlim_{\alpha < \beta}{M_\alpha}.$ We denote by $W$ the image of $g$ in $\Ht$. We then get the following commutative diagram with exact rows. 

$$\xymatrix{ 0 \ar[r] & \varinjlim_{\alpha < \beta}{H^{-1}(M_\alpha)} \ar[r] \ar[d] & \varinjlim{F_\alpha} \ar[r] \ar[d]^{\wr}& \varinjlim{F_{\alpha}^{'}} \ar[r] \ar[d]^{\wr}& \varinjlim{H^{0}(M_\alpha)} \ar[r] \ar[d]&0 \\ 0 \ar[r] & H^{-1}(W) \ar[r] & H^{-1}(\limite{F_\alpha[1]}) \ar[r] & H^{-1}(\limite{F_{\alpha}^{'}[1]}) \ar[r] & H^{0}(W) \ar[r] & 0 }$$

All the vertical arrows are then isomorphisms since so are the two central ones. But the left vertical arrow is the composition $\xymatrix{ \varinjlim_{\alpha < \beta}{H^{-1}(M_{\alpha})} \ar[r] & H^{-1}(\varinjlim_{\alpha < \beta }{M_{\alpha}}) \ar[rr]^{\hspace{0.9cm}H^{-1}(p)}  & & H^{-1}(W)}$, where $\xymatrix{p:\varinjlim_{\alpha < \beta}{M_\alpha} \ar[r] & W}$ is the obvious epimorphism in $\Ht$. It follows that the canonical map $\xymatrix{\varinjlim_{\alpha < \beta}{H^{-1}(M_\alpha)} \ar[r] & H^{-1}(\varinjlim_{\alpha < \beta}{M_\alpha})}$ is a monomorphism. Then it is an isomorphism due to proposition \ref{description of stalk}(1) and, by this same proposition and the isomorphic condition of the right vertical arrow in the above diagram, we conclude that $H^{k}(p)$ is an isomorphism, for all $k \in $ \Z. Then $p$ is an isomorphism in $\Ht$ and the desired short exact sequence for $\beta$ is defined. \\

The argument of the previous paragraph, when applied to $\lambda$ instead of $\beta$, shows that the induced morphism $\xymatrix{\varinjlim_{\alpha < \lambda}{H^{-1}(M_\alpha)} \ar[r] & H^{-1}(\varinjlim_{\alpha < \lambda}{M_\alpha})}$ is a monomorphism. \\

Finally, suppose that condition c) holds. An argument dual to the one used for condition b), shows that, $\T[0]$ generates $\Ht$. On the other hand, by lemma \ref{lema presentados de V}, we have an object $V\in \T$ such that $\T=\Pres(V)$. We easily derive that, for each $T\in \T$, the kernel of the canonical epimorphism $V^{(\Hom_\G(V,T))} \epic T$ is in $\T$, so that the induced morphism $V^{(\Hom_\G(V,T))}[0] \epic T[0]$ is an epimorphism in $\Ht$. Therefore $V[0]$ is a generator of $\Ht$. \\

Note that the assignment $M\rightsquigarrow \Psi(M):=V[0]^{(\Hom_{\Ht}(V[0],M))}$ is functorial. Indeed if $\xymatrix{f:M \ar[r] & N}$ is a morphism in $\Ht$, we define $\xymatrix{\Psi(f):V[0]^{(\Hom_{\Ht}(V[0],M))} \ar[r] & V[0]^{(\Hom_{\Ht}(V[0],N))}}$ using the universal property of the coproduct in $\Ht$. By definition, $\Psi(f)$ the unique morphism in $\Ht$ such that $\Psi(f) \circ \iota_{\alpha}^{M}=\iota^{N}_{f \circ \alpha}$, where $\xymatrix{\iota_{\alpha}^{M}:V[0] \ar[r] & V[0]^{(\Hom_{\Ht}(V[0], M))}}$ is the $\alpha$-injection into the coproduct, where $\alpha \in \Hom_{\Ht}(V[0],M),$ and similarly for $\iota^{N}_{f \circ \alpha}.$ \\

The functor $\xymatrix{\Psi:\Ht \ar[r] & \Ht}$ comes with natural transformation $p:\Psi \epic \text{id}_{\Ht}$ which is epimorphic. Note that $\xi(M):=\Ker(p_M)=T_M[0],$ for some $T_M\in \T$, since $\T[0]$ is closed under taking subobjects in $\Ht.$ We then get functors $\Psi,\xi:\Ht \flecha \T\cong \T[0] \monic \Ht$, together with an exact sequence of functors $\xymatrix{0 \ar[r] & \xi \hspace{0.07 cm}\ar@{^(->}[r]^{\mu} & \Psi \ar@{>>}[r]^{p} & \text{id}_{\Ht} \ar[r] & 0}$. In particular, if $f:L \flecha M$ is a morphism in $\Ht$, we have a commutative diagram in $\T\cong \T[0]$

$$\xymatrix{\xi(L) \ar[r]^{\mu_L} \ar[d]^{\xi(f)} & \Psi(L) \ar[d]^{\Psi(f)} \\ \xi(M) \ar[r]^{\mu_M} & \Psi(M)}$$

But note that if $M$ is an object of $\Ht$, then, viewing $\xi(M)$ and $\Psi(M)$ as objects of $\T$, the complex $C(M):=\cdots \flecha 0 \flecha \xi(M) \flecha \Psi(M) \flecha 0 \flecha \cdots$ ($\Psi(M)$ in degree 0) is isomorphic to $M$ in $\D(\G)$. The diagram above tells us that the assignment $M \rightsquigarrow C(M)$ gives a functor $C:\Ht \flecha \C(\G)$ such that if $q:\C(\G) \flecha \D(\G)$ is the canonical functor, then the composition $\xymatrix{\Ht \ar[r]^{C} & \C(\G) \ar[r]^{q} & \D(\G)}$ is naturally isomorphic to the inclusion $\Ht \monic D(\G).$ \\

Suppose now that $(M_i)_{i \in I}$ is a direct system in $\Ht$. Then $(C(M_i)_{i\in I})$ is a direct system in $\C(\G)$. By lemma \ref{limites directos en C}, we know that $\varinjlim_{\C(\G)}(C(M_i))\cong \limite{C(M_i)} \cong \limite{M_i}$. Then we get an isomorphism  $\varinjlim{H^{-1}(M_i)}=H^{-1}(\varinjlim_{\C(\G)}{C(M_i)}) \iso H^{-1}(\limite{M_i})$, as desired.
\end{proof}

\begin{definition}\rm{
Let $\mathcal{A}$ be an abelian category. A class $\mathcal{C}\subseteq \text{Ob}(\mathcal{A})$ will be called a \emph{generating}\index{class! generating} (resp. \emph{cogenerating}\index{class! cogenerating}) \emph{class} of $\mathcal{A}$ when, for each object $X$ of $\G$, there is an epimorphism $C \epic X$ (resp. monomorphism $X \monic C$), for some $C\in \mathcal{C}.$ }
\end{definition}
\vspace{0.3cm}

On a first look, conditions b) and c) of last theorem seem to be rarely satisfied. Our next corollary shows that it is not so.

\begin{corollary}\label{ejemplos b) y c)}
Let $\mathbf{t}=(\T,\F)$ be a torsion pair such that either $\F$ is generating or $\T$ is cogenerating. The heart $\Ht$ is a Grothendieck category if, and only if, $\F$ is closed under taking direct limits in $\G$.
\end{corollary}
\begin{proof}
We assume that $\F$ is closed under taking direct limits, because, by theorem \ref{caracterizacion AB5}, we only need to prove the ''if'' part of the statement. \\

Suppose first that $\F$ is a generating class and let $M \in \Ht$ be any object, which we represent by a complex $\xymatrix{\cdots \ar[r] & 0 \ar[r]& \ar[r]^{d} M^{-1} & M^{0} \ar[r] & 0 \ar[r] & \cdots }$. By fixing an epimorphism $p:F^{0} \epic M^{0}$ and taking the pullback of this morphism along $d$, we may and shall assume that $M^{0}=F^{0}\in \F.$ But then $\Imagen(d)$ is in $\F$, which implies that $M^{-1}\in \F$ since we have an exact sequence $\xymatrix{0 \ar[r] & H^{-1}(M) \hspace{0.05cm}\ar@{^(->}[r] & M^{-1} \ar[r]^{\overline{d} \hspace{0.15cm}}& \Imagen(d) \ar[r] & 0}$, where the outer nonzero terms are in $\F$. Then condition b) of the theorem \ref{Grothendieck characterization} holds.  \\

Suppose that $\T$ is a cogenerating class. Then the injective objects of $\G$ are in $\T$. Moreover each object $M$ of $\Ht$ is isomorphic in $\D(\G)$ to a complex
$$\xymatrix{\cdots \ar[r] & 0 \ar[r] & E^{-1} \ar[r]^{d} & E^{0} \ar[r]^{p} & Y \ar[r] & 0 \ar[r] & \cdots}$$
concentrated in degrees -1,0,1, such that $p$ is an epimorphism and the $E^{k}$ are injectives objects. The result follows from the fact that such a complex is isomorphic to the complex  
$$\xymatrix{\cdots \ar[r] & 0 \ar[r] & E^{-1} \ar[r] & \Ker(p)\ar[r] & 0 \ar[r] & \cdots}$$
and $\Ker(p)\in \T$. Indeed, in the exact sequence $\xymatrix{0 \ar[r] & \Imagen(d) \ar[r] & \Ker(p) \ar[r] & H^{0}(M) \ar[r] & 0}$, the outer nonzero terms are in $\T$.
\end{proof}

\vspace{0.3cm}

The classical examples of torsion pairs in a module category whose torsion class (resp. torsionfree class) is cogenerating (resp. generating) are the cotilting (resp. tilting) torsion pairs. For this reason the corollary above leads us to study those torsion pairs and, in fact, to discover new things about them.

\section{Tilting and cotilting torsion pairs}
All throughout this section $\G$ is a Grothendieck category. We refer the reader to subsection \ref{section:torsion pair} for the definition of a (co)tilting object in an abelian category.

\begin{definition}\rm{
Let $\mathcal{A}$ be an AB3 (resp. AB3*) abelian category. Two 1-tilting (resp. 1-cotilting) objects of $\mathcal{A}$ are said to be \emph{equivalent} when their associated torsion pairs coincide.}
\end{definition}

\begin{remark}\rm{
Recall that idempotents split in any abelian category. As a consequence, two 1-tilting objects $V$ and $V^{'}$ are equivalent if, and only if, $\Add(V)=\Add(V^{'})$. Similarly, two 1-cotilting objects $Q$ and $Q^{'}$ are equivalent if, and only if, $\Prod(Q)=\Prod(Q^{'}).$}
\end{remark}
\vspace{0.3cm}

A consequence of the results in the previous section is the following:
\begin{proposition}\label{prop. V[0] progenerator}
Let $\G$ be a Grothendieck category and let $\te=(\T,\F)$ be a torsion pair in $\G$. Consider the following assertions:
\begin{enumerate}
\item[1)] $\te$ is a classical tilting torsion pair;
\item[2)] $\T$ is a cogenerating class and $\Ht$ is a module category;
\item[3)] $\T$ is a cogenerating class and $\Ht$ is a Grothendieck category with a projective generator;
\item[4)] $\te$ is a tilting torsion pair such that $\Ht$ is an AB5 abelian category;
\item[5)] $\te$ is a tilting torsion pair such that $\F$ is closed under taking direct limits in $\G$.
\end{enumerate}
Then the implications 1) $\Longleftrightarrow$ 2) $\Longrightarrow$ 3) $\Longleftrightarrow$ 4) $\Longleftrightarrow$ 5) hold.
\end{proposition}
\begin{proof}
Let $V$ be any 1-tilting object and $\te=(\T,\F)$ be its associated torsion pair. For each $M\in \G$, its injective envelope $E(M)$ and its first cosyzygy $\Omega^{-1}(M)=\frac{E(M)}{M}$ are in $\T:=\Gen(V)=\Ker(\Ext^{1}_{\G}(V,?))$. It follows that $\Ext^{2}_{\G}(V,M)\cong \Ext^{1}_{\G}(V, \Omega^{-1}(M))=0$. Thus, for each $F\in \F$, we get that $\Ext^{1}_{\Ht}(V[0],F[1])=\Ext^{2}_{\G}(V,F)=0$. On the other hand, we have $\Ext^{1}_{\Ht}(V[0], T[0])\cong \Ext^{1}_{\G}(V,T)=0$, for all $T\in \T$. It follows that $V[0]$ is a projective object of $\Ht$ since we have the following exact sequence in $\Ht$
$$\xymatrix{0 \ar[r] & H^{-1}(N)[1] \ar[r] & N \ar[r] & H^{0}(N)[0] \ar[r] & 0}$$
for each $N\in \Ht$. On the other hand, by the proof of theorem \ref{Grothendieck characterization} under its condition c), we know that $V[0]$ is a generator of $\Ht$.\\

1) $\Longrightarrow$ 2) Let us assume that $V$ is classical 1-tilting. We then have an isomorphism 
\begin{center}
$\Hom_{\Ht}(V[0],V[0])^{(I)}\cong \Hom_{\G}(V,V)^{(I)}\cong \Hom_{\G}(V,V^{(I)})\cong \Hom_{\Ht}(V[0],V[0]^{(I)}), \ \ \ (\ast)$
\end{center}
for each set $I$. That is, $V[0]$ is a self-small object of $\Ht$ and, since it is a projective generator, it is easily seen that $V[0]$ is a compact object of $\Ht$. It follows that $V[0]$ is a compact object of $\Ht$ and, hence, that $V[0]$ is a progenerator of $\Ht$. Then $\Ht$ is a module category (see theorem \ref{teo. Gabriel-Michell}).\\

3) $\Longrightarrow$ 4) By the proofs of corollary \ref{ejemplos b) y c)} and theorem \ref{Grothendieck characterization}, we know that $\T[0]$ generates $\Ht$. If $G$ is a projective generator of $\Ht$, then we have a split exact sequence in $\Ht$ of the form 
$$\xymatrix{0 \ar[r] & T^{'}[0] \ar[r] & T[0] \ar[r] & G \ar[r] & 0}$$
where $T,T^{'}\in \T$. Then we necessarily have that $G=V[0]$, where $V\in \T$, because $\T[0]$ is closed under taking subobjects in $\Ht$. It easily follows from this that $\T=\Pres(V)=\Gen(V)$ and, hence, that $\F=\Ker(\Hom_{\G}(V,?))$. From the projectivity of $V[0]$ in $\Ht$ we get that $0=\Ext^{1}_{\Ht}(V[0],T[0])=\Ext^{1}_{\G}(V,T),$ for each $T\in \T$. Therefore we get that $\Gen(V)\subseteq \Ker(\Ext^{1}_{\G}(V,?))$. On the other hand, the injective objects of $\G$ are in $\T$ since $\T$ is a cogenerating class. Hence, $\Ext^{2}_{\G}(V,M)=\Ext^{1}_{\G}(V,\Omega^{-1}(M))=0$, for each object $M\in \G$. In particular, if $M\in \Ker(\Ext^{1}_{\G}(V,?))$ and we apply the exact sequence of Ext to the canonical sequence $0 \flecha t(M) \flecha M \flecha (1:t)(M) \flecha 0$, we get that $0=\Ext^{1}_{G}(V,(1:t)(M))=\Hom_{\Ht}(V[0],(1:t)(M)[1])$. This implies that $(1:t)(M)=0$ since $V[0]$ is a generator of $\Ht$. It follows that $M\in \T$ so that $\T=\Gen(V)=\Ker(\Ext^{1}_{\G}(V,?))$. Then $\te$ is a tilting torsion pair.  \\

2) $\Longrightarrow$ 1) By the argument in the implication 3) $\Longrightarrow$ 4), we can assume that $\te$ is a tilting torsion pair induced by a 1-tilting object $V$ such that $V[0]$ is a progenerator of $\Ht$. From the fact that $V[0]$ is compact in $\Ht$ we derive that the isomorphism $(\ast)$ above still holds. Then $V$ is self-small and, hence, $V$ is a classical 1-tilting object.\\

The implication 2) $\Longrightarrow$ 3) is clear and 4) $\Longrightarrow$ 5) follows from theorem \ref{caracterizacion AB5}. \\

5) $\Longrightarrow$ 3) That $\T=\Ker(\Ext^{1}_{\G}(V,?))$ is a cogenerating class is clear since it contains all injective objects. The fact that $\Ht$ is a Grothendieck category follows then from corollary \ref{ejemplos b) y c)}. Finally, by the first paragraph of this proof, we know that $V[0]$ is a projective generator of $\Ht$. 
\end{proof}

\vspace{0.3 cm}

In a first look to last proposition, we wonder if the converse of the implication \linebreak 2) $\Longrightarrow $ 3) is satisfied. The following example shows that the converse is no true.

\begin{example}\rm{
A projective generator $P$ of $\G$ is always a 1-tilting object. However $P$ is self-small if, and only if, it is compact. Therefore if $\G$ has a projective generator but is not a module category, then the trivial torsion pair $\te=(\G,0)$ satisfies assertion 5, but not assertion 2 of last proposition.}
\end{example}

\vspace{0.3 cm}

However, the following is a natural question whose answer seems to be unknown.

\begin{question}
Let $R$ be a ring and $V$ be a 1-tilting $R$-module such that $\Ker(\Hom_{R}(V,?))$ is closed under taking direct limits in $R$-Mod. Is $V$ equivalent to a classical 1-tilting module?. Note that a 1-tilting $R$-module is classical if, and only if, it is finitely presented (cf. \cite[Proposition 1.3]{CT}).
\end{question}

\vspace{0.3cm}

In the case ``cotilting'', we have to introduce some definitions and notations. Indeed, if $I$ is any set, then the product functor $\prod:\G^{I}=[I,\G]\flecha \G$ is left exact, but need not be right exact. We shall denote by $\prod^{1}:=\prod^{1}_{i\in I}:G^{I} \flecha \G$ its first right derived functor. Given a family $(X_i)_{i \in I}$, we have that $\prod^{1}_{i \in I} X_i$ is the cokernel of the canonical morphism $\prod_{i \in I}E(X_i) \flecha \prod_{i\in I}\frac{E(X_i)}{X_i}$.

\begin{remark}\rm{
Note that if $\G$ is AB4*, then $\prod^{1}$ is the zero functor. All Grothendieck categories with enough projectives are AB4*, but the converse is not true and there exist AB4* Grothendieck categories with no nonzero projective objects (see \cite[Theorem 4.1]{Roo}).}
\end{remark}

\begin{definition}\rm{
An object $Q$ of the Grothendieck category $\G$ will be called \emph{strong 1-cotilting}\index{object! strong 1-cotilting} when it is 1-cotilting and $\prod^{1}_{i \in I}Q$ is in $\F:=\Cogen(Q)$, for each set $I$. The corresponding torsion pair is called a \emph{strong cotilting torsion pair}\index{torsion pair! strong cotilting }.}
\end{definition}

Let $\G$ be a locally finitely presented Grothendieck category (see definition \ref{def. locally fp G}) in this paragraph. An exact sequence $\xymatrix{0 \ar[r] & X \ar[r]^{u} & Y \ar[r]^{p} & Z \ar[r] & 0}$ is called \emph{pure-exact}\index{exact sequence! pure-exact} when it is kept exact when applying the functor $\Hom_{\G}(U,?)$, for every finitely presented object $U$. An object $E$ of $\G$ is \emph{pure-injective}\index{object! pure-injective} when $\Hom_{\G}(?,E)$ preserves the exactness of all pure-exact sequences (see, e.g. \cite{CB} or \cite{Pr} for details).

\begin{lemma}\label{lemma Bazzoni}
Let $Q$ be a 1-cotilting object of $\G$. The following assertions hold:
\begin{enumerate}
\item[1)] If $\G$ is AB4* then $Q$ is strong 1-cotilting and the class $\F:=\Cogen(Q)$ is generating.
\item[2)] If $\G$ is locally finitely presented, then $Q$ is a pure-injective object and $\F$ is closed under taking direct limits in $\G$. In particular, the equivalence classes of 1-cotilting objects form a set.
\item[3)] If there exists a strong 1-cotilting object $Q^{'}$ which is equivalent to $Q$, then $Q$ is itself strong 1-cotilting.
\end{enumerate}
\end{lemma}
\begin{proof}
1) That $Q$ is strong 1-cotilting is straightforward since $\prod^{1}$ vanishes when $\G$ is AB4*. In order to prove that $\F$ is generating, it is enough to prove that all injective objects of $\G$ are homomorphic image of objects in $\F$. Indeed, if that is the case and $U$ is any object of $\G$, then fixing an epimorphism $p:F \epic E(U)$, with $F\in \F$, and pulling it back along the inclution $U\monic E(U)$, we obtain an epimorphism $F^{'} \epic U$, for some $F^{'}\in \F.$ But, by the dual of proposition \ref{description tilting object}, we get that each injective cogenerator $E$ is an homomorphic image of an object in $\F:=Cogen(Q)$. \\

2) We follow Bazzoni's argument (see \cite{B}) and see that it also works in our context. First of all, note that an object $Y$ of $\G$ is pure-injective if, and only if, for every set $S$, each morphism $f:Y^{(S)} \flecha Y$ extends to $Y^{S}$ (cf. \cite[Theorem 1]{CB}, \cite[Theorem 5.2]{Pr}). Then lemmas 2.1, 2.3 and 2.4, together with corollary 2.2 of [op.cit] are valid here. We next consider proposition 2.5 in Bazzoni's paper. For it to work in our situation, we just need to check that if $\lambda$ is an infinite cardinal and $(A_{\beta})_{\beta \in \lambda^{\aleph_{0}}}$ is a family of $\lambda^{\aleph_{0}}$ subsets of $\lambda$ such that $A_{\alpha}\cap A_{\beta}$ is finite, for all $\alpha\neq \beta$, then the images of the compositions $M^{A_{\beta}} \monic M^{\lambda} \xymatrix{\ar@{>>}[r]^{pr \hspace{0.5cm}} & \frac{M^{\lambda}}{M^{(\lambda)}}}$ form a family $(Y_{\beta})_{\beta \in \lambda^{\aleph_{0}}}$ of subobjects of $\frac{M^{\lambda}}{M^{(\lambda)}}$ which have direct sum. Note that this amounts to prove that, for each $\beta \in \lambda^{\aleph_{0}}$, we have $(M^{(\lambda)}+ M^{A_{\beta}}) \cap (M^{(\lambda)} + \underset{\gamma \neq \beta}{\sum}M^{A_{\gamma}})=M^{(\lambda)}$. By the modular law, which is a consequence of the AB5 condition, we need to prove that $M^{(\lambda)}+ [(M^{(\lambda)}+M^{A_\beta}) \cap  (\underset{\gamma \neq \beta}{\sum}M^{A_{\gamma}})]=M^{(\lambda)}$. That is, we need to prove that $[(M^{(\lambda)}+M^{A_\beta}) \cap  (\underset{\gamma \neq \beta}{\sum}M^{A_{\gamma}})] \subseteq M^{(\lambda)}.$ \\

For simplicity, call an object $X$ of $\G$ finitely generated when it is homomorphic image of a finitely presented one. Clearly, $\Hom_{\G}(X,?)$ preserves direct unions of subobjects in that case. Due to the locally finitely presented condition of $\G$, each object of this category is a directed union of finitely generated subobjects. Our task reduces to prove that if $X$ is a finitely generated subobject of $[(M^{(\lambda)}+M^{A_\beta}) \cap  (\underset{\gamma \neq \beta}{\sum}M^{A_{\gamma}})]$, then $X\subseteq M^{(\lambda)}$. To do that, we denote by $\Supp(X)$ the set of $\alpha \in \lambda$ such that the composition $X \monic M^{\lambda} \xymatrix{\ar[r]^{\pi_{\alpha}} &} M$ is nonzero, where $\pi_{\alpha}:M^{\lambda} \epic M$ is the $\alpha$-projection, for each $\alpha \in \lambda$. Bearing in mind that $\underset{\gamma \neq \beta}{\sum}M^{A_{\gamma}}=\underset{F}{\bigcup}(\underset{\gamma \in F}{\sum}M^{A_{\gamma}})$, with $F$ varying on the set of finite subsets of $\lambda \setminus \{\beta\} ,$ the AB5  condition (see \cite[V.1]{S}) gives:

$$X=X \cap \underset{\gamma \neq \beta}{\sum}M^{A_{\gamma}}=X \cap [\underset{F}{\bigcup}(\underset{\gamma \in F}{\sum}M^{A_{\gamma}})]=\underset{F}{\bigcup}[X \cap (\underset{\gamma \in F}{\sum}M^{A_{\gamma}})] $$
and the finitely generated condition of $X$ implies that $X=X \cap \underset{\gamma \neq \beta}{\sum}M^{A_{\gamma}}\subseteq \underset{\gamma \in F}{\sum}M^{A_{\gamma}}$, for some $F \subset \lambda \setminus \{\beta\}$ finite. As a consequence, we have $\Supp(X) \subseteq \underset{\gamma \in F}{\bigcup}A_{\gamma}$.\\

On the other hand, exactness of direct limits gives that $M^{(\lambda)}+M^{A_{\beta}}=\underset{F^{'} \subset \lambda \text{, finite}}{\bigcup}[M^{(F^{'})} + M^{A_{\beta}}]$ and, again by the AB5 condition and the finitely generated condition of $X$, we get that $X\subset M^{(F^{'})} + M^{A_{\beta}},$ for some finite subset $F^{'}\subset \lambda$. This implies that $\Supp(X)\subseteq F^{'} \cup A_{\beta}$. Together with the conclusion of the previous paragraph, we then get that $\Supp(X) \subseteq (F^{'} \cup A_{\beta}) \cap ( \underset{\gamma \in F}{\bigcup}A_{\gamma})$, and so $\Supp(X)$ is a finite set and $X \subseteq M^{(\lambda)}.$ \\

The previous two paragraphs show that proposition 2.5 and corollary 2.6 of \cite{B} go on in our context. To complete Bazzoni's argument in our situation, it remains to check the truth of her lemma 2.7. This amount to prove that if $0\neq X \subset \frac{M^{\lambda}}{M^{(\lambda)}}$ is a finitely generated subobject, then there exists a morphism $f:\frac{M^{\lambda}}{M^{(\lambda)}} \flecha Q$ such that $f(X)\neq 0.$ Indeed, take the subobject $\tilde{X}$ of $M^{\lambda}$ such that $X=\frac{\tilde{X}}{M^{(\lambda)}}$. Then $\Supp(\tilde{X})$ is an infinite subset of $\lambda$, and this allows us to fix a subset $A\subseteq \Supp(\tilde{X})$ such that $|A|=\aleph_{0}.$ If now $p:M^{\lambda} \epic M^{A}$ is the canonical projection, then we get an induced morphism $\overline{p}:\frac{M^{\lambda}}{M^{(\lambda)}} \flecha \frac{M^{A}}{M^{(A)}}$ such that $\overline{p}(X)\neq 0.$ Since $\frac{M^{A}}{M^{(A)}}\in \Cogen(Q)$ we get a morphism $h:\frac{M^{A}}{M^{(A)}} \flecha Q$ such that $h(\overline{p}(X))\neq 0.$ We take $f=h \circ \overline{p}$ and have $f(X)\neq 0$, as desired. Therefore $Q$ is pure-injective. \\

Finally, if $(F_i)_{i\in I}$ is a direct system in $\F=\Ker(\Ext^{1}_{\G}(?,Q))$ then the induced sequence $\xymatrix{0 \ar[r] & K \ar[r] & \underset{i \in I}{\coprod} F_i \ar[r]^{p} & \varinjlim{F_i} \ar[r] & 0}$ is pure-exact. The fact that $\F=\varinjlim{\F}$ follows, as in module categories, by applying to this sequence the long exact sequence of $\Ext(?,Q).$ Indeed, we get the following monomorphism $\Ext^{1}_{\G}(\varinjlim{F_i},Q) \monic \Ext^{1}_{\G}(\coprod{F_i},Q)=0$, this implies that $\varinjlim{F_i}\in \F$.  Moreover, equivalence classes of 1-cotilting objects are in bijection with the cotilting torsion pairs. Now apply lemma \ref{lema presentados de V}.\\

3) For any set $I$, the functor $\prod^{1}:[I,\G] \flecha \G$ is additive. This and the fact that $\F=\Cogen(Q)$ is closed under direct summands imply that the class of objects $X$ such that $\prod^{1}X\in \F$ is closed under taking direct summands. This reduces the proof to cheek that if $Q$ is strong 1-cotilting, then $Q^{J}$ is strong 1-cotilting, for every set $J$. To do that, for such a set $J$, we consider the following commutative diagram, where the upper right square is bicartesian and where the vertical sequences are split exacts.
$$\xymatrix{0 \ar[r] & Q^{J} \ar@{=}[d] \ar[r] & *+<1em>{E(Q^{J})} \pushoutcorner \ar[r] \ar@{^(->}[d] & \frac{E(Q^{J})}{Q^{J}} \ar[r] \ar@{^(->}[d] & 0 \\ 0 \ar[r] & Q^{J} \ar[r] & E(Q)^{J} \ar@{>>}[d]\ar[r] &  \hspace{0.3 cm}\frac{E(Q)^{J}}{Q^{J}} \pullbackcorner \ar[r] \ar@{>>}[d] & 0 \\ & & E \ar@{=}[r] & E }$$ 
For each set $I$, the product functor $\prod:[I,G] \flecha \G$ preserves pullbacks since it is left exact. It also preserves split short exact sequences. It follows that the central square of the following commutative diagram is bicartesian since the cokernels of its two vertical arrows are isomorphic:

$$\xymatrix{0 \ar[r] & (Q^{J})^{I} \ar@{=}[d] \ar[r] & *+<1em>{(E(Q^{J}))^{I}} \pushoutcorner \ar[r] \ar@{^(->}[d] & (\frac{E(Q^{J})}{Q^{J}})^{I} \ar[r] \ar@{^(->}[d] & N \ar[r] \ar[d]^{u} & 0 \\ 0 \ar[r] & (Q^{J})^{I} \ar[r] & (E(Q)^{J})^{I} \ar[r] &  \hspace{0.55 cm}(\frac{E(Q)^{J}}{Q^{J}})^{I} \pullbackcorner \ar[r]  & N^{'} \ar[r] &0 }$$ 

Then $u$ is an isomorphism, which allows us to put $N^{'}=N$ and $u=1_N$. Our goal is to prove that $N$ is in $\F.$ But the lower row of the last diagram fits in a new commutative diagram with exact rows, where the two left vertical arrows are isomorphism:

$$\xymatrix{0 \ar[r] & (Q^{J})^{I} \ar[d]^{\wr} \ar[r] & (E(Q)^{J})^{I}  \ar[r] \ar[d]^{\wr} & (\frac{E(Q)^{J}}{Q^{J}})^{I} \ar[r] \ar[d] & N \ar[r] \ar[d]^{v} & 0  \\ 0 \ar[r] & Q^{J \times I} \ar[r] & E(Q)^{J \times I} \ar[r] &  (\frac{E(Q)}{Q})^{J \times I}  \ar[r]  & F \ar[r] &0 }$$ 
Due to the left exactness of the product functor, the second vertical arrow from right to left is a monomorphism. It the follows that $v$ is also a monomorphism. But $F$ is in $\F$, because $Q$ is a strong 1-cotilting. We then get that $N\in \F$, as desired. 
 \end{proof}

\begin{proposition}\label{proposition clave}
Let $\G$ be a Grothendieck category and let $\mathbf{t}=(\T,\F)$ be a torsion pair in $\G$ such that $\F$ is a generating class. Consider the following assertions:
\begin{enumerate}
\item[1)] The heart $\Ht$ is a Grothendieck category;
\item[2)] $\F$ is closed under taking direct limits in $\G$;
\item[3)] $\te$ is a (strong) cotilting torsion pair.
\end{enumerate}
Then the implications 1) $\Longleftrightarrow$ 2) $\Longrightarrow$ 3) hold. When $\G$ is locally finitely presented, all assertions are equivalent.
\end{proposition}
\begin{proof}
Let $(F_i)_{i\in I}$ be a family in $\F$. Note that, in order to calculate its product $\underset{\D(\G)}{\prod}F_{i}[1]$ in $\D(\G)$, we first replace each $F_i$ by an injective resolution, which we assume to be the minimal one, and then take products in $\C(\G).$ When $\G$ is not AB4*, the resulting complex can have nonzero homology in degrees greater than or equal to zero 0. However $\underset{\D(\G)}{\prod}F_{i}[1]$ is $\mathcal{U}_{\mathbf{t}}^{\perp}[1]$ and, using lemma \ref{lemma de adjunctions}, we easily see that $P:=\underset{\Ht}{\prod}F_{i}[1]=\tau_{\mathcal{U}}(\underset{\D(\G)}{\prod}F_{i}[1])$. Using the octahedral axiom, we have the diagram
$$\xymatrix{\tau_{\mathcal{U}}(\tau^{\leq 0}(\underset{\D(\G)}{\prod}F_{i}[1])) \ar@{=}[d] \ar[r] & \tau^{\leq 0}(\underset{\D(\G)}{\prod}F_{i}[1]) \ar[r] \ar[d] & \tau^{\mathcal{U}^{\perp}}(\tau^{\leq 0}(\underset{\D(\G)}{\prod}F_{i}[1])) \ar[r]^{\hspace{1.6cm}+} \ar[d] & \\ \tau_{\mathcal{U}}(\tau^{\leq 0}(\underset{\D(\G)}{\prod}F_{i}[1])) \ar[r] & \underset{\D(\G)}{\prod}F_{i}[1] \ar[r] \ar[d]& Z \ar[d] \ar[r]^{+} & \\ & \tau^{>0}(\underset{\D(\G)}{\prod}F_{i}[1]) \ar@{=}[r] \ar[d]^(.65){+}&  \tau^{>0}(\underset{\D(\G)}{\prod}F_{i}[1]) \ar[d]^(.65){+}\\ &&&}$$
where $\tau^{\leq 0}$ and $\tau^{>0}$ denote the left and right truncation functors with respect to the canonical t-structure $(\D^{\leq 0}(\G),\D^{\geq 0}(\G))$, respectively. It follows that $Z\in \mathcal{U}_{\mathbf{t}}^{\perp}$, since $\D^{>0}(\G)\subseteq \mathcal{U}_{\mathbf{t}}^{\perp}.$ Using once again octahedral axiom, we get the diagram
$$\xymatrix{\tau_{\mathcal{U}}(\tau^{\leq 0}(\underset{\D(\G)}{\prod}F_{i}[1])) \ar@{=}[d] \ar[r] & \tau_{\mathcal{U}}(\underset{\D(\G)}{\prod}F_{i}[1]) \ar[r] \ar[d] & Z^{'} \ar[r]^{+} \ar[d] & \\ \tau_{\mathcal{U}}(\tau^{\leq 0}(\underset{\D(\G)}{\prod}F_{i}[1])) \ar[r] & \underset{\D(\G)}{\prod}F_{i}[1] \ar[r] \ar[d]& Z \ar[d] \ar[r]^{+} & \\ & \tau^{\mathcal{U}^{\perp}}(\underset{\D(\G)}{\prod}F_{i}[1]) \ar@{=}[r] \ar[d]^(.65){+}&  \tau^{\mathcal{U}^{\perp}}(\underset{\D(\G)}{\prod}F_{i}[1]) \ar[d]^(.65){+}\\ &&&}$$

From the first horizontal triangle, we get that $Z^{'}\in \mathcal{U}_{\mathbf{t}}$. On the other hand, from the second vertical triangle and the fact that $Z\in \mathcal{U}_{\mathbf{t}}^{\perp}$ we get that $Z^{'}\in \mathcal{U}_{\te}^{\perp}$ and, hence that $Z^{'}=0$. It follows that $P\cong \tau_{\mathcal{U}}(\tau^{\leq 0}(\underset{\D(\G)}{\prod}F_{i}[1]))$. But $\tau^{\leq 0}(\underset{\D(\G)}{\prod}F_{i}[1])$ is quasi-isomorphic to the complex 
$$\xymatrix{Y:=\cdots \ar[r] & 0 \ar[r] & \underset{i \in I}{\prod}E(F_i) \ar[r]^{\text{can}} & \underset{i \in I}{\prod}\frac{E(F_i)}{F_i} \ar[r] & 0 \ar[r] & \cdots}$$  
concentrated in degrees -1 and 0. Note that $H^{0}(Y)=\prod^{1}F_i$. By pulling back the complex $Y$ along the inclusion $t(\prod^{1}F_i) \monic \prod^{1}F_i$, we obtain a new complex $Y^{'}$ which is in $\Ht$ and fits into a triangle $\xymatrix{Y^{'} \ar[r] & Y \ar[r] & (1:t)(\prod^{1}F_i)[0] \ar[r]^{\hspace{1.3 cm}+} & }$. It follows that $Y^{'}\cong \tau_\mathcal{U}(Y)\cong P$ and, hence, that $H^{-1}(P)=\underset{i \in I}{\prod}F_i$ and $H^{0}(P)\cong t(\prod^{1}_{i\in I}F_i)$.\\


1) $\Longleftrightarrow$ 2) is a direct consequence of corollary \ref{ejemplos b) y c)}. \\
1),2) $\Longrightarrow$ 3) By the proofs of corollary \ref{ejemplos b) y c)} and theorem \ref{Grothendieck characterization}, we know that $\F[1]$ cogenerates $\Ht.$ Then each injective object $I$ of $\Ht$ is a direct summand of an object in $\F[1]$, and hence $I$ is itself in $\F[1]$ since this class is closed under taking quotients. 
In particular, any injective cogenerator of $\Ht$ is of the form $Q[1]$, for some $Q\in \F$. Fixing such a $Q$, we get that $Q[1]^{S}$ is an injective cogenerator of $\Ht$, for each set $S$. This in turn implies that $Q[1]^{S}\in \F[1]$. By the initial paragraph of this proof, we then get that $t(\prod^{1}_{s\in S}Q)=0$, which implies that the proof is reduced to show that $\F=\Cogen(Q)=\Ker \Ext^{1}_{\G}(?,Q).$ We also get that $Q[1]^{S}\cong Q^{S}[1]$, for each set $S$. From this last isomorphism we immediately deduce that $\F=\Copres(Q)=\Cogen(Q).$ \\

On the other hand, the injectivity of $Q[1]$ in $\Ht$ implies that $\Ext^{1}_{\G}(F,Q)\cong \Ext^{1}_{\Ht}(F[1],Q[1])=0$, for each $F\in \F$. From this equality we derive that $\Cogen(Q)=\F \subseteq \Ker^{1}_{\G}(?,Q).$ Let now $Z$ be any object in $\Ker(\Ext^{1}_{\G}(?,Q))$. The generating condition of $\F$ gives us an epimorphism $p:F \epic Z$, with $F\in \F$. Putting $F^{'}:=\Ker(p),$ we then get the following commutative diagram, where the upper right square is bicartesian:
$$\xymatrix{0 \ar[r] &F^{'} \ar[r] \ar@{=}[d]& \tilde{F} \pushoutcorner \ar[r] \ar@{^(->}[d] & t(Z) \ar[r] \ar@{^(->}[d]& 0 \\ 0 \ar[r] &F^{'} \ar[r] & F   \ar[r]^{p} \ar@{>>}[d] & \pullbackcorner Z \ar[r] \ar@{>>}[d]& 0 & (\ast) \\ && \frac{Z}{t(Z)} \ar@{=}[r] & \frac{Z}{t(Z)} }$$

If we apply the long exact sequence of $\Ext(?,Q)$ to the central row and the central column of the last diagram, we get the following commutative diagram with exact rows:

$$\xymatrix{0 \ar[r] & \Ext^{2}_{\G}(\frac{Z}{t(Z)},Q) \ar[r] \ar[d]^{\alpha}& \Ext^{2}_{\G}(F,Q) \ar[r] \ar@{=}[d]& \Ext^{2}_{\G}(\tilde{F},Q) \ar[d] \\ 0 \ar[r] & \Ext^{2}_{\G}(Z,Q) \ar[r] & \Ext^{2}_{\G}(F,Q) \ar[r] & \Ext^{2}_{\G}(F^{'},Q) }$$

It follows that $\alpha:\Ext^{2}_{\G}(\frac{Z}{t(Z)},Q) \flecha \Ext^{2}_{\G}(Z,Q)$ is a monomorphism. If we now apply the long exact sequence of $\Ext$ to the right column of the diagram $(\ast)$ above, we get that the canonical morphism $\Ext^{1}_{\G}(Z,Q) \flecha \Ext^{1}_{\G}(t(Z),Q)$ is an epimorphism. This implies that $\Ext^{1}_{\G}(t(Z),Q)=0$ due to the choice of $Z$. It follows from this that $\Hom_{\Ht}(t(Z)[0],Q[1])=0$, which implies that $t(Z)=0$ since $Q[1]$ is a cogenerator of $\Ht$. We then get $\Ker(\Ext^{1}_{\G}(?,Q)) \subseteq \F =\Cogen(Q)$ and, hence, this last inclusion is an equality.\\

3) $\Longrightarrow$ 2) (assuming that $\G$ is locally finitely presented). It follows directly from lemma \ref{lemma Bazzoni}.
\end{proof}

\begin{remark}\rm{
\begin{enumerate}
\item [1)] When $\G$ is locally finitely presented, by proposition \ref{proposition clave} and lemma \ref{lemma Bazzoni}(3), we know that if $Q$ is a 1-cotilting object such that $\F=\Cogen(Q)$ is a generating class of $\G$, then $Q$ is strong 1-cotilting.
\item [2)] When $\G$ is AB4*, it follows from lemma \ref{lemma Bazzoni} and proposition \ref{proposition clave} that the following assertions are equivalent for a torsion pair $\te=(\T,\F)$:
\begin{enumerate}
\item[(a)] $\F$ is generating and closed under taking direct limits in $\G$;
\item[(b)] $\te$ is a cotilting torsion pair such that $\F$ is closed under taking direct limits in $\G$.
\end{enumerate}
\end{enumerate} }
\end{remark}

The following direct consequence of proposition \ref{proposition clave} extends \cite[Corollary 6.3]{CMT} (see also corollary \ref{cor. tiltilng implies cotilting}).

\begin{corollary}
Let $V$ be a 1-tilting object such that $\F=\Ker(\Hom_{\G}(V,?))$ is closed under taking direct limits in $\G$ (e.g., when $V$ is self-small). If $\F$ is a generating class, then the torsion pair $\mathbf{t}=(\Gen(V),\Ker(\Hom_{\G}(V,?)))$ is strong cotilting.
\end{corollary}

We now make explicit what proposition \ref{proposition clave} says in case $\G$ is locally finitely presented and AB4*. In fact, the next corollary extends the main result of \cite{CG} (see \cite[Theorem 6.2]{Maa}).

\begin{corollary}
Let $\G$ be locally finitely presented and AB4* and let $\te=(\T,\F)$ be a torsion pair in $\G$. The following assertions are equivalent:
\begin{enumerate}
\item[1)] $\F$ is a generating class and the heart $\Ht$ is a Grothendieck category;
\item[2)] $\F$ is a generating class closed under taking direct limits in $\G$;
\item[3)] $\te$ is a cotilting torsion pair.
\end{enumerate}
\end{corollary}


Recall that a Grothendieck category is called \emph{locally noetherian}\index{category! locally noetherian} when it has a set of noetherian generators. The following result extends \cite[Theorem A]{BK} (see lemma \ref{lemma Bazzoni}).

\begin{corollary}
Let $\G$ be a locally finitely presented Grothendieck category which is locally noetherian and denote by $fp(\G)$ its full subcategory of finitely presented (=noetherian) objects. There is a one-to-one correspondence between:
\begin{enumerate}
\item[1)] The torsion pairs $(\mathcal{X,Y})$ of $fp(\G)$ such that $\mathcal{Y}$ contains a set of generators;
\item[2)] The equivalence classes of 1-cotilting objects $Q$ of $\G$ such that $\Cogen(Q)$ is a generating class. \\

When, in addition, $\G$ is an AB4* category, they are also in bijection with 
\item[3)] The equivalence classes of 1-cotilting objects of $\G$.
\end{enumerate}
The map form 1 to 2 takes $(\mathcal{X,Y})$ to the equivalence class $[Q]$, where $Q$ is a 1-cotilting object such that $\Cogen(Q)=\{F\in \text{Ob}(\G): \Hom_{\G}(X,F)=0, \text{ for all }X\in \mathcal{X}\}$. The map from 2 to 1 takes $[Q]$ to $(\Ker(\Hom_{\G}(?,Q))\cap fp(\G), \Cogen(Q) \cap fp(\G))$.
\end{corollary}
\begin{proof}
By lemma \ref{lemma Bazzoni}, when $\G$ is AB4*, the classes in 2) and 3) are the same. We then prove the bijection between 1) and 2). Given a torsion pair $(\mathcal{X,Y})$ in $fp(\G)$ as in 1), by \cite[Lemma 4.4]{CB}, we know that the torsion pair in $\G$ generated by $\mathcal{X}$ is $\mathbf{t}=(\T,\F)=(\varinjlim{\mathcal{X}},\varinjlim{\mathcal{Y}})$. Thus, $\F$ is a generating class that is closed under direct limits. It follows from proposition \ref{proposition clave} that $\mathbf{t}$ is a cotilting torsion pair. We then get a 1-cotilting object $Q$, uniquely determined up to equivalence, such that $\Cogen(Q)=\varinjlim{\mathcal{Y}}=\{F\in \G : \Hom_{\G}(X,F)=0, \text{ for all }X\in \mathcal{X}\}$. \\

Suppose now that $Q$ is any 1-cotilting object and its associated torsion pair $\mathbf{t}=(\T,\F)$ has the property that $\F$ is a generating class. Then $(\mathcal{X,Y}):=(\T\cap fp(\G), \F \cap fp(\G))$ is a torsion pair in $fp(\G)$. We claim that $\mathcal{Y}$ contains a set of generators. Indeed, by hypothesis $\F$ contains a generator $G$ of $\G$. By the locally noetherian condition of $\G$, we know that $G$ is the union of its noetherian (=finitely presented) subobjects. Then the finitely presented subobjects of $G$ form a set of generators of $\G$ which is in $\mathcal{Y}$, thus settling our claim. \\

On the other hand, in the situation of last paragraph, we have that $(\varinjlim{\mathcal{X}},\varinjlim{\mathcal{Y}})$ is a torsion pair in $\G$ such that $\varinjlim{\mathcal{X}}\subseteq \T$ and $\varinjlim{\mathcal{Y}}\subseteq \F.$ Then these inclusions are equalities and, hence, $\te$ is the image of $(\mathcal{X,Y})$ by the map from 1 to 2 defined in the first paragraph of this proof. That the two maps, from 1 to 2 and from 2 to 1, are mutually inverse is then a straightforward consequence of this.
\end{proof}

\begin{example}\rm{
The category $Qcoh(X)$ of quasi-coherent sheaves over quasi-compact scheme $X$ which is locally Noetherian, is a locally finitely presented Grothedieck category which is locally Noetherian (see \cite[I.6.9.12]{Gr2}).}
\end{example}

\chapter{$\Ht$ as a module category}

In this chapter, given any torsion pair $\mathbf{t}$ in a module category in $R$-Mod, we give necessary and sufficient conditions for the heart $\Ht$ to be a module category and, simultaneously, compare this property with that of $\te$ being an HKM torsion pair. \\

All throughout this section, $R$ will be a ring and $\te$ will be a torsion pair in $R\text{-Mod}$. Recall that each complex of $\Ht$ is quasi-isomorphic to a complex (see examples  \ref{example t-structure}(2)) $\xymatrix{\cdots \ar[r] & 0 \ar[r] & X \ar[r]^{j} & Q \ar[r]^{d} & P \ar[r] & 0 \ar[r] & \cdots}$, concentrated -2,-1,0, such that $j$ is a monomorphism and $P,Q$ are projective modules. 
 
\section{Progenerator}
 
\begin{lemma}\label{lem. homology 0 module}
The following assertions hold:
\begin{enumerate}
\item[1)] For each $G\in \D^{\leq 0}(R)$ and for each $R$-module $M$, there is an isomorphism \linebreak $\Hom_{R}(H^{0}(G),M) \iso \Hom_{\D(R)}(G,M[0])$, which is natural in $M$.

\item[2)] If $V$ is a $R$-module in $\T$, such that $\Hom_{R}(V,?)$ preserves direct limits of objects in $\T$ and $\varinjlim{\F}=\F$, then $V$ is a finitely presented $R$-module.

\item[3)] Let $G:= \xymatrix{\cdots \ar[r] & 0 \ar[r] & X \ar[r]^{j} & Q \ar[r]^{d} & P \ar[r] & 0 \ar[r] & \cdots}$, be a complex of $R$-modules with $P$ in degree 0, where $P$ and $Q$ are projective and $j$ is a monomorphism, and let $M$ be any $R$-module. When we view $X$ as a submodule of $Q$ and $j$ as the inclusion, there are natural in $M$ exact sequences of abelian groups:
\begin{enumerate}
\item[a)] $\xymatrix{\Hom_{R}(P,M) \ar[r] & \Hom_R(Q/X,M) \ar[r] & \Hom_{\D(R)}(G,M[1]) \ar[r] & 0}$
\item[b)] $\xymatrix{\Hom_R(Q,M) \ar[r] & \Hom_{R}(X,M) \ar[r] & \Hom_{\D(R)}(G,M[2]) \ar[r] &0}$
\end{enumerate}
\end{enumerate}
\end{lemma}
\begin{proof}
1) For each $G\in \D^{\leq 0}(R)$, we have the following triangle in $\D(R)$:

$$\xymatrix{\tau^{\leq -1}(G) \ar[r] & G \ar[r] & H^{0}(G)[0] \ar[r]^{\hspace{0.6 cm}+} &}$$

Applying the cohomological functor $\Hom_{\D(R)}(?,M[0])$ to the previous triangle, we obtain the following exact sequence:

$$\xymatrix{ \Hom_{\D(R)}(\tau^{\leq -1}(G)[1],M[0])=0 \ar[r] & \Hom_{R}(H^{0}(G),M) \ar[r] & \Hom_{\D(R)}(G,M[0]) \ar[d] \\ & & \Hom_{\D(R)}(\tau^{\leq -1}(G),M[0])=0}$$

2) Let now $(M_i)_{i \in I}$ be any direct system in $R$-Mod. Note that $\varinjlim{t(M_{i})}\cong t(\varinjlim{M_i})$ since $\varinjlim{\F}=\F$. We now have isomorphisms:

$$\xymatrix{\varinjlim \Hom_{R}(V,M_i)  & \varinjlim{\Hom_{R}(V,t(M_i))} \ar[l]_{\sim} \ar[r]^{\sim \hspace{2.1cm}}& \Hom_{R}(V,\varinjlim{t(M_i)})=\Hom_{R}(V,t(\varinjlim{M_i})) \ar@<11ex>[d]^{\wr}\\ && \hspace{4 cm} \Hom_{R}(V,\varinjlim{M_{i}})}$$

3) We have the following triangle in $\D(R)$:
$$\xymatrix{Q/X[0] \ar[r] & P[0] \ar[r] & G \ar[r]^{+} &  }$$
Applying the cohomological functor $\Hom_{\D(R)}(?,M)$ and looking at the corresponding long exact sequences, we obtain 3.a). On the other hand, one easily sees that a morphism $G \flecha M[2]$ in $\D(R)$ is represented by an $R$-homomorphism $f:X \flecha M$. The former morphism is the zero morphism in $\D(R)$ precisely when $f$ factors through $j$. Then the exact sequence in 3.b) follows immediately.
\end{proof}

\begin{lemma}\label{lem. progenerator of Ht}
If $G$ is a progenerator of $\Ht$, then the following assertions hold, where \linebreak $V:=H^{0}(G):$
\begin{enumerate}
\item[1)] $\T=\Gen(V)=\Pres(V)$, and hence $\F=\Ker(\Hom_{R}(V,?))$;

\item[2)] $V$ is a finitely presented $R$-module;

\item[3)] $V$ is a classical quasi-tilting $R$-module.
\end{enumerate}
\end{lemma}
\begin{proof}
By hypothesis $\Ht$ is a module category, in particular it is an AB5 category. Then from theorem \ref{caracterizacion AB5} we get that $\F$ is closed under taking direct limits in $R \text{- Mod}$. On the other hand, by lemma \ref{exactness of H}, the functor $H^{0}:\Ht \flecha R\text{-Mod}$ is right exact and preserves coproducts. When applied to an exact sequence $G^{(I)} \flecha G^{(J)} \flecha T[0] \flecha 0$  in $\Ht$, we get that $T\in \Pres(V)$, for each $T\in \T$. We then get that $\T=\Pres(V)$, and assertion 1) follows from \cite[Proposition 2.2]{MT}. \\

From proposition \ref{description of stalk}, we have that $\limite{T_{i}[0]} \cong (\varinjlim{T_i}[0])$, for each direct system $(T_i)_{i \in I}$ in $\T$. By the previous lemma, we then get that $\Hom_{R}(V,?)$ preserves direct limits of objects in $\T$ since $G$ is a finitely presented object of $\Ht$. Using again the previous lemma, we get that $V$ is a finitely presented $R$-module. Finally, assertion 3 follows from \cite[Proposition 2.4]{MT}, from assertions 1 and 2 and from \cite[Proposition 2.1]{CDT1}.
\end{proof}

\vspace{0.3 cm}

The following result is inspired by \cite[Proposition 5.9]{CMT}.

\begin{lemma}\label{lem. stalk generates}
Let $\xymatrix{G:=\cdots \ar[r] & 0 \ar[r] & X \ar[r]^{j} & Q \ar[r]^{d} & P \ar[r] & 0 \ar[r] & \cdots}$, be a complex of $R$-modules with $P$ in degree 0, where $P$ and $Q$ are projective and $j$ is a monomorphism. \\
If $G$ is a projective object of $\Ht$ such that $\T=\Gen(V)=\Pres(V)$, where $V:=H^{0}(G)$, then the following assertions hold:

\begin{enumerate}
\item[1)] $\frac{M}{t(M)}[1]\in \Gen_{\Ht}(G)$, for each $M \in \overline{\Gen}(V)$;
\item[2)] $\frac{R}{t(R)}[1] \in \Gen_{\Ht}(G)$ if, and only if, there exists a set $I$ and a morphism \linebreak $h:(\frac{Q}{X})^{(I)} \flecha \frac{R}{t(R)}$ such that the cokernel of the restriction of $h$ to $(H^{-1}(G))^{(I)}$, belongs to $\overline{\Gen}(V)$.
\end{enumerate}
\end{lemma}
\begin{proof}
Assertion 1 essentially follows from \cite[Lemma 5.8 and Proposition 5.9]{CMT}, but, for the sake of completeness, we give a short proof. It is clear that $\overline{Gen}(V)$ is closed under quotients, so that $\frac{M}{t(M)}\in \overline{\Gen}(V)$, for each $M\in \overline{\Gen}(V)$. Then there is an exact sequence of the form (see \cite[Lemma 5.6]{CMT}):

$$0 \flecha \frac{M}{t(M)} \flecha T \flecha V^{(I)} \flecha 0$$
for some set $I$, where $T\in \T$. Thus, we have the following exact sequence in $\Ht$:

$$\xymatrix{0 \ar[r] & T[0] \ar[r] & V^{(I)}[0] \ar[r] & \frac{M}{t(M)}[1] \ar[r] & 0}$$

Now the assertion follows from the fact that we have an epimorphism $G^{(I)} \epic V^{(I)}[0]$ in $\Ht.$ \\

 2) For the implication $\Longrightarrow$, note that, by hypothesis, there is an exact sequence of the form 
 $$0 \flecha K \flecha G^{(I)} \xymatrix{\ar@{>>}[r]^{p} &} \frac{R}{t(R)}[1] \flecha 0$$
in $\Ht$, for some set $I$. It is easy to see that $p$ is represented by an $R$-homomorphism $p^{-1}:(\frac{Q}{X})^{(I)} \flecha \frac{R}{t(R)}$. Moreover, we have that $H^{-1}(p)$ coincides with the restriction of $p^{-1}$ to $(H^{-1}(G))^{(I)}$. Now, we consider the long exact sequence associated to the above triangle:
$$\xymatrix{0 \ar[r] & H^{-1}(K) \ar[r] & H^{-1}(G)^{(I)} \ar[rr]^{\hspace{0.5 cm}H^{-1}(p)} && \frac{R}{t(R)} \ar[r] & H^{0}(K) \ar[r] & V^{(I)} \ar[r] & 0}$$

The result follows by putting $h=p^{-1}$ since $H^{0}(K)\in \T$. \\

For the implication $\Longleftarrow$, suppose that there exist a set $I$ and $h\in \Hom_{R}((\frac{Q}{X})^{(I)}, \frac{R}{t(R)})$, such that the cokernel of restriction of $h$ to $(H^{-1}(G))^{(I)}$, which we denote by $Z$, belongs to $\overline{\Gen}(V)$. Clearly, we can extend $h$ to a morphism from $G^{(I)}$ to $\frac{R}{t(R)}[1]$ in $\D(R)$, which we denote by $\overline{h}$. We now complete $\overline{h}$ to a triangle in $\D(R)$ and get:

$$\xymatrix{M \ar[r] & G^{(I)} \ar[r]^{\overline{h}\hspace{0.3 cm}} & \frac{R}{t(R)}[1] \ar[r]^{\hspace{0.45 cm}+} & }$$

Using the long exact sequence of homologies, we then obtain that $H^{-1}(M)$ is a submodule of $H^{-1}(G)^{(I)}$ and hence $H^{-1}(M)\in \F$. Moreover, we also get an exact sequence

$$\xymatrix{0 \ar[r] & Z \ar[r] & H^{0}(M) \ar[r] & V^{(I)} \ar[r] & 0}$$

By \cite[Lemma 5.6]{CMT}, we know that $H^{0}(M)\in \overline{\Gen}(V)$ and then, by assertion 1, we also get that $\frac{H^{0}(M)}{t(H^{0}(M))}[1]\in \Gen_{\Ht}(G)$. Consider now the following commutative diagram:

$$\xymatrix{&&& &\\ &&& t(H^{0}(M))[1] \ar[ur]^{+} \ar[dd]\\ & N \ar[rru] \ar[rd] &&\\ H^{-1}(M)[2] \ar[rr] \ar[ru]&& M[1] \ar[r] \ar[rd] & H^{0}(M)[1] \ar[d] \ar[r]^{\hspace{0.6 cm}+} & \\ &&&\frac{H^{0}(M)}{t(H^{0}(M))}[1] \ar[dr]^(0.6){+} \ar[d]^(0.55){+} \\&&&& }$$

Note that $N\in \Ht[1]$, since $H^{-1}(M)\in \F$. Thus, by \cite{BBD}, we get that $\Coker_{\Ht}(\overline{h}) \cong \frac{H^{0}(M)}{t(H^{0}(M))}[1]$. Then we have the following diagram with exact row in $\Ht$:

$$\xymatrix{&&G^{(J)} \ar@{>>}[d]^(0.35){p} \ar@{-->}[dl]_(0.45){p'} \\ G^{(I)} \ar[r]^{\overline{h} \hspace{0.2 cm}} & \frac{R}{t(R)}[1] \ar[r] & \frac{H^{0}(M)}{t(H^{0}(M))}[1] \ar[r] & 0}$$

where $p$ is an epimorphism and $p'$ is obtained by the projectivity of $G^{(J)}$ in $\Ht$. It follows that $(\overline{h} \ \ p'):G^{(I)} \coprod G^{(J)} \flecha \frac{R}{t(R)}[1]$ is also an epimorphism in $\Ht$.
\end{proof}

\vspace{0.3 cm}

We are now able to give a general criterion for $\Ht$ to be a module category.

\begin{theorem}\label{teo. Ht is a module category}
The heart $\Ht$ is a module category if, and only if, there is a chain complex of $R$-modules 

$$\xymatrix{G:=\cdots \ar[r] & 0 \ar[r] & X \ar[r]^{j} & Q \ar[r]^{d} & P \ar[r] & 0 \ar[r] & \cdots }$$

with $P$ in degree 0, satisfying the following properties, where $V:=H^{0}(G)$:

\begin{enumerate}
\item[1)] $\T=\Pres(V)\subseteq \Ker(\Ext^{1}_{R}(V,?))$;

\item[2)] $Q$ and $P$ are finitely generated projective $R$-modules and $j$ is a monomorphism such that $H^{-1}(G)\in \F$;

\item[3)] $H^{-1}(G)\subseteq \Rej_{\T}(\frac{Q}{X})$;

\item[4)] $\Ext^{1}_{R}(\Coker(j),?)$ vanishes on $\F$;

\item[5)] there is a morphism $h:(\frac{Q}{X})^{(I)} \flecha \frac{R}{t(R)}$, for some set $I$, such that the cokernel of its restriction to $(H^{-1}(G))^{(I)}$ is in $\overline{\Gen}(V)$.
\end{enumerate}
\end{theorem}
\begin{proof}
Let us assume that $G$ is a complex of the form:

$$\xymatrix{\cdots \ar[r] & 0 \ar[r] & X \ar[r]^{j} & Q \ar[r]^{d} & P \ar[r] & 0 \ar[r] & \cdots & (\ast)}$$
with $P$ in degree 0, such that $G\in \Ht$, where $P$ and $Q$ are projective $R$-modules and $j$ is a monomorphism. By lemma \ref{lem. progenerator of Ht}, if $G$ is a progenerator of $\Ht$, then $V:=H^{0}(G)$ is finitely presented. This allows us, for both implications in the proof, to assume that $P$ is a finitely generated projective $R$-module. \\

We claim that $G$ is a projective object in $\Ht$ if, and only if, $\T\subseteq \Ker(\Ext^{1}_{R}(V,?))$ and conditions 3 and 4 hold. Indeed each object $M\in \Ht$ fits into an exact sequence in this category:

$$\xymatrix{0 \ar[r] & H^{-1}(M)[1] \ar[r] & M \ar[r] & H^{0}(M)[0] \ar[r] & 0 & (\ast \ast)}$$

Then $G$ is projective in $\Ht$ if, and only if, $0=\Ext^{1}_{\Ht}(G,T[0])=\Hom_{\D(R)}(G,T[1])$ and $0=\Ext^{1}_{\Ht}(G,F[1])=\Hom_{\D(R)}(G,F[2]),$ for all $T\in \T$ and $F\in \F$. Since each morphism $G \flecha T[1]$ in $\D(R)$ is a morphism in $\mathcal{H}(R)$, the first equality holds if, and only if, the map $\Hom_{R}(P,T) \xymatrix{ \ar[r]^{\overline{d}^{*} \hspace{1.1 cm}}& \Hom_{R}(\frac{Q}{X},T)}$ is surjective, where $\overline{d}:\frac{Q}{X} \flecha P$ is the obvious $R$-homomorphism. But, in turn, this last condition is equivalent to the sum of the following two conditions, for each $T\in \T$:

\begin{enumerate}
\item[i)] Each $R$-homomorphism $f:\frac{Q}{X} \flecha T$ vanishes on $H^{-1}(G)=\frac{\Ker(d)}{X}$; 
\item[ii)] Each morphism $g:\Imagen(\overline{d})=\Imagen(d) \flecha T$ extends to $P$.
\end{enumerate}

Condition i) is equivalent to saying that $H^{-1}(G)\subseteq \Rej_{\T}(\frac{Q}{X})$. On the other hand, from the exact sequence 

$$\Hom_{R}(P,T) \flecha \Hom_{R}(\Imagen(d),T) \flecha \Ext^{1}_{R}(H^{0}(G),T) \flecha 0$$
the condition ii) above is equivalent to saying that $\Ext^{1}_{R}(H^{0}(G),T)=0$, for all $T\in \T.$ Now, from lemma \ref{lem. homology 0 module}, we obtain that the equality $\Hom_{\D(R)}(G,F[2])=0$ holds when each $R$-homomorphism $g:X \flecha F$ extends to $Q$, for all $F\in \F$. This is clearly equivalent to condition 4 in the list since we have the following exact sequence:

$$\Hom_{R}(Q,F) \flecha \Hom_{R}(X,F) \flecha \Ext^{1}_{R}(\Coker(j),F) \flecha 0$$

Suppose that $G$ is projective in $\Ht$ or its equivalent conditions mentioned in the previous paragraph. From example \ref{exam. Ht AB4} we know that $\Ht$ is AB4. Applying this fact to any family of exact sequences as $(\ast \ast)$, we see that $G$ is compact object of $\Ht$ if, and only if, the canonical morphisms
\begin{center}
\hspace{1.8 cm}$\xymatrix{\underset{i\in I}{\coprod} \Hom_{\D(R)}(G,T_i[0]) \ar[r] & \Hom_{\D(R)}(G,(\underset{i \in I}{\coprod}T_i)[0])} \newline \xymatrix{\underset{i \in I}{\coprod} \Hom_{\D(R)}(G,F_i[1]) \ar[r] & \Hom_{\D(R)}(G, (\underset{i\in I}{\coprod}F_i)[1])} $
\end{center}
are isomorphisms, for all families $(T_i)$ in $\T$ and $(F_i)$ in $\F$. By lemma \ref{lem. homology 0 module}(1), we easily get that the first of these morphisms is an isomorphism precisely when $V$ is a compact object of $\T$. Similarly, by the assertion 2) of such lemma, we have that the second centered homomorphism is an isomorphism whenever $P$ and $Q/X$ are finitely generated modules. Therefore, if $G$ satisfies the conditions 1-4 of the list, then $G$ is a compact projective object of $\Ht$. \\

Suppose now that these last conditions hold. Then, due to the canonical sequence ($\ast \ast$), we know that $G$ is a generator if, and only if, each $M\in \T[0] \cup \F[1]$ is generated by $G$. Note that we have an epimorphism $q:G \epic V[0]$ in $\Ht$, which implies that $\Gen_{\Ht}(V[0]) \subseteq \Gen_{\Ht}(G)$. But the equality $\T=\Gen(V)=\Pres(V)$ easily gives that $\T[0]\subseteq \Gen_{\Ht}(V[0])$. On the other hand, each $F\in \F$ gives rise to an exact sequence $0 \flecha F^{'} \monic (\frac{R}{t(R)})^{(\alpha)} \epic F \flecha 0$ in $R$-Mod which, in turn, yields an exact sequence in $\Ht$:

$$0 \flecha F^{'}[1] \flecha (\frac{R}{t(R)})^{(\alpha)}[1] \flecha F[1] \flecha 0$$
Thus, $G$ generates $\Ht$ if, and only if, it generates $\frac{R}{t(R)}[1]$. By lemma \ref{lem. stalk generates}, this equivalent to condition 5 in the list. \\

Note that the ``if'' part of the proof follows from the previous paragraphs. By lemma \ref{lem. progenerator of Ht} and lemma \ref{lem. stalk generates}, in order to prove the ``only if'' part, we only need to prove that if $\Ht$ is a module category and $G$ is a complex like ($\ast$) which is a progenerator of $\Ht$, then $G$ can be represented by a complex as indicated in the statement satisfying the properties 1-5. \\

Lemma \ref{lem. Q/X f.g.} below shows that $\frac{Q}{X}$ is finitely generated, which allows us to replace $Q$ by an appropriate finitely generated direct summand $Q^{'}$ such that the composition \linebreak $p^{'}: Q^{'} \monic Q \xymatrix{ \ar@{>>}[r] &} \frac{Q}{X}$ is surjective. Then the following complex 
$$\xymatrix{ \cdots \ar[r] & 0 \ar[r] & \Ker(p^{'}) \ar[r] & Q^{'} \ar[r]^{d|_{Q^{'}}} & P \ar[r] & 0 \ar[r] & \cdots}$$ 
is quasi-isomorphic to the complex $(\ast)$ and satisfies all conditions 1-5 in the list.
\end{proof}

\begin{lemma}\label{lem. Q/X f.g.}
Let $\xymatrix{G:=\cdots \ar[r] & 0 \ar[r] & X \ar[r]^{j} & Q \ar[r]^{d} & P \ar[r] & 0 \ar[r] & \cdots}$ be a complex concentrated in degrees -2,-1,0 such that $j$ is a monomorphism, $Q$ and $P$ are projective and $P$ is finitely generated. Suppose that the complex represents a progenerator of $\Ht$. Then $\Coker(j)$ is a finitely generated $R$-module.
\end{lemma}
\begin{proof}
As customary, we view $j$ as an inclusion. We identify $G$ with complex

$$\xymatrix{\cdots \ar[r] & 0 \ar[r] & \frac{Q}{X} \ar[r]^{\overline{d}} & P \ar[r] & 0 \ar[r] & \cdots}$$

By lemma \ref{lem. progenerator of Ht}, we know that $V:=H^{0}(G)$ is a finitely presented $R$-module. Then $\Imagen(d)=\Imagen(\overline{d})$ is a finitely generated submodule of $P$ and we can select a finitely generated submodule $A^{'} < \frac{Q}{X}$ such that $\overline{d}(A^{'})=\Imagen(d).$ We fix a direct system $(A_{\lambda})_{\lambda \in \Lambda}$ of finitely generated submodules of $\frac{Q}{X}$, such that $A^{'}< A_{\lambda}$, for all $\lambda \in \Lambda$, and $\varinjlim{A_{\lambda}}=\frac{Q}{X}$. For each $\lambda$, we denote $G_{\lambda}$ to the following complex:

$$\xymatrix{G_{\lambda}:=\cdots \ar[r] & 0 \ar[r] & A_{\lambda} \ar[r]^{\hspace{0.1 cm}\overline{d}|_{A_{\lambda}}} & P \ar[r] & 0 \ar[r] & \cdots}$$

It is clear that $(G_{\lambda})_{\lambda \in \Lambda}$ is a direct system in $\mathcal{C}(R)$ and in $\Ht$, and that we have $\varinjlim_{\mathcal{C}(R)}G_{\lambda}\cong G$. On the other hand, from lemma \ref{lem. progenerator of Ht} we get that $\F=\Ker(\Hom_{R}(V,?))$ is closed under direct limits. Thus, from lemma \ref{limites directos en C}, we have that $\limite{G_{\lambda}} \cong\varinjlim_{\mathcal{C}(R)}{G_{\lambda}} = G$. But, since $G$ is a finitely presented object of $\Ht$, the identify map $1_{G}: G \flecha G$ factors in this category in the form $\xymatrix{G \ar[r]^{f} & G_{\mu} \ar[r]^{\iota_\mu} & G}$, for some $\mu \in \Lambda$. It follows that $H^{-1}(\iota_{\mu})$ is an epimorphism and, therefore, it is an isomorphism. We then get a commutative diagram with exact rows:

$$\xymatrix{0 \ar[r] & H^{-1}(G_{\mu}) \ar[r] \ar[d]_{H^{-1}(\iota_{\mu})}^{\wr} & A_{\mu} \ar[r]  \ar@{^(->}[d]^{\iota_{\mu}^{-1}} & \overline{d}(A_{\mu}) \ar@{=}[d] \ar[r] & 0 \\ 0 \ar[r] & H^{-1}(G) \ar[r] & \frac{Q}{X} \ar[r]^{\overline{d} \  \ } & \Imagen(d) \ar[r] & 0}$$

This show that $A_{\mu}\cong \frac{Q}{X}$ is a finitely generated $R$-module.
\end{proof}

\vspace{0.3 cm}

We refer the reader to section \ref{section:torsion pair} for the definition of HKM torsion pair. Every HKM complex gives a HKM torsion pair, whose associated heart is a module category (see \cite[Theorem 3.8]{HKM}). However, the complex HKM need not belong to the heart (see example \ref{exam. HKM torsion} below). Then, it comes the question of how should be a progenerator for such a heart. Our next result in this section gives a criterion for a torsion pair to be HKM:

\begin{proposition}\label{pro. HKM complex}
Let $\xymatrix{P^{\bullet}:=\cdots \ar[r] & 0 \ar[r] & Q \ar[r]^{d} & P \ar[r] & 0 \ar[r] & \cdots}$ be a complex of finitely generated projective modules concentrated in degrees -1 and 0, where $V:=H^{0}(P^{\bullet})$, and let $\te=(\T,\F)$ be a torsion pair in $R$-Mod. The following assertions are equivalent:

\begin{enumerate}

\item[1)] $P^{\bullet}$ is an HKM complex such that $\te$ is its associated HKM torsion pair;

\item[2)] $V\in \mathcal{X}(P^{\bullet})$ and the complex $\xymatrix{G:=\cdots \ar[r] & 0 \ar[r] & t(\Ker(d)) \ar[r]^{\hspace{0.7 cm}j} & Q \ar[r]^{d \hspace{0.1 cm}} & P \ar[r] & 0 \ar[r] & \cdots }$, concentrated in degrees -2,-1,0, is a progenerator of $\Ht$;

\item[3)] The following conditions hold:

\begin{enumerate}

\item[a)] $\T=\Pres(V)\subseteq \Ker(\Ext^{1}_{R}(V,?))$ and $\mathcal{X}(P^{\bullet})\subseteq \T$;

\item[b)] $\Ker(d) \subseteq \Rej_{\T}(Q)$;

\item[c)] There is a homomorphism $h:Q^{(I)} \flecha \frac{R}{t(R)}$, for some set $I$, such that \linebreak $\Coker(h_{| \Ker(d)^{(I)}})\in \overline{\Gen}(V).$
\end{enumerate}
\end{enumerate}
\end{proposition}
\begin{proof}
Note first that we have an exact sequence $0 \flecha t(\Ker(d))[1] \flecha P^{\bullet} \flecha G^{'} \flecha 0$ in $\mathcal{C}(R)$, where $G^{'}$ is quasi-isomorphic to $G$. We then get a triangle in $\D(R)$:

$$\xymatrix{t(\Ker(d))[1] \ar[r] & P^{\bullet} \ar[r] & G \ar[r]^{\hspace{0.1 cm}+} & }$$

In particular, we get a natural isomorphism 
$$\Hom_{\Ht}(G,?)=\Hom_{\D(R)}(G,?)_{| \Ht} \iso \Hom_{\D(R)}(P^{\bullet},?)_{|\Ht}$$ of functors $\Ht \flecha $Ab since $\Hom_{\D(R)}(t(\Ker(d))[k],?)$ vanishes on $\Ht$, for $k=1,2.$ Also, we have an exact sequence of functors $R$-Mod$ \flecha $Ab:

$$\xymatrix{0 \ar[r] & \Hom_{\D(R)}(G,?[1]) \ar[r] & \Hom_{\D(R)}(P^{\bullet},?[1]) \ar[r] & \Hom_{R}(t(\Ker(d)),?) \ar@{>>}[d] \\ &&& \Hom_{\D(R)}(G,?[2])}$$
since $P^{\bullet}$ is a complex of projective $R$-modules concentrated in degrees -1 and 0. \\

1) $\Longrightarrow$ 2) From \cite[Theorem 3.8]{HKM} (see also theorem \ref{teo. HKM principal}), we know that 
$$\Hom_{\D(R)}(P^{\bullet},?)_{|\Ht}:\Ht \flecha \text{End}_{\D(R)}(P^{\bullet})\text{-Mod}$$ 
is an equivalence of categories, hence the functor $\Hom_{\Ht}(G,?)\cong \Hom_{\D(R)}(P^{\bullet},?)_{|\Ht}:\Ht \flecha $Ab is faithful, i.e., $G$ is a generator of $\Ht$. Moreover, this functor preserves coproducts since $P^{\bullet}$ is a compact object of $\D(R)$. Thus, $G$ is a compact object of $\Ht$. On the other hand, from the initial comments of this proof and the fact that $\T=\mathcal{X}(P^{\bullet})=\Ker(\Hom_{\D(R)}(P^{\bullet}, ?[1])_{| R\text{-Mod}})$, we get a monomorphism $\Ext^{1}_{\Ht}(G,T[0])=\Hom_{\D(R)}(G,T[1]) \monic \Hom_{\D(R)}(P^{\bullet}, T[1])=0$, for each $T\in \T$. Furthermore, $\Ext^{1}_{\Ht}(G,F[1])\linebreak =\Hom_{\D(R)}(G,F[2])$ is a homomorphic image of $\Hom_{R}(t(\Ker(d)),F)=0$, for all $F\in \F$. We conclude that $G$ is a projective object, and hence a progenerator of $\Ht$ since $\Ext^{1}_{\Ht}(G,?)$ vanishes on $\T[0]$ and on $\F[1]$. \\

2) $\Longrightarrow$ 1) As in the initial argument we get a natural isomorphism of functors \linebreak $R\Mode \flecha$Ab given by:
$$\Hom_{\D(R)}(G, ?[0]) \iso \Hom_{\D(R)}(P^{\bullet}, ?[0])$$

From lemma \ref{lem. homology 0 module}, we obtain that $\mathcal{Y}(P^{\bullet})$ consists of the modules $F$ such that \linebreak $\Hom_{R}(H^{0}(P^{\bullet}),F)\cong \Hom_{\D(R)}(G,F[0])\cong \Hom_{\D(R)}(P^{\bullet},F[0])=0$. But theorem \ref{teo. Ht is a module category} and its proof tell us that $H^{0}(G)=V$ generates $\T$, so that we have $\mathcal{Y}(P^{\bullet})=\F$. On the other hand, if $F\in \mathcal{X}(P^{\bullet}) \cap \mathcal{Y}(P^{\bullet})=\mathcal{X}(P^{\bullet}) \cap \F$ then, again, our initial comments in this proof show that $\Hom_{\D(R)}(G,F[1])=0$. But this implies that $F[1]=0$, and hence $F=0$, since $G$ is a generator of $\Ht$. Now, assertion 1 follows from \cite[Theorem 2.10]{HKM} and the fact that $\mathcal{Y}(P^{\bullet})=\F$. \\

1), 2) $\Longrightarrow$ 3) From theorem \ref{teo. Ht is a module category} and its proof we know that the complex $G$ satisfies conditions 1-5 of such theorem. In particular, we get condition 3.a). As for 3.c), note that $\Hom_{R}(t(\Ker(d)),?)$ vanishes on $\F$, so that we have isomorphisms of functors:

$$\xymatrix{ \Hom_{R}(\Coker(j)^{(J)},?)_{| \F } \ar[r]^{  \hspace{0.6 cm}\sim} &  \Hom_{R}(Q^{(J)},?)_{|\F} \\ \Hom_{R}(H^{-1}(G),?)_{| \F} \ar[r]^{\sim \hspace{0.2 cm}} &  \Hom_{R}(\Ker(d)^{(J)},?)_{|\F} }$$
Then condition 5 of theorem \ref{teo. Ht is a module category} is exactly our condition 3.c) in this case. Finally, any homomorphism $f:Q \flecha T$, with $T\in \T$, gives a morphism $P^{\bullet} \flecha T[1]$ in $\D(R)$. But this is the zero morphism since $\T=\mathcal{X}(P^{\bullet})$. This implies that $f$ factors through $d$, so that $f(\Ker(d))=0$ and hence the condition 3.b) holds. \\

3) $\Longrightarrow$ 2) We show that $G$ satisfies the conditions 1-5 of theorem \ref{teo. Ht is a module category}. It clearly satisfies conditions 1) and 2) of that theorem. Moreover, arguing as in the proof of 1),2) $\Longrightarrow$ 3), also condition 5 in that theorem holds. Hence, in order to finish the proof it is enough to show that $H^{-1}(G)\subseteq \Rej_{\T}(\Coker(j))$ and $\Ext^{1}_{R}(\Coker(j),?)$ vanishes on $\F$. This last condition follows by applying the long exact sequence of $\Ext^{1}_R(?,F)$, with $F\in \F$, to the exact sequence
$$\xymatrix{0 \ar[r] & t(\Ker(d)) \ar[r]^{\hspace{0.7 cm}j} & Q \ar[r] & \Coker(j) \ar[r] & 0 }$$

Finally, we consider an $R$-homomorphism $g:\frac{Q}{t(\Ker(d))} =\Coker(j)\flecha T$ with $T\in\T$. By hypothesis $\xymatrix{Q \ar[r]^{g \hspace{0.05 cm}\circ \hspace{0.05 cm} q} & T}$ vanishes on $\Ker(d)$, where $q$ is the canonical epimorphism. From the following commutative diagram, we obtain that the composition $\xymatrix{H^{-1}(G) \ar@{^(->}[r] & \frac{Q}{t(\Ker(d))} \ar[r]^{\hspace{0.5 cm}g} & T}$ is the zero morphism, and hence the proof is finished.

$$\xymatrix{ 0 \ar[r] & t(\Ker(d)) \ar[r] \ar@{=}[d] & \Ker(d) \ar@{>>}[r] \ar@{^(->}[d] & H^{-1}(G) \ar[r] \ar@{^(->}[d] & 0 \\ 0 \ar[r] & t(\Ker(d)) \ar[r] & Q \ar[r]^{q \hspace{0.3 cm}} \ar[d]^{g \hspace{0.05 cm}\circ \hspace{0.05 cm} q} & \frac{Q}{t(\Ker(d))} \ar[r] \ar[d]^{g} & 0 \\ && T \ar@{=}[r] & T}$$ 
\end{proof}

\vspace{0.3 cm}

Using later results in this manuscript, we can show that HKM complexes need not be in the heart.

\begin{example}\label{exam. HKM torsion}\rm{
Let $Q_n:1 \flecha 2 \flecha \cdots \flecha n$ ($n >1$) be the Dynkin quiver of type $\mathbf{A}$ and let $R=KQ_n$ the corresponding path algebra, where $K$ is any field. \linebreak If $\mathbf{a}=Re_nR$ and $\te$ is the right constituent pair of the TTF triple associated to $\mathbf{a}$, then $\te$ is an HKM torsion pair in $R$-Mod such that its HKM complex associated doesn't belong to $\Ht$.}
\end{example}

\begin{proof}
Note that $\mathbf{a}=Re_n$, so that $e_n R(1-e_n)=0$ and 
$$R \cong \begin{pmatrix} e_nRe_n & 0 \\ (1-e_n)Re_n & (1-e_n)R(1-e_n)\end{pmatrix}\cong \begin{pmatrix} K & 0 \\ E & KQ_{n-1} \end{pmatrix}$$
where $E\cong KQ_{n-1}e_{n-1}$ is an injective-projective $KQ_{n-1}$-module. By \cite[Theorem 3.1]{NS2}, we know that the associated TTF-triple is left split. Note that $t(\mathbf{a})=\{x\in \mathbf{a} : e_nx=0\}=Je_n$, where $J$ is the Jacobson radical of $R$. It follows that $\frac{\mathbf{a}}{t(\mathbf{a})}=S_n$ is the simple $R$-module corresponding to the vertex $n$. By example \ref{exam. R/a+a as progenerator}, we know that $G^{'}=\frac{R}{\mathbf{a}}[0] \oplus \frac{\mathbf{a}}{t(\mathbf{a})}[1]\cong R(1-e_n)[0]\oplus S_n[1]$ is a progenerator of $\Ht$. \\


On the other hand, from \cite[Theorem 2.12]{MT} we know that $\te$ is an HKM torsion pair. Let now $P^{\bullet}:=\cdots \flecha 0 \flecha P^{-1} \xymatrix{\ar[r]^{f} &} P^{0} \flecha 0 \flecha \cdots$ be an HKM complex whose associated torsion pair is $\te$ and suppose that $P^{\bullet}\in \Ht$. From the previous proposition we would get a triangle in $\D(R)$ of the form: 

$$ t(\Ker(f))[1] \flecha P^{\bullet} \xymatrix{\ar[r]^{\rho} &  G \ar[r]^{+} & }$$
where $G$ is a progenerator of $\Ht$. By the explicit construction of cokernels in $\Ht$ (see \cite{BBD}), we would deduce that $\rho$ is an isomorphism, and hence $P^{\bullet}$ is a progenerator of $\Ht$. Then we would have $\text{add}_{\D(R)}(P^{\bullet})=\text{add}_{\D(R)}(G)=\text{add}_{\D(R)}(G^{'})$, and this would imply that $\text{add}_{R\text{-Mod}}(\Ker(f))=\text{add}_{R\text{-Mod}}(H^{-1}(P^{\bullet}))=\text{add}_{R\text{-Mod}}(H^{-1}(G))=\text{add}_{R\text{-Mod}}(H^{-1}(G^{'}))=\text{add}_{R\text{-Mod}}(S_n).$ This is absurd since $S_n$ is not isomorphic to a submodule of a projective $R$-module.
\end{proof}

\vspace{0.3 cm}

The following result tells us the conditions that an HKM complex must satisfy in order to belong to $\Ht$, where $\te$ is the torsion pair associated to such complex.

\begin{corollary}\label{cor. HKM=Prog. complex}
Let $\te=(\T,\F)$ be a torsion pair in $R$-Mod and let 
$$\xymatrix{P^{\bullet}:=\cdots \ar[r] & 0 \ar[r] & Q \ar[r]^{d} &  P \ar[r] & 0 \ar[r] & \cdots}$$ 
be a complex of finitely generated projective $R$-modules concentrated in degrees -1 and 0. The following assertions are equivalent, where $V:=H^{0}(P^{\bullet})$:
\begin{enumerate}
\item[1)] $P^{\bullet}$ is a progenerator of $\Ht$;
\item[2)] The following conditions hold:
\begin{enumerate}
\item[a)] $\T=\Pres(V)\subseteq \Ker(\Ext^{1}_{R}(V,?))$; 
\item[b)] $\Ker(d)\subseteq \Rej_{\T}(Q)$ and $\Ker(d)\in \F$;
\item[c)] There is a homomorphism $h:Q^{(I)} \flecha \frac{R}{t(R)}$, for some set $I$, such that \linebreak 
$\Coker(h_{|\Ker(d)^{(I)}})\in \overline{\Gen}(V)$.
\end{enumerate}
\end{enumerate}
In this case $P^{\bullet}$ is a classical tilting complex and an HKM complex whose associated torsion pair is $\te$.
\end{corollary}
\begin{proof}
1 $\Longleftrightarrow$ 2) follows directly from theorem \ref{teo. Ht is a module category}. On the other hand, to see that $P^{\bullet}$ is an \emph{exceptional object} of $\D(R)$, i.e., $\Hom_{\D(R)}(P^{\bullet},P^{\bullet}[k])=0$, for each integer $k\neq 0$, it is enough to show that $\Hom_{\D(R)}(P^{\bullet},P^{\bullet}[1])=0=\Hom_{\D(R)}(P^{\bullet},P^{\bullet}[-1])$, since $P^{\bullet}$ is a complex of projective $R$-modules concentrated in degrees -1 and 0. But this follows from the projective condition of $P^{\bullet}$ in $\Ht$ and the fact that $P^{\bullet}[-1]\in \mathcal{U}_{\mathbf{t}}^{\perp}$. \\

Now, let $M$ be a complex of $R$-modules such that $\Hom_{\D(R)}(P^{\bullet}[n],M)=0$, for all $n\in $ \Z. We put $X_{n}:=\tau^{>-2} \circ \tau^{\leq 0}(M[n])$, for each $n\in $ \Z, and we consider the following triangle in $\D(R)$:

$$\xymatrix{\tau^{\leq -2}(M[n])=\tau^{\leq -2}(\tau^{\leq 0}(M[n])) \ar[r] & \tau^{\leq 0}(M[n]) \ar[r] & \tau^{>-2}(\tau^{\leq 0}(M)[n])=X_n \ar[r] & \tau^{\leq -2}(M[n])[1]}$$

Thus, we clearly have isomorphisms

$$\xymatrix{\Hom_{\D(R)}(P^{\bullet},X_{n})=\Hom_{\D(R)}(P^{\bullet}, \tau^{>-2}(\tau^{\leq 0}(M[n]))) & \Hom_{\D(R)}(P^{\bullet}, \tau^{\leq 0}(M[n])) \ar[l]_{\hspace{2.65 cm}\sim} \ar[d]^(0.46){\wr}\\  & \Hom_{\D(R)}(P^{\bullet}, M[n])}$$
This show that $\Hom_{\D(R)}(P^{\bullet},X_{n})=0$, for all $n \in $ \Z. But $X_{n}$ fits in a triangle

$$\xymatrix{H^{-n-1}(M)[1] \ar[r] & X_{n} \ar[r] & H^{-n}(M)[0] \ar[r]^{\hspace{0.85 cm}+} &}$$

The fact that $\Hom_{\D(R)}(P^{\bullet}, ?[k]):R$-Mod$\flecha$ Ab is the zero functor for $k=-1,2$ implies that $\Hom_{\D(R)}(P^{\bullet},H^{-n-1}(M)[1])=0=\Hom_{\D(R)}(P^{\bullet}, H^{-n}(M)[0])$, for all $n \in $ \Z. The proof is hence reduced to cheek that if $N$ is a left $R$-module such that \linebreak $\Hom_{\D(R)}(P^{\bullet},N[k])=0$, for $k=0,1$, then $N=0$. Indeed, from lemma \ref{lem. homology 0 module} we get that $\Hom_{R}(H^{0}(P^{\bullet}),N)=0$. By 2.a), this implies that $N\in \F$. But then $N[1]$ is an object of $\Ht$, which implies that $N=0$ since $P^{\bullet}$ is a progenerator of this category. Therefore $P^{\bullet}$ is a classical tilting complex since it is an exceptional compact generator of $\D(R)$ (see \cite{Ri}). Finally, the complex $P^{\bullet}$ satisfies condition 2 of proposition \ref{pro. HKM complex}, and, hence, it is an HKM complex. 
\end{proof}

\begin{definition}\label{def. progenerator which is classical tilting complex}\rm{
We shall say that $\Ht$ \emph{has a progenerator which is a classical tilting complex}\index{complex! progenerator classical tilting} when it has a progenerator $P^{\bullet}$ as in corollary \ref{cor. HKM=Prog. complex}.}
\end{definition}

\section{The case of a hereditary torsion pair}

Suppose now that $\te=(\T,\F)$ is hereditary. We will show that the conditions making its heart a module category give more precise information than in the general case. Recall that $\te$ is called \emph{bounded}\index{torsion pair! bounded} when its associated Gabriel topology has a basis consisting of two-sided ideals (\cite[Chapter VI]{S}). Equivalently, when $R/\text{ann}_{R}(T)\in \T$, for each finitely generated $R$-module $T\in \T$.

\begin{examples}\label{exam. bounded torsion pair}\rm{

\begin{enumerate}
\item[1)] If $R$ is a commutative ring, then every hereditary torsion pair is bounded, since every ideal of $R$ is two-sided.

\item[2)] If $\te$ is the right constituent pair of a TTF triple in $R$-Mod, then $\te$ is bounded (see \cite[Proposition VI.6.12]{S}).

\end{enumerate}}
\end{examples}

\begin{theorem}\label{teo. Ht module c. with t here.}
Let $\te=(\T,\F)$ be a hereditary torsion pair in $R$-Mod and let 
$$\progen$$
be a complex of $R$-modules which is in $\Ht$, where $P$ and $Q$ are finitely generated projective. The following assertions are equivalent, where $V:=H^{0}(G)$:
\begin{enumerate}
\item[1)] $G$ is a progenerator of $\Ht$;

\item[2)] The following conditions are satisfied:
\begin{enumerate}
\item[a)] $\T=\Pres(V)\subseteq \Ker(\Ext^{1}_{R}(V,?))$;

\item[b)] $H^{-1}(G)\subseteq \Rej_{\T}(\Coker(j))$;

\item[c)] $\Ext^{1}_{R}(\Coker(j),?)$ vanishes on $\F$;

\item[d)] If $\mathbf{b}=\text{ann}_{R}(V/t(R)V)$, then $\mathbf{b}/t(R)$ is an idempotent ideal of $R/t(R)$ (which is finitely generated on the left) and $R/\mathbf{b}$ is in $\T$;

\item[e)] there is a morphism $h:\Coker(j)^{(J)} \flecha \frac{\mathbf{b}}{t(R)}$, for some set $J$, such that $h_{|H^{-1}(G)^{(J)}}:H^{-1}(G)^{(J)} \flecha \frac{\mathbf{b}}{t(R)}$ is an epimorphism.
\end{enumerate}
When $\te$ is bounded, the assertions are also equivalent to:

\item[3)] There is an idempotent ideal $\mathbf{a}$ of $R$, which is finitely generated on the left, such that:

\begin{enumerate}
\item[a)] $\add(V)=\add(R/\mathbf{a})$ and $\te$ is the right constituent torsion pair of the TTF triple defined by $\mathbf{a}$;

\item[b)] $\Ker(d) \subseteq \Imagen(j)+\mathbf{a}Q$;

\item[c)] $\Ext^{1}_{R}(\Coker(j),?)$ vanishes on $\F$:

\item[d)] There is a morphism $h:\Coker(j)^{(J)} \flecha \frac{\mathbf{a}}{t(\mathbf{a})}$, for some set $J$, such that $h_{|H^{-1}(G)^{(J)}}:H^{-1}(G)^{(J)} \flecha \frac{\mathbf{a}}{t(\mathbf{a})}$ is an epimorphism.
\end{enumerate}
\end{enumerate}
\end{theorem}
\begin{proof}
1) $\Longrightarrow$ 2) The complex $G$ satisfies all properties 1-5 of theorem \ref{teo. Ht is a module category}. In particular, properties 2.a) 2.b) and 2.c) are automatic. To check properties 2.d) and 2.e), we proceed in several steps. \\

Step 1: $(\T \cap \overline{R}\text{-Mod},\F)$ is the right constituent torsion pair of a TTF triple in $\overline{R}$-Mod, where $\overline{R}:=\frac{R}{t(R)}$. Indeed, by property 5 of theorem \ref{teo. Ht is a module category}, there is a morphism $h:(Q/X)^{(I)} \flecha \overline{R}$ such that if $h^{'}=h_{|H^{-1}(G)^{(I)}}$, then $\Coker(h^{'}) \in \overline{\Gen}(V)$. But in this case $\overline{\Gen}(V)=\Gen(V)=\T$ since $\te$ is hereditary. Thus, if we put $\overline{\mathbf{b}}=\frac{\mathbf{b}}{t(R)}=\Imagen(h^{'})$, then $R/\mathbf{b}\in \T$. \\

We claim that a $\overline{R}$-module $T$ is in $\T$ if, and only if, $\Hom_{\overline{R}}(\overline{\mathbf{b}},T)=0$. This will imply that $\mathcal{T} \cap \overline{R}$-Mod is also a torsionfree class in $\overline{R}$-Mod, since $\Ker(\Hom_{\overline{R}}(\overline{\mathbf{b}},?))$ is closed under subobjects, products and extensions. Let $f:\overline{\mathbf{b}} \flecha T$ be any morphism, where $T\in \T$. We then get a pushout commutative diagram
$$\xymatrix{& H^{-1}(G)^{(I)} \hspace{0.04 cm}\ar@{^(->}[r] \ar[d]^(0.4){h^{'}} & (\frac{Q}{X})^{(I)} \ar[d]^(0.4){h}\\0 \ar[r] & \overline{\mathbf{b}} \hspace{0.04 cm} \ar@{^(->}[r]^{j} \ar[d]^{f} & \overline{R} \ar[r] \ar[d]^(0.4){g^{'}} & R/\mathbf{b} \ar[r] \ar@{=}[d]& 0 \\ 0 \ar[r] & T \ar[r]^{\lambda} & T^{'} \ar[r] \pullbackcorner & R/\mathbf{b} \ar[r] & 0}$$
Note that $T^{'}\in \T$, and hence by property 3) of theorem \ref{teo. Ht is a module category}, we get $g^{'}\hspace{0.03 cm} \circ \hspace{0.03 cm} h_{|H^{-1}(G)^{(I)}}=0$. But $g^{'}\hspace{0.03 cm} \circ \hspace{0.03 cm} h_{|H^{-1}(G)^{(I)}}$ is equal to the composition $\xymatrix{H^{-1}(G)^{(I)} \ar[r]^{\hspace{0.75 cm}h^{'}} & \overline{\mathbf{b}} \ar[r]^{f \hspace{0.1 cm}} & T \ar[r]^{\lambda  \hspace{0.15 cm}} & T^{'}}$, which is then the zero map. This implies that $f=0$ since $h^{'}$ is an epimorphism and $\lambda$ is a monomorphism. This show the `only if' part of our claim. For the `if' part, suppose that $\Hom_{\overline{R}}(\overline{\mathbf{b}},T)=0$ and fix an epimorphism $q:\overline{R}^{(J)}\epic T$. Then $q(\overline{\mathbf{b}}^{(J)})=0$, which gives an induced epimorphism $\overline{q}:\frac{\overline{R}^{(J)}}{\overline{\mathbf{b}}^{(J)}}\cong (\frac{R}{\mathbf{b}})^{(J)} \epic T$. It follows that $T\in \T$, which settles our claim.\\

Step 2: The two-sided idempotent ideal of $\overline{R}$ which defines the TTF triple in $\overline{R}$-Mod is $\overline{\mathbf{b}^{'}}=\mathbf{b}^{'}/t(R)$, where $\mathbf{b}^{'}=\ann_{R}(V/t(R)V)$. Let $\overline{\mathbf{b}^{'}}=\mathbf{b}^{'}/t(R)$ be the two-sided idempotent ideal of $\overline{R}$ which defines the TTF triple mentioned above. We then know (see \cite[Section VI.8]{S}) that $\Gen(\overline{\mathbf{b}^{'}})=\{\overline{\mathbf{b}^{'}}C:C\in \overline{R}\text{-Mod}\}=\{C\in \overline{R}\text{-Mod}: \Hom_{R}(C,T)=0, \text{ for all }T\in \T\cap \overline{R}\text{-Mod}\}$ and $\T\cap \overline{R}\text{-Mod}=\{T\in \overline{R}\text{-Mod}:\overline{\mathbf{b}^{'}}T=0\}=\Gen(R/ \mathbf{b^{'}})$. In particular, for the ideal $\mathbf{b}$ of $R$ given in the first step, we have that $\overline{\mathbf{b}^{'}}\overline{\mathbf{b}}=\overline{\mathbf{b}}$ and $\overline{\mathbf{b}^{'}}\frac{R}{\mathbf{b}}=0$. From the first equality we get that $\overline{\mathbf{b}}\subseteq \overline{\mathbf{b}^{'}}$, since $\overline{\mathbf{b}^{'}}$ is two-sided, while from the second one we get $\mathbf{b}^{'}\subseteq \mathbf{b}$, and hence $\overline{\mathbf{b}^{'}}\subseteq \overline{\mathbf{b}}$. It follows that $\overline{\mathbf{b}^{'}}= \overline{\mathbf{b}}$. We then get that $\Gen(V/t(R)V)=\T \cap \overline{R}$-Mod$=\Gen(R/\mathbf{b})$, from which we deduce that $\mathbf{b}^{'}=\mathbf{b}=\ann_{R}(V/t(R)V)=\ann_{R}(R/\mathbf{b})$, since $\ann_{\overline{R}}(R/\mathbf{b})=\mathbf{b}/t(R).$ \\

Step 3: Verification of properties 2.d) and 2.e). Except for the finite generation of $\overline{\mathbf{b}}$, property 2.d) follows immediately from the previous steps. Note that we get a natural isomorphism $\Hom_{\overline{R}}(V/t(R)V,?) \iso \Hom_{R}(V,?)_{|\overline{R}\text{-Mod}}$, thus, $\overline{V}=V/t(R)V$ is a finitely presented $\overline{R}$-module. On the other hand, $R/\mathbf{b}$ is finitely generated and we have an epimorphism $\overline{V}^{(n)} \epic R/\mathbf{b}$. This epimorphism splits since both its domain and codomain are annihilated by $\mathbf{b}$ and $R/\mathbf{b}$ is projective in $R/\mathbf{b}$-Mod. Therefore, $R/\mathbf{b}$ is finitely presented as a left $\overline{R}$-module, which is equivalent to say that $\overline{b}$ is finitely generated as a left ideal of $\overline{R}$. Let us fix an epimorphism $\pi: \overline{R}^{(n)} \epic \overline{\mathbf{b}}$. Using the canonical map $h:(Q/X)^{(I)} \flecha \overline{R}$ (see step 1), we obtain a morphism $g:[(Q/X)^{(I)}]^{(n)} \xymatrix{\ar[r]^{h^{(n)}} &} \overline{R}^{(n)} \xymatrix{\ar@{>>}[r]^{\pi}&} \overline{\mathbf{b}}$. Put $Y=H^{-1}(G)$, we have
\begin{center}
$g([Y^{(I)}]^{(n)})=\pi(\Imagen(h^{'})^{(n)})=\pi(\overline{\mathbf{b}}^{(n)})=\pi(\overline{\mathbf{b}}\overline{R}^{(n)})=\overline{\mathbf{b}}.\overline{\mathbf{b}}=\overline{\mathbf{b}}$,
\end{center}
which proves 2.e).\\

2) $\Longrightarrow$ 1) The complex $G$ satisfies all properties 1-4 of theorem \ref{teo. Ht is a module category}. Moreover, if $h:\Coker(j)^{(J)} \flecha \overline{\mathbf{b}}$ is the homomorphism given in 2.e), then $h$ is an epimorphism and the composition $\Coker(j)^{(J)} \flecha \overline{\mathbf{b}} \monic \overline{R}$ has $R/\mathbf{b}$ as its cokernel. By property 2.d), this cokernel is in $\T=\Gen(V)$.\\

We assume in the rest of the proof that $\te$ is bounded.\\

1), 2) $\Longrightarrow$ 3) We know that $\T=\Gen(V)\subseteq \Ker(\Ext^{1}_{R}(V,?))$ and that $_{R}V$ is finitely presented. Note that $\mathbf{a}:=\ann_{R}(V)$ annihilates all modules in $\T$, hence, if $\mathbf{b}$ is a ideal on the left of $R$ such that $R/\mathbf{b}\in \T$, then $\mathbf{a} \subseteq \ann_{R}(R/\mathbf{b}) \subseteq \mathbf{b}$. Moreover, the bounded condition of $\te$ implies that $R/\mathbf{a}\in \T$. By \cite[Proposition VI.6.12]{S}, we know that $\mathbf{a}$ is idempotent, so that $\te$ is the right constituent torsion pair of the TTF triple defined by $\mathbf{a}$. This allows to identify $\T$ with $R/\mathbf{a}$-Mod and, using that also $\T=\Gen(V)\subseteq \Ker(\Ext^{1}_{R}(V,?)),$ we conclude that $\add(V)=\add(R/\mathbf{a}).$ We then get condition 3.a). We also get that $R/\mathbf{a}$ is a finitely presented $R$-module, so that $\mathbf{a}$ is finitely generated on the left. \\

Note that condition 3.c) holds and $\Rej_{\T}(M)=\mathbf{a}M$, for each $R$-module $M.$ Then condition 3.b) is equivalent to saying that $H^{-1}(G)\subseteq \Rej_{\T}(\Coker(j))$. Finally, following the proof of the implication 1) $\Longrightarrow$ 2), we see that the ideal $\mathbf{b}$ obtained in assertion 2 is identified by the properties that $\overline{\mathbf{b}}=\mathbf{b}/t(R)$ is idempotent and a $\overline{R}$-module $T$ is in $\T$ if, and only if, $\mathbf{b}T=0$. Then we have $\mathbf{b}=\mathbf{a}+t(R)$ and so condition 3.d) follows by using the isomorphism $\frac{\mathbf{b}}{t(R)}=\frac{\mathbf{a}+t(R)}{t(R)}\cong \frac{\mathbf{a}}{\mathbf{a} \cap t(R)}=\frac{\mathbf{a}}{t(\mathbf{a})}$. \\

3) $\Longrightarrow$ 1) Due to the fact that $\Rej_{\T}(M)=\mathbf{a}M$, for each $R$-module $M$, it is easily verified that $G$ satisfies all conditions 1-5 of theorem \ref{teo. Ht is a module category}.
\end{proof}

\begin{corollary}\label{cor. hereditary + fateful ---> right constituent}
If $\te=(\T,\F)$ is a faithful hereditary torsion pair such that its heart $\Ht$ is a module category, then $\te$ is the right constituent pair of a TTF triple in $R$-Mod defined by an idempotent ideal $\mathbf{a}$ which is finitely generated on the left.
\end{corollary}
\begin{proof}
By hypothesis $t(R)=0$. From theorem \ref{teo. Ht module c. with t here.}, we have that, if $G$ is a progenerator of $G$, then $R/\mathbf{b}\in \T$, where $\mathbf{b}=\ann_{R}(V/t(R)V)=\ann_{R}(V)$ is an idempotent ideal of $R$ which is finitely generated on the left, and $V:=H^{0}(G)$. The result follows using the fact that $\T=\Gen(V)=\Pres(V)$ and following the proof 1), 2) $\Longrightarrow$ 3) of theorem \ref{teo. Ht module c. with t here.}. 
\end{proof}

\begin{corollary}\label{cor. here commutative ring}
Let $R$ be a commutative ring and let $\te=(\T,\F)$ be a hereditary torsion pair in $R$-Mod. The heart $\Ht$ is a module category if, and only if, $\te$ is (left or right) constituent pair of a centrally split TTF triple in $R$-Mod. In that case $\Ht$ is equivalent to $R$-Mod.
\end{corollary}
\begin{proof}
Since $\te$ is bounded (see examples \ref{exam. bounded torsion pair}), last theorem says that $\te$ is the right constituent torsion pair of a TTF triple in $R$-Mod defined by an idempotent ideal $\mathbf{a}$ which is finitely generated. But each finitely generated idempotent ideal of a commutative ring is generated by an idempotent element (see the proof of \cite[Lemma VI.8.6]{S}). Then the TTF triple is centrally split. Moreover, by Corollary \ref{cor. left split} below, we have that $\Ht$ is equivalent to $R$-Mod.
\end{proof}

\subsection{The right constituent of a TTF triple}

Among the hereditary torsion pairs we have as a particular case the right constituent torsion pair of a TTF triple. For this reason it is natural to ask if in this case the conditions in theorem \ref{teo. Ht module c. with t here.} are simplified. Furthermore, as shown in theorem \ref{teo. Ht module c. with t here.} and corollary \ref{cor. hereditary + fateful ---> right constituent}, hereditary torsion pairs which are the right constituent of a TTF triple appear quite naturally when studying the modular condition of the heart. In this subsection we fix an idempotent ideal $\mathbf{a}$ of $R$ and its associated TTF triple $(\mathcal{C,T,F})$ and want to study when the pair $\te=(\T,\F)$ has the property that its heart $\Ht$ is a module category. When this is the case, by theorem \ref{teo. Ht module c. with t here.}, we know that $\mathbf{a}$ is finitely generated on the left. We start with the HKM condition.

\begin{corollary}
Let $\mathbf{a}$ be an idempotent ideal of $R$ which is finitely generated on the left, let $\te=(\T,\F)$ be the right constituent torsion pair of the associated TTF triple and let 
$\xymatrix{P^{\bullet}:=\cdots \ar[r] & 0 \ar[r] & Q \ar[r]^{d} & P \ar[r] & 0 \ar[r] & \cdots}$ be a complex of finitely generated projective modules concentrated in degrees -1 and 0. The following assertions are equivalent:
\begin{enumerate}
\item[1)] $P^{\bullet}$ is an HKM complex such that $\te$ is its associated HKM torsion pair;
\item[2)] The following conditions hold:
\begin{enumerate}
\item[a)] $\mathcal{X}(P^{\bullet}) \subseteq \T$;

\item[b)] $H^{0}(P^{\bullet})$ is a progenerator of $R/\mathbf{a}$-Mod;

\item[c)] $\Ker(d)\subseteq \mathbf{a}Q$;

\item[d)] there is a homomorphism $p:Q^{(I)} \epic \mathbf{a}/t(\mathbf{a})$, for some set $I$, such that its restriction to $\Ker(d)^{(I)}$ is an epimorphism.
\end{enumerate}
\end{enumerate}
\end{corollary}
\begin{proof}
By hypothesis, we have that $\Rej_{\T}(Q)=\mathbf{a}Q$, and hence condition 2.c) coincides with condition 3.b) of proposition \ref{pro. HKM complex}.\\

2) $\Longrightarrow$ 1) Let us consider the epimorphism $p:Q^{(I)} \epic \mathbf{a}/t(\mathbf{a})$ given by our property 2.d). If $\overline{j}:\mathbf{a}/(t(\mathbf{a})) \flecha R/t(R)$ is the homomorphism induced by the inclusion \linebreak $j:\mathbf{a} \monic R$, then the composition 
$$\xymatrix{h: Q^{(I)} \ar@{>>}[r]^{p} & \mathbf{a}/t(\mathbf{a}) \ar[r]^{\overline{j}} & R/t(R)}$$
satisfies the condition 3.c) of the proposition \ref{pro. HKM complex}. Indeed we have
$\Imagen(h_{|\Ker(d)^{(I)}})=\Imagen(h)=\Imagen(\overline{j})=\frac{\mathbf{a}+t(R)}{t(R)}$, so that the cokernel of $h_{|\Ker(d)^{(I)}}$ is isomorphic to $\frac{R}{\mathbf{a}+t(R)},$ which is a module in $\T=R/\mathbf{a}-$Mod. The implication follows from proposition \ref{pro. HKM complex} since $\Ext^{1}_{R}(H^{0}(P),T)\cong \Ext^{1}_{R/\mathbf{a}}(H^{0}(P),T)=0$, for each $T\in \T=R/\mathbf{a}-$Mod. \\

1) $\Longrightarrow$ 2) From proposition \ref{pro. HKM complex}, we have that the complex  

$$\xymatrix{G:=\cdots \ar[r] & 0 \ar[r] & t(\Ker(d)) \hspace{0.03cm} \ar@{^(->}[r]^{\hspace{0.65 cm}j} & Q \ar[r]^{d} & P \ar[r] & 0 \ar[r] & \cdots }$$ 
concentrated in degrees -2,1,0, is a progenerator of $\Ht$. On the other hand, since $\T$ is a TTF class, we know that $\te$ is bounded (see examples \ref{exam. bounded torsion pair}). By theorem \ref{teo. Ht module c. with t here.}, we get that $H^{0}(G)=H^{0}(P^{\bullet})$ is a progenerator of $R/\mathbf{a}$-Mod, and there is a morphism $h:(\frac{Q}{t(\Ker(d))})^{(I)} \flecha \mathbf{a}/t(\mathbf{a})$, for some set $I$, such that $h_{|H^{-1}(G)^{(I)}}:H^{-1}(G)^{(I)} \flecha \mathbf{a}/t(\mathbf{a})$ is an epimorphism. It is easy to see that the composition 
$$\xymatrix{p:Q^{(I)} \ar@{>>}[r]^{q \hspace{0.5 cm}} & (\frac{Q}{t(\Ker(d))})^{(I)} \ar[r]^{\hspace{0.3 cm}h} & \mathbf{a}/(t(\mathbf{a}))}$$
is a homomorphism that satisfies the condition 2.d), where $q:Q^{(I)} \epic (\frac{Q}{t(\Ker(d))})^{(I)}$ is the canonical map. Condition 2.a) follows the definition of HKM complex.
\end{proof}

\begin{definition}\rm{
Let $R$ be a ring. We shall denoted by $V(R)$ the \emph{additive monoid}\index{additive monoid} whose elements are the isoclasses of finitely generated projective $R$-modules, where $|P|+ |Q|=|P \oplus Q|$. }
\end{definition}

\begin{remark}\rm{
For each two-sided ideal $\mathbf{a}$ of the ring $R$, we have an obvious morphism of monoids $V(R) \flecha V(R/\mathbf{a})$ taking $|P|  \rightsquigarrow |P/\mathbf{a}P|$. This morphism need not be surjective. However, the class of rings $R$ for which it is surjective, independently of $\mathbf{a}$, is very large and includes the so-called exchange rings (see \cite{A}). In particular, it includes all semiperfect and, more generally, semiregular rings (i.e. those rings $R$ such that $R/J(R)$ is Von Neumann regular and idempotents lift module $J(R)$, where $J(R)$ denotes the Jacobson radical). }
\end{remark}

\vspace{0.3 cm}

Our next result, very important in the sequel, shows that the conditions for $\Ht$ to be a module category get rather simplified if we assume that the monoid morphism $V(R) \flecha V(R/\mathbf{a})$ is surjective.

\begin{theorem}\label{teo. progenerator Imd=aP}
Let $\mathbf{a}$ be an idempotent ideal of the ring $R$ which is finitely generated on the left, and let $(\mathcal{C,T,F})$ be the associated TTF triple. Consider the following assertions for $\te=(\T,\F):$
\begin{enumerate}
\item[1)] There is a finitely generated projective $R$-module $P$ satisfying the following conditions:
\begin{enumerate}
\item[a)] $P/\mathbf{a}P$ is a (pro)generator of $R/\mathbf{a}-Mod$;
\item[b)] There is an exact sequence $\xymatrix{0 \ar[r] & F \ar[r] & C \ar[r]^{q \hspace{0.1 cm}} & \mathbf{a}P \ar[r] & 0}$ in $R$-Mod, where $F\in \F$ and $C$ is a finitely generated module which is in $\C \cap \Ker(\Ext^{1}_{R}(?,\F))$ and generates $\C \cap \F$.
\end{enumerate}
\item[2)] The heart $\Ht$ is a module category. 

\end{enumerate}

Then 1) implies 2) and, in such a case, if $j:\mathbf{a}P \monic P$ is the inclusion, then the complex concentrated in degrees -1,0

$$\xymatrix{G:=\cdots \ar[r] & 0 \ar[r] & C \oplus \frac{C}{t(C)} \ar[rr]^{ \hspace{0.5 cm}(j \hspace{0.01 cm}\circ \hspace{0.01cm} q \ \ 0)}& & P \ar[r] & 0 \ar[r] & \cdots}$$
is a progenerator of $\Ht$. \\

When the monoid morphism $V(R) \flecha V(R/\mathbf{a})$ is surjective, the implication \linebreak 2) $\Longrightarrow$ 1) is also true.
\end{theorem}
\begin{proof}
1) $\Longrightarrow$ 2) Fix an exact sequence $\xymatrix{0 \ar[r] & F \ar[r] & C \ar[r]^{q} & \mathbf{a}P \ar[r] & 0}$ as indicated in condition 1.b). Taking $F^{'}\in \F$ arbitrary, applying the functor $\Hom(?,F^{'})$ to the sequence $\xymatrix{0 \ar[r] & t(C) \hspace{0.02 cm} \ar@{^(->}[r] & C \ar@{>>}[r]^{\text{pr.} \hspace{0.5 cm}} & C/t(C) \ar[r] & 0}$ and using condition 1.b), we get the following exact sequence:
$$\xymatrix{\cdots \ar[r] & \Hom_{R}(t(C),F^{'})=0 \ar[r] & \Ext^{1}_{R}(C/t(C),F^{'}) \ar[r] & \Ext^{1}_{R}(C,F)=0 \ar[r] & \cdots}$$
This show that $\Ext^{1}_{R}(C/t(C),?)_{|\F}=0$. But then any epimorphism $(\frac{R}{t(R)})^{n} \epic \frac{C}{t(C)}$ splits, which implies that $U:=C/t(C)$ is a finitely generated projective $R/t(R)$-module which is in $\C$. Moreover it generates $\C\cap \F$ since so does $C$. \\

Let $\pi_{U}:Q_{U} \epic U$ and $\pi_{C}:Q_{C} \epic C$ be two epimorphisms from finitely generated projective modules, whose respective kernels are denoted by $K_{U}$ and $K_{C}$. We will prove that the following complex $G$, concentrated in degrees -2,-1,0 and which is clearly quasi-isomorphic to the one in the statement, is a progenerator of $\Ht$:

$$\xymatrix{\cdots \ar[r] & 0 \ar[r] & K_U \oplus K_C \hspace{0.03 cm} \ar@{^(->}[r]^{\hspace{0.05 cm}j} & Q_U \oplus Q_C \ar[rr]^{\hspace{0.6 cm}(0 \ \ j \hspace{0.01 cm} \circ \hspace{0.01 cm}q \hspace{0.01 cm} \circ \pi_{C})} & & P \ar[r] & 0 \ar[r] & \cdots}$$

Note that $\Imagen((0 \ \ j \hspace{0.01 cm} \circ \hspace{0.01 cm}q \hspace{0.01 cm} \circ \pi_{C}))\cong \Imagen(j \hspace{0.01 cm} \circ \hspace{0.01 cm}q \hspace{0.01 cm} \circ \pi_{C})=\mathbf{a}P$, and hence $H^{0}(G)=P/\mathbf{a}P\in \T$. On the other hand, we consider the following commutative diagram with exact rows:
$$\xymatrix{0 \ar[r] & \Ker(q \hspace{0.01 cm} \circ \pi_{C} ) \ar[r] \ar@{>>}[d] & Q \ar[r]^{q \hspace{0.01 cm} \circ \hspace{0.01 cm}\pi_{C}} \ar@{>>}[d]^{\pi_{C}} & \mathbf{a}P \ar@{=}[d] \ar[r] & 0 \\ 0 \ar[r] & F \ar[r] & C \ar[r]^{q \hspace{0.15 cm}} & \mathbf{a}P \ar[r] & 0}$$
Using snake lemma, we obtain the exact sequence $\xymatrix{0 \ar[r] & K_C \ar[r] & \Ker(q \hspace{0.01cm} \circ \hspace{0.01 cm} \pi_{C}) \ar[r] & F \ar[r] & 0}$. It then follows that 

$$H^{-1}(G)=\frac{\Ker((0 \ \ j \hspace{0.01 cm} \circ \hspace{0.01 cm}q \hspace{0.01 cm} \circ \pi_{C}))}{K_{U}\oplus K_C}=\frac{Q_U\oplus \Ker(q \hspace{0.01cm} \circ \hspace{0.01 cm} \pi_{C})}{K_U \oplus K_C} \cong U \oplus F$$
which is in $\F$. This proves that $G$ is an object of $\Ht$. We next check all conditions 3.a)-3.d) of theorem \ref{teo. Ht module c. with t here.}. Clearly, condition 3.a) in that theorem is a consequence of our condition 1.a). Note next that $(K_U \oplus K_C)+ \mathbf{a}(Q_{U}\oplus Q_C)=Q_U \oplus Q_C$ because $U\cong \frac{Q_U}{K_U}$ and $C\cong \frac{Q_C}{K_C}$ are both in $\C=\{X \in R\text{-Mod}: \mathbf{a}X=X\}.$ In particular, condition 3.b) of the mentioned theorem is automatic, as so is condition 3.c) since $\Ext^{1}_{R}(U,?)$ and $\Ext^{1}_{R}(C,?)$ vanish on $\F$, and $\Coker(j) = \frac{Q_U\oplus Q_C}{K_U \oplus K_C}\cong U \oplus C$. \\

Using the fact that $U$ generates $\C \cap \F$, fix an epimorphism $p:U^{(J)} \epic \mathbf{a}/t(\mathbf{a})$. Identifying $\Coker(j)=U \oplus C$, we clearly have that $(p \ \ 0):\Coker(j)^{(J)}=U^{(J)} \oplus C^{(J)} \flecha \mathbf{a}/t(\mathbf{a})$ is a homomorphism whose restriction to $H^{-1}(G)^{(J)}=U^{(J)} \oplus F^{(J)}$ is an epimorphism. Then also condition 3.d) of theorem \ref{teo. Ht module c. with t here.} holds. \\

2) $\Longrightarrow$ 1) (assuming that the monoid morphism $V(R) \flecha V(R/\mathbf{a})$ is surjective). Let $\xymatrix{G:=\cdots \ar[r] & 0 \ar[r] & X^{'} \ar[r]^{j^{'}} & Q^{'} \ar[r]^{d^{'}} & P^{'} \ar[r] & 0 \ar[r] & \cdots}$ be a progenerator of $\Ht$. By theorem \ref{teo. Ht module c. with t here.}, we know that $\add(H^{0}(G)) =\add(R/\mathbf{a})$ and our extra hypothesis gives a finitely generated projective $R$-module $P$ such that $P/\mathbf{a}P\cong H^{0}(G)$, so that condition 1.a) holds. Now, due to the projective condition of $P$, we get the following commutative diagram with exact rows:

$$\xymatrix{0 \ar[r] & \mathbf{a}P \ar[r] \ar@{-->}[d] & P\ar[r] \ar@{-->}[d] & P/\mathbf{a}P \ar[r] \ar[d]^{\wr} & 0 \\ 0 \ar[r] & \Imagen(d^{'}) \ar[r] & P^{'} \ar[r] & H^{0}(G) \ar[r] & 0 }$$

After suitable taking pullbacks, we see that we can represent $G$ by a chain complex 

$$\xymatrix{G:=\cdots \ar[r] & 0 \ar[r] & X \ar[r]^{j} & Q \ar[r]^{d} & P \ar[r] & 0 \ar[r] & \cdots}$$

where $Q$ is finitely generated projective, $j$ is a monomorphism, and $\Imagen(d)=\mathbf{a}P$. Then $G$ satisfies all conditions 3.a)-3.d) of theorem \ref{teo. Ht module c. with t here.}. Note that $\mathbf{a}P=\mathbf{a}^{2}P=\mathbf{a}\Imagen(d)=d(\mathbf{a}Q),$ which implies that $Q=\Ker(d)+ \mathbf{a}Q$ and, by condition 3.b) of the mentioned theorem, that $Q=X+\mathbf{a}Q$. That is, the module $Q/X$ is in $\mathcal{C}$ and, by condition 3.c) of that theorem, we also have that $Q/X\in \Ker(\Ext^{1}_{R}(?,\F))$. The exact sequence needed for our condition 1.b) is then $\xymatrix{0 \ar[r] & H^{-1}(G) \ar@{^(->}[r] & Q/X \ar[r]^{\overline{d}} & \mathbf{a}P \ar[r] & 0}$ since $\mathcal{C}\cap \F$ is generated by $\mathbf{a}/t(\mathbf{a})$  and $\mathbf{a}/t(\mathbf{a})\in \Gen(Q/X)$ (see condition 3.d in theorem \ref{teo. Ht module c. with t here.}). 
\end{proof}

\vspace{0.3 cm}

We now give some applications of last theorem.

\begin{corollary}\label{cor. Trace of projective HKM torsion pair}
Let $Q$ be a finitely generated projective $R$-module and let us consider the hereditary torsion pair $\te=(\T,\F)$, where $\T=\Ker(\Hom_{R}(Q,?))$. If the trace of $Q$ in $R$ is finitely generated on the left, then $\te$ is an HKM torsion pair and $\Ht$ is a module category.
\end{corollary}
\begin{proof}
We will check assertion 1 of theorem \ref{teo. progenerator Imd=aP} for the suitable choices. We have that $\T$ fits into a TTF triple (see example \ref{exam. TTF projective module}) $(\mathcal{C},\T,\F)$, where $\mathcal{C}=\Gen(Q)=\{T\in R\text{-Mod}: \mathbf{a}T=T\}$, where $\mathbf{a}=tr_{Q}(R)$. In particular, we have an obvious epimorphism $p:Q^{(n)}\epic \mathbf{a}$ for some integer $n>0$. Now, we consider the following short exact sequence

$$\xymatrix{0 \ar[r] & t(\Ker(p)) \hspace{0.02 cm} \ar@{^(->}[r] & Q^{(n)} \ar@{>>}[r]^{\pi_{C}} & C \ar[r] & 0}$$
where $C:=\frac{Q^{n}}{t(\Ker(p))}$. For each $F\in\F$, if applying the functor $\Hom_{R}(?,F)$ to the previous sequence, we obtain the following exact sequence
$$\xymatrix{\cdots \ar[r] & \Hom_{R}(t(\Ker(p)),F)=0 \ar[r] & \Ext^{1}_{R}(C,F) \ar[r] & \Ext^{1}_{R}(Q,F)=0 \ar[r] & \cdots}$$

Hence, $\Ext^{1}_{R}(C,?)$ vanishes on $\F$. By taking $P=R$, the short exact sequence given by the property 1.b) in theorem \ref{teo. progenerator Imd=aP} is the lowest row of the following commutative diagram, since $\Ker(q)\cong \frac{\Ker(p)}{t(\Ker(p))}\in\F$:

$$\xymatrix{& t(\Ker(p)) \ar@{=}[r] \ar@{^(->}[d]& t(\Ker(p)) \ar@{^(->}[d] \\ 
0 \ar[r] & \Ker(p) \ar@{>>}[d] \ar[r] & Q^{n} \ar@{>>}[r]^{p}  \ar@{>>}[d]^{\pi_{C}}& \mathbf{a} \ar@{=}[d]\ar[r] &0 \\ 0 \ar[r] & \Ker(q)  \ar[r] & C \ar[r]^{q} \ar@{-->>}[ru]^{q} & \mathbf{a} \ar[r] &  0}$$





By theorem \ref{teo. progenerator Imd=aP} the complex
$$\xymatrix{\cdots \ar[r] & 0 \ar[r] & \frac{Q^{(n)}}{t(\Ker(p))}\oplus \frac{Q^{(n)}}{t(Q^{(n)})} \ar[rr]^{\hspace{1.1 cm}(j\circ q \hspace{0.1 cm} 0)} & & R \ar[r] & 0 \ar[r] & \cdots}$$
is a progenerator of $\Ht$ since $\frac{C}{t(C)}\cong \frac{Q^{(n)}}{t(Q^{(n)})}$. This complex is clearly quasi-isomorphic to 
$$\xymatrix{\cdots \ar[r] & 0 \ar[r] & t(\Ker(p)) \oplus t(Q)^{(n)} \hspace{0.02 cm} \ar@{^(->}[r] & Q^{(n)} \oplus Q^{(n)}\ar[rr]^{\hspace{0.95 cm}( j \hspace{0.01 cm} \circ \hspace{0.01 cm} p\ \ 0 )} & & R \ar[r] & 0 \ar[r] & \cdots}$$

Since that $R/\mathbf{a} \in \T=\Ker(\Hom_{R}(Q,?))$, we have that the complex 
$$\xymatrix{P^{\bullet}:=\cdots \ar[r] & 0 \ar[r] & Q^{(n)} \oplus Q^{(n)} \ar[rr]^{\hspace{0.9 cm}(j \circ p \ \ 0)} && R \ar[r] & 0 \ar[r] & \cdots}$$
satisfies the condition 2 of the proposition \ref{pro. HKM complex}.      
\end{proof}

\vspace{0.3 cm}

The following result is folklore. We include a short proof for the convenience of the reader. We refer the reader to \cite[p. 124 and Chapter 11]{Ka} for the definition of projective cover and semiperfect rings. 

\begin{lemma}\label{lem. clave semiperfect rings}
Let $R$ be a ring and let $\mathbf{a}$ be an idempotent ideal. The following assertions hold:
\begin{enumerate}
\item[1)] If $p:P \epic M$ is a projective cover and $\mathbf{a}M=M$, then $\mathbf{a}P=P$;
\item[2)] Suppose that $R$ is semiperfect and let $\{e_1, \dots ,e_{n}\}$ be a family of primitive orthogonal idempotents such that $\underset{1 \leq i \leq n}{\sum} e_{i} =1$. If $\mathbf{a}$ is finitely generated on the left, then there is an idempotent element $e\in R$ (which is a sum of $e_{i}$'s) such that $\mathbf{a}=ReR$. 
\end{enumerate}
\end{lemma}
\begin{proof}
1) We have an equality $\Ker(p)+\mathbf{a}P=P$ since $p(\mathbf{a}P)=\mathbf{a}M=M$. We then get $P=\mathbf{a}P$ because $\Ker(p)$ is a superfluous submodule of $P$. \\

2) After reordering the primitive idempotents, we can fix $1\leq m \leq n$ such that $e_{i}\notin \mathbf{a}$, for $i\leq m$, while $e_{i}\in \mathbf{a}$, for $i>m$. Then we have $\mathbf{a}=(\underset{1 \leq i \leq m}{\oplus} \mathbf{a}e_i) \oplus (\underset{m < j \leq n}{\oplus}R{e_{j}})$, where the inclusion $\mathbf{a}e_{i} \subset Re_{i}$ is strict, for each $i=1,\dots,m$.\\

Let now $p:P \epic \mathbf{a}$ be the projective cover. By part 1), we know that $\mathbf{a}P=P$. But we know that $P$ is finite direct sum of copies of the $Re_i$. It follows that $P=\underset{m < j \leq n}{\oplus}Re_j^{(m_j)}$, for some nonnegative integers $m_j$, because if there were an index $1 \leq i \leq m$ such that $Re_i$ is a summand of $P$, we would have $Re_i=\mathbf{a}e_i$ and would get a contradiction. But any morphism $Re_j \flecha R$ has image contained in $Re_{j}R$. We then get that $\mathbf{a}\subseteq \underset{m < j \leq n}{\sum}Re_jR$. The converse inclusion is obvious by the way we chose the $e_i$'s. Putting $e=\underset{m < j \leq n}{\sum}e_j$, we have $ReR=\underset{m<j\leq n}{\sum}Re_jR=\mathbf{a}$.
\end{proof}

\vspace{0.3 cm}

For semiperfect rings, we have the following result.

\begin{corollary}\label{cor. caract. rings semiperfect}
Let $R$ be a semiperfect ring, let $\{e_1,\dots,e_n\}$ be a complete family of primitive orthogonal idempotents, and let $\te=(\T,\F)$ be the right constituent torsion pair of a TTF triple in $R$-Mod. The following assertions are equivalent:
\begin{enumerate}
\item[1)] The heart $\Ht$ is a module category;
\item[2)] $\te$ is an HKM torsion pair;
\item[3)] There is an idempotent element $e\in R$ (which is a sum of $e_i$'s) such that $ReR$ is finitely generated on the left and $ReR$ is the idempotent ideal which defines the TTF triple.
\end{enumerate}
\end{corollary}
 \begin{proof}
 1) $\Longrightarrow$ 3) By theorem \ref{teo. Ht module c. with t here.}, we know that $\mathbf{a}$ is finitely generated on the left. Then assertion 3) follows from lemma \ref{lem. clave semiperfect rings}. \\
 2) $\Longrightarrow$ 1) It follows from \cite[Theorem 3.8]{HKM}. \\
 3) $\Longrightarrow$ 2) is a direct consequence of corollary \ref{cor. Trace of projective HKM torsion pair} since $ReR$ is the trace of $Re$ in $R$.
 \end{proof}

\begin{lemma}\label{lem. ImB<t(R)V}
Let $V$ be a classical quasi-tilting $R$-module such that $\Gen(V)$ is closed under taking submodules and let $t(R)$ be the trace of $V$ in $R$. An endomorphism $\beta$ of $V$ satisfies that $\Imagen(\beta)\subseteq t(R)V$, if and only if, it factors through a (finitely generated) projective $R$-module.
\end{lemma}
\begin{proof}
We only need to prove the ``only if'' part of the statement, because the ``if'' part is clear. Now, from lemma \ref{lem. quasi-tilting}, we know that $\Gen(V)$ is a torsion class. Furthermore, by hypothesis $\te=(\Gen(V), \Ker(\Hom_{R}(V,?)))$ is a hereditary torsion pair. Let $q:V^{(\Hom_{R}(V,R))} \epic t(R)=tr_{V}(R)$, $i:t(R) \monic R$, $\pi:R^{(V)}\epic V$ and $j:t(R)V \monic V$ be the canonical morphisms and let $\pi^{'}:t(R)^{(V)} \epic t(R)V$ be the epimorphism given by the restriction of $\pi$ to $t(R)^{(V)}$. We have a commutative diagram 

$$\xymatrix{V^{(\Hom_{R}(V,R) \times V)} \ar[rr]^{\hspace{0.9 cm}q^{(V)}} \ar@{>>}[rrd]^(0.55){\rho}&& t(R)^{(V)} \ar[r]^{\hspace{0.3 cm}i^{(V)}} \ar@{>>}[d]^(0.4){\pi^{'}} & R^{(V)} \ar@{>>}[d]^(0.41){\pi}\\ & & t(R)V \hspace{0.03cm} \ar@{^(->}[r]^{\hspace{0.2 cm}j} & V}$$

where $\rho:=\pi^{'} \circ q^{(V)}$. Note that, due to the hereditary condition of $\te$, we know that $\Ker(\rho)\in \T \subseteq \Ker(\Ext^{1}_{R}(V,?)).$ \\

Suppose that $\Imagen({\beta})\subseteq t(R)V$. In such case, we have a factorization $j \circ \overline{\beta}=\beta$, for some $\overline{\beta}\in \Hom_{R}(V,t(R)V)$. Taking now the pullback of $\rho$ along $\overline{\beta}$, we have the following commutative diagram

$$\xymatrix{0 \ar[r] & \Ker(\rho) \ar[r] \ar@{=}[d] & Z \ar[d] \ar[r] \pushoutcorner & V \ar[r] \ar[d]^{\overline{\beta}} & 0 \\ 0 \ar[r] & \Ker(\rho) \ar[r] & V^{(\Hom_{R}(V,R)\times V)} \ar[r]^{\hspace{0.7 cm}\rho} & t(R)V \ar[r] & 0}$$
Due to the fact that $\Ker(\rho)\in \Ker(\Ext^{1}_{R}(V,?))$, we have that $\overline{\beta}$ factors throughout $\rho$. Fix a morphism $\gamma:V \flecha V^{(\Hom_{R}(V,R)\times V)}$ such that $\overline{\beta}=\rho \circ \gamma$. Then we have:
$$\beta=j \circ \overline{\beta}=j \circ \rho \circ \gamma = \pi \circ i^{(V)} \circ q^{(V)} \circ \gamma$$
so that $\beta$ factors through $R^{(V)}$.
\end{proof}

\vspace{0.3 cm}

As a consequence of theorem \ref{teo. Ht module c. with t here.} and corollary \ref{cor. caract. rings semiperfect}, we get some significative classes of rings for which we can identify all the hereditary torsion pairs whose heart is a module category.

\begin{corollary}
Let $\te=(\T,\F)$ be a hereditary torsion pair in $R$ and let $\Ht$ be its heart. The following assertions hold:
\begin{enumerate}
\item[1)] If $R$ is a local ring and $\Ht$ is a module category, then $\te$ is either $(R\text{-Mod},0)$ or $(0,R\text{-Mod})$;

\item[2)] When $R$ is right perfect, $\Ht$ is a module category if, and only if, there is an idempotent element $e\in R$ such that $\T=\{T\in \text{R}-Mod: eT=0\}$ and $ReR$ is finitely generated on the left;

\item[3)] If $R$ is left Artinian (e.g. an Artin algebra), then $\Ht$ is always a module category.
\end{enumerate}
\end{corollary}
\begin{proof}
1) By theorem \ref{teo. Ht is a module category}, we have a finitely presented $R$-module $V$ such that \linebreak $\T=\Gen(V)\subseteq \Ker(\Ext^{1}_{R}(V,?))$. Using theorem \ref{teo. Ht module c. with t here.} and its proof, we get that \linebreak $\te^{'}=(\T \cap \overline{R}\text{-Mod},\F)$ is the right constituent of a TTF triple in $\overline{R}$-Mod defined by an idempotent ideal $\overline{\mathbf{b}}=\mathbf{b}/t(R)$ of $\overline{R}:=R/t(R)$ which is finitely generated on the left, where $\mathbf{b}=\ann_{R}(V/t(R)V)$. Since $\overline{R}$ is also a local ring, and hence semiperfect, lemma \ref{lem. clave semiperfect rings} says that $\overline{\mathbf{b}}=\overline{R}\overline{e}\overline{R}$, for some idempotent element $\overline{e}\in \overline{R}$, which is necessarily equal to $\overline{1}$ or 0. The fact that $\overline{R}\in \F$ implies that $\overline{e}=1$, so that $\overline{\mathbf{b}}=\overline{R}$ and $\mathbf{b}=R=\ann_{R}(V/t(R)V)$. It follows that $V=t(R)V$ and, by lemma \ref{lem. ImB<t(R)V}, we have that the identity morphism $1_V:V \flecha V$ factors through a projective $R$-module. Hence $V$ is projective. But all finitely generated projective modules over a local ring are free. Then we have $V=0$ or $V=R^{(\alpha)}$, so that either $\te=(0,R\text{-Mod})$ or $\te=(R\text{-Mod},0)$. \\

2) Assume now that $R$ is right perfect and $\Ht$ is a module category. By \cite[Corollary VIII.6.3]{S}, we know that $\te$ is the right constituent torsion pair of a TTF triple and, in particular $\te$ is bounded. By theorem \ref{teo. Ht module c. with t here.}, the associated idempotent ideal is finitely generated on the left and, by \cite[Corollary VIII.6.4]{S}, we know that it is of the form $ReR$. Conversely, if $e\in R$ is idempotent and $ReR$ is finitely generated on the left and $\T=\{T \in R\text{-Mod}:eT=0\}$, then corollary \ref{cor. caract. rings semiperfect} says that $\Ht$ is a module category.\\

3) From \cite[Examples VI.8.2]{S} we know that $\te$ is the right constituent torsion pair of a TTF triple and, since all left ideals are finitely generated, the result follows from lemma \ref{lem. clave semiperfect rings} and corollary \ref{cor. caract. rings semiperfect}.
\end{proof}
 
\vspace{0.3 cm}

Another consequence of theorem \ref{teo. progenerator Imd=aP} is the following.

\begin{corollary}\label{cor. a in F}
Let $\mathbf{a}$ be an idempotent ideal of $R$, which is finitely generated on the left and let $\te=(\T,\F)$ be the right constituent torsion pair of the associated TTF triple in $R$-Mod. Suppose that $\mathbf{a}\in \F$ and that the monoid morphism $V(R) \flecha V(R/\mathbf{a})$ is surjective. Consider the following assertions:
\begin{enumerate}
\item[1)] $\Ht$ is a module category;

\item[2)] There is an epimorphism $M \epic \mathbf{a}$, where $M$ is a finitely generated projective $R/t(R)$-module which is in $\C$.

\item[3)] $\mathbf{a}$ is the trace of some finitely generated projective left $R$-module;

\item[4)] $\te$ is an HKM torsion pair;

\item[5)] $\Ht$ has a progenerator which is a classical tilting complex.
\end{enumerate} 
Then the implications 5) $\Longrightarrow$ 4) and 3) $\Longrightarrow$ 4) $\Longrightarrow$ 1) $\Longleftrightarrow$ 2) hold true. When the monoid morphism $V(R) \flecha V(R/t(R))$ is also surjective, all assertions are equivalent.
\end{corollary}
\begin{proof}
The implications 3) $\Longrightarrow$ 4) $\Longrightarrow$ 1) and 5) $\Longrightarrow$ 4) follow from corollaries \ref{cor. Trace of projective HKM torsion pair} and \ref{cor. HKM=Prog. complex} (see definition \ref{def. progenerator which is classical tilting complex}). \\

1) $\Longrightarrow$ 2) Consider the finitely generated projective module $P$ and the exact sequence $\xymatrix{0 \ar[r] & F \ar[r] & C \ar[r] & \mathbf{a}P \ar[r] & 0}$ given by theorem \ref{teo. progenerator Imd=aP}. Note that there is a monomorphism of the form $P \monic R^{(n)}$, for some integer $n$, in particular we get $\mathbf{a}P \monic \mathbf{a}R^{(n)}=(\mathbf{a})^{(n)}$, and hence $\mathbf{a}P\in \F$. It follows that $C\in \F$ since so do $F$ and $\mathbf{a}P$. But, then, the fact that $\Ext^{1}_{R}(C,?)$ vanishes on $\F$, implies that $C$ is a finitely generated projective $R/t(R)$-module. Since $C$ generates $\mathcal{C}\cap \F$ we get an epimorphism $M:=C^{(n)}\epic \mathbf{a}$ as desired. \\

2) $\Longrightarrow$ 1) is a direct consequence of theorem \ref{teo. progenerator Imd=aP}., by taking $P=R$. \\

2) $\Longrightarrow$ 3),5) (Assuming that the monoid map $V(R) \flecha V(R/t(R))$ is surjective). We have a finitely generated projective $R$-module $Q$ such that $M\cong \frac{Q}{t(R)Q}=\frac{Q}{t(Q)}$. But we then get $Q=\mathbf{a}Q + t(Q)$ since $\mathbf{a}M=M$. Moreover, we then have $\mathbf{a}Q\cap t(Q)=0$, because $\mathbf{a}Q$ is in $\F$, and hence $Q=\mathbf{a}Q \oplus t(Q)$. It follows that $M\cong \mathbf{a}Q$ is also projective as an $R$-module. Since $M\cong \mathbf{a}Q\in\Gen(\mathbf{a})$, we obtain that $\Gen(M)=\Gen(\mathbf{a})$, so that $\mathbf{a}=tr_{M}(R)$. This show the assertions 3). \\

On the other hand, by taking $C=M$ in theorem \ref{teo. progenerator Imd=aP}, we know that the complex 

$$\xymatrix{G:=\cdots \ar[r] & 0 \ar[r] & M\oplus M \ar[rr]^{\hspace{0.5 cm}(0 \ \ j \hspace{0.01cm} \circ q)} && R \ar[r] & 0 \ar[r] & \cdots}$$
concentrated in degrees -1 and 0, is a progenerator of $\Ht$, where $q:M \epic \mathbf{a}$ is an epimorphism given by hypothesis and $j:\mathbf{a} \monic R$ is the inclusion.
\end{proof}

\vspace{0.3 cm}

In \cite[Corollary 2.13]{MT} (see also \cite[Lemma 4.1]{CMT}) the authors proved that a faithful (not necessarily hereditary) torsion pair in $R$-Mod has a heart which is a module category if, and only if, it is an HKM torsion pair. Our next result shows that, for hereditary torsion pairs, we can be more precise.

\begin{corollary}\label{cor. faith. here. ---> module}
Let $R$ be a ring and let $\te=(\T,\F)$ be a faithful hereditary torsion pair in $R$-Mod. Consider the following assertions:

\begin{enumerate}
\item[1)] There is a finitely generated projective $R$-module $Q$ such that $\T=\Ker(\Hom_{R}(Q,?))$ and the trace $\mathbf{a}$ of $Q$ in $R$ is finitely generated as a left ideal;

\item[2)] $\Ht$ is a module category;

\item[3)] $\te$ is an HKM torsion pair;

\item[4)] $\Ht$ has a progenerator which is a classical tilting complex;

\item[5)] There is an idempotent ideal $\mathbf{a}$ of $R$ which satisfies the following properties:
\begin{enumerate}
\item[a)] $\te$ is the right constituent torsion pair of the TTF triple defined by $\mathbf{a}$;
\item[b)] there is a progenerator $V$ of $R/\mathbf{a}-$Mod which admits a finitely generated projective presentation $\xymatrix{Q \ar[r]^{d} & P \ar[r] & V \ar[r] & 0 }$ in $R$-Mod satisfying the following two properties:
\begin{enumerate}
\item[i.] $\Ker(d) \subseteq \mathbf{a}Q$;
\item[ii.] there is a morphism $Q^{(J)} \xymatrix{\ar[r]^{h} &} \mathbf{a}$, for some set $J$, such that \linebreak $h_{|\Ker(d)^{(J)}}:\Ker(d)^{(J)} \flecha \mathbf{a}$ is an epimorphism.
\end{enumerate}
\end{enumerate}
\end{enumerate}
Then the implications 1) $\Longrightarrow$ 2) $\Longleftrightarrow$ 3) $\Longleftrightarrow$ 4) $\Longleftrightarrow$ 5) hold true. When the monoid morphism $V(R) \flecha V(R/I)$ is surjective, for all idempotent two-sided ideals $I$ of $R$, all assertions are equivalent.
\end{corollary}
\begin{proof}
1) $\Longrightarrow$ 2) follows from corollary \ref{cor. Trace of projective HKM torsion pair}. \\

2) $\Longleftrightarrow$ 3) follows from \cite[Corollary 2.13]{MT}. Moreover, from \cite[Theorem 6.1]{CMT}, we obtain that there is a progenerator of $\Ht$ of the form
$$\xymatrix{\cdots \ar[r] & 0 \ar[r] & Q \ar[r]^{d} & P \ar[r] & 0 \ar[r] & \cdots }$$
concentrated in degrees -1 and 0, where $Q,P$ are finitely generated projective $R$-modules. \\

By the argument of the previous paragraph, the implications 2),3) $\Longrightarrow$ 4) and 4) $\Longrightarrow$ 3) follow from corollary \ref{cor. HKM=Prog. complex} (see definition \ref{def. progenerator which is classical tilting complex}). \\

2),4) $\Longrightarrow$ 5) By corollary \ref{cor. hereditary + fateful ---> right constituent}, we know that $\te$ is the right constituent pair of a TTF triple in $R$-Mod defined by an idempotent ideal $\mathbf{a}$ which is finitely generated on the left. Let now fix a classical tilting complex $\xymatrix{G:=\cdots \ar[r] & 0 \ar[r] & Q \ar[r]^{d} & P \ar[r] & 0 \ar[r] & \cdots}$ which is a progenerator of $\Ht$. Assertion 5 follows by applying  theorem \ref{teo. Ht module c. with t here.} to G, since $\mathbf{a}\in \F$ and hence $t(\mathbf{a})=0$.\\

5) $\Longrightarrow$ 2) By hypothesis $\te$ is bounded and $\add(V)=\add(R/\mathbf{a})$. Taking the complex $G:=\xymatrix{\cdots \ar[r] & 0 \ar[r] & Q \ar[r]^{d} & P \ar[r] & \ar[r] & 0 \ar[r] & \cdots}$, concentrated in degrees -1 and 0. We easily check all conditions 3.a)-3.d) of theorem \ref{teo. Ht module c. with t here.}. \\

2) $\Longrightarrow$ 1) (Assuming that $V(R) \flecha V(R/I)$ is surjective, for all idempotent two-sided ideals $I$ of $R$) It is a consequence of corollaries \ref{cor. hereditary + fateful ---> right constituent} and \ref{cor. a in F}.
\end{proof}

\section{Stalks which are progenerators}

It is natural to expect that the ``simplest'' case in which the heart is a module category appears when the progenerator of the heart may be chosen to be a sum of stalk complexes. We already got a first sample of this phenomenon in the proof of proposition \ref{prop. V[0] progenerator} for tilting torsion pairs. Our next result gives criteria for that to happen. \\

Recall that if $M$ and $N$ are $R$-modules, then $\Ext^{1}_{R}(M,N)=\Hom_{D(R)}(M,N[1])$ has a canonical structure of $\text{End}_{R}(N)-\text{End}_{R}(M)-$bimodule given by composition of morphisms in $\mathcal{D}(R)$. But then it has also a structure of $\text{End}_{R}(M)^{\text{op}}-\text{End}_{R}(N)^{\text{op}}-$bimodule, by defining $\alpha^{o} \circ \epsilon \circ f^{o}=f \circ \epsilon \circ \alpha$, for all $\alpha \in \text{End}_{R}(M)$ and $f\in \text{End}_{R}(N).$

\begin{proposition}\label{prop. progenerator V[0]+Y[1]}
The following assertions are equivalent:
\begin{enumerate}
\item[1)] $\Ht$ has a progenerator of the form $V[0]\oplus Y[1]$, where $V\in \T$ and $Y\in \F$;

\item[2)] There are $R$-modules $V$ and $Y$ satisfying the following properties:
\begin{enumerate}
\item[a)] $V$ is finitely presented and $\T=\Pres(V) \subseteq \Ker(\Ext^{1}_{R}(V,?))$;
\item[b)] $\Ext^{2}_{R}(V,?)$ vanishes on $\F$;
\item[c)] $Y$ is a finitely generated projective $R/t(R)$-module which is in $^{\perp}\T$;
\item[d)] For each $F\in \F$, the module $\frac{F/tr_{Y}F}{t(F/tr_{Y}(F))}$ embeds into a module in $\T$, where $tr_{Y}(F)$ denotes the trace of $Y$ in $F$.
\end{enumerate}
\end{enumerate}

In this case $\Ht$ is equivalent to $S$-Mod, where $S=\begin{pmatrix} \text{End}_{R}(Y)^{op} & 0 \\ \Ext^{1}_{R}(V,Y) & \text{End}_{R}(V)^{op}\end{pmatrix}$, when viewing $\Ext^{1}_{R}(V,Y)$ as a $\text{End}_{R}(V)^{op}-\text{End}_{R}(Y)^{op}-$bimodule in the usual way.
\end{proposition}
\begin{proof}
By theorem \ref{caracterizacion AB5} and by condition 2.a), all throughout the proof we can assume that $\F$ is closed under taking direct limits in $R$-Mod.\\

1) $\Longrightarrow$ 2) Put $G=V[0]\oplus Y[1]$. By lemma \ref{lem. progenerator of Ht}, we get condition 2.a). On the other hand, the projective condition of $V[0]$ (since $V[0]$ is a direct summand of $G$ in $\Ht$) implies that $0=\Ext^{1}_{\Ht}(V[0],F[1])=\Hom_{\D(R)}(V[0],F[2])=\Ext^{2}_{R}(V,F)$, for all $F\in \F$. Then condition 2.b) also holds. \\

Similarly, the projective condition of $Y[1]$ in $\Ht$ implies that $0=\Ext^{1}_{\Ht}(Y[1],F[1])=\Hom_{\D(R)}(Y[1],F[2])=\Ext^{1}_{R}(Y,F)$, for all $F\in \F$ and that $0=\Ext^{1}_{\Ht}(Y[1],T[0])=\Hom_{\D(R)}(Y[1],T[1])\cong \Hom_{R}(Y,T)$, for all $T\in \T$. Then we have that $Y\in$$^{\perp}\T$. Moreover, if $f:\frac{R}{t(R)}^{(I)} \epic Y$ is any epimorphism, then $f$ is a retraction, which implies that $Y$ is a projective $R/t(R)$-module. The fact that $Y$ is finitely generated as $R/t(R)$-module follows from the compactness of $Y[1]$ in $\Ht$ since $\Hom_{R}(Y,?)_{|\F}\cong \Hom_{\D(R)}(Y[1],?[1])_{|\F}$ preserves coproducts of modules in $\F$. We then get condition 2.c).\\

For each $F\in \F$, let us consider the canonical morphism $g:Y^{(\Hom_{R}(Y,F))} \flecha F$. We then get the morphism $\xymatrix{Y[1]^{(\Hom_{\Ht}(Y[1],F[1]))}\cong Y^{(\Hom_{R}(Y,F))}[1] \ar[r]^{\hspace{2.9 cm}g[1]}& F[1]}$ whose image is the trace of $Y[1]$ in $F[1]$ within the category $\Ht$. By \cite{BBD}, from the following diagram  we obtain that the cokernel of $g[1]$ in $\Ht$ is precisely the stalk complex $\frac{\Coker(g)}{t(\Coker(g))}[1]=\frac{F/tr_{Y}(F)}{t(F/tr_{Y}(F))}[1]$:
$$\xymatrix{&&t(\Coker(g))[1] \ar[d] & \\ Y[1]^{(\Hom_{\Ht}(Y[1],F[1]))}\cong Y^{(\Hom_{R}(Y,F))}[1] \ar[r]^{\hspace{2.9 cm}g[1]}& F[1] \ar[r] & \Coker(g)[1] \ar[r]^{\hspace{0.9cm}+} \ar[d] &\\ & & \frac{\Coker(g)}{t(\Coker(g))}[1] \ar[d]^{+}& \\ & & & }$$
Due to the projectivity of $Y[1]$ in $\Ht$, we have that $\Hom_{R}(Y,\frac{\Coker(g)}{t(\Coker(g))})\cong \Hom_{\Ht}(Y[1],\frac{\Coker(g)}{t(\Coker(g))}[1])$\linebreak$=\Hom_{\Ht}(Y[1],\frac{F[1]}{tr_{Y[1]}(F[1])})=0$. It follows that the canonical morphism \linebreak$q:V[0]^{(\Hom_{\Ht}(V[0],\frac{\Coker(g)}{t(\Coker(g))}[1]))} \flecha \frac{\Coker(g)}{t(\Coker(g))}[1]$ is an epimorphism in $\Ht$, since $V[0]\oplus Y[1]$ is a projective generator of $\Ht$. We necessarily have $\Ker(q)=T[0]$, for some $T\in \T$. Condition 2.d) follows then from the long exact sequence of homologies associated to the triangle 

$$\xymatrix{T[0] \ar[r] & V[0]^{(\Hom_{\Ht}(V[0],\frac{\Coker(g)}{t(\Coker(g))}[1]))} \ar[r]^{\hspace{1.3 cm}q} & \frac{\Coker(g)}{t(\Coker(g))}[1] \ar[r]^{\hspace{0.9 cm}+} & }$$

2) $\Longrightarrow$ 1) From conditions 2.a) and 2.b) we deduce that $\Ext^{1}_{\Ht}(V[0],?)$ vanishes on stalk complexes $T[0]$ and $F[1]$, for each $T\in \T$ and $F\in \F$. Similarly, from condition 2.c) we deduce that $\Ext^{1}_{\Ht}(Y[1],?)$ vanishes on all those stalk complexes. It follows that $G:=V[0] \oplus Y[1]$ is a projective object of $\Ht$. \\

Knowing that $G$ is a projective object, in order to prove that $G$ is generator of $\Ht$, we just need to prove that it generates all complexes $X$, with $X\in \T[0] \cup \F[1]$. Note that from condition 2.a) we get that $V[0]$ generates all stalk complexes $T[0]$ and, hence that $T[0]\in \Gen_{\Ht}(G)$. If now we take $F\in \F$, then the argument in the proof of the other implication shows that the canonical morphism 
$\xymatrix{Y[1]^{(\Hom_{\Ht}(Y[1],F[1]))}\cong Y^{(\Hom_{R}(Y,F))}[1] \ar[r]^{\hspace{2.9 cm}g[1]}& F[1]}$ has as cokernel $F^{'}[1]$, where $F^{'}=\frac{F/tr_{Y}F}{t(F/tr_{Y}(F))}$. By hypothesis we have a monomorphism $F^{'} \monic T$ and, hence, an exact sequence $0 \flecha F^{'} \flecha T \flecha T^{'} \flecha 0$, where $T$ and $T^{'}$ are in $\T$. We then get an exact sequence in $\Ht$:

$$0 \flecha T[0] \flecha T^{'}[0] \flecha F^{'}[1] \flecha 0$$ 
which shows that $F^{'}[1]\in \Gen_{\Ht}(G)$, since $T^{'}[0]\in \Gen_{\Ht}(G)$. But we have an exact sequence in $\Ht$:

$$0 \flecha \Imagen_{\Ht}(g[1]) \flecha F[1] \flecha F^{'}[1] \flecha 0$$

Note that $\Imagen_{\Ht}(g[1])\in \Gen_{\Ht}(Y[1])\subseteq \Gen_{\Ht}(G)$. The projective condition of $G$ in $\Ht$ proves then that also $F[1]\in \Gen_{\Ht}(G)$ and, hence $G$ is a generator of $\Ht$. \\

We finally prove that $G$ is compact in $\Ht$, which is equivalent to proving that $V[0]$ and $Y[1]$ are compact in this category. For each family $(M_i)_{i\in I}$ of objects in $\Ht$, we have a direct system of exact sequences in $\Ht$:

$$0 \flecha H^{-1}(M_i)[1] \flecha M_i \flecha H^{0}(M_i)[0] \flecha 0 \hspace{1 cm} (i\in I)$$ 
Using this and the projectivity of $V[0]$ and $Y[1]$, the task is reduced to check the following facts:
\begin{enumerate}
\item[i)] $\Hom_{R}(Y,?)$ preserves coproducts of modules in $\F$;
\item[ii)] $\Hom_{R}(V,?)$ preserves coproducts of modules in $\T$;
\item[iii)] $\Ext^{1}_{R}(V,?)$ preserves coproducts of modules in $\F$.
\end{enumerate}
Conditions i) and ii) automatically hold since $Y$ and $V$ are finitely generated modules. Condition iii) follows from the fact that $V$ is finitely presented. \\

The final statement of the proposition is clear, because the ring $S=\begin{pmatrix}\End_{R}(Y)^{op} & 0 \\ \Ext^{1}_{R}(V,Y) & \End_{R}(V)^{op}\end{pmatrix}$ is isomorphic to $\End_{\Ht}(V[0]\oplus Y[1])^{op}$.
\end{proof}

\vspace{0.3 cm}

We have now the following consequence of last proposition.

\begin{corollary}\label{cor. progenerator casi tilting}
Let $V$ be an $R$-module and consider the following conditions
\begin{enumerate}
\item[1)] $V$ is a classical 1-tilting module;
\item[2)] $\te=(\Pres(V),\Ker(\Hom_{R}(V,?)))$ is a torsion pair in $R$-Mod and $V[0]$ is a progenerator of $\Ht$;

\item[3)] $V$ is a finitely presented and satisfies the following conditions:

\begin{enumerate}
\item[a)] $\T:=\Pres(V)=\Gen(V)\subseteq \Ker(\Ext^{1}_{R}(V,?))$;
\item[b)] $\Ext^{2}_{R}(V,?)$ vanishes on $\F:=\Ker(\Hom_{R}(V,?))$;
\item[c)] Each module of $\F$ embeds into a module of $T$.
\end{enumerate}
\end{enumerate}
Then the implications 1) $\Longrightarrow$ 2) $\Longleftrightarrow$ 3) hold true. Moreover, when conditions 2 or 3 hold, $\te$ is also a torsion pair in the Grothendieck category $\G:=\overline{\Gen}(V)$, $V$ is a classical 1-tilting object of $\G$ and the canonical functor $\D(\G) \flecha \D(R)$ gives by restriction an equivalence of categories $\Ht(G) \iso \Ht$, where $\Ht(\G)$ is the heart of the torsion pair in $\G$. 
\end{corollary}
\begin{proof}
1) $\Longrightarrow$ 2) is a particular case of proposition \ref{prop. V[0] progenerator}. \\

2) $\Longrightarrow$ 3) is a direct consequence of proposition \ref{prop. progenerator V[0]+Y[1]} by taking $Y=0$. \\

3) $\Longrightarrow$ 2) From condition 3.a) and lemma \ref{lem. quasi-tilting} we obtain that $\Gen(V)$ is a torsion class and $V$ is a classical quasi-tilting $R$-module. The implication follows from proposition \ref{prop. progenerator V[0]+Y[1]}. \\

Let us prove now the final statement. From lemma \ref{lem. quasi-tilting}, we get that $V$ is a classical 1-tilting object of $\G$. 
On the other hand, the inclusion functor $\G \monic R$-Mod is exact and, hence, extends to a triangulated functor $j:\D(\G) \flecha \D(R)$, which need be neither faithful nor full, but induces by restriction a functor $\overline{j}:\Ht(\G) \flecha \Ht:=\Ht(R\text{-Mod})$. By the argument of the previous paragraph, we have that $\Ht(\G)$ is a module category with $V[0]$ as progenerator (see the proof of proposition \ref{prop. V[0] progenerator}). Furthermore $\Hom_{\Ht(\G)}(V[0],?):\Ht(\G) \flecha S$-Mod is an equivalence of categories, where $S=\End_{R}(V)^{op}$. We claim that, up to natural isomorphism, the following diagram of functors is commutative:

$$\xymatrix{\Ht(\G) \ar[rr]^{\overline{j}} \ar[dr]_{\Hom_{\D(\G)}(V[0],?) \hspace{0.3 cm}}&& \Ht \ar[dl]^{\hspace{0.2cm}\Hom_{\D(R)}(V[0],?)} \\ & S\text{-Mod} & }$$

Due to the projective condition of $V[0]$ both in $\Ht(\G)$ and $\Ht$, we just need to see that the maps induced by the functor $j$:

$$\xymatrix{\Hom_{\G}(V,T)\cong \Hom_{\D(\G)}(V[0],T[0]) \ar[r] & \Hom_{\D(R)}(V[0],T[0])\cong \Hom_{R}(V,T) \\ \Ext^{1}_{\G}(V,F)\cong \Hom_{\D(\G)}(V[0],F[1]) \ar[r] & \Hom_{\D(R)}(V[0],F[1])\cong \Ext^{1}_{R}(V,F)}$$
are isomorphism, for each $T\in \T$ and for each $F\in \F$. The first one is clear and the second one follows from the next paragraph. \\

We show that the canonical map $\varphi:\Ext^{1}_{\G}(V,X) \flecha \Ext^{1}_{R}(V,X)$ is an isomorphism, for each $X \in \G$. It is clearly injective. To prove that it is surjective, \linebreak let $\xymatrix{0 \ar[r] & X \ar[r] & M \ar[r] & V \ar[r] & 0} \ (\ast)$ be an exact sequence in $R$-Mod. We take a monomorphism $u:X \monic T$, with $T\in \T$. By pushing out the sequence $(\ast)$ along the monomorphism $u$ and using the fact that $\Ext^{1}_{R}(V,T)=0$, we get a monomorphism $M\monic T \oplus V$, which implies that $M\in \G$. Then the sequence $(\ast)$ lives in $\G$ and, hence, $\varphi$ is an isomorphism.  \\

By assertion 2, the functor $\Hom_{\Ht}(V[0],?):\Ht \flecha S$-Mod is an equivalence of categories. Therefore $\overline{j}:\Ht(\G) \flecha \Ht$ is an equivalence of categories.   
\end{proof}

\vspace{0.3 cm}

The following is now very natural.

\begin{question}
Let $\te=(\T,\F)$ be a torsion pair in $R$-Mod satisfying the equivalent conditions 2 and 3 of corollary \ref{cor. progenerator casi tilting}. Is $\te$ a classical tilting torsion pair?
\end{question}

\vspace{0.3 cm}

In the following section, we will show that last question has a negative answer.

\begin{corollary}\label{cor. here. torsion pair with V+Y as proge.}
Let us assume that $\te=(\T,\F)$ is a hereditary torsion pair in $R$-Mod. The following assertions are equivalent:
\begin{enumerate}
\item[1)] $\Ht$ has a progenerator of the form $V[0]\oplus Y[1]$, where $V\in \T$ and $Y\in \F$;
\item[2)] There are $R$-modules $V$ and $Y$ satisfying the following properties:
\begin{enumerate}
\item[a)] $V$ is finitely presented and $\T=\Pres(V)\subseteq \Ker(\Ext^{1}_{R}(V,?))$;
\item[b)] $\Ext^{\hspace{0.03 cm}2}_{R}(V,?)$ vanishes on $\F$;
\item[c)] $Y$ is a finitely generated projective $R/t(R)$-module which is in $^{\perp}\T$;
\item[d)] For each $F\in \F$, the module $F/tr_{Y}(F)$ is in $\T$, where $tr_Y(F)$ denotes the trace of $Y$ in $F$.
\end{enumerate}
\end{enumerate}
In this case, if $I=t(R)$ then $\te^{'}=(\T\cap R/I-\text{Mod},\F)$ is a torsion pair in $R/I-$Mod which is the right constituent of a TTF triple in this category and has the property that $\frac{V}{IV}[0]\oplus Y[1]$ is a progenerator of $\mathcal{H}_{\mathbf{t}^{'}}$. Moreover, the forgetful functor $\mathcal{H}_{\mathbf{t}^{'}} \flecha \Ht$ is faithful. It is full if, and only if, each short exact sequence $0 \flecha Y \flecha M \flecha V/IV \flecha 0$ in $R$-Mod satisfies that $IM=0$.
\end{corollary} 
\begin{proof}
All throughout proof we put $\overline{M}=M/IM$, for each $R$-module $M$. The equivalence of assertions 1 and 2 is a direct consequence of proposition \ref{prop. progenerator V[0]+Y[1]}. From theorem \ref{teo. Ht module c. with t here.} and its proof, we know that $(\T \cap \overline{R}\text{-Mod},\F)$ is the right constituent torsion pair of a TTF triple $(\mathcal{C}_{I},\T_{I},\F_{I})$ in $\overline{R}$-Mod. Moreover, by property 2.c), the class $\Ker(\Hom_{\overline{R}}(Y,?))$ contains $\T\cap \overline{R}$-Mod and is closed under taking quotients. Thus, if $Z\in\Ker(\Hom_{\overline{R}}(Y,?))$ then $F:=Z/t(Z)$ also is in $\Ker(\Hom_{\overline{R}}(Y,?)$, so that $tr_{Y}(F)=0.$ Using property 2.d), we get $F\cong F/tr_{Y}(F)$ is in $\T$ and, hence $F=0$ and $Z\cong t(Z)$. It then follows that the inclusion $\Ker(\Hom_{\overline{R}}(Y,?)) \subseteq \T \cap \overline{R}$-Mod also holds, which implies that $\mathcal{C}_{I}=\Gen(Y)$. If now $\mathbf{a}$ is the two-sided ideal of $R$ given by the equality $\overline{\mathbf{a}}=\frac{\mathbf{a}}{I}=tr_{Y}(\frac{R}{I})$, then $\overline{\mathbf{a}}$ is the idempotent ideal of $\overline{R}$ which defines the TTF triple and, by the proof of theorem \ref{teo. Ht module c. with t here.}, we know that $\mathbf{a}=\ann_{R}(\overline{V})$ and that $\add(\overline{V})=\add(R/\mathbf{a})$, so that $\overline{V}$ is a progenerator of $\frac{R}{\mathbf{a}}-$Mod.\\

Now the $\overline{R}$-modules $\overline{V}$ and $Y$ satisfy the conditions 2.a), 2.c) and 2.d) with respect to the torsion pair $\te^{'}=(\T_{I},\F)$ of $\overline{R}$-Mod. On the other hand, $\te$ and $\te^{'}$ are hereditary torsion pairs in $R$-Mod and $\overline{R}$-Mod, respectively. Then, for each $F\in \F$, the injective envelope $E(F)$ in $R$-Mod is also in $\F$ (see \cite[Proposition VI.3.2]{S}). In particular, we have that $E(F)\in \overline{R}$-Mod, so that $E(F)$ is also the injective envelope of $F$ as a $\overline{R}$-module and, hence, the first cosyzygy $\Omega^{-1}(F)$ is the same in $R$-Mod and $\overline{R}$-Mod. In order to check condition 2.b) for $\overline{V}$, we need to cheek that $\Ext^{1}_{R}(\overline{V},\Omega^{-1}(F))=0$. But using condition 2.b) for $V$, our needed goal will follow from something stronger that we will prove. Namely, that if $p=p_V:V \epic \overline{V}$ is the canonical projection, then the composition

$$\xymatrix{\varphi_M:\Ext^{1}_{\overline{R}}(\overline{V},M) \ar[r]^{\hspace{0.2 cm}\text{can}} & \Ext^{1}_{R}(\overline{V},M) \ar[rr]^{\Ext^{1}(p,M)} && \Ext^{1}_{R}(V,M)}$$
is a monomorphism for all $M\in \overline{R}$-Mod. \\

Let $\xymatrix{0 \ar[r] & M \ar[r]^{j} & N \ar[r]^{q} & \overline{V} \ar[r] & 0}$ be an exact sequence in $\overline{R}$-Mod which represents an element of $\Ker(\varphi_M)$. Then the projection $p:V \epic \overline{V}$ factors through $q$. Fixing a morphism $g:V \flecha N$ such that $q \circ g=p$ and taking into account that $IN=0$, we get a morphism $\overline{g}: \overline{V} \flecha N$ such that $g=\overline{g} \circ p$ and so $p=q \circ g=q \circ \overline{g} \circ p$. It follows that $q$ is a retraction since $p$ is an epimorphism. \\

In order to prove the final assertion, we consider the following composition of morphisms of abelian groups, where $F\in \F$:

$$\xymatrix{\Ext^{1}_{R}(j \circ \rho, F): \Ext^{1}_{R}(V,F) \ar[rr]^{\hspace{1.1 cm}\Ext^{1}_{R}(j,F)} && \Ext^{1}_{R}(t(R)V,F) \ar[rr]^{\Ext^{1}_{R}(\rho,F)\hspace{0.8 cm}} && \Ext^{1}_{R}(V^{(\Hom_{R}(V,R)\times V)},F)}$$
where $q:V^{(\Hom_{R}(V,R))} \flecha t(R)=tr_{V}(R)$, $\pi: R^{(V)} \epic V$ and $j:t(R)V \monic V$ be the canonical morphism and $\pi^{'}:t(R)^{(V)} \epic t(R)V$ is the epimorphism given by the restriction of $\pi$ to $R^{(V)}$ and $\rho:=\pi^{'} \circ q^{(V)}$. We have that $\Ext^{1}_{R}(\rho, F)$ is a monomorphism, because $\Ker(\rho)\in \T$ and hence $\Hom_{R}(\Ker(\rho),F)=0$. But $\Ext^{1}_{R}(j \circ \rho, F)=0$ since $j \circ \rho$ factors through a projective $R$-module (see the proof of lemma \ref{lem. ImB<t(R)V}). We then get that $\Ext^{1}_{R}(j,F)$ is the zero map, for each $F\in \F$. By considering the canonical exact sequence $\xymatrix{0 \ar[r] & t(R)V \hspace{0.03cm}\ar@{^(->}[r]^{\hspace{0.4cm}j} & V \ar@{>>}[r] & \overline{V} \ar[r] & 0}$ and applying to it the long exact sequence of $\Ext^{1}(?,F)$, we get:

$$\xymatrix{0 = \Hom_{R}(t(R)V,F) \ar[r] & \Ext^{1}_{R}(\overline{V},F) \ar[r] & \Ext^{1}_{R}(V,F) \ar[r]^{0 \hspace{0.4 cm}} & \Ext^{1}_{R}(t(R)V,F)}$$
which proves that $\Ext^{1}_{R}(\overline{V},F)\cong \Ext^{1}_{R}(V,F)$, for each $F\in \F$. Moreover, by the two previous paragraphs, we get that the map $\Ext^{1}_{\overline{R}}(\overline{V},F) \xymatrix{\ar[r]^{\text{can}} &} \Ext^{1}_{R}(\overline{V},F)$ is a monomorphism. \\

Let us put $\overline{G}:=\overline{V}[0]\oplus Y[1]$. We claim that the map $\Hom_{\mathcal{H}_{\mathbf{t}^{'}}}(\overline{G},M) \flecha \Hom_{\Ht}(\overline{G},M)$ is injective for all $M\in \mathcal{H}_{\mathbf{t}^{'}}$. Bearing in mind that we have isomorphisms of abelian groups 
$$\Hom_{\mathcal{H}_{\mathbf{t}^{'}}}(Y[1],M)\cong \Hom_{\overline{R}}(Y,H^{-1}(M))=\Hom_{R}(Y,H^{-1}(M))\cong \Hom_{\Ht}(Y[1],M)$$
our task reduces to check that the canonical map $\Hom_{\mathcal{H}_{\mathbf{t}^{'}}}(\overline{V}[0],M) \flecha \Hom_{\Ht}(\overline{V}[0],M)$ is injective. But we have the following commutative diagram with exact rows:

$$\xymatrix{0 \ar[r] & \Ext^{1}_{\overline{R}}(\overline{V},H^{-1}(M)) \ar[r] \ar@{^(->}[d]^{\text{can}} & \Hom_{\mathcal{H}_{\mathbf{t}^{'}}}(\overline{V}[0],M) \ar[r] \ar[d]& \Hom_{\overline{R}}(\overline{V},H^{0}(M)) \ar[r] \ar[d]^{\wr}& \Ext^{2}_{\overline{R}}(\overline{V},F)=0\\ 0 \ar[r] & \Ext^{1}_{R}(\overline{V},H^{-1}(M)) \ar[r] & \Hom_{\Ht}(\overline{V}[0],M) \ar[r] & \Hom_{R}(\overline{V},H^{0}(M)) \ar[r] & \Ext^{2}_{R}(\overline{V},F)}$$
The right vertical arrow is an isomorphism since $H^{0}(M)$ is a $\overline{R}$
-module, and the left vertical arrow is a monomorphism. It then follows that the central vertical arrow is a monomorphism, as desired.\\

Let us fix any object $M\in \mathcal{H}_{\mathbf{t}^{'}}$ and consider the full subcategory $\mathcal{C}_{M}$ of $\mathcal{H}_{\mathbf{t}^{'}}$ consisting of the objects $N$ such that the canonical map $\Hom_{\mathcal{H}_{\mathbf{t}^{'}}}(N,M) \flecha \Hom_{\Ht}(N,M)$ is a monomorphism. This subcategory is closed under taking coproducts and cokernels and, by the previous paragraph, it contains $\overline{G}$. We then have $\mathcal{C}_{M}=\mathcal{H}_{\mathbf{t}^{'}}$ (because $\overline{G}$ generates $\mathcal{H}_{\mathbf{t}^{'}}$) and, since this is true for any $M\in \mathcal{H}_{\mathbf{t}^{'}}$, we conclude that the forgetful functor $\mathcal{H}_{\mathbf{t}^{'}} \flecha \Ht$ is faithful. \\

Finally, note that the argument giving the faithfulness also shows that the functor is full if, and only if, the canonical map $\Ext^{1}_{\overline{R}}(\overline{V},F) \flecha \Ext^{1}_{R}(\overline{V},F)$ is surjective (and hence bijective), for each $F\in \F$. But we have proved that $\Ext^{1}_{R}(\overline{V},F)\cong \Ext^{1}_{R}(V,F)$. Hence, the task reduces to prove that $\varphi_{F}:\Ext^{1}_{\overline{R}}(\overline{V},F) \flecha \Ext^{1}_{R}(V,F)$ is surjective, for all $F\in \F$, knowing that by hypothesis, $\varphi_Y$ is an isomorphism. Note first that, by property 2.d), if $F\in \F$ then we have an exact sequence 
$$0 \flecha tr_{Y}(F) \monic F \epic T \flecha 0$$
in $R/I-$Mod, with $T\in \T$. Then we get the following commutative diagram 

$$\xymatrix{0 \ar[r] & \Hom_{\overline{R}}(\overline{V},T) \ar[r] \ar[d]^{\wr} & \Ext^{1}_{\overline{R}}(\overline{V},tr_{Y}(F)) \ar[r] \ar[d] & \Ext^{1}_{\overline{R}}(\overline{V},F) \ar[r] \ar[d] & 0\\ 0 \ar[r] & \Hom_{R}(V,T) \ar[r] & \Ext^{1}_{R}(V,tr_{Y}(F)) \ar[r] & \Ext^{1}_{R}(V,F) \ar[r] & 0 }$$

The rows are exact since $\Ext^{1}_{\overline{R}}(\overline{V},T)=0=\Ext^{1}_{R}(V,T)$, while the vertical left arrow is clearly an isomorphism. Then we can assume, without loss of generality, that $F=tr_{Y}(F)$ in the sequel. Fix for such $F$ an epimorphism $\delta:Y^{(J)} \epic F$, for some set $J$. Bearing in mind that $\Ext^{2}_{\overline{R}}(\overline{V}.?)$ and $\Ext^{2}_{R}(V,?)$ both vanish on $\F$ and that $_{\overline{R}}\overline{V}$ and $_{R}V$ are finitely presented modules, we get a commutative square, where the horizontal arrows are epimorphisms and the left vertical one is an isomorphism since it is $\varphi_{Y}^{(J)}$:

$$\xymatrix{\Ext^{1}_{\overline{R}}(\overline{V},Y^{(J)}) \cong \Ext^{1}_{\overline{R}}(\overline{V},Y)^{(J)} \ar@{>>}[r] \ar[d]^{\wr} & \Ext^{1}_{\overline{R}}(\overline{V},F) \ar[d] \ar[r] & 0\\ \Ext^{1}_{R}(V,Y^{(J)})\cong \Ext^{1}_{R}(V,Y)^{(J)} \ar@{>>}[r] & \Ext^{1}_{R}(V,F) \ar[r] & 0}$$ 
Then the right vertical arrow is an epimorphism, as desired.
\end{proof}

\begin{proposition}\label{prop. Left constituent hereditary}
Let $\te=(\T,\F)$ be hereditary and suppose that it is the {\bf left} constituent torsion pair of a TTF triple. Then $\Ht$ is a module category if, and only if, there is a finitely generated projective $R$-module $P$ such that $\T=\Gen(P)$. In such case, the following assertions hold:
\begin{enumerate}
\item[1)] $\te$ is HKM if, and only if, there is a finitely generated projective $R$-module $Q^{'}$ such that $\Hom_{R}(Q^{'},P)=0$ and $\add(\frac{Q^{'}}{t(Q^{'})})=\add(\frac{R}{t(R)})$;
\item[2)] $\te$ is the right constituent of a TTF triple in $R$-Mod if, and only if, $P$ is finitely generated over its endomorphism ring.
\end{enumerate}
\end{proposition}
\begin{proof}
The ``if'' part: Let $P$ be a finitely generated projective $R$-module such that $\T=\Gen(P)$. We will check that $V=P$ and $Y=\frac{R}{t(R)}$ satisfy conditions 2.a)-2.d) of corollary \ref{cor. here. torsion pair with V+Y as proge.}. All these properties are trivially satisfied, except the fact that $Y\in$$^{\perp}\T$. For that, we consider the TTF triple $(\T,\F,\F^{\perp})$. By \cite[Lemma VI.8.3]{S}, we know that $\T \subseteq \F^{\perp}$. It particular, $Y=R/t(R)\in \F=$$^\perp(\F^{\perp})\subseteq $$^{\perp}\T$. \\

The ``only if'' part: Let $\mathbf{a}$ be the idempotent ideal which defines the TTF triple, so that $\T=\{ T\in R\text{-Mod}: \mathbf{a}T=T\}$. By theorem \ref{teo. Ht is a module category}, we have a progenerator in $\Ht$ of the form:
$$\xymatrix{G:=\cdots \ar[r] & 0 \ar[r] & X \ar[r]^{j} & Q \ar[r]^{d} & P \ar[r] & 0 \ar[r] & \cdots}$$
where $P$ and $Q$ are finitely generated and $\T=\Gen(V)$, where $V=H^{0}(G)$. We then have $\mathbf{a}M=t(M)=tr_{V}(M)$, for each $R$-module $M$. In particular, we have $\mathbf{a}=t(R)=tr_{V}(R)$ and, by applying lemma \ref{lem. ImB<t(R)V} to the identity $1_V:V \flecha V$, we conclude that $V$ is a finitely generated projective module (since $\Imagen(1_V)=V=\mathbf{a}V=t(R)V$).\\

We next prove assertions 1 and 2:

1) If $Q^{'}$ exists, then complex $\xymatrix{P^{\bullet}:=\cdots \ar[r] & 0 \ar[r] & Q^{'} \ar[r]^{0} & P \ar[r] & 0 \ar[r] & \cdots}$, concentrated in degrees -1 and 0, satisfies assertion 2 of proposition \ref{pro. HKM complex} since we know that $P[0]\oplus \frac{R}{t(R)}[1]$ is a progenerator of $\Ht$. Conversely, suppose that $\te$ is HKM and let $\xymatrix{P^{\bullet}:=\cdots \ar[r] & 0 \ar[r] & Q \ar[r]^{d} & P^{'} \ar[r] & 0 \ar[r] & \cdots}$ be an HKM complex whose associated torsion pair is $\te$. Then, by proposition \ref{pro. HKM complex}, we know that the complex
$$\xymatrix{G:= \cdots \ar[r] & 0 \ar[r] & t(\Ker(d)) \ar[r] & Q \ar[r]^{d} & P^{'} \ar[r] & 0 \ar[r] & \cdots}$$
concentrated in degrees -2,-1,0, is a progenerator of $\Ht$. We then have that $\add_{\Ht}(G)=\add_{\Ht}(P[0]\oplus \frac{R}{t(R)}[1])$. In particular, we get that $V:=H^{0}(G)$ is a projective module, which implies that $\Imagen(d)$ is projective and that $\Ker(d)$ is projective too. It follows that, up to isomorphism in the category $\mathcal{C}(R)$, we can rewrite $Q$ as:

$$\xymatrix{\cdots \ar[r] & 0 \ar[r] & t(Q^{'}) \ar[r]^{\begin{pmatrix} 0 \\ \iota \end{pmatrix}\hspace{0.6 cm}} \ar[r] & \Imagen(d)\oplus Q^{'} \ar[rr]^{\begin{pmatrix} 1 & 0 \\ 0 & 0 \end{pmatrix}\hspace{0.35cm}}& & \Imagen(d) \oplus V \ar[r] & 0 \ar[r] & \cdots}$$ 
where $Q^{'}:=\Ker(d)$ and $\iota:t(Q^{'}) \monic Q^{'}$ is the inclusion. Similarly, $P^{\bullet}$ is isomorphic in $\mathcal{C}(R)$ to the complex

$$\xymatrix{\cdots \ar[r] & 0 \ar[r] & \Imagen(d)\oplus Q^{'} \ar[rr]^{\begin{pmatrix} 1 & 0 \\ 0 & 0 \end{pmatrix}\hspace{0.35cm}}& & \Imagen(d) \oplus V \ar[r] & 0 \ar[r] & \cdots }$$ 
Each $R$-homomorphism $g:Q^{'} \flecha V$ give rise to a morphism $P^{\bullet} \flecha V[1]$ which is the zero morphism since $V\in \mathcal{X}(P^{\bullet})$. In particular, we get the following commutative diagram

$$\xymatrix{\Imagen(d)\oplus Q^{'} \ar[d]_{\begin{pmatrix}0 & g \end{pmatrix}}\ar[rr]^{\begin{pmatrix} 1 & 0 \\ 0 & 0 \end{pmatrix}\hspace{0.35cm}}& & \Imagen(d) \oplus V \ar[dll]^{\hspace{0.5 cm}\begin{pmatrix}h_1 & h_2\end{pmatrix}} \\ V && }$$
which shows that $g=0$. Using the projective condition of $Q^{'}$, we deduce that $Q^{'}\in \hspace{0.01 cm} ^{\perp} \T$ and, hence $\Hom_{R}(Q^{'},P)=0$. On the other hand, note that $\add(\frac{Q^{'}}{t(Q^{'})})=\add(H^{-1}(G))=\add(H^{-1}(P[0]\oplus \frac{R}{t(R)}[1]))=\add(\frac{R}{t(R)})$. \\

2) $\te$ is the right constituent pair of a TTF triple if, and only if, $\T=\Gen(P)$ is closed under taking products in $R$-Mod. But this is equivalent to saying that each product of copies of $P$ is in $\Gen(P)$. By \cite[Lemma, Section 1]{CM}, this happens exactly when $P$ is finitely generated over its endomorphism ring.
\end{proof}

\vspace{0.3 cm}

We are ready to answer question 3 of the introduction in the negative.

\begin{corollary}\label{cor. Ht module not HKM}
There are torsion pairs $\te$ in $R$-Mod such that $\Ht$ is a module category, but $\te$ is not an HKM torsion pair.
\end{corollary}
\begin{proof}
We consider a field $K$, an infinite dimensional $K$-vector space $P$ and view it as left module over $R=\End_{K}(P)$. It is well-know that $P$ is a faithful simple projective $R$-module, so that $\T=\Add($$_{R}P)=\Gen($$_{R}P)$ is closed under taking submodules and, hence, $\te$ is hereditary. By the previous proposition we get that $\Ht$ is a module category. However, the faithful condition of $_{R}P$ implies that each projective $R$-module embeds into a direct product of copies of $P$. Then it does not exist a finitely generated projective $R$-module $Q^{'}$ such that $\Hom_{R}(Q^{'},P)=0$ and $\add(Q^{'}/t(Q^{'}))=\add(R/t(R))$. By the previous proposition, we conclude that $\te$ is not HKM. 
\end{proof}


\vspace{0.3 cm}

Recall that a ring is left \emph{semi-hereditary}\index{ring! semi-hereditary} when its finitely generated left ideals are projective.

\begin{example}\label{exam. R/a+a as progenerator}\rm{
Let $\mathbf{a}$ be an idempotent two-sided ideal of $R$, let $(\mathcal{C},\T,\F)$ be the associated TTF triple in $R$-Mod and let $\te=(\T,\F)$ be its right constituent torsion pair. The following assertions are equivalent:
\begin{enumerate}
\item[1)] $\frac{R}{\mathbf{a}}[0] \oplus \frac{\mathbf{a}}{t(\mathbf{a})}[1]$ is a progenerator of $\Ht$;

\item[2)] $\Ht$ has a progenerator of the form $V[0]\oplus Y[1]$, with $V\in \T$ and $Y\in F$;

\item[3)] $\mathbf{a}$ is finitely generated on the left and $\Ext^{2}_{R}(R/\mathbf{a},?)$ vanishes on $\F$.
\end{enumerate}
In particular, if $R$ is left semi-hereditary and $\te$ is the right constituent pair of a TTF triple in $R$-Mod, then $\Ht$ is a module category, if and only if, the associated idempotent ideal is finitely generated on the left.}
\end{example}
\begin{proof}
1) $\Longrightarrow $ 2) is clear. \\

2) $\Longrightarrow$ 3) By lemma \ref{lem. progenerator of Ht}, we know that $V$ is finitely presented and $\T=\Gen(V) \subseteq \Ker(\Ext^{1}_{R}(V,?))$. On the other hand, by hypothesis we also have that $\T=\{T\in R\text{-Mod}:\mathbf{a}T=0\}\cong R/\mathbf{a}-$Mod. In particular, $V$ is a finitely presented generator of $R/\mathbf{a}-$Mod such that $\Ext^{1}_{R/\mathbf{a}}(V,?)=0$ since the canonical morphism $\Ext^{1}_{R/\mathbf{a}}(V,T) \flecha \Ext^{1}_{R}(V,T)$ is a monomorphism, for each $T\in \T$. Hence, $V$ is a progenerator of $\frac{R}{\mathbf{a}}-$Mod, which implies that $\add_{R\text{-Mod}}(R/\mathbf{a})=\add_{R\text{-Mod}}(V)$. Then $R/\mathbf{a}$ is a finitely presented left $R$-module and, hence, $\mathbf{a}$ is finitely generated as a left ideal. The fact that $\Ext^{2}_{R}(R/\mathbf{a},?)$ vanishes on $\F$ follows from the fact that, by corollary \ref{cor. here. torsion pair with V+Y as proge.}, we know that $\Ext^{2}_{R}(V,?)$ vanishes on $\F$. \\

3) $\Longrightarrow$ 1) We take $V=\frac{R}{\mathbf{a}}$ and $Y=\frac{\mathbf{a}}{t(\mathbf{a})}$. Note that $tr_{Y}(F)=\mathbf{a}F$, for all $F\in\F$ and, hence, conditions 2.a), 2.b) and 2.d) of corollary \ref{cor. here. torsion pair with V+Y as proge.} hold, since $F/tr_{Y}(F)=F/\mathbf{a}F$ is in $\T$, for all $F\in \F$. We just need to prove that $\frac{\mathbf{a}}{t(\mathbf{a})}$ is a projective $R/t(R)$-module since it is clearly in $^{\perp}\T=\mathcal{C}$. Indeed, let $0 \flecha K \hspace{0.03 cm} \xymatrix{\ar@{^(->}[r]^{j} & } Q \xymatrix{\ar[r]^{q}&} \mathbf{a} \flecha 0$ be an exact sequence, with $Q$ a finitely generated projective $R$-module. The canonical projection $p:K \epic K/t(K)$ extends to Q since $\Ext^{1}_{R}(\mathbf{a},\frac{K}{t(K)})\cong \Ext^{2}_{R}(\frac{R}{\mathbf{a}},\frac{K}{t(K)})=0$. We then have a morphism $g:Q \flecha \frac{K}{t(K)}$ such that $g \circ j=p$. Consider now the following commutative diagram
$$\xymatrix{0 \ar[r] & t(K) \ar[r] \ar@{^(->}[d] & K \ar@{>>}[r]^{p} \ar@{^(->}[d]^{j} & \frac{K}{t(K)} \ar[r] \ar[d]^{\iota}& 0\\ 0 \ar[r] & t(Q) \ar[r] & Q \ar[r]^{p^{'}} \ar[ur]^{g} & \frac{Q}{t(Q)} \ar@<0.7ex>[u]^{g^{'}}\ar[r] & 0}$$

where $g^{'}$ is the morphism given by the universal property of cokernels, such that $g=g^{'} \circ p^{'}$. We then have $p= g \circ j = g^{'} \circ p^{'} \circ j= g^{'} \circ \iota \circ p$, hence $\iota$ is a split monomorphism, since $p$ is an epimorphism. But, the cokernel of $\iota$ is $\frac{\frac{Q}{t(Q)}}{\frac{K+t(Q)}{t(Q)}}\cong \frac{Q}{K+t(Q)} \cong \frac{\frac{Q}{K}}{\frac{K+t(Q)}{K}}\cong \frac{\mathbf{a}}{q(t(Q))}$. It follows that this latter one is a projective $R/t(R)$-module, which implies that it is in $\F$ when viewed as an $R$-module. But then $\frac{t(\mathbf{a})}{q(t(Q))}\in \T\cap \F=0$. Therefore we have $q(t(Q))=t(\mathbf{a})$ and $\frac{\mathbf{a}}{q(t(Q))}=\frac{\mathbf{a}}{t(\mathbf{a})}$ is projective as a left $R/t(R)$-module.
\end{proof}

\vspace{0.3 cm}

The following is both a consequence of last example and of corollary \ref{cor. Trace of projective HKM torsion pair}.

\begin{corollary}\label{cor. a projective}
Let $\mathbf{a}$ be an idempotent two-sided ideal of $R$, let $(\mathcal{C},\T,\F)$ be the associated TTF triple in $R$-Mod and let $\te=(\T,\F)$ be its right constituent torsion pair. If $\mathbf{a}$ is a finitely generated projective left $R$-module, then $\Ht$ is a module category.
\end{corollary}

\begin{corollary}\label{cor. left split}
Let $\mathbf{a}$ be a two-sided idempotent ideal of $R$ whose associated TTF triple is left split and put $\te=(\T,\F)$. Then $\Ht$ is equivalent to $\frac{R}{t(R)} \times \frac{R}{\mathbf{a}}-$Mod. \\
When the TTF triple is centrally split, $\Ht$ is equivalent to $R$-Mod.
\end{corollary}
\begin{proof}
In this case $\mathbf{a}$ is a direct summand of $_{R}R$, whence projective, so that the previous corollary and example \ref{exam. R/a+a as progenerator} apply. Note that then $V=R/\mathbf{a}$ is a projective left $R$-module, which implies that $\Ht$ is equivalent to $S$-Mod, where $S\cong \text{End}_R(\frac{R}{t(R)})^{op} \times \text{End}_{R}(\frac{R}{\mathbf{a}})^{op}\cong \frac{R}{t(R)}\times \frac{R}{\mathbf{a}}$. \\

When the TTF triple is centrally split, we have a central idempotent $e$ such that $\mathbf{a}=Re$ and $T=R(1-e)$, so that $R/\mathbf{a}\cong R(1-e)$ and $R/t(R)\cong Re$. The result in this case follows immediately since we have a ring isomorphism $R\cong Re \times R(1-e)$. 
\end{proof}

\vspace{0.3 cm}

We are now able to give another significative class of rings for which we are able to identify all hereditary torsion pairs whose heart is a module category.

\begin{proposition}
Let $R$ be a left semihereditary ring and let $V$ be a finitely presented quasi-tilting $R$-module whose associated torsion pair $\te=(\Gen(V), \Ker(\Hom_{R}(V,?)))$ is hereditary. The following assertions are equivalent:
\begin{enumerate}
\item[1)] if $\mathbf{a}=\ann_{R}(V/t(R)V)$ then $\mathbf{a}/t(R)$ is an idempotent ideal of $R/t(R)$, which is finitely generated on the left, and there is a monomorphism $R/\mathbf{a} \monic (V/t(R)V)^{(n)}$, for some natural number $n$.

\item[2)] The heart $\Ht$ is a module category.
\end{enumerate}
In this case $\Ht$ is equivalent to $S$-Mod, where $S=\begin{pmatrix} \End_{R}(\frac{\mathbf{a}}{t(R)})^{\text{op}} & 0 \\ \Ext^{1}_{R}(V,\frac{\mathbf{a}}{t(R)}) & \End_{R}(V)^{\text{op}} \end{pmatrix}$.
\end{proposition}
\begin{proof}
1) $\Longrightarrow$ 2) Put $\overline{M}=\frac{M}{t(R)M}$, for each $R$-module $M$. Note that $\T \cap \overline{R}-$Mod$=\Gen(\overline{V})$ and that $\ann_{\overline{R}}(\overline{V})=\overline{\mathbf{a}}$. Now, Let $f:\overline{\mathbf{a}} \flecha T$ be an arbitrary $R$-homomorphism, where $T\in \T$. Note that $ \Imagen(f)\in \T \cap \overline{R}$-Mod and, hence, the induced morphism \linebreak $\overline{f}:\mathbf{a} \flecha \Imagen(f)$ is a morphism in $\overline{R}$-Mod such that $\overline{f}(\overline{\mathbf{a}})=\overline{f}(\overline{\mathbf{a}}^{2})=\overline{\mathbf{a}}\Imagen(f)=0$. This show that $\Hom_{R}(\overline{\mathbf{a}},T)=0$, for each $T\in \T$, thus, $\overline{\mathbf{a}}$ is in $^{\perp}\T$. \\

On the other hand, since $\overline{\mathbf{a}}$ is finitely generated on the left, we have a finitely generated left ideal $\mathbf{a}^{'}$ of $R$ contained in $\mathbf{a}$ such that the canonical composition $\mathbf{a}^{'} \monic \mathbf{a} \epic \overline{\mathbf{a}}$ is an epimorphism. We consider the following commutative diagram:

$$\xymatrix{ 0 \ar[r] & \mathbf{a}^{'} \ar@{^(->}[r] \ar@{>>}[d] & \mathbf{a} \ar[r] \ar@{>>}[d] & \mathbf{a}/\mathbf{a}^{'} \ar[r] \ar[d] & 0 \\ 0 \ar[r] & \overline{\mathbf{a}} \ar@{=}[r] & \overline{\mathbf{a}} \ar[r] & 0 \ar[r] & 0}$$
It is easy to see that $\frac{\mathbf{a}^{'}}{t(\mathbf{a}^{'})}=\frac{\mathbf{a}^{'}}{\mathbf{a}^{'}\cap t(R)} \cong \overline{\mathbf{a}}$. Since $\mathbf{a}^{'}$ is projective, we conclude that $\overline{\mathbf{a}}$ is a finitely generated projective left $\overline{R}$-Mod. Now, note that $V$ and $Y:=\overline{\mathbf{a}}$ satisfy all conditions 2.a)-2.c) of corollary \ref{cor. here. torsion pair with V+Y as proge.}. Moreover if $F\in \F$ then $F/\mathbf{a}F$ is generated by $R/\mathbf{a}$ and, due to our hypotheses, we know that $R/\mathbf{a}\in \T$. Since $\te$ is a hereditary torsion pair, we get that $F/\mathbf{a}F$ is in $\T$. Thus, the property 2.d) of the mentioned corollary holds. Then $\Ht$ is a module category, actually equivalent to $S$-Mod (see proposition \ref{prop. progenerator V[0]+Y[1]}). \\

2) $\Longrightarrow$ 1) Let $\xymatrix{G:=\cdots \ar[r] & 0 \ar[r] & X \ar[r]^{j} & Q \ar[r]^{d} & P \ar[r] & 0 \ar[r] & \cdots}$ be a complex as in theorem \ref{teo. Ht module c. with t here.}, which is then a progenerator of $\Ht$. We know that $\Imagen(d)$ is a finitely generated $R$-module, since $H^{0}(G)$ is a finitely presented $R$-module (see lemma \ref{lem. progenerator of Ht}). From \cite[proposition I.6.9]{S}, we then get $\Imagen(d)$ is a projective $R$-module and, then $G$ is isomorphic to $H^{0}(G)[0]\oplus H^{-1}(G)[1]$ (use the argument of the proof of proposition \ref{prop. Left constituent hereditary}). Putting $V=H^{0}(G)$ and $Y=H^{-1}(G)$ for simplicity, corollary \ref{cor. here. torsion pair with V+Y as proge.} and its proof show that $\te^{'}=(\T\cap \overline{R}\text{-Mod},\F)$ is the right constituent torsion pair of a TTF triple in $\overline{R}$-Mod defined by the idempotent ideal $\overline{\mathbf{a}}=\ann_{\overline{R}}(\overline{V})$, which is finitely generated on the left. Then we have $\overline{\mathbf{a}}=\frac{\mathbf{a}}{t(R)}$, where $\mathbf{a}=\ann_{R}(V/t(R)V)$. Moreover, we have $R/\mathbf{a}-$Mod$=\T \cap \overline{R}$-Mod$=\Gen(\overline{V})$, so that $R/\mathbf{a}\in \add(\overline{V})$.
\end{proof}

\section{When $\T$ is closed under taking products}

Our next result shows that if the torsion class is closed under taking products in $R$-Mod, then classical tilting theory appears quite naturally. Nevertheless, in this section, we will show a way of building torsion pairs $\te$ in $R$-Mod, such that $\Ht$ has a progenerator of the form $V[0]$, although $V$ is not a tilting $R$-module.

\begin{theorem}\label{teo. T closed under products}
Let $\te=(\T,\F)$ be a torsion pair in $R$-Mod. The following assertions are equivalent:
\begin{enumerate}
\item[1)] $\T$ is closed under taking products in $R$-Mod and the heart $\Ht$ is a module category;

\item[2)] $\T=\Gen(V)$, where $V$ is a module which is classical 1-tilting over $R/\ann_{R}(V)$ and admits a finitely generated projective presentation $\xymatrix{Q\ar[r]^{d} & P \ar[r] & V \ar[r] & 0}$ in $R$-Mod and a submodule $X\subseteq \Ker(d)$ such that :
\begin{enumerate}
\item[a)] $\frac{\Ker(d)}{X} \in \F$ and $\Ker(d)\subseteq X + \mathbf{a}Q$, where $\mathbf{a}:=\ann_{R}(V)$;

\item[b)] $\Ext^{1}_{R}(Q/X,?)$ vanishes on $\F$;

\item[c)] There is a $R$-homomorphism $h:(\frac{Q}{X})^{(I)} \flecha R/t(R)$, for some set $I$, such that $h((\frac{\Ker(d)}{X})^{(I)})=\frac{\mathbf{a}+t(R)}{t(R)}$.
\end{enumerate} 
\end{enumerate}
In this case $\te^{'}=(\Gen(V),\F \cap \frac{R}{\mathbf{a}}\text{-Mod})$ is a classical tilting torsion pair in $R/\mathbf{a}-$Mod and the forgetful functor $\mathcal{H}_{\mathbf{t}^{'}} \flecha \Ht$ is faithful. 
\end{theorem}
\begin{proof}
2) $\Longrightarrow$ 1) The classical 1-tilting condition of $V$ over $R/\mathbf{a}$ implies that $\T$ consists of $R/\mathbf{a}$-modules $T$ such that $\Ext^{1}_{R/\mathbf{a}}(V,T)=0$. This class is clearly closed under taking products. We next consider the complex
$$\xymatrix{G:=\cdots \ar[r] & 0 \ar[r] & 0 \ar[r] & X \hspace{0.03cm}\ar@{^(->}[r] & Q \ar[r]^{d} & P \ar[r] & 0 \ar[r] & \cdots}$$
concentrated in degrees -2,-1,0. By condition 2.a) and by the equality $\T=\Gen(V)$, we have that $G\in \Ht$. We shall check that $G$ satisfies all properties 1-5 of theorem \ref{teo. Ht is a module category}. We immediately derive properties 2, 3 and 4 of the mentioned theorem since $\mathbf{a}M\subseteq \Rej_\T(M)$, for all $R$-module $M$. On the other hand, our condition 2.c) implies that: 
$$C:=\Coker(h_{|(H^{-1}(G))^{(I)}})=\frac{\frac{R}{t(R)}}{\frac{\mathbf{a}+t(R)}{t(R)}}\cong \frac{R}{\mathbf{a}+t(R)}$$ 
hence, C is in $R/\mathbf{a}-$Mod. But we have that $R/\mathbf{a}$-Mod$=\overline{\Gen}(V)$ since $R/\mathbf{a}\in \overline{\Gen}(V)$ due to the 1-tilting condition of $V$ over $R/\mathbf{a}$. Then property 5 of theorem \ref{teo. Ht is a module category} holds. It remains to prove that $\T\subseteq \Ker(\Ext^{1}_{R}(V,?))$. Let
$$\xymatrix{0 \ar[r] & T \ar[r] & M \ar[r] & V \ar[r] & 0}$$
be any exact sequence in $R$-Mod, with $T \in \T$. Since $\T=\Gen(V)$ is closed under taking extensions in $R$-Mod, we get that $M\in \T$ and, hence, the exact sequence lives in $R/\mathbf{a}$-Mod. But then it splits since we have that $\T=\Ker(\Ext^{1}_{R/\mathbf{a}}(V,?))$.\\

1) $\Longrightarrow$ 2) Let $\xymatrix{G:=\cdots \ar[r] & 0 \ar[r] & X \hspace{0.03cm}\ar@{^(->}[r] & Q \ar[r]^{d} & P \ar[r] & 0 \ar[r] & \cdots }$ be a complex satisfying properties 1-5 of theorem \ref{teo. Ht is a module category}, which is then a progenerator of $\Ht$, and let us put $V:=H^{0}(G)$ and $\mathbf{a}=\ann_{R}(V)$. There is an obvious monomorphism $R/\mathbf{a} \flecha V^{V}$ and, by hypothesis, we have that $V^{V}\in \T=\Gen(V)$. Hence, we get that $\overline{\Gen}(V)=\frac{R}{\mathbf{a}}$-Mod. Moreover, from lemma \ref{lem. progenerator of Ht} and \cite[Proposition 3.2]{CMT} we get that $V$ is a classical 1-tilting $R/\mathbf{a}$-module. On the other hand, the equality $\overline{\Gen}(V)=\frac{R}{\mathbf{a}}$-Mod implies that $\Rej_{\T}(M)=\Rej_{\frac{R}{\mathbf{a}}\text{-Mod}}(M)=\mathbf{a}M$. Then conditions 2.a) and 2.b) follow directly. \\

Now, we consider the morphism $h:\frac{Q}{X}^{(I)} \flecha \frac{R}{t(R)}$ in property 5 of theorem \ref{teo. Ht is a module category} and put $h^{'}$ the restriction of $h$ to $(\frac{\Ker(d)}{X})^{(I)}$. The fact that $\Coker(h^{'})$ is in $\overline{Gen}(V)=\frac{R}{\mathbf{a}}-$Mod is equivalent to saying that $\mathbf{a}\frac{R}{t(R)} \subseteq \Imagen(h^{'})$. But, by the already proved condition 2.a), we know that $\Imagen(h^{'})\subseteq h(\mathbf{a}(\frac{Q}{X})^{(I)})=\mathbf{a}\Imagen(h) \subseteq \mathbf{a}\frac{R}{t(R)}$. We then get that $\Imagen(h^{'})=\mathbf{a}\frac{R}{t(R)}=\frac{\mathbf{a}+t(R)}{t(R)}$. \\

Finally, it is clear that $\mathbf{t}^{'}=(\T,\F \cap \frac{R}{\mathbf{a}}-\text{Mod})$ is a classical tilting torsion pair of $R/\mathbf{a}-$Mod. If $j:\mathcal{H}_{\mathbf{t}^{'}} \flecha \Ht$ is the forgetful functor then, arguing as in the final part of the proof of corollary \ref{cor. here. torsion pair with V+Y as proge.}, in order to prove that $j$ is faithful, we just need to check that the canonical map $\Hom_{\mathcal{H}_{\mathbf{t}^{'}}}(V[0],M) \flecha \Hom_{\Ht}(V[0],M)$ is injective, for all $M\in \mathcal{H}_{\mathbf{t}^{'}}$. Similar as there, this in turn reduces to check that the canonical map
\begin{center}
$ \Ext^{1}_{R/\mathbf{a}}(V,F)\cong \Hom_{\mathcal{H}_{\mathbf{t}^{'}}}(V[0],F[1]) \flecha \Hom_{\Ht}(V[0],F[1])\cong \Ext^{1}_{R}(V,F)$
\end{center}
is injective, for all $F\in \F \cap \frac{R}{\mathbf{a}}-$Mod. But this is clear.
\end{proof}

\begin{remark}\rm{
Let $A$ be a ring and $V$ be a non-projective classical 1-tilting $A$-module. Then, by \cite{Mi}, we know that $V$ is also a classical tilting right $S$-module, where $S=\End($$_{A}V)^{op}$, such that the canonical algebra morphism $A \flecha \End(V_{S})$ is an isomorphism. Due to the tilting theorem (see the beginning of section \ref{remark tilting theorem}), we then know that $(\Ker(? \otimes_{A} V), \Ker(\text{Tor}_{1}^{\hspace{0.04cm}A}(?,V)))$ is a torsion pair in $\Mod A$. If we had $\Ker(? \otimes_{A} V)=0$ we would have that $\Ker(\text{Tor}_{1}^{\hspace{0.04cm}A}(?,V))=\Mod A$, i.e., $\text{Tor}_{1}^{\hspace{0.04cm}A}(?,V)=0$, and hence $V$ would be a flat left $A$-module. This is a contradiction since finitely presented flat $A$-modules are projective (see \cite[Corollarie 1.3]{L}). We then have a right $A$-module $X \neq 0 $ such that $X \otimes _{A}V=0 \neq \text{Tor}^{\hspace{0.04cm}A}_{1}(X,V).$ Considering an epimorphism $X \epic X^{'}$, with $X^{'}_{A}$ simple, and replacing $X$ by $X^{'}$, we can even choose $X$ to be a simple right $A$-module.}
\end{remark}

\vspace{0.3 cm}

Recall that if $A$ is a ring and $M$ is an $A$-bimodule, then the \emph{trivial extension}\index{trivial extension} of $A$ by M, denoted $A \rtimes M$, is the ring whose underlying $A$-bimodule is $A \oplus M$ and the multiplication is given by $(a,m).(a^{'},m^{'})=(aa^{'},am^{'}+ma^{'})$. The next result is a systematic way of constructing torsion pairs $\te$ in $R$-Mod, such that $\Ht$ has a progenerator of the form $V[0]$, although $V$ is not a tilting $R$-module.

\begin{theorem}\label{teo. V[0] prig. which is not tilting}
Let $A$ be a finite dimensional algebra over an algebraically closed field $K$, let $V$ be a classical 1-tilting left $A$-module such that $\Hom_{A}(V,A)=0$, let $X$ be a simple right $A$-module such that $X \otimes_{A}V=0$ and let us consider the trivial extension $R=A \rtimes M$, where $M=V \otimes_{K}X$. Viewing $V$ as a left $R$-module annihilated by $0 \rtimes M$, the pair $\mathbf{t}=(\Gen(V),\Ker(\Hom_{R}(V,?)))$ is a non-tilting torsion pair in $R$-Mod such that $V[0]$ is a progenerator of $\Ht$.
\end{theorem}
\begin{proof}
All throughout the proof, for any two-sided ideal $\mathbf{a}$ of a ring $R$, we view $R/\mathbf{a}$-modules as $R$-modules annihilated by $\mathbf{a}$. Note that if $M$ is any such module and we apply $? \otimes_{R} M :\Mod R \flecha $Ab to the canonical exact sequence:
$$0 \flecha \mathbf{a} \monic R \epic R/\mathbf{a} \flecha 0$$
we obtain the following exact sequence 
\begin{small}
$$\xymatrix{\cdots \ar[r] & \text{Tor}_{1}^{\hspace{0.04 cm} R}(R,M)=0 \ar[r] & \text{Tor}_{1}^{R}(\frac{R}{\mathbf{a}},M ) \ar[r] & \mathbf{a}\otimes_{R}M \ar[r]^{\delta} & R\otimes_{R}M \ar[r] & \frac{R}{\mathbf{a}} \otimes_{R} M \ar[r] & 0}$$
\end{small}
where $\Imagen(\delta)\cong \mathbf{a}M=0$. Hence we get two canonical isomorphisms $R \otimes_{R} M \iso \frac{R}{\mathbf{a}} \otimes_{R} M$ and $\text{Tor}_{1}^{\hspace{0.04cm}R}(\frac{R}{\mathbf{a}},M) \iso \mathbf{a} \otimes_{R} M$. \\

If $\xymatrix{0 \ar[r] & Q^{'} \ar[r]^{u} & P^{'} \ar[r] & V \ar[r] & 0}$ is the minimal projective resolution of $_{A}V$, then the map $1_X \otimes u: X \otimes_{A} Q^{'} \flecha X \otimes_{A}P^{'}$ is the zero map. Indeed the minimality of the resolution implies that the image of the map $1_X \otimes u$ is contained in the image of the map  $1_X \otimes j: X \otimes_A J(A)P^{'} \flecha X \otimes_A P^{'}$, where $J(A)$ is the Jacobson radical of $A$. But we have $X \otimes_A J(A)P^{'}=XJ(A) \otimes_A P^{'}=0$ since $X_A$ is simple. We then get an isomorphism of vector spaces $X\otimes_{A} P^{'} \cong X \otimes_{A} V=0$ and $\text{Tor}_{1}^{\hspace{0.04cm} R}(X,V) \cong X \otimes_{A} Q^{'}$. \\

On the other hand, note that we have an isomorphism $_{A}M \cong _{A}V$ in $A$-Mod because, due to the algebraically closed condition of $K$, the simple right $A$-module $X$ is one-dimensional over $K$. As a right $A$-module, $M_A$ is in $\add(X_{A})$. Moreover, we have \linebreak $M\otimes_{A} V \cong (V \otimes_{K} X) \otimes_{A}V \cong V \otimes_{K} (X \otimes_{A} V)=V \otimes_{K}0=0$. We will prove the following facts:

\begin{enumerate}
\item[i)] $\ann_{R}(V)=0 \rtimes M:=\mathbf{a}$ and this ideal is the trace of $V$ in $R$;

\item[ii)] $\mathbf{a} \otimes_{R} V=0;$

\item[iii)] $\T:=\Gen(V)$ is closed under taking extensions and, hence, it is a torsion class, in $R$-Mod;

\item[iv)] $\Ker(\Hom_{A}(V,?))=\Ker(\Hom_{R}(V,?))$ and, hence, $\te:=(\Gen(V), \Ker(\Hom_{R}(V,?)))$ is a torsion pair in $R$-Mod which lives in $A$-Mod;

\item[v)] There is a finitely generated projective presentation of $_{R}V$

$$\xymatrix{Q \ar[r]^{d} & P \ar@{>>}[r] & V \ar[r] & 0}$$
such that $\Ker(d)$ is a non-projective $R$-module in $\T$. 
\end{enumerate}

Suppose that all these facts have been proved. Then $\te$ is a torsion pair in $R$-Mod whose torsion class is closed under taking products. We claim that $V$ satisfies all conditions of assertion 2 in theorem \ref{teo. T closed under products}, by taking $X=\Ker(d)$. The only nontrivial things to check are conditions 2.b) and 2.c) in that assertion. We consider the exact sequence:

$$0 \flecha \Ker(d) \monic Q \xymatrix{ \ar[r]^{d} &} \Imagen(d) \flecha 0$$
Applying the functor $\Hom_{R}(?,F)$ to the previous exact sequence, for each $F\in \F$, we obtain the following exact sequence (recall that $\Ker(d)\in \T$):
$$\xymatrix{\cdots \ar[r] & \Hom_{R}(\Ker(d),F)=0 \ar[r] & \Ext^{1}_{R}(\Imagen(d),F) \ar[r] & \Ext^{1}_{R}(Q,F)=0 \ar[r] & \cdots}$$
Hence, condition 2.b) holds. On the other hand, $t(R)$ is the trace of $V$ in $R$ and, by fact i), we know that $t(R)=\mathbf{a}$. Then condition 2.c) of assertion 2 in theorem \ref{teo. T closed under products} also holds, simply by taking as $h$ the zero map. Looking at the proof of implication 2) $\Longrightarrow$ 1) in that theorem, we see that the complex

$$\xymatrix{G:=\cdots \ar[r] & 0 \ar[r] & \Ker(d) \ar@{^(->}[r] & Q \ar[r]^{d} & P \ar[r] & 0 \ar[r] & \cdots}$$
concentrated in degrees -2, -1, 0, is a progenerator of $\Ht$. But we have an isomorphism $G\cong V[0]$ in $\Ht$. Therefore $V[0]$ is a progenerator of $\Ht$. Note that this torsion pair $\te$ in $R$-Mod is not tilting, because the projective dimension of $_{R}V$ is greater than 1 and all 1-tilting $R$-modules have projective dimension less than or equal to 1 (see proposition \ref{description tilting object}). We now move to prove the facts i)-v) in the list above:\\

i) By definition of the $R$-module structure on $V$, we have $(a,m)v=av$, for each $(a,m)\in A \rtimes M$ and each $v \in V$. Then $\ann_{R}(V)=\ann_{A}(V) \rtimes M$. But $\ann_{A}(V)=0$ due to the 1-tilting condition of $_{A}V$. On the other hand, if $f:V \flecha R$ is a morphism in $R$-Mod, then we have two $K$-linear maps $g:V \flecha A$ and $h:V\flecha M$ such that $f(v)=(g(v),h(v)) \in A \rtimes M=R$, for all $v \in V$. Direct computation shows that $g$ is a morphism in $A$-Mod, and hence $g=0$ (since $\Hom_{A}(V,A)=0$). Then one immediately sees that $h\in \Hom_{A}(V,M)$ and, since $_{A}M$ is in $\Gen(_{A}V)$, we conclude that $t(R)=0 \rtimes M= \mathbf{a}$. \\

ii) Since $\mathbf{a}^{2}=0$ and $\mathbf{a}V=0$, we have an equality $\mathbf{a} \otimes_{R} V=\mathbf{a}\otimes_{\frac{R}{\mathbf{a}}} V = \mathbf{a} \otimes_{A} V$. But, as a right $A$-module, we have that $\mathbf{a}_{A} \cong M_{A}$. It follows that $\mathbf{a} \otimes_{A} V \cong M \otimes_{A}V=0$.\\

iii) Consider any exact sequence $0 \flecha T \flecha N \xymatrix{\ar[r]^{q} &} T^{'} \flecha 0$ $(\ast)$ in $R$-Mod, with $T,T^{'}\in \Gen(V)=\T$. Taking an epimorphism $p:V^{(I)} \epic T^{'}$, and considering the pullback of $p$ along the $q$, we easily reduce the problem to the case when $T^{'}=V^{(I)}$. We now apply the functor $\frac{R}{\mathbf{a}}\otimes_{R}?:R$-Mod $\flecha \frac{R}{\mathbf{a}}-$Mod to the sequence $(\ast)$, with $T^{'}=V^{(I)}$, and use the fact that, by the initial paragraph of this proof, we have $\text{Tor}_{1}^{\hspace{0.04cm}R}(R/\mathbf{a},V^{(I)}) \cong \mathbf{a} \otimes_{R}V^{(I)}=0$. We then get a commutative diagram with exact rows, where the vertical arrows are the canonical maps:

$$\xymatrix{0 \ar[r] & T \ar[r] \ar[d]^(0.4){\wr}& N \ar[r] \ar[d] & V^{(I)} \ar[r] \ar[d]^(0.4){\wr}& 0 \\ 
0 \ar[r] & \frac{R}{\mathbf{a}} \otimes_{R} T \ar[r] & \frac{R}{\mathbf{a}}\otimes_{R} N \ar[r] & \frac{R}{\mathbf{a}} \otimes_{R} V^{(I)} \ar[r] & 0 }$$   
The central vertical arrow $N \flecha \frac{R}{\mathbf{a}} \otimes_{R} N \cong \frac{N}{\mathbf{a}N}$ is then an isomorphism. This implies that $\mathbf{a}N=0$ and, hence, the sequence $(\ast)$ above lives in $A$-Mod and $N\in \T$. \\

iv) If $F\in \Ker(\Hom_{R}(V,?))$ then $t(R)F=tr_V(R)F=0$. By fact i), we get that $\mathbf{a}F=0$. Then $F$ is an $A$-module, and hence $\Ker(\Hom_{R}(V,?)) \subseteq \Ker(\Hom_{A}(V,?))$. The converse inclusion is obvious. \\

v) The multiplication map $\mu: R\otimes_{A} V \flecha V$ is surjective. Moreover, we have an isomorphism of $K$-vector spaces:
\begin{center}
$R \otimes_{R}V \cong (A \oplus M) \otimes_{A} V \cong V \oplus(M \otimes_{A} V)=V \oplus 0 \cong V$.
\end{center}
Since $V$ is a finite dimensional $K$-vector space we get that $\mu$ is an isomorphism of left $R$-modules. \\

Let $\xymatrix{0 \ar[r] & Q^{'} \ar[r]^{d^{'}} & P^{'} \ar[r] & V \ar[r] & 0}$ be a finitely generated projective presentation of $V$ in $A$-Mod. Applying the functor $R \hspace{0.02  cm} \otimes_{A} ?$ to the previous exact sequence, we obtain the following exact sequence:

$$\xymatrix{\cdots \ar[r] & \text{Tor}_{1}^{\hspace{0.04 cm} A}(R,P^{'})=0 \ar[r] & \text{Tor}_{1}^{\hspace{0.04 cm}A}(R,V) \ar[r] & R \otimes_{A} Q^{'} \ar[rr]^{1\otimes d^{'}} && R \otimes_{A} P^{'} \ar[r] & V \ar[r] & 0}$$


Then we have isomorphisms of $K$-vector spaces:

\begin{center}
$\Ker(1 \otimes d^{'})=\text{Tor}_{1}^{\hspace{0.04 cm} A}(R,V)\cong \text{Tor}_{1}^{\hspace{0.04 cm} A}(A \oplus M,V)\cong \text{Tor}_{1}^{\hspace{0.04 cm} A}(M,V)$.
\end{center}

It is easy to deduce from this fact that $\mathbf{a}\Ker(1 \otimes d^{'})=0$. Then we can view $\Ker(1 \otimes d^{'})$ as a left $A$-module isomorphic to $\text{Tor}_{1}^{\hspace{0.04 cm} A}(M,V)$. Since $M=V\otimes_{K}X$ we get that $\text{Tor}_{1}^{\hspace{0.04 cm} A}(M,V)\cong V \otimes_{K} \text{Tor}_{1}^{\hspace{0.04 cm} A}(X,V)$ which is nonzero and isomorphic to $V \otimes_{K}(X \otimes_{A} Q^{'}) \in \add(_{A}V)$ due to the first paragraph of this proof. It follows that $\Ker(1\otimes d^{'})$ is a nonzero left $R$-module in $\add(_{R}V)=\add(_{A}V)$. The following is a finitely generated projective presentation of $V$ in $R$-Mod:

$$\xymatrix{ R \otimes_{A} Q^{'} \ar[rr]^{1\otimes d^{'}} && R \otimes_{A} P^{'} \ar[r] & V \ar[r] & 0}$$

We finally prove that $W:=\Ker(1 \otimes d^{'})$ is not a projective in $R$-Mod. If it were so, we would have that $W=tr_{V}(R)W$. By fact i), we would get that $\mathbf{a}W=W$, which would imply that $W=0$ since $\mathbf{a}^{2}=0$.
\end{proof}

\section{Some examples}
We now give a few examples which illustrate the results obtained in this chapter. All of them refer to finite dimensional algebras over a field which are given by quivers and relations. We refer the reader to \cite[Chapter II]{ASS} for the terminology that we use.

\begin{example}
Let $K$ be a field and $R$ be the $K$-algebra given by the following
quiver and relations:

$$\xymatrix{ 1 \ar@<1ex>[r]^{\alpha} \ar@<-0.5ex>[r]_{\beta}
& 2  \ar@<1ex>[r]^{\gamma} \ar@<-0.5ex>[r]_{\delta} & 3 } \hspace{2
cm} \alpha \delta\text{, } \beta\gamma \text{ and
}\alpha\gamma \text{-}\beta\delta$$

Let $\mathbf{a}$ be an idempotent ideal of $R$, let
$(\mathcal{C},\mathcal{T},\mathcal{F})$ be the associated TTF triple
in $R-\text{Mod}$ and let $\mathbf{t}=(\mathcal{T},\mathcal{F})$ be
its right constituent torsion pair. The following facts are true:

\begin{enumerate}
\item[1)] If $\mathbf{a}=Re_1R$ then $\mathcal{H}_\mathbf{t}$ has a progenerator which is a sum of stalk complexes and
 $\mathcal{H}_\mathbf{t}\cong K\Gamma-\text{Mod}$, where $\Gamma$ is the quiver $\xymatrix{2
\ar@<1ex>[r] \ar@<-1ex>[r]& 3 \ar[r] & 1}$.  The algebra $R$ is then
tilted of type $\Gamma$.
\item[2)] If $\mathbf{a}=Re_2R$ then $\mathcal{H}_\mathbf{t}$ is
equivalent to $K\times K\times K-\text{Mod}$.
\item[3)] If $\mathbf{a}=R(e_1+e_2)R$ then $\mathcal{H}_\mathbf{t}$ does not have a sum of stalk complexes as a progenerator and
$\mathcal{H}_\mathbf{t}\cong S-\text{Mod}$, where $S$ is the algebra
given by the following quiver and relations

$$\xymatrix{ 2 \ar@<0ex>[rr]^{\mu_2} \ar@<2.2ex>[rr]^{\mu_1} \ar@<-2.2ex>[rr]^{\mu_3}  &&  3  \ar@<1ex>[rr]^{\pi_1} \ar@<-1ex>[rr]_{\pi_2} && 1}$$
\begin{center}

$\mu_1\pi_2=\mu_3\pi_1=0$;

$\mu_1\pi_1=-\mu_2\pi_2$

$\mu_2\pi_1=\mu_3\pi_2$.
\end{center}
The algebras $R$ and $S$ are derived equivalent and, hence, $S$ is
piecewise hereditary (i.e. derived equivalent to a hereditary algebra).
\end{enumerate}
\end{example}
\begin{proof}
Using a classical visualization of modules via diagrams (see, e.g.,
\cite{F}), the indecomposable projective left $R$-modules can be
depicted as:

$$\xymatrix{1 & & & 2 \ar@{-}[ld] \ar@{-}[rd] &  & & 3 \ar@{-}[dr] \ar@{-}[dl] & \\ && _{1}\alpha & & _{1}\beta &  _{2}\gamma \ar@{-}[dr] && _{2}\delta \ar@{-}[dl] &&&\\ &&&&&& _{1}(\alpha \gamma=\beta \delta)&&}$$

By corollary \ref{cor. faith. here. ---> module}, whenever
$\mathbf{t}$ is faithful, we know that if $\mathcal{H}_\mathbf{t}$ is equivalent
to $S-\text{Mod}$, then $S$ and $R$ are derived equivalent. Even
more, in that case $R[1]$ is a tilting object of
$\mathcal{H}_\mathbf{t}$, which implies that $R$ and $S$ are
tilting-equivalent.\\

1) In this case we have  $\mathbf{a}=\text{Soc}(_RR)\cong
S_1^{(4)}\cong Re_1^{(4)}$, so that $\mathbf{a}$ is projective in
$R-\text{Mod}$ and $\mathbf{t}$ is a faithful torsion pair. From corollary \ref{cor. a projective} and example \ref{exam. R/a+a as progenerator}, we obtain that $\frac{R}{\mathbf{a}}\oplus\mathbf{a}[1]$ is a progenerator of
$\mathcal{H}_\mathbf{t}$. It follows that
$G:=\frac{R}{\mathbf{a}}\oplus S_1[1]$ is also a progenerator of
$\mathcal{H}_\mathbf{t}$. \\

We put $A:=\frac{R}{\mathbf{a}}$, which is isomorphic to the
Kronecker algebra: $\xymatrix{2  \ar@<1ex>[r] \ar@<-0.5ex>[r] & 3 }$
and which we also view as a left $R$ module annihilated by
$\mathbf{a}$. Then $\mathcal{H}_\mathbf{t}$ is equivalent to
$S-\text{Mod}$, where
$S=\text{End}_{\mathcal{H}_\mathbf{t}}(G)^{op}\cong\begin{pmatrix}
\text{End}_A(S_1)^{op} & 0\\ \text{Ext}_R^1(A,S_1) & A
\end{pmatrix}\cong\begin{pmatrix}
K & 0\\ \text{Ext}_R^1(A,S_1) & A
\end{pmatrix}$.\\

Note that $A=S_2\oplus\frac{Re_3}{R\alpha\gamma}$ as left
$R$-module, so that we have a vector space decomposition
$\text{Ext}_R^1(A,S_1)\cong\text{Ext}_R^1(S_2,S_1)\oplus\text{Ext}_R^1(\frac{Re_3}{R\alpha\gamma},S_1)$.
Let $\epsilon'$ be the element of $\text{Ext}_R^1(A,S_1)$
represented by the short exact sequence

\begin{center}
 $\xymatrix{0 \ar[r] & S_1 \hspace{0.08cm}\ar@{^(->}[r] &
Re_3\ar[r]&\frac{Re_3}{R\alpha\gamma}\ar[r]& 0}$
\end{center}
The assignment $a\rightsquigarrow a\epsilon'$ gives an isomorphism
of left $A$-modules
$\xymatrix{Ae_3\ar[r]^{\sim\hspace{0.7cm}}&\text{Ext}_R^1(A,S_1)}$,
so
that the algebra $S$ is isomorphic to $\begin{pmatrix} K & 0\\
Ae_3 & A
\end{pmatrix}\cong K\Gamma$, where $\Gamma$ is
the quiver $\xymatrix{2 \ar@<1ex>[r] \ar@<-1ex>[r]& 3 \ar[r] & 1}$.\\

2) We have that $\mathbf{a}=Re_2R=Re_2\oplus Je_3$ is not projective
 and that $R/Re_2R\cong S_1\oplus S_3$ in $R-\text{Mod}$. We then
 get that $\mathcal{F}=\{F\in R-\text{Mod}:$
 $\text{Soc}(F)\in\text{Add}(S_2)\}$, so that $\mathbf{t}$ is not
 faithful. The minimal projective resolution of $Je_3$ is of the form $\xymatrix{0\ar[r]& P_1^{(3)}\ar[r] & P_2^{(2)}\ar[r]&Je_3\ar[r]&
 0}$, where $P_1=S_1$ is simple projective. It follows that

\begin{center}
 $\text{Ext}_R^2(\frac{R}{\mathbf{a}},F)\cong\text{Ext}_R^2(S_1\oplus
 S_3,F)\cong\text{Ext}_R^2(S_3,F)\cong\text{Ext}_R^1(Je_3,F)=0$,
 \end{center}
 for each $F\in\mathcal{F}$. By example \ref{exam. R/a+a as progenerator},
 we know that
 $\frac{R}{\mathbf{a}}[0]\oplus\frac{\mathbf{a}}{t(\mathbf{a})}[1]\cong (S_1\oplus S_3)[0]\oplus S_1^{(3)}[1]$
 is a progenerator of $\mathcal{H}_\mathbf{t}$, which implies that
 $G:=(S_1\oplus S_3)[0]\oplus S_1[1]$ is also a progenerator of
 $\mathcal{H}_\mathbf{t}$. Since we have
 $\text{Ext}_R^1(S_1,S_1)=0=\text{Ext}_R^1(S_3,S_1)$, we get from
 proposition \ref{prop. progenerator V[0]+Y[1]} that
 $\mathcal{H}_\mathbf{t}$ is equivalent to $S-\text{Mod}$, where $S=\text{End}_R(S_1)^{op}\times\text{End}_R(S_1\oplus S_3)^{op}\cong K\times K\times
 K$.\\

3) We have isomorphisms $\mathbf{a}=R(e_1+e_2)R\cong Re_1\oplus
Re_2\oplus Je_3$ and $R/\mathbf{a}\cong S_3$ in $R-\text{Mod}$, and
this implies that $\mathcal{F}=\{F\in R-\text{Mod}:$
 $\text{Soc}(F)\in\text{Add}(S_1\oplus S_2)\}$ and that $\mathcal{F}$ is
 faithful. Since $\text{Ext}_R^2(S_3,S_1)\neq
 0$, this time we do not have a sum of stalk complexes as a
 progenerator of $\mathcal{H}_\mathbf{t}$ (see proposition \ref{prop. progenerator V[0]+Y[1]}). Instead, inspired by the proof of corollary \ref{cor. a in F}, we consider the
 minimal projective resolution of $S_3\cong R/\mathbf{a}$

 \begin{center}
 $\xymatrix{0\ar[r]&P_1^{(3)}\ar[r]&P_2^{(2)}\ar[r]^{d'}&P_3\ar[r]&S_3 \ar[r] & 0}$
 \end{center}
 and take the complex $\xymatrix{G':= \cdots \ar[r] & 0 \ar[r] &  P_2^{(2)}\ar[r]^{d'}&P_3\ar[r] &  0 \ar[r] & \cdots}$, concentrated in degrees $-1$ and $0$. Now the complex $G:=G'\oplus
 P_1[1]\oplus P_2[1]$ satisfies all conditions in assertion 2 of corollary \ref{cor. HKM=Prog. complex}, so that it is a progenerator of
 $\mathcal{H}_\mathbf{t}$ and $\mathcal{H}_\mathbf{t}\cong
 S-\text{Mod}$, where $S=\begin{pmatrix} \text{End}_R(P_1\oplus P_2)^{op} & \text{Hom}_{\mathcal{D}(R)}(P_1[1]\oplus P_2[2],G')
 \\ \text{Hom}_{\mathcal{D}(R)}(G',P_1[1]\oplus P_2[2]) &
 \text{End}_{\mathcal{D}(R)}(G')^{op} \end{pmatrix}$.\\

We clearly have $\text{End}_R(P_1\oplus
P_2)^{op}\cong\begin{pmatrix} \text{End}_R(P_2)^{op} & 0\\
\text{Hom}_R(P_1,P_2) & \text{End}_R(P_1)^{op}\end{pmatrix}\cong \begin{pmatrix} K & 0\\
K^2 & K\end{pmatrix}=:A$, which is isomorphic to the Kronecker
algebra. Moreover, the $0$-homology functor defines an isomorphism
$\xymatrix{\text{End}_{\mathcal{D}(R)}(G')\ar[r]^{\sim\hspace{1.6
cm}}&\text{End}_R(H^0(G'))\cong\text{End}_R(S_3)\cong K}$. On the
other hand, $\text{Hom}_{\mathcal{D}(R)}(G',P_1[1]\oplus P_2[1])$ is
a $2$-dimensional vector space,  where a basis $\{\pi_1,\pi_2\}$ is
induced by the two projections $\xymatrix{P_2^{(2)}\ar@{>>}[r]&
P_2}$. Similarly, $\text{Hom}_{\mathcal{D}(R)}(P_1[1]\oplus
P_2[2],G')$ is a 3-dimensonal vector space with a basis $\{\mu_i:$
$i=1,2,3\}$ induced by the monomorphisms $\xymatrix{P_1=Re_1
\hspace{0.07cm}\ar@{^(->}[r]&P_2^{(2)}}$ which map $e_1$ onto
$(\beta ,0)$, $(\alpha ,-\beta )$ and $(0,\alpha )$, respectively.
Since the multiplication in   $S$ is given by anti-composition of
the entries, we easily get that $\pi_i\mu_j=\mu_j\circ\pi_i=0$, for
all $i,j$. On the other hand, we have:

\begin{center}

 $\mu_3\pi_1=\pi_1\circ\mu_3=0=\pi_2\circ\mu_1=\mu_1\pi_2$

$\mu_1\pi_1=\pi_1\circ\mu_1=-\pi_2\circ\mu_2=-\mu_2\pi_2=\begin{pmatrix}
0 & 0\\ \rho_\beta & 0
\end{pmatrix}:\xymatrix{P_1\oplus P_2\ar[r]&P_1\oplus P_2}$

$\mu_2\pi_1=\pi_1\circ\mu_2=\pi_3\circ\mu_2=\mu_2\pi_3=\begin{pmatrix}
0 & 0\\ \rho_\alpha & 0
\end{pmatrix}:\xymatrix{P_1\oplus P_2\ar[r]&P_1\oplus P_2}$
\end{center}
where $\xymatrix{\rho_x:P_1=Re_1\ar[r] & P_2}$ maps
$a\rightsquigarrow ax$, for each $x\in P_2$. It easily follows that
$S$ is given by quivers and relations as claimed in the statement.
\end{proof}

\vspace{0.3 cm}

Given a finite quiver $Q$ with no oriented cycle, a path $p$ will be
called a \emph{maximal path} when its origin is a source and its
terminus is a sink. We put
$D=\text{Hom}_K(?,KQ)=KQ-\text{mod}^{op}\stackrel{\cong}{\longleftrightarrow}\text{mod}-KQ$
to denote the usual duality between finitely generated left and
right $KQ$-modules.

\begin{example} \label{exem.progenerator V not tilting}
Let $Q$ be a finite connected quiver with no oriented cycles which
is different from $1\rightarrow \dots \rightarrow n$, and let $i\in
Q_0$ be a source. Let us form a new quiver $\hat{Q}$ as follows. We
put $\hat{Q}_0=Q_0$ and the arrows of $\hat{Q}$ are the arrows of
$Q$ plus an arrow $\alpha_p:t(p)\rightarrow i$, for each maximal
path $p$ in $Q$. Given a field $K$, we consider the $K$-algebra $R$
with quiver $\hat{Q}$ and relations:

\begin{enumerate}
\item $\alpha_p\beta =0$, for each  $\beta\in Q_1$ and each maximal path $p$ in $Q$;

\item $p'\alpha_p=q'\alpha_q$, whenever $p'$ and $q'$  are paths in $Q$ such
that $s(p')=s(q')$ and there is a path $\pi:j\rightarrow \dots
\rightarrow s(p')=s(q')$ in $Q$ such that $\pi p'=p$ and $\pi q'=q$.
\end{enumerate}
We identify $KQ-\text{Mod}$ with the full subcategory of
$R-\text{Mod}$ consisting of the $R$-modules annihilated by the
two-sided ideal generated by the $\alpha_p$. Then
$\mathbf{t}=(KQ-\text{Inj},(KQ-\text{Inj})^\perp)$ is a non-tilting
torsion pair in $R-\text{Mod}$ such that $D(KQ)[0]$ is a
progenerator of $\mathcal{H}_\mathbf{t}$.
\end{example}
\begin{proof}
Let $\mathcal{M}$ be the set of maximal path in $Q$ and consider the
paths of $Q$ as the canonical basis $B$ of $KQ$. Its dual basis is
denoted by $B^*$. Consider the assignment
$\sum_{p\in\mathcal{M}}a_p\alpha_p\rightsquigarrow
(\sum_{p\in\mathcal{M}}a_pp^*)\otimes\bar{e}_i$, where $a_p\in
KQe_{t(p)}$ for each $p\in\mathcal{M}$ and
$\bar{e}_i=e_i+e_iJ\in\frac{e_iKQ}{e_iJ}$ is the canonical element.
Here $J=J(KQ)$ is the Jacobson radical,  a basis of which is given
by the paths of length $>0$. This assignment defines an isomorphism
of $KQ$-bimodules

\begin{center}
$\mathbf{a}:=\sum_{p\in\mathcal{M}}R\alpha_pR=\sum_{p\in\mathcal{M}}R\alpha_p\stackrel{\cong}{\longrightarrow}D(KQ)\otimes_K\frac{e_iKQ}{e_iJ}=:M$.
\end{center}
Moreover, it is the restriction of an algebra isomorphism
$R\stackrel{\cong}{\longrightarrow}KQ\rtimes M$ which maps
$e_i\rightsquigarrow (e_i,0)$, $\beta\rightsquigarrow (\beta ,0)$
and   $\alpha_p\rightsquigarrow (0,p^*\otimes\bar{e}_i)$, for all
$i\in Q_0$, all $\beta\in Q_1$ and all $p\in\mathcal{M}$.\\

Our hypotheses on $Q$ guarantee that there is no
projective-injective $KQ$-module. Then $D(KQ)$ is a classical
1-tilting $KQ$-module whose associated torsion pair in
$KQ-\text{Mod}$, namely
$\mathbf{t}=(KQ-\text{Inj},(KQ-\text{Inj})^\perp)$,  is faithful. On
the other hand, the fact that $i$ is a source and $Q$ is connected
implies that $\frac{e_iKQ}{e_iJ}\otimes_{KQ}D(KQ)=0$. Then theorem \ref{teo. V[0] prig. which is not tilting} applies, with
$A=KQ$, $V=D(KQ)$ and $X=\frac{e_iKQ}{e_iJ}$.
\end{proof}

\chapter[Compactly generated t-structures]{Compactly generated t-structures in the derived category of a commutative Noetherian ring}
\section{Properties of $\Hp$}
All throughout this section $R$ is a commutative Noetherian ring and $\phi$ is a sp-filtration of $\Spec(R)$. Furthermore, for each integer number $k$, we will denote by $(\T_{k},\F_{k})$ the hereditary torsion pair $(\T_{\phi(k)},\F_{\phi(k)})$. We will keep the terminology and notation introduced in sections \ref{sec. Commutative algebra} and \ref{sec. commutative Noetherian rings}. In this chapter, we study when $\Hp$ is a Grothendieck category or a module category. We start by studying the stalk complexes in the heart.


\subsection{Stalk complexes in the heart}
Similar to the case of the t-structure of Happel-Reiten-Smal\o, the stalks of $\Hp$ play an important role in the structure of $\Hp$. A first step to verify that $\Hp$ is an AB5 abelian category, would be to guarantee the exactness of direct systems of short exact sequences whose terms are stalks of the heart $\Hp$. In this subsection, we will settle this first obstacle. The following result is standard and is useful for our purposes. We include a short proof for completeness.

\begin{lemma}\label{lem. sutileza}
Let $Z$ be a sp-subset of $\Spec(R)$ and let $\T_{Z}$ be the associate hereditary torsion class in $R\Mode$. For each integer $i\geq 0$ and each $R$-module $M$, the following assertions are equivalent:
\begin{enumerate}
\item[1)] $\Ext^{\hspace{0.03cm}i}_{R}(T,M)=0$, for all $T\in \T_{Z};$
\item[2)] $\Ext^{\hspace{0.03 cm}i}_{R}(R/\mathbf{p},M)=0$, for all $\mathbf{p} \in Z$.
\end{enumerate}
\end{lemma}
\begin{proof}
Every finitely generated module in $\T_Z$ admits a finite filtration whose factors are isomorphic to modules of the form $R/\mathbf{p}$, with $\p\in Z$ (see \cite[Theorem 6.4]{Ma}). In particular, a module $F$ is in $\T_Z^{\perp}$ if, and only if, $\Hom_R(R/\p,F)=0$, for all $\p\in Z$. This proves the case $i=0$.\\

For $i>0$, we denote by $\Phi_{i}(1)$ and $\Phi_{i}(2)$ the classes of modules $M$ which satisfy condition 1 and 2, respectively. Note that if $i>1$ then, by dimension shifting, we have that $M\in \Phi_{i}(k)$ if, and only if, $\Omega^{1-i}(M)\in \Phi_{1}(k)$, for $k=1,2$. Hence, the proof is reduced to the case $i=1$. Then using the finite filtration mentioned in the previous paragraph, given $T \in \T_Z$, we can construct, by transfinite induction, an ascending transfinite chain $(T_{\alpha})_{\alpha < \lambda}$ of submodules of $T$ , for some ordinal $\lambda$, such that:
\begin{enumerate}
\item[i)] $\underset{\alpha < \lambda}{\bigcup} T_{\alpha}=T;$
\item[ii)] $T_{\alpha}=\underset{\beta < \alpha}{\bigcup} T_{\beta}$, whenever $\alpha$ is a limit ordinal smaller than $\lambda$;
\item[iii)] $\frac{T_{\alpha +1}}{T_{\alpha}}$ isomorphic to $R/\mathbf{p}$, for some $\mathbf{p}\in Z$, whenever $\alpha +1 < \lambda$.
\end{enumerate}
The result is then a direct consequence of Eklof's lemma (see \cite[3.1.2]{GT}).
\end{proof}

\begin{notation}\rm{
For each integer $m$, we will denote by $\T\F_{m}$ the class of $R$-modules $Y$ such that $Y[-m]\in \Hp$. Note that $\T\F_m\subseteq \T_m$.}
\end{notation}

\vspace{0.3 cm}

The following result gives us some information about the stalks of the heart $\Hp$.

\begin{proposition}\label{pro. stalks on the heart}
For each $Y\in \T_{m}$, the following assertions are equivalent:
\begin{enumerate}
\item[1)] $Y\in \T\F_{m}$;
\item[2)] $\Ext^{\hspace{0.04 cm}k-1}_{R}(R/\mathbf{p},Y)=0$, for each integer $k\geq 1$ and for each $\mathbf{p}\in \phi(m+k)$;
\item[3)] $\Ext^{\hspace{0.04 cm}k-1}_{R}(T,Y)=0$, for each integer $k \geq 1$ and each $T\in \T_{m+k}$.
\end{enumerate}
\end{proposition}
\begin{proof}
1) $\Longleftrightarrow$ 2) Note that $Y[-m]\in \mathcal{U}_{\phi}$. Bearing in mind that $\phi$ is a decreasing map, we then have the following chain of double implications:\\

$Y\in \T\F_{m} \Longleftrightarrow Y[-m]\in \Hp \Longleftrightarrow Y[-m] \in \mathcal{U}_{\phi}^{\perp}[1] \Longleftrightarrow Y[-m-1]\in \mathcal{U}_{\phi}^{\perp} \Longleftrightarrow  \linebreak \Hom_{\D(R)}(R/\mathbf{p}[-i],Y[-m-1])=0, \text{ for all }i \in \text{ \Z  \ and }\mathbf{p} \in \phi(i) \Longleftrightarrow \Ext_{R}^{\hspace{0.02 cm} i-m-1}(R/\mathbf{p},Y)=0, \text{ for all } i\in \text{ \Z \ and }\mathbf{p}\in \phi(i).$ \\

By the nonexistence of negative extensions between modules, putting $k=i-m$ for $i>m$, we conclude that $Y\in\T\F_{m}$ if, and only if, $\Ext^{\hspace{0.02 cm}k-1}_{R}(R/\mathbf{p},Y)=0$, for all $k>0$ and $\mathbf{p}\in \phi(m+k)$. \\

2) $\Longleftrightarrow$ 3) is a direct consequence of the previous lemma.
\end{proof}

\begin{examples}\label{exam. TF}\rm{
Suppose that the filtration $\phi$ has the property that 
$$\dots \supseteq \phi(0) \supsetneq \phi(1)=\phi(2)=\dots=\emptyset$$
It follows from the previous proposition that $\T\F_{0}=\T_{0}$ and $\T\F_{-1}=\T_{-1}\cap \F_{0}$. }
\end{examples}

\begin{corollary}\label{cor. TF closed cop rod. and lim.}
For each integer $m$, the class $\T\F_{m}$ is closed under taking coproducts, direct limits and kernels of epimorphisms in $R$-Mod.
\end{corollary}
\begin{proof}
Due to the fact that $R$ is a Noetherian ring, for each $\mathbf{p}\in \Spec(R)$, we have that all syzygies of $R/\mathbf{p}$ are finitely presented (=finitely generated) $R$-modules. We then get that the functor $\Ext^{\hspace{0.02 cm}k}_{R}(R/\mathbf{p},?)$ commutes with coproducts and direct limits in $R$-Mod. By proposition \ref{pro. stalks on the heart}, we deduce that $\T\F_{m}$ is closed under taking coproducts and direct limits in $R$-Mod.\\

On the other hand, let now $\xymatrix{0 \ar[r] & X \ar[r] & Y \ar[r] & Z \ar[r] & 0}$ be an exact sequence in $\Mod R$, where $Y$ and $Z$ are in $\T\F_{m}.$ Note that $X\in \T_{m}$, since $(\T_{m},\F_{m})$ is a hereditary torsion pair. Let us fix $k\geq 1$ and $\mathbf{p}\in \phi(m+k)\subseteq\phi(m+k-1)$. If we apply the long exact sequence of $\Ext(R/\mathbf{p},?)$ to the exact sequence above, we get:
$$\xymatrix{0 \ar[r] & \Hom_{R}(R/\mathbf{p},X) \ar[r] & \Hom_{R}(R/\mathbf{p}, Y) \ar[r] & \Hom_{R}(R/\mathbf{p},Z) \ar[r] & \cdots \\
\cdots \ar[r] & \Ext^{k-2}_{R}(R/\mathbf{p},Z) \ar[r] & \Ext^{k-1}_{R}(R/\mathbf{p},X) \ar[r] & \Ext^{k-1}_{R}(R/\mathbf{p},Y) \ar[r] & \cdots}$$
From proposition \ref{pro. stalks on the heart} we then get that $\Ext^{k-2}_{R}(R/\mathbf{p},Z)=0=\Ext^{k-1}_{R}(R/\mathbf{p},Y)$, and hence $\Ext^{k-1}_{R}(R/\mathbf{p},X)=0$. Using again proposition \ref{pro. stalks on the heart}, we obtain that $X \in \T\F_{m}$.
\end{proof}

\begin{proposition}\label{prop. limits stalk}
Let $m$ be any integer, let $(Y_{\lambda})$ be a direct system in $\T\F_{m}$ and \linebreak $g:\coprod Y_{\lambda \mu} \flecha \coprod Y_{\lambda}$ be the colimit-defining morphism. The following assertions hold:
\begin{enumerate}
\item[1)] $\Ker(g)\in \T\F_{m}$;
\item[2)] $\varinjlim_{\Hp}(Y_{\lambda}[-m])=(\varinjlim Y_{\lambda})[-m]$.
\end{enumerate} 
\end{proposition} 
\begin{proof}
Since $\coprod Y_{\lambda \mu}$, $\coprod Y_{\lambda}$ and $\varinjlim{Y_{\lambda}}$ are in $\T\F_{m}$, it follows corollary \ref{cor. TF closed cop rod. and lim.} that $\Imagen{g}$ and $\Ker(g)$ are also in $\T\F_{m}$. We then consider the following diagram

\begin{footnotesize}
$\xymatrix{&&&&\Ker(g)[-m+1]=H^{m-1}(\Cone(g[-m]))[-m+1] \ar[d]\\ \coprod Y_{\lambda \mu}[-m] \ar[rr]^{g[-m]} & & \coprod Y_{\lambda}[-m] \ar[rr] & & \Cone(g[-m]) \ar[d] \ar[r]^{+} & \\ &&&& \varinjlim{Y_{\lambda}}[-m]=H^{m}(\Cone(g[-m]))[-m] \ar[d]^{+} \\ &&&&&}$
\end{footnotesize}
By the explicit construction of cokernels in $\Hp$ (\cite{BBD}), we deduce that $\varinjlim_{\Hp}(Y_{\lambda}[-m])=\Coker_{\Hp}(g[-m])\cong (\varinjlim{Y_{\lambda}})[-m].$
\end{proof} 

\begin{remark}\rm{
Due to the fact that $R$-Mod is an AB5 abelian category, the previous proposition implies that the direct limits of short exact sequences whose terms are stalks of the heart $\Hp$ (at the same degree) is exact.}
\end{remark}

\begin{remark}\label{rem. F closed under direct limits}\rm{
Since the Gabriel topology associated to $(\T_m,\F_m)$ consists of finitely generated ideals, the class $\F_m$ is closed under taking direct limits, for each integer $m$. Of course, this same property holds for $\T_m$. Then the associated torsion radical (resp. coradical) functor $t_{m}:R \text{-Mod} \flecha R\text{-Mod}$ (resp. $(1:t_{m}): R\text{-Mod} \flecha R\text{-Mod}$) preserves direct limits.}
\end{remark} 

\subsection{$\Hp$ has a generator}
In this subsection we will prove that $\Hp$ always has a generator. In the sequel $\phi$ is a sp-filtration of $\Spec(R)$.

\begin{notation}\rm{
For each $k\in \mathbb{Z}$, we will denoted by $\phi_{\leq k}$ the sp-filtration given by:
\begin{enumerate}
\item[1)] $\phi_{\leq k}(j):=\phi(j)$,\hspace{0.4 cm} if \hspace{0.4 cm} $j \leq k$;
\item[2)] $\phi_{\leq k}(j):=\emptyset$, \hspace{0.7 cm} if \hspace{0.4 cm} $j>k$.
\end{enumerate}}
\end{notation}

\begin{lemma}\label{lem. menor igual dos}
Let $\phi: \mathbb{Z} \flecha \mathcal{P}(\Spec(R))$ be a sp-filtration of $\Spec(R)$. \\

If $X \in \mathcal{H}_{\phi}$ and $k$ is a integer, then $\tau^{\leq k}(X)\in \mathcal{H}_{\phi_{ \leq k+j}}$ for each $j \in \{ 0,1,2\}$.
\end{lemma}
\begin{proof}
Let us fix $X\in \mathcal{H}_{\phi}$ and let $k$ be a integer. It is clear that $\tau^{\leq k}(X) \in \mathcal{U}_{\phi_{\leq k+j}}$, for each $j$ integer. On the other hand, we consider the next triangle:

$$\xymatrix{\tau^{>k}(X)[-2] \ar[r] & \tau^{\leq k}(X)[-1] \ar[r] & X[-1] \ar[r]^{\hspace{0.55 cm}+}&}$$
Applying the cohomological functor $\Hom_{\D(R)}(R/\mathbf{p}[-m],?)$ to the previous triangle, we get the following exact sequence:
$$\xymatrix{\cdots \ar[r] & \Hom_{\D(R)}(R/\mathbf{p}[-m],\tau^{>k}(X)[-2])\ar[r] & \Hom_{\D(R)}(R/\mathbf{p}[-m],\tau^{\leq k}(X)[-1])  \ar[d] \\ & \cdots & \Hom_{\D(R)}(R/\mathbf{p}[-m],X[-1])\ar[l] }$$
Note that $\Hom_{\D(R)}(R/\mathbf{p}[-m], \tau^{\leq k}(X)[-1])=0$, for each $\mathbf{p}\in \phi(m)$ and $m \leq k+2$, since  $\Hom_{\D(R)}(R/\mathbf{p}[-m],\tau^{>k}(X)[-2])=0$ and $\Hom_{\D(R)}(R/\mathbf{p}[-m],X[-1])=0$, because \linebreak $X\in \mathcal{H}_{\phi}=\mathcal{U}_{\phi} \cap \mathcal{U}_{\phi}^{\perp}[1]$. This shows that $\tau^{\leq k}(X)\in \mathcal{U}_{\phi_{\leq k+j}}^{\perp}[1]$, for each $j\in \{0, 1,2\}$. 
\end{proof}




\begin{lemma}\label{lem. U'-->U-->H}
Let $\xymatrix{U^{'} \ar[r]^{f} & U \ar[r]^{g} & Y \ar[r]^{+} & }$ be a triangle in $\D(R)$, where $U^{'},U\in \mathcal{U}_{\phi}$ and $Y \in \Hp$. Then $\xymatrix{0 \ar[r] & L(U^{'}) \ar[r]^{L(f)} & L(U) \ar[r]^{\hspace{0.15 cm}L(g)} & Y \ar[r] & 0}$ is an exact sequence in $\Hp$, where $L:\mathcal{U}_{\phi} \flecha \Hp$ is the left adjoint of the inclusion $\Hp \monic \mathcal{U}_{\phi}$.
\end{lemma}
\begin{proof}
The octahedron axiom and the definition of $L$ (see lemma \ref{lemma de adjunctions}) give a commutative diagram in $\D(R)$:
$$\xymatrix{&&&&\\ &&& M \ar[dd] \ar@{-->}[ur]^{+} & \\ & U^{'} \ar[dr] \ar@{-->}[urr] && \\ \tau^{\leq}_{\phi}(U[-1])[1] \ar[rr] \ar@{-->}[ur] & & U \ar[r] \ar[dr] & jL(U) \ar[r]^{\hspace{0.4 cm}+} \ar[d] & \\ &&& j(Y) \ar[d]^{+} \ar[dr]^{+} & \\ &&&&}$$

From the right vertical triangle we deduce that $M\in \mathcal{U}_{\phi}^{\perp}[1]$, since $jL(U)$ and $j(Y)$ are in $\Hp=\mathcal{U}_{\phi} \cap \mathcal{U}_{\phi}^{\perp}[1]$. This implies that $M\cong \tau^{>}_{\phi}(U^{'}[-1])[1]$, by using the dotted triangle. By definition of $L$, we then get that $M\in \Hp$ and $M\cong jL(U^{'})$. Then the right vertical triangle has its three vertices in $\Hp$, which ends the proof. 
\end{proof}

\vspace{0.3 cm}

The following result, tells us that that heart of every compactly generated t-structure in $\D(R)$ has a generator.

\begin{proposition}\label{prop. H has a generator}
For each $\phi$ sp-filtration of $\Spec(R)$, $\Hp$ has a generator.
\end{proposition}
\begin{proof}
Let $\mathcal{U}_{\phi}^{fg}$ be the class of all complexes quasi-isomorphic to a complex $X$ in $\mathcal{U}_{\phi}$, such that $X^{i}$ is a finitely generated $R$-module for each $i\in \mathbb{Z}$. Note that the isoclasses of objects $L(X)$ such that $X\in \mathcal{U}^{fg}$, up to isomorphism in $\Hp$, form a set. We will prove that such set form a set of generators of $\Hp$, so that the corresponding coproduct is a generator of $\Hp$. \\

Let $(X,d)\in \mathcal{H}_{\phi}$ be any object. Our goal is to prove that it is an epimorphic image in $\Hp$ of a (set-indexed) coproduct of objects of the form $L(Y)$, with $Y\in \mathcal{U}_{\phi}^{fg}$. In order to do that, we first consider the set $\Delta$ of all subcomplexes $Y$ of $X$ (in the abelian category $\C(R)$) such that $Y^{k}$ is finitely generated, for all $k\in $ \Z, that $Y\in \mathcal{U}_{\phi}$, when we view $Y$ as an object of $\D(R)$, and that the inclusion $\iota_{Y}:Y \monic X$ induces a monomorphism $H^{k}(Y) \monic H^{k}(X)$, for each $k \in \mathbb{Z}$. \\

We claim that $\underset{Y\in \Delta}{\cup} Y^{k}=X^{k}$, for each $k\in \mathbb{Z}$. Let us fix an integer $k$ and let $z\in X^{k}$. Put $Y^{k}:=Rz$ and $Y^{k+1}:=Rd^{k}(z)$, we get that $\Imagen(d^{k-1}) \cap Y^{k}$ is a finitely generated $R$-module, since it is a submodule of $Y^{k}$ and $R$ is a Noetherian ring. Hence there exists a finitely generated submodule $Y^{k-1}$ of $X^{k-1}$ such that $d^{k-1}(Y^{k-1})=\Imagen(d^{k-1})\cap Y^{k}$. By using this argument in a recursive way, we obtain a subcomplex $Y$ of $X$ concentrated in degrees less or equal to $k+1$, whose homology in degree $k+1$ vanishes, so that $Y\in \D^{\leq k}(R)$, and the homology in degrees less or equal to $k$ is given by:

\begin{align*}
          H^{j}(Y)=\frac{\Ker(d^{j}) \cap Y^{j}}{d^{j-1}(Y^{j-1})}=\frac{\Ker(d^{j})\cap Y^{j}}{\Imagen(d^{j-1})\cap Y^{j}} &= \frac{\Ker(d^{j}) \cap Y^{j}}{\Imagen(d^{j-1})\cap \Ker(d^{j}) \cap Y^{j}} \\
&=    \frac{(\Ker(d^{j})\cap Y^{j})+\Imagen(d^{j-1})}{\Imagen(d^{j-1})}       \\
     &\leq \frac{\Ker(d^{j})}{\Imagen(d^{j-1})}=H^{j}(X)\in \T_{j}  \\
\end{align*}

It follows that $Y\in \Delta$, since $(\T_{j},\F_{j})$ is a hereditary torsion pair, for each integer $j$. Our claim follows due the fact that $z\in Y^{k}$. Consider now the epimorphism \linebreak $p:\underset{Y\in \Delta}{\coprod} Y \flecha X$ in the abelian category $\C(R)$ whose components are the inclusions $\iota_{Y}:Y \monic X$. Now, for each integer $k$ and for each $Y\in \Delta$, we consider the following commutative diagram 
$$\xymatrix{0 \ar[r] & \Ker(d^{k})\cap Y^{k} \ar[r] \ar@{^(->}[d] & Y^{k} \ar[r] \ar@{^(->}[d] & d^{k}(Y^{k}) \ar@{^(->}[d] \ar[r] & 0 \\ 0 \ar[r] & \Ker(d^{k}) \ar[r] & X^{k} \ar[r] & \Imagen(d^{k}) \ar[r] & 0}$$
Taking now the coproduct of the upper rows in the diagrams above, we obtain the following commutative diagram, where $W:=\underset{Y\in \Delta}{\coprod} Y$.
\begin{footnotesize}
$$\xymatrix{0 \ar[r] & (\underset{Y\in \Delta}{\coprod}(\Ker(d^{k})\cap Y^{k}))=\Ker(d^{k}_{W}) \ar[rr] \ar[d]^{\rho_{k}} & &  \underset{Y\in \Delta}{\coprod}Y^{k} \ar[rr] \ar[d]^{p^{k}} & & (\underset{Y\in \Delta}{\coprod}d^{k}(Y^{k}))=\Imagen(d^{k}_{W})\ar[d] \ar[r] & 0 \\ 0 \ar[r] & \Ker(d^{k}) \ar[rr] && X^{k} \ar[rr] & & \Imagen(d^{k}) \ar[r] & 0}$$
\end{footnotesize}

It follows that the canonical morphism $\rho_{k}:\Ker(d^{k}_{W}) \flecha \Ker(d^{k})$ is an epimorphism, since $X^{k}=\underset{Y \in \Delta}{\cup}{Y^{k}}$. Now, we consider the following commutative diagram
$$\xymatrix{0 \ar[r] & \Imagen(d^{k-1}_{W}) \ar[r] \ar[d] & \Ker(d^{k}_{W}) \ar[r] \ar@{>>}[d]^(0.45){\rho_{k}} & H^{k}(W) \ar[r] \ar[d]^(0.45){H^{k}(p)} & 0\\ 0 \ar[r] & \Imagen(d^{k-1}) \ar[r] & \Ker(d^{k}) \ar[r] & H^{k}(X) \ar[r] & 0}$$ 
Using the snake lemma, we get that $H^{k}(p)$ is an epimorphism, for each integer $k$. If $N$ is the kernel of $p$ in $\C(R)$, then we get a triangle $\xymatrix{N \ar[r] & \underset{Y\in \Delta}{\coprod} Y \ar[r]^{\hspace{0.3 cm}p} & X \ar[r]^{+} & }$ in $\D(R)$. From the long exact sequence of homologies of this triangle, we deduce that the sequence

$$\xymatrix{0 \ar[r] & H^{k}(N) \ar[r] & \underset{Y\in \Delta}{\coprod} H^{k}(Y) \ar[r]^{\hspace{0.4 cm}H^{k}(p)} & H^{k}(X) \ar[r] & 0}$$
is exact, for each $k\in \mathbb{Z}$. It follows that $H^{k}(N)$ is in $\T_k$, for each $k\in \mathbb{Z}$, so that $N$ is in $\mathcal{U}_{\phi}$. By applying lemma \ref{lem. U'-->U-->H} and using the fact that a left adjoint functor preserves coproducts, we get a short exact sequence in $\Hp$:

$$\xymatrix{0 \ar[r] & L(N) \ar[r] & \underset{Y\in \Delta}{\coprod} L(Y) \ar[r] & X \ar[r] & 0}$$ 
This ends the proof since $\Delta \subset \mathcal{U}_{\phi}^{fg}$.
\end{proof}

\section{Localization}

In this section, we use proposition \ref{prop. Localization AJS}, for $S=R \setminus \mathbf{p}$, where $\mathbf{p}\in \Spec(R)$. In fact, we show that the condition AB5 of $\Hp$ is a local property.\\


All throughout this chapter, whenever necessary, we view $\Spec(R_{\mathbf{p}})$ as the subset of $\Spec(R)$ consisting of the $\mathbf{q}\in \Spec(R)$ such that $\mathbf{q}\subseteq \mathbf{p}$. The following result shows that hearts of t-structures behave rather well with respect to localization at primes (see proposition \ref{prop. Localization AJS}).

\begin{lemma}\label{lem. first localization}
Let $\phi$ be an sp-filtration of $\Spec(R)$, let $\mathbf{p}$ be a prime ideal of $R$ and let $\phi_{\mathbf{p}}$ denote the sp-filtration of $\Spec(R_{\mathbf{p}})$, defined by $\phi_{\mathbf{p}}(i):=\phi(i) \cap \Spec(R_{\mathbf{p}})$, for each $i\in \mathbb{Z}$. Let $\Hp$ and $\mathcal{H}_{\phi_{\mathbf{p}}}$ denote the hearts of the associated t-structures in $\D(R)$ and $\D(R_{\mathbf{p}})$, respectively. The functors $f^{*}$ and $f_{*}$ induce by restriction an adjoint pair of exact functors $(f^{*}:\Hp \flecha \mathcal{H}_{\phi_{\mathbf{p}}},f_{*}: \mathcal{H}_{\phi_{\mathbf{p}}} \flecha \Hp)$ whose count is a natural isomorphism (equivalently, $f_{*}$ is fully faithful).  
\end{lemma}
\begin{proof}
Note that $\mathbf{L}f^{*}=f^{*}: \D(R) \flecha \D(R_\p)$ is left adjoint of $\mathbf{R}f_{*}=f_{*}:\D(R_\p) \flecha \D(R)$ and the counit of the corresponding adjunction is an isomorphism. By proposition \ref{prop. Localization AJS}, we have that $f^{*}(\Hp)\subseteq \mathcal{H}_{\phi_{\mathbf{p}}}$ and $f_{*}(\mathcal{H}_{\phi_{\mathbf{p}}}) \subseteq \Hp$. The result follows immediately from this fact since a short exact sequence in the heart of a t-structure is the same as a triangle in the ambient triangulated category which has its three vertices in the heart. 
\end{proof}

\begin{corollary}\label{cor. localization AB5}
The following assertions are equivalent:
\begin{enumerate}
\item[1)] $\Hp$ is an AB5 abelian category; 
\item[2)] $\mathcal{H}_{\phi_{\mathbf{p}}}$ is an AB5 abelian category, for all $\mathbf{p}\in \Spec(R)$.
\end{enumerate}
\end{corollary}
\begin{proof}
1) $\Longrightarrow$ 2) Let us fix a $\mathbf{p}\in \Spec(R)$ and let us consider a direct system 
$$\xymatrix{\epsilon_{\lambda}:= \hspace{0.1 cm}0 \ar[r] & X_{\lambda} \ar[r] & Y_{\lambda} \ar[r] & Z_{\lambda} \ar[r] & 0}$$
of short exact sequences in $\mathcal{H}_{\phi_{\mathbf{p}}}$. By applying the exact functor $f_{*}$ to previous exact sequences, we obtain a direct system $(f_{*}(\epsilon_{\lambda}))$ of short exact sequences in $\Hp$. Taking direct limits in $\Hp$, due to the AB5 condition of this latter category, we get an exact sequence in $\Hp$:
$$\xymatrix{0 \ar[r] & \varinjlim_{\Hp}f_{*}(X_{\lambda}) \ar[r] & \varinjlim_{\Hp}f_{*}(Y_{\lambda}) \ar[r] & \varinjlim_{\Hp}f_{*}(Z_{\lambda}) \ar[r] & 0}$$
But the functor $f^{*}:\Hp \flecha \mathcal{H}_{\phi_{\mathbf{p}}}$ preserves direct limits since it is a left adjoint. This, together with the fact the counit $f^{*} \circ f_{*} \flecha 1_{\mathcal{H}_{\phi_{\mathbf{p}}}}$ is an isomorphism, imply that the sequence
  $$\xymatrix{0 \ar[r] & \varinjlim_{\mathcal{H}_{\phi_\p}}X_{\lambda} \ar[r] & \varinjlim_{\mathcal{H}_{\phi_\p}}Y_{\lambda} \ar[r] & \varinjlim_{\mathcal{H}_{\phi_\p}}Z_{\lambda} \ar[r] & 0}$$
is exact in $\mathcal{H}_{\phi_{\mathbf{p}}}$.\\

2) $\Longrightarrow$ 1) Let us consider a direct system of short exact sequence in $\Hp$
$$\xymatrix{\epsilon_{\lambda}:=\hspace{0.1 cm}0 \ar[r] & X_{\lambda} \ar[r] & Y_{\lambda} \ar[r] & Z_{\lambda} \ar[r] & 0}$$
By right exactness of colimits, we then get an exact sequence in $\Hp$:
$$\xymatrix{ \varinjlim_{\Hp}X_{\lambda} \ar[r]^{h} & \varinjlim_{\Hp}Y_{\lambda} \ar[r]^{g} & \varinjlim_{\Hp}Z_{\lambda} \ar[r] & 0}$$
Consider now any prime ideal $\mathbf{p}$ and the associated ring homomorphism $f:R \flecha R_{\mathbf{p}}$. By the exactness of $f^{*}$, we have a direct system $(f^{*}(\epsilon_{\lambda}))$ of short exact sequences in $\mathcal{H}_{\phi_{\mathbf{p}}}$. Using the fact that this functor preserves direct limits, together with the fact that $\mathcal{H}_{\phi_{\mathbf{p}}}$ is an AB5 abelian category, we get a commutative diagram with exact rows 
$$\xymatrix{ & f^{*}(\varinjlim_{\Hp}X_{\lambda}) \ar[r]^{f^{*}(h)} \ar[d]^{\wr} & f^{*}(\varinjlim_{\Hp}Y_{\lambda}) \ar[r]^{f^{*}(g)} \ar[d]^{\wr} & f^{*}(\varinjlim_{\Hp}Z_{\lambda}) \ar[r] \ar[d]^{\wr} & 0 \\ 0 \ar[r] & \varinjlim_{\mathcal{H}_{\phi_{\mathbf{p}}}} f^{*}(X_{\lambda}) \ar[r] & \varinjlim_{\mathcal{H}_{\phi_{\mathbf{p}}}} f^{*}(Y_{\lambda}) \ar[r] & \varinjlim_{\mathcal{H}_{\phi_{\mathbf{p}}}} f^{*}(Z_{\lambda}) \ar[r] & 0 }$$

where the vertical arrows are isomorphisms. It follows $R_{\mathbf{p}} \otimes_{R} \Ker_{\Hp}(h)=f^{*}(\Ker_{\Hp}(h))=0$ in $\D(R_{\mathbf{p}})$, for all $\mathbf{p}\in \Spec(R)$. This implies that the complex $\Ker_{\Hp}(h)$ is acyclic and, hence, it is the zero object of $\Hp$.
\end{proof}

\begin{corollary}\label{cor. localizando (H) categorías de modulos}
If $\Hp$ is a module category, then $\mathcal{H}_{\phi_{\mathbf{p}}}$ is a module category, for each $\mathbf{p}\in \Spec(R)$.
\end{corollary}
\begin{proof}
Note that both hearts are AB3 abelian categories (see proposition \ref{AB3 t-structure}). Let $P$ be a progenerator of $\Hp$. Fix a prime ideal $\mathbf{p}$ and consider the canonical ring homomorphism $f:R \flecha R_{\mathbf{p}}$. We will prove that $f^{*}(P)$ is a progenerator of $\mathcal{H}_{\phi_{\mathbf{p}}}$. Indeed $f^{*}(P)$ is a projective object of this category since $f^{*}$ is left adjoint of an exact functor, which implies that it preserves projective objects. On the other hand, if $Y$ is any object of $\mathcal{H}_{\phi_{\mathbf{p}}}$, then we have an epimorphism $p:P^{(\Lambda)} \epic f_{*}(Y)$ in $\Hp$, for some set $\Lambda$. Exactness of $f^{*}$ and the fact that the counit $f^{*} \circ f_{*} \flecha 1_{\mathcal{H}_{\phi_{\mathbf{p}}}}$ is an isomorphism imply that $f^{*}(p):f^{*}(P)^{(\Lambda)}\cong f^{*}(P^{(\Lambda)}) \flecha f^{*}f_{*}(Y)\cong Y$ is an epimorphism in $\mathcal{H}_{\phi_{\mathbf{p}}}$. Then $f^{*}(P)$ is a generator of $\mathcal{H}_{\phi_{\mathbf{p}}}$. \\

We finally prove that $f^{*}(P)$ is compact. Let $(Y_{\lambda})_{\lambda \in \Lambda}$ be a family of objects of $\mathcal{H}_{\phi_{\mathbf{p}}}$. From the adjunction $(f^{*},f_{*})$ together with the fact that $f_{*}:\mathcal{H}_{\phi_{\mathbf{p}}} \flecha \Hp$ preserves coproducts, we obtain the following chain of isomorphism:
\begin{eqnarray*}
 \Hom_{\mathcal{H}_{\phi_{\mathbf{p}}}}(f^{*}(P),\coprod_{\lambda \in \Lambda} Y_{\lambda}) \cong \Hom_{\Hp}(P,f_{*}(\coprod_{\lambda \in \Lambda} Y_{\lambda} )) &\cong& \Hom_{\Hp}(P, \coprod_{\lambda \in \Lambda} f_{*}(Y_{\lambda} )) \\ & \cong & \coprod_{\lambda \in \Lambda}  \Hom_{\Hp}(P, f_{*} (Y_{\lambda}) )\\
  &\cong& \coprod_{\lambda \in \Lambda}  \Hom_{\mathcal{H}_{\phi_{\mathbf{p}}}}(f^{*}(P),  Y_{\lambda})
\end{eqnarray*}
Hence $f^{*}(P)$ is a progenerator of $\mathcal{H}_{\phi_{\mathbf{p}}}$.
\end{proof}

\section{Left bounded filtrations}
In some situations, the complexes in $\Hp$ are one or two-sided bounded, and they are then more apt to induction arguments. \\

All throughout this section, $\phi$ is a sp-filtration which is \emph{left bounded}, i.e., there is an integer $n$ such that $\phi(n)=\phi(n-k)$, for all integers $k \geq 0$. In the sequel, $\phi$ will denote the sp-filtration given by 
$$\cdots=\phi(n-1) =\phi(n) \supseteq \phi(n+1) \supseteq \phi(n+2) \supseteq \cdots$$ 

\begin{lemma}\label{lem. bounded heart}
For each $X\in \Hp$, the following assertions holds.
\begin{enumerate}
\item[1)] $H^{j}(X)=0$, for each integer $j<n$.
\item[2)] If $\phi(m)=\emptyset$, for some integer $m$, then $H^{j}(X)=0$, for each $j \geq m$.
\end{enumerate}
\end{lemma}
\begin{proof}
1) By the boundedness hypothesis on $\phi$, we have that $H^{k}(X)\in \T_{k}\subseteq \T_{n}$, for all integer $k$. It follows from theorem \ref{teo. Right derived AJS} that $\mathbf{R}\Gamma_{\phi(n)}(X)\cong X$. On the other hand, $X[-1]\in \mathcal{U}_{\phi}^{\perp}$, and hence $\mathbf{R}\Gamma_{\phi(k)}X \in \D^{>k-1}(R)$ for all integer $k$ (see theorem \ref{teo. main AJS}). In particular, putting $k=n$  we get that $\mathbf{R}\Gamma_{\phi(n)}X=X\in \D^{>n-1}(R)$.\\

2) It is clear since $H^{j}(X)\in \T_{j} \subseteq \T_{m} =0$, for all $j\geq m$.
\end{proof}

\begin{lemma}\label{lem. last homology}
For each $X \in \Hp$, the following assertions holds.
\begin{enumerate}
\item[1)] $\Hom_{\D(R)}(Y[-k],\tau^{\leq k}(X))=0$, for all $k \in$ \Z \ and for each $Y\in \T_{k+1}$.
\item[2)] $\Hom_{\D(R)}(Y[-k],\tau^{\leq k-2}(X)[-1])=0$, for all $k \in $ \Z \ and for each $Y\in \T_{k}$.
\item[3)] If $m$ is the smallest integer, such that $H^{m}(X)\neq 0$, then 
$$H^{m}(X)\in \F_{m+1} \bigcap \Ker(\Ext^{1}_{R}(\T_{m+2},?)).$$
\end{enumerate}
\end{lemma}
\begin{proof}
1) and 2) follow from the fact that $\tau^{\leq k}(X)\in \mathcal{H}_{\phi_{ \leq k+1}}$ and $\tau^{\leq k-2}(X)\in \mathcal{H}_{\phi_{\leq k}}$ for all integer $k$ (see lemma \ref{lem. menor igual dos}). \\

3) Let $m$ be the smallest integer, such that $H^{m}(X)\neq 0$ (such integer exists by the previous lemma). By 1), we have:
$$0=\Hom_{\D(R)}(Y[-m],\tau^{\leq m}(X))=\Hom_{\D(R)}(Y[-m],H^{m}(X)[-m])=\Hom_{R}(Y,H^{m}(X))$$
for each $Y\in \T_{m+1}$, and hence $H^{m}(X)\in \F_{m+1}$. Similarly, by 2), we have:

\begin{eqnarray*}
 0=\Hom_{\D(R)}(Y[-m-2],\tau^{\leq m+2-2}(X)[-1])&=&\Hom_{\D(R)}(Y[-m-2],H^{m}(X)[-m-1])\\
  &=&\Ext^{1}_{R}(Y,H^{m}(X))
\end{eqnarray*}
for each $Y\in \T_{m+2}$.
\end{proof}

\begin{definition}\rm{
Let $\phi$ be a sp-filtration of $\Spec(R)$. We shall say that it is \emph{eventually trivial}, when there is a $j\in \mathbb{Z}$ such that $\phi(j)=\emptyset$. We will put $\phi=\emptyset$ whenever $\phi(i)=\emptyset$, for all $i \in \mathbb{Z}$.}
\end{definition}

\vspace{0.4 cm}

In the following proposition, we gather a few properties that will be needed later on.

\begin{proposition}\label{prop. last iso efect domino}
Let $\phi: \mathbb{Z} \flecha P(\Spec(R))$ be a left bounded sp-filtration and let us put $m:=\text{min}\{ i \in \mathbb{Z}: \phi(i) \supsetneq \phi(i+1)\}$. Consider a direct system $(X_{\lambda})$ in $\Hp$ and let $\varphi_{j}:\varinjlim{H^{j}(X_{\lambda})} \flecha H^{j}(\varinjlim_{\Hp}{X_{\lambda}})$ be the canonical map, for each $j \in \mathbb{Z}$. The following assertions hold:
\begin{enumerate}
\item[1)] If $\phi$ is eventually trivial and $t=\text{max}\{i \in \mathbb{Z}:\phi(i)\neq 0\},$ then $\varphi_{t}$ is an isomorphism and $\varphi_{t-1}$ is an epimorphism;

\item[2)] $\Supp(\Ker(\varphi_{j})) \subseteq \phi(j+1)$ and $\Supp(\Coker(\varphi_{j})) \subseteq \phi(j+2)$, for all $j\in \mathbb{Z}$;

\item[3)] $\varphi_{m}$ is an isomorphism.
\end{enumerate}
\end{proposition}
\begin{proof}
1) By taking the exact sequence of homologies associated to any exact sequence in $\Hp$, one readily sees that the functor $H^{t}:\Hp \flecha  R\text{-Mod}$ is right exact (see lemma \ref{lem. bounded heart}). Since it preserves coproducts we conclude that it preserves colimits and, in particular, direct limits. On the other hand, let us consider the canonical exact sequence
$$\xymatrix{0 \ar[r] & K \ar[r]^{\iota \hspace{0.25cm}} & \coprod X_{\lambda} \ar[r]^{p \hspace{0.25 cm}} & \varinjlim_{\Hp} X_{\lambda} \ar[r] & 0}$$
By \cite[Lemma 3.5]{CGM}, we get that $H^{t}(\iota)$ is a monomorphism. But, applying the exact sequence of homology to the exact sequence above, we obtain the following exact sequence in $R\Mode$
$$\xymatrix{\cdots \ar[r] & H^{t-1}(\coprod X_{\lambda}) \ar[r]^{H^{t-1}(p) \hspace{0.2 cm}} \ar[r] & H^{t-1}(\varinjlim_{\Hp}{X_{\lambda}}) \ar[r] & H^{t}(K) \ar@{^(->}[r]^{H^{t}(\iota) \hspace{0.2 cm}} & H^{t}(\coprod X_{\lambda}) \ar[r] & \cdots}$$
Thus $H^{t-1}(p)$ is an epimorphism, but this morphism is the composition 
$$\xymatrix{\coprod H^{t-1}(X_{\lambda}) \ar@{>>}[r] & \varinjlim{H^{t-1}(X_{\lambda})} \ar[r]^{\text{can}\hspace{0.2 cm}} & H^{t-1}(\varinjlim_{\Hp} M_i)}$$
since $H^{t-1}$ preserves coproducts. Therefore the second arrow is also an epimorphism. \\

2) Let $\mathbf{p}\in \Spec(R)\setminus  \phi(j+2)$. Then the sp-filtration $\phi_{\mathbf{p}}$ satisfies $\phi_{\mathbf{p}}(j+2)=\phi(j+2)\cap \Spec(R_{\mathbf{p}})=\emptyset$ (see proposition \ref{prop. Localization AJS}). Indeed if $\mathbf{q}\in \phi_{\mathbf{p}}(j+2)=\phi(j+2)\cap \Spec(R_{\mathbf{p}})$, then $\mathbf{q}\subseteq \mathbf{p}$, and hence $\mathbf{p}\in \phi(j+2)$, which is a contradiction. Now, applying assertion 1 to $\phi_{\mathbf{p}}$, we get that $(\varphi_{j})_{R_{\mathbf{p}}}$ is an epimorphism, i.e., $R_{\mathbf{p}} \otimes_{R} \Coker(\varphi_{j}) =0$. Similarly, if $\mathbf{p}\in \Spec(R)\setminus \phi(j+1)$ then assertion 1 for $\phi_{\mathbf{p}}$ gives that $  R_{\mathbf{p}} \otimes_{R} \Ker(\varphi_{j}) =0$. Hence assertion 2) holds. \\

3) By lemmas \ref{lem. last homology} and \ref{lem. sutileza} and the fact that $\Ext^{i}_{R}(R/\mathbf{p},?)$ preserves direct limits, for each $\mathbf{p}\in \Spec(R),$ we know that $\varinjlim{H^{m}(X_{\lambda})}$ is in $\F_{m+1}\cap \Ker(\Ext^{1}_{R}(\T_{m+2},?))$. This in turn implies that $\Ker(\varphi_{m})\in \F_{m+1}$. But assertion 2 tells us that $\Ker(\varphi_{m})\in \T_{m+1}$, so that $\varphi_{m}$ is a monomorphism. Assertion 2 also says that $\Coker(\varphi_{m})\in \T_{m+2}$, which implies then that the sequence:
$$\xymatrix{0 \ar[r] & \varinjlim{H^{m}(X_{\lambda})} \ar[r]^{\varphi_{m}\hspace{0.2 cm}} & H^{m}(\varinjlim_{\Hp}{X_{\lambda}}) \ar[r] & \Coker(\varphi_{m}) \ar[r] & 0}$$
splits. By lemma \ref{lem. last homology}, we know that $H^{m}(\varinjlim_{\Hp}X_{\lambda})\in \F_{m+1}$, which implies then that $\Coker(\varphi_{m})\in \F_{m+1}$. But we also have that $\Coker(\varphi_{m})\in \T_{m+2}\subseteq \T_{m+1}$, so that $\varphi_{m}$ is an epimorphism. 
\end{proof}

\subsection{The eventually trivial case}
Let $Z\subseteq \Spec(R)$ be a nonempty sp-subset and let $\phi$ denote the sp-filtration given by $\phi(j)=\Spec(R)$, if $j\leq -1$, $\phi(0)=Z$ and $\phi(j)=\emptyset$, if $j\geq 0$. Then the t-structure $(\mathcal{U}_{\phi},\mathcal{U}_{\phi}^{\perp}[1])$ is precisely the Happel-Reiten-Smal\o \ t-structure associated to the torsion pair $(\T_Z,\F_Z)$. By theorem \ref{Grothendieck characterization}, we know that $\Hp$ is a Grothendieck category in such case. This fact motivates us to study left bounded and eventually trivial sp-filtrations. \\ 


All throughout this subsection, $\phi$ is a left bounded eventually trivial sp-filtration which, without loss of generality, we assume that is of the form
$$\cdots=\phi(-n-1)=\phi(-n)\supsetneq \phi(-n+1) \supseteq \cdots \phi(-1) \supseteq \phi(0) \supsetneq \phi(1)=\emptyset= \cdots$$
for some $n\in \mathbb{N}$. In such case, we have that $\Hp \subset \D^{[-n,0]}(R)$ (see lemma \ref{lem. bounded heart}) and $L(\phi)=n+1$ (see definition \ref{def. finite sp-filtration}).

\begin{lemma}\label{lem. T0 closed under subobjects}
$\T_{0}[0]$ is a full subcategory of $\Hp$ closed under taking subobjects.
\end{lemma}
\begin{proof}
By examples \ref{exam. TF}, we know that $\T_{0}[0]$ is a full subcategory of $\Hp$. If now $u:X \monic T[0]$ is a monomorphism in $\Hp$, where $T\in \T_{0}$, then we have an exact sequence in $\Hp$ of the form:

$$\xymatrix{0 \ar[r] & X \ar[r]^{u\hspace{0.2 cm}} & T[0] \ar[r] & \Coker_{\Hp}(u) \ar[r] & 0}$$

From the sequence of homologies applied to the previous triangle, we get that $H^{j}(\Coker_{\Hp}(u)) \cong H^{j+1}(X)\in \T_{j+1}$, for all $j\leq -2$, and hence $\tau^{\leq -2}(\Coker_{\Hp}(u))[-1]\in \mathcal{U}_{\phi}$. But, from lemma \ref{lem. menor igual dos}, we get that $\tau^{\leq -2}(\Coker_{\Hp}(u))\in \mathcal{H}_{\phi \leq 0}=\Hp=\mathcal{U}_{\phi}\cap \mathcal{U}_{\phi}^{\perp}[1]$, so that $\tau^{\leq -2}(\Coker_{\Hp}(u))=0$. Thus, $X$ is a stalk complex concentrated in degree 0.  
\end{proof}

\vspace{0.3 cm}

In this section, we will develop a proof by induction on the length on $\phi$, in which the homology functors preserve direct limits. Considering this, the following lemma, which relates two ``consecutive'' sp-filtrations, will be very useful for the induction.

\begin{lemma}\label{lem. comparate sort consecutive filt.}
If $(X_{\lambda})$ is a direct system in $\Hp\cap \D^{\leq -1}(R)$, then $(X_{\lambda})$ is a direct system of $\mathcal{H}_{\phi \leq -1}$. Moreover, there exists a $T\in \T_{0}$ and a triangle in $\D(R)$ of the form:
$$\xymatrix{T[1] \ar[r] & \varinjlim_{\mathcal{H}_{\phi \leq -1}} X_{\lambda} \ar[r] & \varinjlim_{\Hp} X_{\lambda} \ar[r]^{\hspace{0.7cm}+}& } $$
\end{lemma} 
\begin{proof}
From lemma \ref{lem. menor igual dos}, we get that $(\tau^{\leq -1}(X_{\lambda}))=(X_{\lambda})$ is a direct system of $\mathcal{H}_{\phi \leq -1}$. We consider the associated triangle given by the colimit-defining morphism:
$$\xymatrix{\coprod X_{\lambda \mu} \ar[r]^{g} & \coprod X_{\lambda} \ar[r] & Z \ar[r]^{+} & }$$
By definition, we have that $\Coker_{\Hp}(g)=\varinjlim_{\Hp}X_{\lambda}$ and $\Coker_{\mathcal{H}_{\phi \leq -1}}(g)=\varinjlim_{\mathcal{H}_{\phi \leq -1}}X_{\lambda}$. Now, using lemma \ref{lemma de adjunctions}, together with \cite{BBD}, we get that $\varinjlim_{\Hp}X_{\lambda}=\tau_{\phi}^{>}(Z[-1])[1]$ and $\varinjlim_{\mathcal{H}_{\phi \leq -1}}X_{\lambda}=\tau_{\phi \leq -1}^{>}(Z[-1])[1]$, since $Z\in \mathcal{U}_{\phi} \cap \mathcal{U}_{\phi}^{\perp}[2]$. On the other hand, from proposition \ref{prop. coro AJS}, we then obtain the following triangle:
$$\xymatrix{H^{0}(\tau_{\phi}^{\leq}(Z[-1]))[0] \ar[r] & \tau_{\phi \leq -1}^{>}(Z[-1]) \ar[r] & \tau_{\phi}^{>}(Z[-1]) \ar[r]^{\hspace{0.8 cm}+} & }$$
The result follows by putting $T=H^{0}(\tau_{\phi}^{\leq}(Z[-1]))$, since $\tau_{\phi}^{\leq}(Z[-1])\in \mathcal{U}_{\phi}$. 
\end{proof}

\begin{lemma}\label{lem. X ----> limite}
Let $f=(f_{\lambda}: X_{\lambda} \flecha Y_{\lambda})_{\lambda \in \Lambda}$ be a morphism between the direct systems $(X_{\lambda})_{\lambda \in \Lambda}$ and $(Y_{\lambda})_{\lambda \in \Lambda}$ of $\D(R)$, and suppose that there are integers $m\leq r$ such that $X_{\lambda},Y_{\lambda}\in \D^{[m,r]}(R)$, for all $\lambda\in \Lambda$. Let $\C$ be any class of finitely generated $R$-modules. If the induced map $\varinjlim H^{j}(X_{\lambda}) \flecha \varinjlim H^{j}(Y_{\lambda})$ is an isomorphism, for all $m \leq j \leq r$, then the induced map $\varinjlim{\Hom_{\D(R)}(C[s],X_{\lambda})} \flecha \varinjlim{\Hom_{\D(R)}(C[s],Y_{\lambda})}$ is an isomorphism, for each $C\in \C$ and each integer $s$.
\end{lemma}
\begin{proof}
Due to the hypothesis and the fact that the functor $\Ext^{i}_{R}(C,?):R\text{-Mod} \flecha $Ab preserves direct limits, for each $C\in \C$ and each $i\geq 0$, we have isomorphisms:
\begin{center}
$\varinjlim \Hom_{\D(R)}(C[s],H^{j}(X_{\lambda})[-j]) \xymatrix{\cong \Hom_{\D(R)}(C[s],\varinjlim H^{j}(X_{\lambda})[-j]) \ar@<-16ex>[d]^(0.55){\wr}  & \\ &} \newline \varinjlim \Hom_{\D(R)}(C[s],H^{j}(Y_{\lambda})[-j]) \cong \Hom_{\D(R)}(C[s],\varinjlim H^{j}(Y_{\lambda})[-j]) \hspace{2.3 cm}$
\end{center}
for all $C\in \C$ and $s,j\in \mathbb{Z}$. On the other hand, note that each $X_{\lambda}$ is an iterated finite extension of $H^{m}(X_{\lambda})[-m],H^{m+1}(X_{\lambda})[-m-1],\dots,H^{r}(X_{\lambda})[-r]$ and similarly for $Y_{\lambda}$. Using this and the previous paragraph, one easily proves by induction on $k\geq 0$, that the induced map
$$\varinjlim \Hom_{\D(R)}(C[s],\tau^{\leq m+k}(X_{\lambda})) \flecha \varinjlim \Hom_{\D(R)}(C[s],\tau^{\leq m+k}(Y_{\lambda}))$$
is an isomorphism, for all $C\in \C$ and $s\in \mathbb{Z}$. The case $k=r-m$ gives the desired result.
\end{proof}

\begin{lemma}\label{lem. L' is in H}
Let $(X_{\lambda})$ be a direct system in $\Hp$, let $L^{'}$ be any object of $\D(R)$, which is viewed as a constant $\Lambda$-direct system, and let $(f_{\lambda}:X_{\lambda} \flecha L^{'})_{\lambda \in \Lambda}$ be a morphism of direct systems in $\D(R)$. If the induced morphism $\varinjlim H^{j}(f_{\lambda}):\varinjlim H^{j}(X_{\lambda}) \flecha H^{j}(L^{'})$ is an isomorphism, for all $j\in \mathbb{Z}$, then $L^{'}$ is in $\Hp$.
\end{lemma}
\begin{proof}
For each $j\in \mathbb{Z}$, we have that $H^{j}(L^{'})\in \T_{j}$ since $\T_{j}$ is closed under direct limits in $R\text{-Mod}$. It follows that $L^{'}\in \mathcal{U}_{\phi}$. On the other hand, by the previous lemma and the fact that $X_{\lambda}\in \D^{[-n,0]}(R)$ for each $\lambda \in \Lambda$, we know that the canonical map 
$$\varinjlim \Hom_{\D(R)}(C[s],X_{\lambda}) \flecha \Hom_{\D(R)}(C[s],L^{'})$$
is an isomorphism, for each finitely generated module $C$ and each $s \in \mathbb{Z}$. When taking $s=-j+1$ and $C=R/\mathbf{p}$, with $\mathbf{p}\in \phi(j)$, we get that $\Hom_{\D(R)}(R/\mathbf{p}[-j],L^{'}[-1])=0$, for all $j\in \mathbb{Z}$ and all $\mathbf{p}\in \phi(j)$, since $X_{\lambda}\in \Hp$, for all $\lambda$. Therefore, we get that $L^{'}[-1]\in \mathcal{U}_{\phi}^{\perp}$, since $\phi$ is a decreasing map. It follows that $L^{'}\in \mathcal{U}_{\phi} \cap \mathcal{U}_{\phi}^{\perp}[1]=\Hp$.   
\end{proof}

\vspace{0.3 cm}

The crucial step for our desired induction argument is given by the following result.

\begin{lemma}\label{lem. iso canonical k---->k+1}
If $(X_{\lambda})$ is a direct system of $\Hp \cap \D^{\leq k+1}(R)$ such that the canonical morphism
\begin{center}
$\varphi_{j}:\varinjlim H^{j}(X_{\lambda}) \flecha H^{j}(\varinjlim_{\Hp} X_{\lambda})$
\end{center}
is an isomorphism, for all $j\leq k$, then $\varphi_{j}$ is an isomorphism, for all $j \in \mathbb{Z}$.
\end{lemma}
\begin{proof}
Without loss of generality, we will assume that $-n\leq k+1$ (see lemma \ref{lem. bounded heart}). Let $g:\coprod X_{\lambda \mu} \flecha \coprod X_{\lambda}$ be the colimit-defining morphism and put $L:=\varinjlim_{\Hp}X_{\lambda}$. By \cite{BBD}, there exist a $M \in \Hp$ (in fact $M=\Ker_{\Hp}(g)$) and a diagram of the form, where the vertical and horizontal sequences are both triangles in $\D(R)$:
$$\xymatrix{&&M[1] \ar[d]\\ \coprod X_{\lambda \mu} \ar[r]^{g} & \coprod X_{\lambda} \ar[r] & Z \ar[r]^{+} \ar[d] & \\ && L \ar[d]^{+}\\ &&} $$
By using the octahedral axiom, we then get another commutative diagram, where rows and columns are triangles:
$$\xymatrix{\tau^{\leq k+1}(M)[1] \ar[r] \ar@{=}[d] & M[1] \ar[r] \ar[d] & \tau^{>k+1}(M)[1] \ar[r]^{\hspace{1.1 cm}+} \ar[d] &\\ \tau^{\leq k+1}(M)[1] \ar[d] \ar[r] & Z \ar[r] \ar[d] & L^{'} \ar[r]^{+} \ar[d] & \\ 0 \ar[r] \ar[d]^{+} & L \ar@{=}[r] \ar[d]^{+} & L \ar[d]^{+} \ar[r]^{+} & \\ &&&}$$

We now use the exact sequences of homologies associated to the different triangles appearing in these diagrams. When applied to the horizontal triangle in the first diagram, we get an isomorphism $\varinjlim H^{k+1}(X_{\lambda}) \cong \Coker(H^{k+1}(g)) \iso H^{k+1}(Z)$ and get that $Z\in \D^{\leq k+1}(R)$. When applied to the central horizontal triangle in the second diagram, we get another isomorphism $H^{k+1}(Z)\iso H^{k+1}(L^{'})$ and we also get that $L^{'}\in \D^{[-n-1,k+1]}(R)$. But, when applied to the right vertical triangle of the second diagram, we get that $H^{-n-1}(L^{'})=0$ since $L\in \Hp \subseteq \D^{[-n,0]}(R)$ (see lemma \ref{lem. bounded heart}) and $-n \leq k+1$. It follows that $L^{'}\in \D^{[-n,k+1]}(R)$.\\

Note that, when viewing $Z$ and $L^{'}$ as constant $\Lambda$-direct systems in $\D(R)$, we get morphisms of direct systems $(X_{\lambda} \flecha Z \flecha L^{'})$, which yield a morphism \linebreak $\psi_{j}:\varinjlim H^{j}(X_{\lambda}) \flecha H^{j}(L^{'})$ in $R$-Mod, for each $j \in \mathbb{Z}$. We claim that $\psi_{j}$ is an isomorphism, for all $j\in \mathbb{Z}$. The previous paragraph proves the claim for $j=k+1$ and the fact that $X_{\lambda},L^{'} \in \D^{[-n,k+1]}(R)$ reduces the proof to the case when $j\in \{-n,-n+1,\dots,k\}$. But, when we apply the exact sequence of homologies to the right vertical triangle of the second diagram above, we get an isomorphism $H^{j}(L^{'})\cong H^{j}(L)$, for $-n \leq j <k$, and a monomorphism $H^{k}(L^{'}) \monic H^{k}(L)$, since $H^{j}(\tau^{>k+1}(M)[1])=H^{j+1}(\tau^{>k+1}(M))=0$, for all $j \leq k$. The hypothesis of the present lemma then gives that $\psi_{j}$ is an isomorphism, for all $j<k$, and a composition of morphism $\varinjlim H^{k}(X_{\lambda}) \flecha H^{k}(L^{'}) \flecha H^{k}(L)$, which is an isomorphism, and whose second arrow is a monomorphism. It then follows that both arrows in this composition are isomorphisms, thus settling our claim. \\

Once we know that $\psi_{j}$ is an isomorphism, for all $j\in \mathbb{Z}$, lemma \ref{lem. L' is in H} says that $L^{'}\in \Hp$. Hence, $L^{'}[-1]\in \mathcal{U}_{\phi}^{\perp}$, in particular $L^{'}[-3]\in \mathcal{U}_{\phi}^{\perp}$. Now, using the central horizontal triangle of the second diagram above, we obtain the following triangle in $\D(R)$:

$$\xymatrix{L^{'}[-3] \ar[r] & \tau^{\leq k+1}(M)[-1] \ar[r] & Z[-2] \ar[r]^{\hspace{0.5 cm}+} & }$$

Due to the fact that $Z\in \mathcal{U}_{\phi}^{\perp}[2]$, we deduce that $\tau^{\leq k+1}(M)[-1] \in \mathcal{U}_{\phi}^{\perp}$. It follows that $\tau^{\leq k+1}(M)\in \Hp$, since, due to the fact that $M\in \Hp$, we have that $\Supp(H^{j}(\tau^{\leq k+1}(M)))\subseteq \phi(j)$, for all $j\in \mathbb{Z}$. Using now the explicit calculation of $\Coker_{\Hp}(g)$ (\cite{BBD}), we then get that $L^{'}\cong \Coker_{\Hp}(g)=L$ in $\Hp$.    
\end{proof}

\vspace{0.4 cm}

The following is the first main result of the section.

\begin{theorem}\label{teo. AB5 finite}
Let $\phi$ be a left bounded eventually trivial sp-filtration. The classical cohomological functor $H^{k}:\Hp \flecha R$-Mod preserves direct limits, for each integer $k$. In particular, $\Hp$ is an AB5 abelian category.
\end{theorem}
\begin{proof}
The final part of the statement follows from corollary \ref{All functor homology}. Without loss of generality, if $L(\phi)=n+1$ we will assume that $\phi$ is the filtration given by
$$\cdots=\phi(-n-k)=\phi(-n)\supsetneq \phi(-n+1) \supseteq \phi(-n+2) \supseteq \cdots \supseteq \phi(0)  \supsetneq \phi(1)=\cdots =\emptyset$$

The proof will be done by induction on the length $L(\phi)=n+1$ of the filtration. The cases $L(\phi)=0,1$ are covered by proposition \ref{prop. last iso efect domino} (see lemma \ref{lem. bounded heart}). We then assume in the sequel that $n\geq 1$, that $L(\phi)=n+1$ and that the result is true for finite sp-filtrations of length less than or equal to $n$. Let $(X_{\lambda})$ be a direct system in $\Hp$. The proof is divided into several steps. \\

\emph{Step 1}: The result is true when $X_{\lambda}\in \D^{\leq -2}(R)$, for all $\lambda \in \Lambda$.\\
In such case, by lemma \ref{lem. menor igual dos} we know that $(\tau^{\leq -1}(X_{\lambda}))=(X_{\lambda})$ is a direct system of $\Hp \cap \mathcal{H}_{\phi_{\leq -1}}$ and from lemma \ref{lem. comparate sort consecutive filt.}, we get a triangle of the form:

$$\xymatrix{T[1] \ar[r] & \varinjlim_{\mathcal{H}_{\phi_{\leq -1}}}X_{\lambda} \ar[r] & \varinjlim_{\Hp} X_{\lambda} \ar[r]^{\hspace{0.75 cm}+}&}$$
with $T\in \T_{0}$. Since $L(\phi_{\leq -1}) \leq n$, the induction hypothesis gives that the canonical map $\varinjlim H^{k}(X_{\lambda}) \flecha H^{k}(\varinjlim_{\mathcal{H}_{\phi_{\leq -1}}} X_{\lambda})$ is an isomorphism, for all integers $k$. On the other hand, from the sequence of homologies applied to the previous triangle, we get that the canonical maps $\varinjlim H^{k}(X_{\lambda}) \iso H^{k}(\varinjlim_{\mathcal{H}_{\phi_{\leq -1}}} X_{\lambda}) \iso H^{k}(\varinjlim_{\Hp} X_{\lambda})$ are isomorphisms, for all $k \leq -3$. By lemma \ref{lem. iso canonical k---->k+1}, we conclude that the canonical morphism $\varinjlim H^{k}(X_{\lambda}) \cong H^{k}(\varinjlim_{\Hp}X_{\lambda})$ is an isomorphism all $k\in \mathbb{Z}$. \\

\emph{Step 2}: The result is true when $X_{\lambda} \in \D^{\leq -1}(R)$, for all $\lambda \in \Lambda$. \\
Using the octahedral axiom, for each $\lambda$, we obtain the following commutative diagram:

$$\xymatrix{&&&&\\ &&&t_{0}(H^{-1}(X_{\lambda}))[1] \ar[ur]^{+} \ar[dd] & \\ & Y_{\lambda} \ar[urr] \ar[dr] &&&\\ \tau^{\leq -2}(X_{\lambda}) \ar[ru] \ar[rr] && X_{\lambda} \ar[r] \ar[rd] & H^{-1}(X_{\lambda})[1] \ar[r]^{\hspace{0.8 cm}+} \ar[d] & \\ &&&(1:t_{0})(H^{-1}(X_{\lambda}))[1] \ar[d]^{+} \ar[dr]^(0.59){+}& \\ &&&&& }$$
From the triangle $\xymatrix{\tau^{\leq -2}(X_{\lambda}) \ar[r] & Y_{\lambda} \ar[r] & t_{0}(H^{-1}(X_{\lambda}))[1] \ar[r]^{\hspace{1.3 cm}+} & }$, we get that $Y_{\lambda}\in \mathcal{U}_{\phi}$, while from the triangle $\xymatrix{(1:t_{0})(H^{-1}(X_{\lambda}))[0] \ar[r] & Y_{\lambda} \ar[r] & X_{\lambda} \ar[r]^{+} &}$, and the fact that \linebreak $(1:t_0)(H^{-1}(X_{\lambda}))\in \T\F_{-1}$ (see example \ref{exam. TF}) we get that $Y_{\lambda}\in \mathcal{U}_{\phi}^{\perp}[1]$. We then have that $Y_{\lambda}\in \Hp$. Moreover, using lemma \ref{lem. menor igual dos} we get direct systems of exact sequences in $\Hp$ (note that, in particular $(Y_{\lambda})$ form a direct system of $\Hp$ since it is obtained as the cokernel of a morphisms of direct systems in $\Hp$):
$$\xymatrix{0 \ar[r] & t_{0}(H^{-1}(X_{\lambda}))[0] \ar[r] & \tau^{\leq -2}(X_{\lambda}) \ar[r] & Y_{\lambda} \ar[r] & 0 & \ar@{}[d]^{(\ast)} \\ 0 \ar[r] & Y_{\lambda} \ar[r] & X_{\lambda} \ar[r] & (1:t_0)(H^{-1}(X_{\lambda}))[1] \ar[r] & 0 & }$$

Applying the direct limit functor and putting $T:=t_{0}(\varinjlim H^{-1}(X_{\lambda}) \cong \varinjlim t_{0}(H^{-1}(X_{\lambda})))$ (see remark \ref{rem. F closed under direct limits}), we get an exact sequence in $\Hp$ of the form (see proposition \ref{prop. limits stalk}):
\begin{center}
$\xymatrix{T[0] \ar[r] & \varinjlim_{\Hp} \tau^{\leq -2}(X_{\lambda}) \ar[r]^{\hspace{0.6 cm}p} & \varinjlim_{\Hp} Y_{\lambda} \ar[r] & 0}$
\end{center}
If $W:=\Ker_{\Hp}(p)$ then, by lemma \ref{lem. T0 closed under subobjects}, we have an exact sequence in $\Hp$ of the form:
$$0 \flecha T^{'}[0] \flecha T[0] \flecha W \flecha 0$$ 
with $T^{'} \in \T_{0}$. From the sequence of homologies applied to the previous sequence, we obtain that $W\in \D^{[-1,0]}(R)$ and $H^{k}(W)\in \T_{0}$ for all $k\in\{-1,0\}$, since $\T_{0}$ is closed under taking subobjects and quotients. But assertion 3 of lemma \ref{lem. last homology}, then says that $H^{-1}(W)\in \F_{0}$, and this implies that $W\cong H^{0}(W)[0]\in \T_{0}[0]$. We then have isomorphisms $H^{k}(X_{\lambda})=H^{k}(\tau^{\leq -2}(X_{\lambda})) \cong H^{k}(Y_{\lambda})$ and $H^{k}(\varinjlim_{\Hp}\tau^{\leq -2}(X_{\lambda})) \cong H^{k}(\varinjlim_{\Hp}Y_{\lambda})$, for all $k \leq -2$. By step 1 of this proof, we get an isomorphism 
\begin{center}
$\xymatrix{\varinjlim H^{k}(Y_{\lambda}) \ar[r] & \varinjlim H^{k}(X_{\lambda}) \ar[r]^{\sim}& H^{k}(\varinjlim_{\Hp}\tau^{\leq -2}(X_{\lambda})) \cong H^{k}(\varinjlim_{\Hp} Y_{\lambda})}$
\end{center}
for all $k\leq -2$. By lemma \ref{lem. iso canonical k---->k+1}, we then get that the canonical map 
\begin{center}
$\varinjlim H^{k}(Y_{\lambda}) \flecha H^{k}(\varinjlim_{\Hp}Y_{\lambda})$
\end{center}
is an isomorphism, for all $k\in \mathbb{Z}$. \\

Bearing in mind that $L(\phi_{\leq -1})\leq n$, lemma \ref{lem. comparate sort consecutive filt.} and the induction hypothesis (similar to step 1) imply that we have isomorphisms $\varinjlim H^{k}(X_{\lambda}) \cong H^{k}(\varinjlim_{\mathcal{H}_{\phi \leq -1}}X_{\lambda}) \cong H^{k}(\varinjlim_{\Hp} X_{\lambda})$, for all $k \leq -3$. They also imply the following facts:
\begin{enumerate}
\item[1)] $\varinjlim H^{-2}(X_{\lambda}) \cong H^{-2}(\varinjlim_{\mathcal{H}_{\phi_{\leq -1}}}X_{\lambda}) \flecha H^{-2}(\varinjlim_{\Hp}X_{\lambda})$ is a monomorphism with cokernel in $\T_{0}$;

\item[2)] $\varinjlim H^{-1}(X_{\lambda}) \cong H^{-1}(\varinjlim_{\mathcal{H}_{\phi_{\leq -1}}}X_{\lambda}) \flecha H^{-1}(\varinjlim_{\Hp}X_{\lambda})$ is an epimorphism with kernel in $\T_{0}.$ 
\end{enumerate}
We next apply the direct limit functor to the second direct system in $(\ast)$ to get an exact sequence in $\Hp$ (see proposition \ref{prop. limits stalk}):
$$\xymatrix{\varinjlim_{\Hp}Y_{\lambda} \ar[r]^{\alpha} & \varinjlim_{\Hp}X_{\lambda} \ar[r] & (\varinjlim (1:t_{0})(H^{-1}(X_{\lambda})))[1] \ar[r] & 0}$$
We put $Z:=\Imagen(\alpha)$, and we consider the two induced short exact sequences in $\Hp$:
$$\xymatrix{0 \ar[r] & W^{'} \ar[r] & \varinjlim_{\Hp}Y_{\lambda} \ar[r] & Z \ar[r] & 0 \\ 0 \ar[r] & Z \ar[r] & \varinjlim_{\Hp}X_{\lambda} \ar[r] & (\varinjlim (1:t_{0})(H^{-1}(X_{\lambda})))[1] \ar[r] & 0}$$
From the first one we get that $Z\in \D^{\leq -1}(R)$, since the functor $H^{0}:\Hp \flecha R$-Mod, is right exact and the $Y_\lambda$ are in $\D^{\leq -1}(R)$. Now, using the exact sequence of homologies for the second one, we get that $H^{k}(Z) \cong H^{k}(\varinjlim_{\Hp}X_{\lambda})$, for all $k \leq -2$. On the other hand, when applying the exact sequence of homologies to the triangle:
$$\xymatrix{\tau^{\leq -2}(X_{\lambda}) \ar[r] & Y_{\lambda} \ar[r] & t_{0}(H^{-1}(X_{\lambda}))[1] \ar[r]^{\hspace{1.2 cm}+} &}$$
we get that the canonical map $H^{k}(X_{\lambda}) \flecha H^{k}(Y_{\lambda})$ is an isomorphism, for all $\lambda \in \Lambda$ and all $k \leq -2$. If $p:\varinjlim_{\Hp}Y_{\lambda} \epic Z$ is the epimorphism in the first sequence above, then we get that the morphisms:
\begin{center}
$\varinjlim H^{k}(Y_{\lambda}) \cong H^{k}(\varinjlim_{\Hp}Y_{\lambda}) \xymatrix{\ar[r]^{H^{k}(p)} &} H^{k}(Z)\cong \varinjlim H^{k}(X_{\lambda})$
\end{center}
are an isomorphism, for all $k \leq -3$ and a monomorphism with cokernel in $\T_{0}$, for $k=-2$. This shows that $W^{'}\in \D^{[-1,0]}(R)$. Moreover, we have an induced exact sequence:
$$\xymatrix{0 \ar[r] & \varinjlim H^{-2}(Y_{\lambda}) \ar[rr]^{\hspace{0.3 cm}H^{-2}(p)} && H^{-2}(Z) \ar[r] & H^{-1}(W^{'}) \ar[r] & \varinjlim H^{-1}(Y_{\lambda})}$$
Since $H^{-1}(Y_{\lambda})=t_{0}(H^{-1}(X_{\lambda}))$ and $\Coker(H^{-2}(p))$ are in $\T_{0}$,  we derived that $H^{-1}(W)\in \T_{0}$. But assertion 3 of lemma \ref{lem. last homology}, says that $H^{-1}(W^{'})\in \F_{0}$. It follows that $H^{-1}(W^{'})=0$, which implies that the canonical map $H^{-2}(p)$ is an isomorphism. We then get that the canonical map $\xymatrix{\varinjlim H^{k}(X_{\lambda}) \ar[r] & H^{k}(\varinjlim_{\Hp} X_{\lambda})}$ is an isomorphism, for all $k\leq -2$ (since $H^{-2}(Z)\cong H^{-2}(\varinjlim_{\Hp}X_{\lambda})$). By lemma \ref{lem. iso canonical k---->k+1}, we conclude that this is also an isomorphism, for all $k\in \mathbb{Z}$. \\

\emph{Step 3:} The general case.\\
Let $(X_{\lambda})$ be an arbitrary direct system of $\Hp$. From lemma \ref{lem. menor igual dos} and examples \ref{exam. TF}, we have a direct system of exact sequences in $\Hp$, of the form:
$$\xymatrix{0 \ar[r] & \tau^{\leq -1}(X_{\lambda}) \ar[r] & X_{\lambda} \ar[r] & H^{0}(X_{\lambda})[0] \ar[r] & 0}$$
After taking direct limits, we get the following commutative diagram, with exact rows, in $\Hp$ (see proposition \ref{prop. limits stalk}):
$$\xymatrix{& \varinjlim_{\Hp} \tau^{\leq -1}(X_{\lambda}) \ar[r] \ar[d]^{\epsilon} & \varinjlim_{\Hp}X_{\lambda} \ar[r] \ar@{=}[d] & (\varinjlim{H^{0}(X_{\lambda})})[0] \ar[r] \ar[d]^{\wr} & 0\\ 0 \ar[r] & \tau^{\leq -1}(\varinjlim_{\Hp} X_{\lambda}) \ar[r] & \varinjlim_{\Hp}X_{\lambda} \ar[r] & H^{0}(\varinjlim_{\Hp}X_{\lambda})[0] \ar[r] & 0}$$

Its right vertical arrow is an isomorphism since $H^{0}: \Hp \flecha R$-Mod preserves direct limits. By snake lemma, the morphism $\epsilon$ is an epimorphism. We next consider the exact sequence in $\Hp$:
$$\xymatrix{0 \ar[r] & W^{''} \ar[r] & \varinjlim_{\Hp} \tau^{\leq -1}(X_{\lambda}) \ar[r]^{\epsilon} & \tau^{\leq -1}(\varinjlim_{\Hp} X_{\lambda}) \ar[r] & 0}$$
When localizing at a prime $\mathbf{p}\in \Spec(R)\setminus \phi(0)$, the filtration $\phi_{\mathbf{p}}$ has length less than or equal to $n$, since $\phi_{\mathbf{p}}(0)=\phi(0)\cap \Spec(R_{\mathbf{p}})=\emptyset$. Thus, by the induction hypothesis, we then get isomorphisms $\varinjlim H^{k}( R_{\mathbf{p}} \otimes_{R}\tau^{\leq -1}(X_{\lambda})   ) \iso H^{k}(\varinjlim_{\mathcal{H}_{\phi_\p}}( R_{\mathbf{p}}  \otimes_{R} \tau^{\leq -1} (X_{\lambda})) )$ and $\varinjlim H^{k}(  R_{\mathbf{p}} \otimes_{R} X_{\lambda}) \iso H^{k}(\varinjlim_{\mathcal{H}_{\phi_\p}} ( R_{\mathbf{p}} \otimes_{R} X_{\lambda}  ))$, for all $k\in \mathbb{Z}$ (see lemma \ref{lem. first localization}). But the canonical morphism $\tau^{\leq -1}(R_{\mathbf{p}}  \otimes_{R} X_{\lambda}  )=  R_{\mathbf{p}}  \otimes_{R} \tau^{\leq -1}(X_{\lambda})  \flecha R_{\mathbf{p}} \otimes_{R} X_{\lambda}    $ is an isomorphism in $\D(R_{\mathbf{p}})$ since $R_{\mathbf{p}} \otimes_{R} H^{0}(X_{\lambda})  =0$, for all $\lambda \in \Lambda$. It follows that:

\begin{center} 
$H^{k}(1_{R_{\mathbf{p}}}  \otimes \epsilon  ):H^{k}(  \varinjlim_{\Hp}  (R_{\mathbf{p}} \otimes_{R} \tau^{\leq -1}(X_{\lambda}))) \flecha H^{k}(\varinjlim_{\Hp}( R_{\mathbf{p}} \otimes_{R}X_{\lambda}) )$
\end{center}

is an isomorphism, for all $k \in \mathbb{Z}$, which implies that $ 1_{R_{\mathbf{p}}} \otimes_{R} \epsilon $ is an isomorphism in $\D(R_{\mathbf{p}})$, for all $\mathbf{p} \in \Spec(R) \setminus \phi(0)$. This in turn implies that $\Supp(H^{k}(W^{''})) \subseteq \phi(0)$ (equivalently $H^{k}(W^{''})\in \T_{0}$), for all $k \in \mathbb{Z}$. Hence, we get that $\tau^{\leq -1}(W^{''})[-1]\in \mathcal{U}_{\phi}$. But then we have  $\tau^{\leq -1}(W^{''})=0$ since $\tau^{\leq -1}$$(W^{''})\in \Hp=\mathcal{U}_{\phi} \cap \mathcal{U}_{\phi}^{\perp}[1]$. This shows that $W^{''}\cong T[0]$, for some $T\in \T_{0}$.\\

By applying the exact sequence of homologies to the triangle:
$$\xymatrix{T[0] \ar[r] & \varinjlim_{\Hp} \tau^{\leq -1}(X_{\lambda}) \ar[r]^{\epsilon} & \tau^{\leq -1}(\varinjlim_{\Hp} X_{\lambda}) \ar[r]^{\hspace{1.3 cm}+} & }$$
and, using step 2 of this proof, we get isomorphisms: 
\begin{center}
$\varinjlim H^{k}(X_{\lambda}) \cong H^{k}(\varinjlim_{\Hp}\tau^{\leq -1}(X_{\lambda})) \iso H^{k}(\tau^{\leq -1}(\varinjlim_{\Hp} X_{\lambda}))=H^{k}(\varinjlim_{\Hp}X_{\lambda})$
\end{center}
for all $k \leq -2$, and an exact sequence:
$$\xymatrix{0 \ar[r] & \varinjlim H^{-1}(X_{\lambda}) \ar[r] & H^{-1}(\varinjlim_{\Hp}X_{\lambda}) \ar[r] & T \ar[r] & 0}$$
By proposition \ref{prop. last iso efect domino}, we conclude that the canonical map $\varinjlim H^{k}(X_{\lambda}) \flecha H^{k}(\varinjlim_{\Hp} X_{\lambda})$ is an isomorphism, for all $k \leq -1$ and, hence, for all $k\in \mathbb{Z}$.
\end{proof}


\subsection{Infinite Case}
If $\phi$ is a left bounded sp-filtration (not necessary eventually trivial), its associated heart $\Hp$, can have complexes with nonzero homology in infinitely many degrees. That is, in general, $\Hp$ is not contained in $\D^{b}(R)$. The next result tells us that, despite this fact, $\Hp$ is a Grothendieck category, which is evidently a leap foward in the complexity level of the case Happel-Reiten-Smal\o \ and theorem \ref{teo. AB5 finite}. \\

In this subsection, $\phi$ is a sp-filtration which is left bounded\index{filtration! left bounded}, i.e., there is a $n$ integer such that $\phi(n)=\phi(n-k)$, for all integers $k\leq 0$. \\

We are ready to prove the first main result of this chapter.

\begin{theorem}\label{teo. first main chapter V}
Let $\phi$ be any left bounded sp-filtration. Then $\Hp$ is a Grothendieck category. 
\end{theorem}
\begin{proof}
Without loss of generality, we may and shall assume that $\phi(j)\neq \emptyset$, for all $j\in \mathbb{Z}$ and that $\phi$ is of the form:
$$\cdots=\phi(-n-k)=\cdots=\phi(-n-1)=\phi(-n)\supsetneq \phi(-n+1) \supseteq \cdots \supseteq \phi(0) \supseteq \phi(1) \supseteq \cdots$$
for some $n\in \mathbb{N}$ (see theorem \ref{teo. AB5 finite}). By proposition \ref{prop. H has a generator} and corollary \ref{sufficient AB5 general}, it is enough to check that, for each direct system $(X_{\lambda})_{\lambda \in \Lambda}$ in $\Hp$ and each $j\in \mathbb{Z}$, the canonical map:
$$\xymatrix{\varphi_{j}: \varinjlim H^{j}(X_{\lambda}) \ar[r] & H^{j}(\varinjlim_{\Hp} X_{\lambda})}$$
is an isomorphism. Moreover, by lemma \ref{lem. bounded heart}, it is enough to do this for $j \geq -n$. On the other hand, note that for each $\mathbf{p}\in \Spec(R)$, if $\phi_{\mathbf{p}}(j)=\phi(j)\cap \Spec(R_{\mathbf{p}})\neq \emptyset$, for some $j\in \mathbb{Z}$, then $\mathbf{p}\in \phi(j)$ since $\phi(j)$ is a sp-subset of $\Spec(R)$. Thus, $\phi_{\mathbf{p}}(j)\neq \emptyset $, for all $j\in \mathbb{Z}$ if, and only if, $\mathbf{p}\in \underset{i\in \mathbb{Z}}{\bigcap} \phi(i)$. Therefore, if $\mathbf{p}\in \Spec(R)\setminus (\underset{i \in \mathbb{Z}}{\bigcap} \phi(i))$, then the associated sp-filtration $\phi_{\mathbf{p}}$ of $\Spec(R_{\mathbf{p}})$ is either trivial (i.e. all its members are empty) or eventually tirvial. Since the result is true for eventually trivial filtrations (see theorem \ref{teo. AB5 finite}) we conclude that (see lemma \ref{lem. first localization}):
$$\xymatrix{ 1_{R_{\mathbf{p}}}  \otimes_{R} \varphi_{j}:  R_{\mathbf{p}} \otimes_{R}  (\varinjlim H^{j}(X_{\lambda})) \ar[r] & R_{\mathbf{p}} \otimes_{R} H^{j}(\varinjlim_{\Hp} X_{\lambda})}$$
is an isomorphism, for all $j\in \mathbb{Z}$. We then get that $\Supp(\Ker(\varphi_{j})) \bigcup \Supp(\Coker(\varphi(j))) \subseteq \underset{i \in \mathbb{Z}}{\bigcap} \phi(i)$, for all $j \geq -n$.\\

Given such a direct system $(X_{\lambda})$, we obtain a direct system $(\tau^{\leq j}(X_{\lambda}))_{\lambda \in \Lambda}$ in $\D(R)$, for each $j\in \mathbb{Z}$, together with a morphism of $(\Lambda)$-direct systems $(\tau^{\leq j}(X_{\lambda}) \flecha \tau^{\leq j}(L)),$ where $L:=\varinjlim_{\Hp} X_{\lambda}$. We will prove at once, by induction on $k \geq 0$, the following two facts and the proof will be then finished:
\begin{enumerate}
\item[1)] The canonical map $\varinjlim H^{-n+k}(X_{\lambda}) \flecha H^{-n+k}(L)$ is an isomorphism;

\item[2)] The induced map $$\varinjlim \Hom_{\D(R)}(R/\mathbf{p}[-j], \tau^{\leq -n+k}(X_{\lambda})[s]) \flecha \Hom_{\D(R)}(R/\mathbf{p}[-j], \tau^{\leq -n+k}(L)[s])$$ 
is an isomorphism, for all $k,s \in \mathbb{Z}$ and all $\mathbf{p} \in \Spec(R)$.
\end{enumerate}

For $k=0$, proposition \ref{prop. last iso efect domino} proves fact 1). Once fact 1) is satisfied, we get fact 2) from lemma \ref{lem. X ----> limite}. \\

Let $k\geq 0$ be a natural number and suppose now that facts 1) and 2) are true for the natural numbers less than or equal $k$. Now, Let $Z\in \Hp$ be any object. By lemma \ref{lem. menor igual dos}, we have that $\tau^{\leq -n+k+1}(Z)\in \mathcal{H}_{\phi_{\leq -n+k+3}} \subseteq \mathcal{U}^{\perp}_{\phi_{\leq -n+k+3}}[1]$. We next consider the following triangle in $\D(R)$:

$$\xymatrix{\tau^{\leq -n+k+1}(Z)[-1] \ar[r] & H^{-n+k+1}(Z)[n-k-2] \ar[r] & \tau^{\leq -n+k}(Z) \ar[r] & \tau^{\leq -n+k+1}(Z)}$$
If $\mathbf{p}\in \phi(-n+k+3)$ is any prime, then $\Hom_{\D(R)}(R/\mathbf{p}[n-k-2],\tau^{\leq -n+k+1}(Z)[-1])=0$, since $\mathbf{p}\in \phi(-n+k+3)\subseteq \phi(-n+k+2)$. Furthermore, $$\Hom_{\D(R)}(R/\mathbf{p}[n-k-3],\tau^{\leq -n+k+1}(Z)[-1])=0=\Hom_{\D(R)}(R/\mathbf{p}[n-k-2],\tau^{\leq -n+k+1}(Z))$$
Applying the functor $\Hom_{\D(R)}(R/\mathbf{p}[n-k-2],?)$ to the previous triangle together with the previous paragraph, we then get:
\begin{enumerate}
\item[1')] The morphism  
$$\begin{matrix}\Hom_{R}(R/\mathbf{p},H^{-n+k+1}(Z))\cong  \Hom_{\D(R)}(R/\mathbf{p}[n-k-2], H^{-n+k+1}(Z)[n-k-2]) \\ \xymatrix{\ar[d] & \\ & } \\ \Hom_{\D(R)}(R/\mathbf{p}[n-k-2],\tau^{\leq -n+k}(Z))
\end{matrix}$$
is an isomorphism.
\item[2')] The morphism
$$
\begin{matrix}
 \Ext^{1}_{R}(R/\mathbf{p},H^{-n+k+1}(Z)) \cong \Hom_{\D(R)}(R/\mathbf{p}[n-k-3], H^{-n+k+1}(Z)[n-k-2])  \\ \xymatrix{\ar[d] & \\ & }
   \\ \Hom_{\D(R)}(R/\mathbf{p}[n-k-3], \tau^{\leq -n+k}(Z))\cong \Hom_{\D(R)}(R/\mathbf{p}[n-k-2],\tau^{\leq -n+k}(Z)[1])  
\end{matrix}
$$ 
is a monomorphism. 
\end{enumerate}

We now get the following chain of isomorphisms, where we write over each arrow the reason for it to be an isomorphism:

$$\xymatrix{\Hom_{R}(R/\mathbf{p},\varinjlim H^{-n+k+1}(X_{\lambda})) & \varinjlim \Hom_{R}(R/\mathbf{p},H^{-n+k+1}(X_{\lambda}))\ar[l]_{\text{can}} \ar[d]^{\text{fact }1'}\\  & \varinjlim \Hom_{\D(R)}(R/\mathbf{p}[n-k-2],\tau^{\leq -n+k}(X_{\lambda})) \ar[d]^{\text{induction}} \\ \Hom_{R}(R/\mathbf{p},H^{-n+k+1}(L)) \ar[r]^{\text{fact }1' \hspace{1 cm}} & \Hom_{\D(R)}(R/\mathbf{p}[n-k-2],\tau^{\leq -n+k}(L)) }$$

It is not difficult to see that the resulting isomorphism is the one obtained by applying the functor $\Hom_{R}(R/\mathbf{p},?)$ to the canonical morphism 

$$\varphi_{-n+k+1}:\varinjlim H^{-n+k+1}(X_{\lambda}) \flecha H^{-n+k+1}(L)$$ 
It follows that $\Hom_{R}(R/\mathbf{p}, \Ker(\varphi_{-n+k+1}))=0$, for all $\mathbf{p}\in \phi(-n+k+3)$. This shows that $\Ker(\varphi_{-n+k+1})\in \F_{-n+k+3}$. But, $\Supp(\Ker(\varphi_{-n+k+1}))\subseteq \underset{i \in \mathbb{Z}}{\bigcap} \phi(i) \subseteq \phi(-n+k+3)$, i.e., $\Ker(\varphi_{-n+k+1})$ is in $\T_{-n+k+3}$. Therefore $\varphi_{-n+k+1}$ is a monomorphism. \\

We claim that the morphism $\Ext^{1}_{R}(R/\mathbf{p},\varphi_{-n+k+1})$ is a monomorphism, for all $\mathbf{p}\in \phi(-n+k+3)$. In such case, if we take the sequence of $\Ext(R/\mathbf{p},?)$ associated to the exact sequence:
$$\xymatrix{0 \ar[r] & \varinjlim H^{-n+k+1}(X_{\lambda}) \ar[rr]^{\hspace{0.4 cm}\varphi_{-n+k+1}} & & H^{-n+k+1}(L) \ar[r] & \Coker(\varphi_{-n+k+1}) \ar[r] & 0}$$
we will derive that $\Hom_{R}(R/\mathbf{p},\Coker(\varphi_{-n+k+1}))=0$ (see the sequence below), for all $\mathbf{p}\in \phi(-n+k+3)$.

$$\xymatrix{0 \ar[r] & \Hom_{R}(R/\mathbf{p}, \varinjlim H^{-n+k+1}(X_{\lambda})) \ar[r]^{\sim} & \Hom_{R}(R/\mathbf{p}, H^{-n+k+1}(L)) \ar[d] \\ \Ext^{1}_{R}(R/\mathbf{p},H^{-n+k+1}(L)) & \Ext^{1}_{R}(R/\mathbf{p}, \varinjlim H^{-n+k+1}(X_{\lambda})) \hspace{0.05 cm} \ar@{^(->}[l] & \Hom_{R}(R/\mathbf{p}, \Coker(\varphi_{-n+k+1})) \ar[l]  }$$ 
Then, arguing as in the case of kernel, we will deduce that $\Coker(\varphi_{-n+k+1})=0$, so that the $k+1$-induction step for fact 1) above will be covered. In order to settle our claim, we consider the commutative diagram
$$
\xymatrix{\Ext^{1}_{R}(R/\mathbf{p}, \varinjlim H^{-n+k+1}(X_{\lambda})) & \\ \varinjlim \Ext^{1}_{R}(R/\mathbf{p}, H^{-n+k+1}(X_{\lambda})) \ar[r] \ar[u]^{\wr} \ar[d]& \Ext^{1}_{R}(R/\mathbf{p},H^{-n+k+1}(L)) \ar[d] \\ \varinjlim \Hom_{\D(R)}(R/\mathbf{p}[n-k-2], \tau^{\leq -n+k}(X_{\lambda})[1]) \ar[r]^{\hspace{0.4 cm}\sim} & \Hom_{\D(R)}(R/\mathbf{p}[n-k-2], \tau^{\leq -n+k}(L)[1])}
$$

The induction hypothesis for fact 2) gives that the lower horizontal arrow is an isomorphism. On the other hand, by 2') above, the downward left vertical is a monomorphism (after taking direct limits). Then our claim follows immediately since we obtain the following commutative diagram
$$\xymatrix{ 0 \ar[r] & \varinjlim \Ext^{1}_{R}(R/\mathbf{p}, H^{-n+k+1}(X_{\lambda})) \ar[r] \ar[d]_{\Ext^{1}_{R}(R/\mathbf{p}, \varphi_{-n+k+1})} & \varinjlim \Hom_{\D(R)}(R/\mathbf{p}[n-k-2],\tau^{\leq -n+k}(X_{\lambda})[1]) \ar[d]^{\wr} \\ 0 \ar[r] & \Ext^{1}_{R}(R/\mathbf{p},H^{-n+k+1}(L)) \ar[r] & \Hom_{\D(R)}(R/\mathbf{p}[n-k-2],\tau^{\leq -n+k}(L)[1])}$$
Finally, the k+1-step of the induction for fact 2) part follows from lemma \ref{lem. X ----> limite}, since $\varphi_{j}$ is an isomorphism, for all $j \leq -n+k+1$.
\end{proof}

\begin{corollary}
Let $R$ be a commutative Noetherian ring such that $\Spec(R_\p)$ is a finite set, for all $\p \in \Spec(R_\p)$. The heart of every compactly generated t-structure in $\D(R)$ is a Grothendieck category. In particular, this is true when the Krull dimension of $R$ is $\leq 1$.
\end{corollary}
\begin{proof}
Let $(\mathcal{U},\mathcal{U}^{\perp}[1])$ be a compactly generated t-structure in $\D(R)$. From theorem \ref{teo. main AJS}, we have that there exists a sp-filtration $\phi $ such that $\mathcal{U}=\mathcal{U}_{\phi}$. It is clear that $\phi_{\mathbf{p}}$ is a left bounded sp-filtration (since $\Spec(R_{\mathbf{p}})$ is finite), for all $\mathbf{p}\in \Spec(R)$. Using the previous theorem, we obtain that $\mathcal{H}_{\phi_{\mathbf{p}}}$ is a Grothendieck category, and so it is an AB5 abelian category, for all $\mathbf{p}\in \Spec(R)$. Therefore, $\Hp$ is an AB5 abelian category (see corollary \ref{cor. localization AB5}). The result follows from proposition \ref{prop. H has a generator}.
\end{proof}

\section{The module case}
Once we have covered a large part of the sp-filtrations whose associated heart $\Hp$ is a Grothendieck category, we move on to a stronger question: which are the sp-filtrations $\phi$ such that $\Hp$ is a module category?. Unlike the Grothendieck case, we won't need any preconditions on $\phi$. We start by giving some comments on t-structures in a finite product of triangulated categories.

\begin{remark}\rm{
Let $\D_1, \dots, \D_n$ be triangulated categories with arbitrary (set-indexed) coproducts. The product category $\D=\D_{1} \times \dots \times \D_n $ has an obvious structure of triangulated category with the relevant concepts defined componentwise. Then one can view each $\D_{i}$ as a full triangulated subcategory of $\D$ by identifying each object $X$ of $\D_{i}$ with the object $(0,\dots,0,X,0,\dots,0)$ of $\D$, with $X$ in the $i$-th position. With this convention, we have the following useful result.}
\end{remark}

\begin{lemma}\label{lem. product of heart}
With the terminology of the previous remark, the following assertions hold:
\begin{enumerate}
\item[1)] If $\mathcal{U}$ is the aisle of a t-structure in $\D$, then $\mathcal{U}\cap \D_i$ is the aisle of a t-structure in $\D_{i}$, for $i=1,\dots,n$.

\item[2)] If $\mathcal{U}_{i}$ is the aisle of a t-structure in $\D_{i}$, for $i=1,\dots,n$, then the full subcategory $\mathcal{U}$ of $\D$ consisting of $n$-tuples $(X_{1},\dots, X_{n})$ such that $X_{i}\in \mathcal{U}_{i}$, for all $i=1,\dots,n$ is the aisle of a t-structure in $\D$. We denote this aisle $\underset{1 \leq i \leq n}{\oplus} \mathcal{U}_{i}$.

\item[3)] The assignment $(\mathcal{U}_{1} \times \dots \times \mathcal{U}_{n}) \rightsquigarrow \underset{1 \leq i \leq n}{\oplus} \mathcal{U}_{i}$ defines an isomorphism of ordered classes
$$\text{aisl}(\D_i) \times \dots \times \text{aisl}(\D_n) \iso \text{aisl}(\D).$$

\item[4)] If $\mathcal{H}_{i}$ denotes the heart of the t-structure $(\mathcal{U}_{i},\mathcal{U}_{i}^{\perp}[1])$, for $i=1,\dots, n$, and $\mathcal{H}$ is the heart of the t-structure in $\D$ whose aisle $\underset{1 \leq i \leq n}{\oplus} \mathcal{U}_{i}$, then $\mathcal{H}$ is equivalent to \linebreak $\mathcal{H}_{1} \times \dots \times \mathcal{H}_{n}$.
\end{enumerate}
\end{lemma}
\begin{proof}
1), 2) and 3) is clear, since the triangles in $\D$ are $n$-th tuples of triangles where the $i$-th position is a triangle in $\D_{i}$, for all $i=1,\dots,n$. \\

As for assertion 4, put $\mathcal{U}:=\underset{1 \leq i \leq n}{\oplus} \mathcal{U}_{i}$. We then $\mathcal{U}^{\perp}=\underset{1 \leq i \leq  n}{\bigcap} \mathcal{U}_{i}^{\perp}$, where $(-)^{\perp}$ denotes orthogonality in $\D$ and we are viewing $\mathcal{U}_{i}$ as a full subcategory of $\D$. But, by definition, we have $\mathcal{U}_{i}^{\perp}=\D_1 \times \dots \times \D_{i-1} \times \mathcal{U}_{i}^{\circ} \times \D_{i+1} \times \dots \times \D_n$, where $\mathcal{U}_{i}^{\circ}$ denotes the right orthogonal of $\mathcal{U}_i$ within $\D_{i}$. It then follows that $\mathcal{U}^{\perp}=\mathcal{U}_{i}^{\circ} \times \dots \times \mathcal{U}_{n}^{\circ}$, which implies that $\mathcal{H}=\mathcal{U} \cap \mathcal{U}^{\perp}[1]$ is equal to $\mathcal{H}_1 \times \dots \times \mathcal{H}_n$. Assertion 4 follows immediately from this. 
\end{proof}

\begin{example}\rm{
Let $R$ be a not necessarily commutative ring and let $\{e_{1},\dots,e_{n}\}$ be a family of nonzero orthogonal central idempotents of $R$ such that $1=\underset{1 \leq i \leq n}{\sum}e_i$. If $R_i:=e_iR$ and $X_i$ is a complex of $R_{i}$-modules, for all $i=1,\dots,n$, then $X=\underset{1 \leq i \leq n}{\oplus} X_i$ is a complex of $R$-modules. Moreover, the assignment $(X_{1},\dots,X_{n}) \rightsquigarrow \underset{1 \leq i \leq n}{\oplus}X_{i}$ gives an equivalence of triangulated categories $\D(R_{1}) \times \dots \times \D(R_n) \iso \D(R)$. The previous lemma then says that all aisles of t-structures in $\D(R)$ are direct sums of aisles of t-structures in the $\D(R_i)$, with the obvious meaning.}
\end{example}

\subsection{The eventually trivial case for connected rings}
In this subsection, we assume that $\phi$ is an eventually trivial sp-filtration, but we assume that $\phi\neq \emptyset$, i.e., that there is some $i \in \mathbb{Z}$ such that $\phi(i)\neq \emptyset$.


\begin{theorem}\label{prop. H mod. cat. R. con.}
Let us assume that $R$ is connected and let $\phi$ be an eventually trivial sp-filtration of $\Spec(R)$ such that $\phi(i)\neq \emptyset$, for some $i\in \mathbb{Z}$. The following assertions are equivalent:
\begin{enumerate}
\item[1)] $\Hp$ is a module category; 

\item[2)] There is an $m\in \mathbb{Z}$ such that $\phi(m)=\Spec(R)$ and $\phi(m+1)=\emptyset$;

\item[3)] There is an $m\in \mathbb{Z}$ such that the t-structure associated to $\phi$ is $(\D^{\leq m}(R),\D^{\geq m}(R))$.
\end{enumerate}

In that case $\Hp$ is equivalent to $R$-Mod.
\end{theorem}
\begin{proof}
The implications 2) $\Longleftrightarrow$ 3) $\Longrightarrow$ 1) are clear and assertion 3 implies that $\Hp\cong R\Mode$. Without loss of generality, we may and shall assume that $\phi(0)\neq \emptyset = \phi(1)$. For the implication 1) $\Longrightarrow$ 2), we fix a progenerator $P$ of $\Hp$ and we will prove in several steps that $\T_0=R$-Mod.\\

\emph{Step 1}: $H^{0}(P)$ is a finitely presented(=generated) $R$-module.\\

Indeed, from lemma \ref{lem. bounded heart}, we get that $P\in \D^{\leq 0}(R)$. By lemma \ref{lem. homology 0 module}(1), we have an isomorphism $\Hom_{R}(H^{0}(P),M)\iso \Hom_{\D(R)}(P,M[0])$, which is natural in $M$, for each $R$-module $M$. On the other hand, if $(T_{\lambda})$ is a direct system in $\T_{0}$, then $\varinjlim_{\Hp} T_{\lambda}[0]\cong (\varinjlim{T_{\lambda}})[0]$ (see proposition \ref{prop. limits stalk}). It follows that $\Hom_{R}(H^{0}(P),?)$ preserves direct limits of objects in $\T_{0}$, because $P$ is a finitely presented object of $\Hp$. Using again lemma \ref{lem. homology 0 module}(2), we obtain that $H^{0}(P)$ is a finitely presented $R$-module, since $\F_0$ is closed under taking direct limits (see remark \ref{rem. F closed under direct limits}).\\

\emph{Step 2}: $H^{0}(P)$ is a progenerator of $\T_{0}$.\\

Step 1 proves that $H^{0}(P)$ is a compact object of $\T_{0}$, i.e., that the functor $\Hom_{R}(H^{0}(P),?):\T_0 \flecha $Ab preserves coproducts. On the other hand, the functor $H^{0}:\Hp \flecha R$-Mod is right exact and its essential image is contained in $\T_{0}$. The facts that $\T_{0}[0]\subset \Hp$ and $P$ is a generator of $\Hp$ imply that $\T_0=\Gen(H^{0}(P))$ and, hence, that $H^{0}(P)$ is a generator of $\T_{0}$. Finally, if $T\in \T_{0}$ and we apply the functor $\Hom_{\D(R)}(?,T[1])$ to the triangle $\xymatrix{ P \ar[r] & H^{0}(P)[0] \ar[r] & \tau^{\leq -1}(P)[1] \ar[r]^{\hspace{0.8 cm}+} & }$, we obtain the following exact sequence in $R$-Mod:
$$\xymatrix{\Hom_{\D(R)}(\tau^{\leq -1}(P)[1],T[1])=0 \ar[r] & \Hom_{\D(R)}(H^{0}(P)[0],T[1]) \cong \Ext^{1}_{R}(H^{0}(P),T) \ar[d] \\ & \Hom_{\D(R)}(P,T[1])=\Ext^{1}_{\Hp}(P,T[0])=0}$$ 
This shows that $\Ext^{1}_{R}(H^{0}(P),T)=0$ and, hence, that $H^{0}(P)$ is a projective object of $\T_{0}$.\\

\emph{Step 3}: $\T_{0}=R$-Mod. \\

Using the bijection between hereditary torsion pairs and Gabriel topologies (see \cite[Chapter VI]{S}), we know that there are ideals $\mathbf{a}_{1}, \dots, \mathbf{a}_{m}$ such that $R/\mathbf{a}_{i}\in \T_{0}$, for $i=1,\dots,m,$ and there is an epimorphism $\underset{1 \leq i \leq m}{\oplus}\frac{R}{\mathbf{a}_{i}} \epic H^{0}(P)$. Since $R$ is a commutative ring, we know $\ann_{R}(\frac{R}{\mathbf{b}})=\mathbf{b}$, for each ideal $\mathbf{b}$ of $R$, so that $\mathbf{a}=\ann_{R}(\underset{1 \leq i \leq m}{\oplus}\frac{R}{\mathbf{a}_{i}})=\underset{1 \leq i \leq m}{\cap}\mathbf{a}_{i} \subseteq \ann_{R}(H^{0}(P))$. But $R/\mathbf{a}$ is also in $\T_{0}$, whence generated by $H^{0}(P)$. It follows that $\ann_{R}(H^{0}(P))\subseteq \mathbf{a}$ and, hence, that this inclusion is an equality. We then get that $\mathbf{a}=\ann_{R}(H^{0}(P))$ is contained in any ideal $\mathbf{b}$ of $R$ such that $R/\mathbf{b}\in \T_{0}$, which also implies that $\T_{0}=\Gen(R/\mathbf{a})$. By \cite[Propositions VI.6.12 and VI.8.6]{S}, we know that $\mathbf{a}=Re$, for some idempotent element $e$ in $R$. The connectedness of $R$ gives that $\mathbf{a}=0$ or $\mathbf{a}=R$. But the second possibility is discarded because $\phi(0) \neq \emptyset$, and hence $\T_{0}\neq 0$. This show that $\T_{0}=\Gen(R)=R$-Mod or, equivalently, that $\phi(0)=\Spec(R)$. 
\end{proof}

\subsection{Some auxiliary results}

\begin{proposition}\label{prop. Hp equivalent quotient}
Let $Z$ be a sp-subset of $\Spec(R)$, and let $\phi$ be the sp-filtration given by $\phi(i)=\Spec(R)$, for $i \leq 0$, and $\phi(i)=Z$, for all $i>0$. Then, for each $Y\in \mathcal{H}_{\phi}$, the $R$-module $H^{0}(Y)$ is $Z$-closed $R$-module. Moreover, the assignment $Y  \rightsquigarrow H^{0}(Y)$ defines an equivalence of categories $H^{0}:\Hp \iso \frac{R\text{-Mod}}{\T_Z}$, whose inverse in $L(?[0])$ (see lemma \ref{lemma de adjunctions}). 
\end{proposition}
\begin{proof}
From lemma \ref{lem. bounded heart}, we have that $\Hp\subseteq \D^{\geq 0}(R)$. Hence, for each $Y\in \Hp$, we get a triangle in $\D(R)$ of the form:
$$\xymatrix{\tau^{>0}(Y)[-1] \ar[r] & H^{0}(Y)[0] \ar[r]^{ \hspace{0.7 cm}\mu_Y} & Y \ar[r]^{+} & & (\ast)}$$

Using assertions 3 of lemma \ref{lem. last homology}, we know that $H^{0}(Y)$ is a $Z$-closed $R$-module. On the other hand, note that $H^{j}(\tau^{>0}(Y))\in \T_Z$, for all $j\in \mathbb{Z}.$ Hence, $\tau^{>0}(Y)[k]\in \mathcal{U}_{\phi}$, for all $k\in \mathbb{Z}$. We then have $L(H^{0}(Y)[0])=\tau^{>}_{\phi}(H^{0}(Y)[-1])[1]\cong Y$, since $Y\in \mathcal{U}_{\phi}^{\perp}[1]$. \\

Now, if $Y^{'}$ is an arbitrary object in $\Hp$, then applying the functor $\Hom_{\D(R)}(?,Y^{'})$ to triangle $(\ast)$, we obtain that $\mu_{Y}^{*}:=\Hom_{\D(R)}(\mu_{Y},Y^{'}):\Hom_{\D(R)}(Y,Y^{'}) \flecha \Hom_{\D(R)}(H^{0}(Y)[0],Y^{'})$ is an isomorphism, since $\Hom_{\D(R)}(\tau^{>0}(Y)[k],Y^{'})\cong \Hom_{\D(R)}(\tau^{>0}(Y)[k-1],Y^{'}[-1])=0$, for all $k\in \mathbb{Z}$. Similarly, applying the functor $\Hom_{\D(R)}(H^{0}(Y)[0],?)$ to the triangle $(\ast)$ associated to $Y^{'}$, we obtain that $(\mu_{Y^{'}})_{*}:=\Hom_{\D(R)}(H^{0}(Y)[0],\mu_{Y^{'}}): \Hom_{\D(R)}(H^{0}(Y)[0],H^{0}(Y^{'})[0])\cong \Hom_{R}(H^{0}(Y),H^{0}(Y^{'}))\cong \Hom_{\frac{R\text{-Mod}}{\T_Z}}(H^{0}(Y),H^{0}(Y^{'})) \flecha \Hom_{\D(R)}(H^{0}(Y)[0],Y^{'})$ is an isomorphism. It is not difficult to see that the morphism $(\mu_{Y^{'}})_{*}^{-1} \circ \mu_{Y}^{*}:\Hom_{\D(R)}(Y,Y^{'}) \iso \linebreak \Hom_{\frac{R\text{-Mod}}{\T_Z}}(H^{0}(Y),H^{0}(Y^{'}))$ coincide with the morphism 
$$H^{0}:\Hom_{\D(R)}(Y,Y^{'}) \flecha \Hom_{\frac{R\text{-Mod}}{\T_Z}}(H^{0}(Y),H^{0}(Y^{'}))$$

This show that the functor $H^{0}:\Hp \flecha \frac{R\text{-Mod}}{\T_Z}$ is full and faithful. We claim that it is also dense. Indeed, if $F$ is a $Z$-closed $R$-module, then $F[0]\in \mathcal{U}_{\phi}$ and we have a triangle in $\D(R)$:
$$\xymatrix{\tau^{\leq}_{\phi}(F[-1])[1] \ar[r]^{\hspace{0.8 cm}f} & F[0] \ar[r]^{g\hspace{0.35 cm}} & L(F[0]) \ar[r]^{\hspace{0.6cm}+} & }$$ 
Since $L(F[0])\in \Hp\subseteq \D^{\geq 0}(R)$, we know that $\tau^{\leq}_{\phi}(F[-1])[1]\in \D^{\geq0}(R)$ and $H^{0}(\tau_{\phi}^{\leq}(F[-1])[1])\cong H^{1}(\tau_{\phi}^{\leq}(F[-1])$, which is in $\T_Z$. From the sequence of homologies applied to the previous triangle. we get an exact sequence:

$$\xymatrix{0 \ar[r] & H^{0}(\tau_{\phi}^{\leq}(F[-1])[1]) \ar[r] & F \ar[r]^{H^{0}(g)\hspace{0.8 cm}} & H^{0}(L(F[0])) \ar[r] & H^{1}(\tau_{\phi}^{\leq}(F[-1])[1]) \ar[r] & 0}$$
It follows that $H^{0}(\tau^{\leq}_{\phi}(F[-1])[1]) \in \F_Z\cap \T_Z=0$, and hence $H^{0}(L(F[0]))\cong F \oplus H^{1}(\tau^{\leq}_{\phi}(F[-1])[1])$, since $F$ is $Z$-closed $R$-module. But, $H^{0}(L(F[0]))\in \F_Z$ (see lemma \ref{lem. last homology}), so that $H^{1}(\tau^{\leq}_{\phi}(F[-1])[1])=0$, and hence the morphism $H^{0}(g):F \flecha H^{0}(L(F[0]))$ is an isomorphism.
\end{proof}

\vspace{0.3 cm}

In the rest of the chapter we will use the following terminology.

\begin{definition}\rm{
Let $S\subseteq W$ be subsets of $\Spec(R)$. We will say that $S$ is \emph{stable under specialization (resp. generalization) within $W$}, when the following property holds:
\begin{enumerate}
\item[-] If $\mathbf{p} \subseteq \mathbf{q}$ are prime ideals in $W$ such that $\mathbf{p}\in S$ (resp. $\mathbf{q}\in S$), then $\mathbf{q}\in S$ (resp. $\mathbf{p}\in S$)
\end{enumerate}
Note that in such case $S$ need not be stable under specialization (resp. generalization) in $\Spec(R)$.}
\end{definition}

\begin{lemma}\label{lem. iterative quotiens category}
Let $Z\subseteq \Spec(R)$ be a sp-subset and suppose that $\Spec(R)\setminus Z=\tilde{V} \bigcupdot \tilde{W}$, where $\tilde{V}$ and $\tilde{W}$ are stable under specialization within $\Spec(R)\setminus Z$. If we put $V=\tilde{V} \cup Z$ and $W=\tilde{W} \cup Z$, then the following assertions hold:
\begin{enumerate}

\item[1)] $V$ and $W$ are sp-subsets of $\Spec(R)$;

\item[2)] The category $\frac{R\text{-Mod}}{\T_Z}$ is equivalent to $\frac{\T_V}{\T_Z} \times \frac{\T_W}{\T_Z};$

\item[3)] We have canonical equivalences of categories $\frac{\T_V}{\T_Z} \iso \frac{R\text{-Mod}}{\T_W}$ and $\frac{\T_W}{\T_Z} \iso \frac{R\text{-Mod}}{\T_Z}$.
\end{enumerate}
\end{lemma}
\begin{proof}
All throughout the proof, for any sp-subset $Z$ of $\Spec(R)$ we shall identify $\frac{R\text{-Mod}}{\T_Z}$ with the associated Giraud subcategory $\G_Z$ of $R$-Mod consisting of the $Z$-closed $R$-modules (see subsection \ref{sec. Commutative algebra}).\\

1) It is sufficient to prove that $V$ is a sp-subset of $\Spec(R)$. Let $\mathbf{p} \subseteq \mathbf{q}$ be prime ideals, where $\mathbf{p}\in V$. If $\mathbf{q}\in Z$ there is nothing to prove. So we assume that $\mathbf{q}\notin Z$, and hence $\mathbf{p}\notin Z$. In such case, $\mathbf{p}\in \tilde{V}$. It follows that $\mathbf{q}\in \tilde{V}$, since $\mathbf{q}\in V(\mathbf{p})\cap (\Spec(R)\setminus Z)$, and $\tilde{V}$ is stable under specialization within $\Spec(R) \setminus Z$. In particular, $\mathbf{q}\in V$. \\

 2) Clearly $\tilde{\T}_V:=q(\T_V)=\frac{\T_V}{\T_Z}$ is a hereditary torsion class in $\frac{R\text{-Mod}}{\T_Z}$, which gets identified with $\T_V \cap \G_Z$ when we identify $\frac{R\Mode}{\T_Z}$ with $\G_Z$. A similar fact is true when replacing $V$ by $W$. \\

We claim that $\Hom_{\G_Z}(T_W,T_V)=0=\Ext^{1}_{\G_Z}(T_W,T_V)$, for all $T_W\in \tilde{\T_W}$ and all $T_V\in \tilde{\T_V}$  and, once this is proved, the same will be true with the roles of $V$ and $W$ exchanged. Note that $T_V$ is a $Z$-torsionfree and $V$-torsion $R$-module. Since $\T_V$ is closed under taking injective envelopes and $E(M)\in \text{Add}(\underset{\mathbf{p}\in \text{Ass}(M)}{\oplus}E(R/\mathbf{p}))$, for all $R$-module $M$, it follows that the minimal injective resolution of $T_V$ in $R$-Mod has the form
$$\xymatrix{0 \ar[r] & T_V \ar[r] & I^{0}_{\tilde{V}} \ar[r] & I^{1}_{\tilde{V}} \oplus I^{1}_{Z} \ar[r] & \cdots \ar[r] & I^{n}_{\tilde{V}} \oplus I^{n}_{Z} \ar[r] & \cdots}$$
with the convention that $I^{k}_{\tilde{V}}\in \text{Add}(\underset{\mathbf{p}\in \tilde{V}}{\oplus} E(R/\mathbf{p}))$ and $I^{k}_{Z}\in \text{Add}(\underset{\mathbf{p}\in Z}{\oplus} E(R/\mathbf{p}))$, for all $k\geq 0$. When applying the exact functor $q:R\text{-Mod} \flecha \frac{R\text{-Mod}}{\T_Z}\cong \G_Z$, we get an injective resolution 
$$\xymatrix{0 \ar[r] & T_V \ar[r] & I^{0}_{\tilde{V}} \ar[r] & I^{1}_{\tilde{V}}  \ar[r] & \cdots \ar[r] & I^{n}_{\tilde{V}} \ar[r] & \cdots}$$
in $\G_Z$. All its terms are then injective $W$-torsionfree $R$-modules, since $\tilde{V}\cap \tilde{W}=\emptyset$. This implies that $\Ext^{k}_{\G_Z}(T_W,T_V)=0$, for all $k\geq 0$, and our claim follows. \\

If $F\in \G_Z$ is such that $\Hom_{\frac{R\text{-Mod}}{\T_Z}}(q(T),F)=0$, for all $T\in \T_V$, the adjunction $(q,j)$ gives that $\Hom_{R}(T,j(Z))=0$, for all $T\in \T_V$, i.e., $j(F)$ is a $V$-torsionfree $R$-module, so that its injective envelope $E(j(F))$ is in $\text{Add}(\underset{\mathbf{p}\in \Spec(R)\setminus V}{\oplus} E(R/\mathbf{p}))=\text{Add}(\underset{\mathbf{p}\in \tilde{W}}{\oplus}E(R/\mathbf{p})).$ But then $E(j(F))$ is in $\T_W$, which implies that $F\in \T_W\cap \G_Z=\tilde{\T}_W$. By the previos paragraph, we have that $\Ext^{1}_{\G_Z}(F,?)$ vanishes on $\tilde{\T_V}$. It follows that if $Y\in \G_Z$ is any object and $\xymatrix{0 \ar[r] & \tilde{T}_V \ar[r] & Y \ar[r] & F \ar[r] & 0}$ is the canonical exact sequence in $\G_Z$ associated to torsion pair $(\tilde{\T}_{V}, \tilde{\T}_{V}^{\perp})$, then this sequence splits. Therefore $Y$ decomposes as $Y=\tilde{T}\oplus F$, where $\tilde{T}\in \tilde{\T}_V$ and $F\in \tilde{\T}_W$. Assertion 2 is now clear.\\

3) We shall prove that $\T_V \cap \G_Z=\G_W$. From the equivalences of categories $\frac{\T_V}{\T_Z}\cong \T_V \cap \G_Z$ and $\frac{R\text{-Mod}}{\T_W}\cong \G_W$ the first ``half''of assertion 3 will follow then automatically. The other ``half'' is obtained by symmetry. \\

It follows from the arguments in the proof of assertion 2 that if $\tilde{T}_V\in \T_V\cap \G_Z$ then $\tilde{T}_V$ is $W$-torsionfree. Let us take now $T_W\in \T_W$ and consider any exact sequence in $R$-Mod
$$\xymatrix{0 \ar[r] & \tilde{T}_V \ar[r]^{u} & M \ar[r] & T_W \ar[r] & 0}$$
When applying the functor $q:R\text{-Mod} \flecha \frac{R\text{-Mod}}{\T_Z}\cong \G_Z$, using the proof of assertion 2 we get a split exact sequence in $\G_Z$. In particular, we have $(j \circ q)(u)$ is a section in $R$-Mod. But, bearing in mind that $\tilde{T}_V$ is $Z$-closed, we can identify $(j \circ q)(u)$ with the composition $\tilde{T}_V \xymatrix{\ar[r]^{u\hspace{0.2 cm}} & M \ar[r]^{\mu_M \hspace{0.6 cm}} & (j \circ q)(M)}$. It follows that $u$ is a section in $R$-Mod, so that $\Ext^{1}_R(?,\tilde{T}_V)$ vanishes on $\T_W$. Therefore $\tilde{T}_V$ is in $\G_W$.\\

Now, let us take $Y\in \G_W$. Then $Y$ is $Z$-closed and assertion 2 gives a decomposition $Y=\tilde{T}_{V} \oplus \tilde{T}_W$, where $\tilde{T}_V\in \tilde{\T}_V$ and $\tilde{T}_W\in \tilde{\T}_W$. But $\tilde{T}_W$ is then $W$-torsion and $W$-closed, which implies that it is zero. Therefore we have $Y\cong \tilde{T}_V\in \T_V \cap \G_Z$.  
\end{proof}

\begin{lemma}\label{lem. equivalence with q(H)}
If $\phi$ is a sp-filtration of $\Spec(R)$ and $Z:=\underset{i \in I}{\bigcap} \phi(i)$, then $(q(\mathcal{U}_\phi),q(\mathcal{U}_{\phi}^{\perp})[1])$ is a t-structure in $\D(\frac{R\text{-Mod}}{\T_Z})$, where $q$ denotes the (left derived functor of) quotient functor $q:R\text{-Mod} \flecha \frac{R\text{-Mod}}{\T_Z}$. Furthermore, we have an equivalence of categories $\Hp\cong \mathcal{H}_{q(\phi)}$.
\end{lemma}
\begin{proof}
We denote by $\mathcal{U}_Z$ the (full) triangulated subcategory of $\D(R)$ consisting of the complexes $X$ such that $\Supp(H^{j}(X))\subseteq Z$, for all $j\in \mathbb{Z}$. Recall that if $(q:R\text{-Mod} \flecha \frac{R\text{-Mod}}{\T_Z},j:\frac{R\text{-Mod}}{\T_Z} \flecha R\text{-Mod})$ is the localization adjoint pair, then we get an adjunction of triangulated functors $(q,\mathbf{R}j)$. Moreover, for each object $M\in \D(R)$, we have a triangle 
$$\xymatrix{\mathbf{R}\Gamma_Z(M)\ar[r] & M \ar[r] & (\mathbf{R}j \circ q)(M) \ar[r]^{\hspace{1.1 cm}+} & }$$
This is precisely the triangle associated to the t-structure $(\mathcal{U}_Z,\mathcal{U}_Z^{\perp}[1])$ (see \cite[Section 2]{AJSo3} and \cite[Section 1.6]{AJS}). On the other hand, if $X\in \mathcal{U}_{\phi}$ and $Y\in \mathcal{U}^{\perp}_{\phi}$, then we have an isomorphism $\Hom_{\D(\frac{R\text{-Mod}}{\T_Z})}(q(X),q(Y))\cong \Hom_{\D(R)}(X, (\mathbf{R}j \circ q)(Y)).$ Due to the fact that $Y\in \mathcal{U}_{\phi}^{\perp}$, it is not difficult to see that the canonical map $Y \iso (\mathbf{R}j \circ q)(Y)$ is an isomorphism (since $\mathcal{U}_Z \subseteq \mathcal{U}_{\phi}$). It follows $\Hom_{\D(\frac{R\text{-Mod}}{\T_Z})}(q(X),q(Y))=0$. Now, if $Y\in \D(\frac{R\text{-Mod}}{\T_Z})$, then we get a triangle
$$\xymatrix{\tau^{\leq}_{\phi}(\mathbf{R}j(Y)) \ar[r] & \mathbf{R}j(Y) \ar[r] & \tau^{>}_\phi(\mathbf{R}j(Y))\ar[r]^{\hspace{1.0 cm}+} & }$$
in $\D(R)$, from which we get a triangle in $\D(\frac{R\text{-Mod}}{\T_Z})$:
$$\xymatrix{q(\tau^{\leq}_{\phi}(\mathbf{R}j(Y))) \ar[r] & q(\mathbf{R}j(Y))\cong Y \ar[r] & q(\tau^{>}_\phi(\mathbf{R}j(Y)))\ar[r]^{\hspace{1.0 cm}+} & }$$
where the outer terms are in $q(\mathcal{U}_\phi)$ and $q(\mathcal{U}_{\phi}^{\perp})$ respectively. Thus, $(q(\mathcal{U}_\phi), q(\mathcal{U}_{\phi}^{\perp})[1])$ is a t-structure in $\D(\frac{R\text{-Mod}}{\T_Z})$.\\

On the other hand, we have a commutative diagram
$$\xymatrix{\Hp \ar@{-->}[r]^{\overline{q}} \ar@{^(->}[d]&\mathcal{H}_{q(\phi)} \ar@{^(->}[d]\\ \D(R) \ar[r]^{q \hspace{0.4 cm}} & \D(\frac{R\text{-Mod}}{\T_Z})}$$
We shall prove that the upper horizontal arrow is an equivalence of categories. We will start by showing that $\tilde{q}$ is dense. Let $M\in \mathcal{H}_{q(\phi)}$, we then have $M\cong q(U)$, for some $U\in \mathcal{U}_{\phi}$ and, due to the triangle, $\xymatrix{\mathbf{R}\Gamma_Z(U) \ar[r] & U \ar[r]^{\mu_U\hspace{0.9 cm}} & (\mathbf{R}j \circ q)(U) \ar[r]^{\hspace{1.1 cm}+}&}$ in $\D(R)$, we can assume that $\mu_U:U \flecha (\mathbf{R}j \circ q)(U)$ is an isomorphism, since $(\mathbf{R}j \circ q)(U)\in \mathcal{U}_{\phi}$ and $q(U)\cong q((\mathbf{R}j \circ q)(U))$. We claim that $U\in \Hp$, for which we just need to see that $U[-1]\in \mathcal{U}_{\phi}^{\perp}$. Indeed, for each $U^{'}\in \mathcal{U}_\phi$ we have the following isomorphisms:
\begin{eqnarray*}
\Hom_{\D(R)}(U^{'},U[-1])\cong \Hom_{\D(R)}(U^{'}, (\mathbf{R}j \circ q)(U)[-1]) & \cong & \Hom_{\D(\frac{R\text{-Mod}}{\T_Z})}(q(U^{'}), q(U)[-1]) \\ & \cong & \Hom_{\D(\frac{R\text{-Mod}}{\T_Z})}(q(U^{'}), M[-1])=0
\end{eqnarray*}
This is zero since $q(U^{'})\in q(\mathcal{U}_{\phi})$ and, hence our claim holds. On the other hand, if $M,N$ are in $\Hp$, then we get an isomorphism (since $N[-1]\in \mathcal{U}_{\phi}^{\perp}$):
$$\mu_{N[-1]}:N[-1] \iso (\mathbf{R}j \circ q)(N[-1])$$
Note that we have the following chain of isomorphism:
$$\xymatrix{\Hom_{\D(R)}(M,N) \ar[r]^{\sim \hspace{0.75cm}} & \Hom_{\D(R)}(M[-1],N[-1]) \ar[rr]^{(\mu_{N[-1]})_{*} \hspace{0.75 cm}} && \Hom_{\D(R)}(M[-1],(\mathbf{R}j \circ q)(N[-1])) \ar[d]^{\wr} \\ && & \Hom_{\D(\frac{R\Mode}{\T_Z})}(q(M[-1]), q(N[-1])) \ar[d]^{\wr} \\ && & \Hom_{\D(\frac{R\Mode}{\T_Z})}(q(M), q(N))} $$
It is not difficult to see that the resulting isomorphism is given by \linebreak $\Hom_{\D(R)}(M,N)\cong \xymatrix{ \Hom_{\Hp}(M,N) \ar[r]^{\tilde{q}(?) \hspace{0.8 cm}} & \Hom_{\mathcal{H}_{q(\phi)}}(q(M),q(N))} \iso \Hom_{\D(\frac{R\Mode}{\T_Z})}(q(M),q(N))$\\ 
Thus, $\tilde{q}$ is fully faithful.
\end{proof}

\subsection{The main theorem}
We are now ready for the second main result of this chapter.

\begin{theorem}\label{teo. second main}
Let $R$ be a commutative Noetherian ring, let $(\mathcal{U},\mathcal{U}^{\perp}[1])$ be a compactly generated t-structure in $\D(R)$ such that $\mathcal{U}\neq \mathcal{U}[-1]$ and let $\mathcal{H}$ be its heart. The following assertions are equivalent:
\begin{enumerate}
\item[1)] $\mathcal{H}$ is a module category;

\item[2)] There are a possibly empty stable under specialization subset $Z$ of $\Spec(R)$, a family $\{e_1,\dots,e_t\}$ of nonzero orthogonal idempotents of the ring quotients $R_Z$ and integers $m_1 < m_2 < \dots < m_t$ satisfying the following properties:
\begin{enumerate}
\item[a)] If $\mu_*: R_Z$-Mod$\flecha R$-Mod is the restriction of scalars functor and \linebreak $q:\D(R) \flecha \D(\frac{R\text{-Mod}}{\T_Z})$ is the canonical functor, then $\mathcal{U}$ consists of the complexes $U\in \D(R)$ such that $q(U)$ is in $\underset{1 \leq k \leq t}{\oplus} \D^{\leq m_t}(\frac{R_Ze_k \text{-Mod}}{\mu_{*}^{-1}(\T_Z)\cap R_Ze_k\text{-Mod}})$;

\item[b)] $\frac{R_Ze_k\text{-Mod}}{\mu_{*}^{-1}(\T_Z) \cap R_Ze_k\text{-Mod}}$ is a module category, for $k=1,\dots,t$.
\end{enumerate}
\end{enumerate}
In that case, $\mathcal{H}$ is equivalent to $\frac{R_Ze_1\text{-Mod}}{\mu_{*}^{-1}(\T_Z)\cap R_Ze_1\text{-Mod}} \times \dots \times \frac{R_Ze_t\text{-Mod}}{\mu_{*}^{-1}(\T_Z)\cap R_Ze_t\text{-Mod}}$.
\end{theorem}
\begin{proof}
2) $\Longrightarrow$ 1) By \cite[Chapter X, Section 2]{S} (see subsection \ref{sec. quotients category}), we have an equivalence of categories $\frac{R_Z\text{-Mod}}{\mu_{*}^{-1}(\T_Z)} \iso \frac{R \text{-Mod}}{\T_Z}$. The idempotents $e_i$ give a decomposition $R_Z$-Mod$\cong R_Ze_0\Mode \times R_Ze_1 \Mode \times \dots \times R_Ze_t \Mode$, where $e_0=1-e_1-\dots -e_t$. This in turn yields a decomposition of Grothendieck categories
\begin{small}
$$\frac{R\Mode}{\T_Z}\cong \frac{R_Z\Mode}{\mu^{-1}_{*}(\T_Z)} \iso \frac{R_Ze_0\Mode}{\mu_{*}^{-1}(\T_Z) \cap R_Ze_0\Mode} \times \frac{R_Ze_1\Mode}{\mu_{*}^{-1}(\T_Z) \cap R_Ze_1\Mod} \times \dots \times \frac{R_Ze_t\Mode}{\mu_{*}^{-1}(\T_Z) \cap R_Ze_t\Mode}$$
\end{small}
and a corresponding decomposition of the derived category 
\begin{center}
$
\hspace{4 cm}\xymatrix{\D(\frac{R\Mode}{\T_Z})\cong \D(\frac{R_Z\Mode}{\mu^{-1}_{*}(\T_Z)}) \ar[d]^(0.57){\wr} & \\ & }\newline \D( \frac{R_Ze_0\Mode}{\mu_{*}^{-1}(\T_Z) \cap R_Ze_0\Mode} )\times \D(\frac{R_Ze_1\Mode}{\mu_{*}^{-1}(\T_Z) \cap R_Ze_1\Mod}) \times \dots \times \D(\frac{R_Ze_t\Mode}{\mu_{*}^{-1}(\T_Z) \cap R_Ze_t\Mode})
$
\end{center}

On the other hand, if $\phi$ is the filtration by supports of $\Spec(R)$ associated to $(\mathcal{U},\mathcal{U}^{\perp}[1])$ (see theorem \ref{teo. main AJS}) and $i>m_t$, then we have the following chain of double implications:
$$\mathbf{p}\in \phi(i) \Longleftrightarrow q(R/\p[-i])=0 \Longleftrightarrow R/\p \in \T_Z \Longleftrightarrow \p\in Z$$
so that $\phi(i)=Z$, for all $i>m_t$. By lemma \ref{lem. equivalence with q(H)}, we get that $(q(\mathcal{U}_\phi), q(\mathcal{U}_{\phi}^{\perp})[1])$ is a t-structure in $\D(\frac{R\Mode}{\T_Z})$ whose heart is equivalent to $\Hp \cong \mathcal{H}$. But assertion 2 tells us that $q(\mathcal{U})$ is the direct sum (with respect to the above decompositions) of the aisles $\D^{\leq m_k}(\frac{R_Ze_k\Mode}{\mu^{-1}_{*}(\T_Z)\cap R_Ze_0\Mode})$, with $k=1,\dots, t$. It follows from lemma \ref{lem. product of heart}, that $\mathcal{H}$ is equivalent to the product of the hearts of the canonical t-structures 
$$(\D^{\leq m_t}(\frac{R_Ze_k\Mode}{\mu_{*}^{-1}(\T_Z)\cap R_Ze_0\Mode}),\D^{\geq m_t}(\frac{R_Ze_k\Mode}{\mu_{*}^{-1}(\T_Z)\cap R_Ze_0\Mode}))$$ 
in $\D(\frac{R_Ze_k\Mode}{\mu_{*}^{-1}(\T_Z)\cap R_Ze_0\Mode})$. The final statement of the theorem follows immediately from this and, hence, $\mathcal{H}$ is a module category since so is $\frac{R_Ze_k\Mode}{\mu_{*}^{-1}(\T_Z)\cap R_Ze_k\Mode}$, for $k=1,\dots,t$.\\

1) $\Longrightarrow $ 2) By lemma \ref{lem. product of heart}, we can assume, without loss of generality, that $R$ is connected. Let $\phi$ be the filtration by supports of $\Spec(R)$ associated to $(\mathcal{U},\mathcal{U}^{\perp}[1])$ and put $Z:=\underset{i\in \mathbb{Z}}{\bigcap} \phi(i)$ and $Z^{'}:=\underset{i\in \mathbb{Z}}{\bigcup} \phi(i)$. From the hypothesis of that $\mathcal{U}\neq \mathcal{U}[-1]$, there exists some integer $i$ such that $\phi(i)\supsetneq Z$. if now $\p \in Z^{'}\setminus Z$, then there exists a  greatest integer $m=m_\p$, such that $\p\in \phi(m)$. Consider the induced sp-filtration $\phi_\p$ of $\Spec(R_\p)$. By corollary \ref{cor. localizando (H) categorías de modulos}, the associated t-structure in $\D(R_\p)$ has a heart $\mathcal{H}_{\phi_\p}$ which is a module category. Moreover, we have $\phi_{\p}(m)\neq  \emptyset =\phi_{\p}(m+1)$ since $\p \in \phi(m)\setminus \phi(m+1)$. We are then in the eventually trivial case, and theorem \ref{prop. H mod. cat. R. con.} tell us that $\phi_{p}(m)=\Spec(R_\p)$ and, hence $\phi_{p}(m)=\phi_{\p}(i)=\Spec(R_\p)$ for all $i \leq m$. This proves that $\phi(i) \setminus Z$ is stable under generalization (and specialization) within $\Spec(R) \setminus Z$, for all $i \in \mathbb{Z}$ such that $Z\subsetneq \phi(i)$. It follows from this that if $\mathcal{M}_i=\text{MinSpec}(R)\cap \phi(i)$, then $\phi(i)=\bar{\mathcal{M}_i}\cup Z$, for all integers $i$, where the upper bar denotes the Zariski closure in $\Spec(R)$. Here MinSpec$(R)$ denotes the (finite) set of minimal prime ideals of $R$.\\

Now, we consider the chain $\dots \supseteq \mathcal{M}_i \supseteq \mathcal{M}_{i+1} \supseteq \cdots $. Bearing in mind that MinSpec$(R)$ is a finite set, together with the fact that $\mathcal{M}_i\subseteq \text{MinSpec}(R)$, for all $i\in\mathbb{Z}$, we have that previous chain is stationary. Moreover, we get integers $m_1 < m_2 < \dots <m_t$ satisfying the following properties:
\begin{enumerate}
\item[i)] $\phi(i) \supsetneq \phi(i+1)$ if, and only if, $i\in \{m_1,\dots, m_t\}$. Furthermore, $\phi(i)=Z$, for all $i>m_t$.

\item[ii)] If we put $\tilde{V}_0=\Spec(R)\setminus \phi(m_1)$, $\tilde{V}_t=\phi(m_t)\setminus Z$ and $\tilde{V}_k=\phi(m_k)\setminus \phi(m_{k+1})=\phi(m_k)\setminus \phi(m_{k}+1)$, for $k=1,\dots, t-1$, then each $\tilde{V_k}$ is stable under specialization and generalization in $\Spec(R)\setminus Z$ (since $\phi(i)=\bar{\mathcal{M}_i}\cup Z$, for all $i\in \mathbb{Z}$) and we have $\Spec(R)\setminus Z =\underset{0 \leq k \leq t}{\bigcupdot} \tilde{V}_k$.
\end{enumerate}

By lemma \ref{lem. equivalence with q(H)}, if $q:\D(R) \flecha \D(\frac{R\text{-Mod}}{\T_Z})$ is the canonical functor, then $(q(\mathcal{U}_{\phi}),q(\mathcal{U}_{\phi}^{\perp}))$ is a t-structure in $\D(\frac{R\Mode}{\T_Z})$ whose heart $\mathcal{H}_{q(\phi)}$ is equivalent to $\mathcal{H}=\Hp$. On the other hand, an iterative use of lemma \ref{lem. iterative quotiens category} says that if $V_k=\tilde{V}_k \cup Z$, for $k=0,1\dots,t$, then we have an equivalence of categories
$$\frac{R\text{-Mod}}{\T_Z}\xymatrix{\ar[r]^{\sim} & \ar[l]} \frac{\T_{V_{0}}}{\T_Z} \times \frac{\T_{V_{1}}}{\T_Z} \times \dots \times \frac{\T_{V_{t}}}{\T_Z}$$

If now $R_Z$ is the ring of quotients of $R$ with respect to $Z$, then, when viewed as an object of $\frac{R\Mode}{\T_Z}$, it decomposes in a direct sum $R_Z=Y_0\oplus Y_1 \oplus \dots \oplus Y_t$, where $Y_k\in \tilde{\T}_{V_{k}}:=\frac{\T_{V_{k}}}{\T_Z}$ for each $k=0,1,\dots, t.$ This decomposition corresponds to a decomposition of the identity of $\text{End}_{\frac{R\Mode}{\T_Z}}(R_Z)$ as a sum orthogonal idempotent endomorphisms. But we have a ring isomorphism $R_Z \iso \End_{\frac{R\Mode}{\T_Z}}(R_Z)$ since the section functor $\frac{R \Mode}{\T_Z} \flecha R_Z$-Mod is fully faithful (see \cite[Chapter IX, Section 1]{S}). We then get idempotents $e_0,e_1,\dots,e_t\in R_Z$, which are central since $R_Z$ is commutative (see lemma \ref{lem. R_F is a commutative ring}), such that $1=e_0+e_1+ \dots + e_t$. Note that all the $e_k$ are nonzero except perhaps $e_0$, which is zero when $\phi(m_1)=\Spec(R).$ \\

Using the equivalence of categories $\frac{R_Z\Mode}{\mu_{*}^{-1}(\T_Z)}\xymatrix{\ar[r]^{\sim} & \ar[l]} \frac{R\Mode }{\T_Z}$ given by the restriction of scalars $\mu_{*}: R_Z \Mode \flecha R \Mode$, we can now rewrite the above decomposition of Grothendieck categories as
$$\frac{R_Z\Mode}{\mu_{*}^{-1}(\T_Z)} \xymatrix{\ar[r]^{\sim} & \ar[l]} \frac{R_Ze_0 \Mode}{\mu_{*}^{-1}(\T_Z) \cap R_Ze_0 \Mode} \times \frac{R_Ze_1 \Mode}{\mu_{*}^{-1}(\T_Z) \cap R_Ze_1 \Mode} \times \dots \times \frac{R_Ze_t \Mode}{\mu_{*}^{-1}(\T_Z) \cap R_Ze_t \Mode}$$
That is, each category $\frac{\T_{V_{k}}}{\T_Z}$ is identified with the full subcategory of $Z$-closed $R$-modules that, when viewed as $R_Z$-modules, belong to $R_Ze_k$-Mod (equivalently, are annihilated by $1-e_k$). \\

The last decomposition of Grothendieck categories passes to the corresponding derived categories. Now, bearing in mind that $\mathcal{U}_Z\subseteq \mathcal{U}=\mathcal{U}_\phi$, we then have an isomorphism $\Hom_{\D(R)}(X,M)\cong \Hom_{\D(\frac{R\Mode}{\T_Z})}(q(X),q(M))$, for each $X\in \D(R)$ and $M\in \mathcal{U}^{\perp}_{\phi}$ (see the proof of previous lemma). Therefore, a complex $X\in \D(R)$ is in $\mathcal{U}$ if, and only if, $q(X)\in q(\mathcal{U})$. Hence, the proof will be finished once we check that $q(\mathcal{U})$ is the direct sum of the aisles $\D^{\leq m_k}(\frac{R_Ze_k\Mode}{\mu_{*}^{-1}(\T_Z)\cap R_Ze_k\Mode})$, with $1 \leq k \leq t$. Indeed, if that is the case then condition 2.a) will be automatic and, by lemma \ref{lem. product of heart}, the heart $\mathcal{H}\cong \mathcal{H}_{q(\phi)}$ will be equivalent to the product of categories $\frac{R_Ze_1\Mode}{\mu_{*}^{-1}(\T_Z)\cap R_Ze_1\Mode} \times \dots \times \frac{R_Ze_t \Mode}{\mu^{-1}_{*}(\T_Z)\cap R_Zet\Mode}$. Then also condition 2.b) will hold since $\mathcal{H}$ is a module category.\\

In the sequel, we will keep $H^{i}(?)$ to denote the $i$-th homology $R$-module and will denote by $H^{i}_{\G_Z}(?)$ the corresponding homology object in the category $\G_Z\cong \frac{R\Mode}{\T_Z}$. Let $Y\in \D(\G_Z)$, which we view as a complex of injective $Z$-closed $R$-modules, whence also as a complex of injective $R_Z$-modules. This complex is in $q(\mathcal{U})$ if, and only if, it is in $\mathcal{U}$ when viewed as a complex of $R$-modules. Then $Y\in q(\mathcal{U})$ if, and only if, $\Supp(H^{i}(Y))\subseteq \phi(i)=\phi(m_s)=Z\cup \tilde{V}_s \cup \dots \cup \tilde{V}_{t} \cup \tilde{V}_{t+1}$, whenever $m_s-1 < i \leq m_s$, with the convention that $m_0=-\infty$ and $m_{t+1}=+\infty$ and that $\tilde{V}_{t+1}=\emptyset$. When applying the functor $q:R\Mode \flecha \G_Z$, this is equivalent to saying that $H^{i}_{\G_Z}(Y)$ is in $\frac{\T_{V_s}}{\T_Z}\oplus \frac{\T_{V_{s+1}}}{\T_Z} \oplus \dots \oplus \frac{\T_{V_t}}{\T_Z}$ whenever $m_{s-1}<i \leq m_s$. Using now the equivalence of categories $\frac{\T_{V_{k}}}{\T_Z}\cong \frac{R_Ze_k\Mode}{\mu_{*}^{-1}(\T_Z)\cap R_Ze_k\Mode}$, with $k=0,1,\dots,t$, we conclude that $Y$ is in $q(\mathcal{U})$ if, and only if, we have $H^{i}_{\G_Z}(e_kY)=0$, for all $k=1, \dots, t$ and all $i>m_k$. This is exactly saying that $e_kY\in \D^{\leq m_k}(\frac{R_Ze_k\Mode}{\mu_{*}^{-1}(\T_Z)\cap R_Ze_k\Mode})$, for each $k=1,\dots,t$. Therefore we have the desired equality of aisles $q(\mathcal{U})=\underset{1 \leq k \leq t}{\oplus} \D^{\leq m_k}(\frac{R_Ze_k\Mode}{\mu_{*}^{-1}(\T_Z)\cap R_Ze_k\Mode})$
\end{proof}

\vspace{0.3 cm}

The following example shows that, even when $R$ is connected, the heart of a compactly generated t-structure can be a module category which strictly decomposes as a product of smaller (module) categories.

\begin{example}
Let $R$ be reduced (i.e. with no nonzero nilpotent elements), let $\{\p_0,\dots,\p_r\}$ be a subset of $\text{MinSpec}(R)$ and consider the sp-filtration $\phi$ of $\Spec(R)$ given as follows:
\begin{enumerate}
\item[1)] $\phi(i)=\Spec(R)\setminus \text{MinSpec}(R)$, for all $i>0$;

\item[2)] $\phi(i)=(\Spec(R) \setminus \MinSpec(R))\cup \{\p_0,\dots,\p_i\}$, for $-r<i \leq 0$;

\item[3)] $\phi(i)=(\Spec(R)\setminus \MinSpec(R))\cup \{\p_0,\dots,\p_r\}$, for all $i \leq -r$.
\end{enumerate}
If $(\mathcal{U}_{\phi},\mathcal{U}_{\phi}^{\perp}[1])$ is the associated compactly generated t-structure of $\D(R)$, then the heart $\Hp$ is equivalent to $k(\p_0)\Mode \times \dots \times k(\p_r)\Mode\cong k(\p_0)\times \dots \times k(\p_r)\Mode$, where $k(\p)$ denotes the residue field at $\p$, for each $\p\in \Spec(R)$.
\end{example}
\begin{proof}
If $S$ denotes the set of nonzero divisors of $R$ and $Z=\Spec(R)\setminus \MinSpec(R)$, then $R_Z\cong S^{-1}R$ and we have a ring isomorphism $S^{-1}R\iso \underset{\p\in \MinSpec(R)}{\prod}k(\p)$ (see \cite[Proposition III.4.23]{Ku}). Moreover the localization with respect to $\T_Z$ is a perfect localization in the sense of \cite[Chapter XI]{S}. It follows that $\frac{R\Mode}{\T_Z}$ is equivalent to $S^{-1}R \Mode$ and the canonical functor $q:R\Mode \flecha \frac{R\Mode}{\T_Z}$ gets identified with the localization functor $S^{-1}(?)\cong S^{-1}R \otimes_R ?:R\Mode \flecha S^{-1}R\Mode$. \\

On the other hand, given a complex $U\in \D(R)$, we have that $U\in \mathcal{U}_{\phi}$ if, and only if, $q(U)\in \underset{-r \leq k \leq 0}{\oplus} \D^{\leq k}(k(\p_k))$. Then condition 2 of previous theorem holds, so that $\Hp$ is equivalent to $k(\p_0)\Mode \times \dots \times k(\p_r)\Mode \cong k(\p_0)\times \dots \times k(\p_r)\Mode$.  
\end{proof}

\vspace{0.3 cm}

Some direct consequences can be derived now from last theorem.

\begin{corollary}
Let $R$ be connected and let $(\mathcal{U},\mathcal{U}^{\perp}[1])$ be a compactly generated left non-degenerate t-structure (i.e. $\underset{k\in \mathbb{Z}}{\bigcap} \mathcal{U}[k]=0$) in $\D(R)$ such that $\mathcal{U}\neq 0$. The heart of this structure is a module category if, and only if, $(\mathcal{U},\mathcal{U}^{\perp}[1])=(\D^{\leq m}(R),\D^{\geq m}(R))$, for some integer $m$.
\end{corollary}
\begin{proof}
Note that it is enough to prove the ``only part''. If $\phi$ is the associated filtration by support, the left non-degeneracy of $(\mathcal{U,U}^{\perp}[1])$ translates into the fact that $\underset{i\in \mathbb{Z}}{\bigcap} \phi(i)=\emptyset$. Now, from the proof of the implication 1) $\Longrightarrow$ 2) in theorem \ref{teo. second main}, we obtain that there exists an integers $m$ such that $\phi(j)=Z=\underset{i\in \mathbb{Z}}{\bigcap} \phi(i)=\emptyset$, for all $j>m$, i.e., $\phi$ is eventually trivial and, hence the result follows from theorem \ref{prop. H mod. cat. R. con.}.
\end{proof}

\begin{corollary}
Let $R$ be connected and suppose that its nilradical is a prime ideal. Let $(\mathcal{U},\mathcal{U}^{\perp}[1])$ be a compactly generated t-structure in $\D(R)$ such that $\mathcal{U}\neq \mathcal{U}[-1]$. The heart $\mathcal{H}$ of this t-structure is a module category if, and only if, there are a possibly empty sp-subset $Z$ of $\Spec(R)$ and an integer $m$ such that $\frac{R\Mode}{\T_Z}$ is a module category and $\mathcal{U}$ consists of the complexes $X\in \D(R)$ such that $\Supp(H^{i}(R))\subseteq Z$, for all $i>m$. In this case $\mathcal{H}$ is equivalent to $\frac{R\Mode}{\T_Z}$.
\end{corollary}
\begin{proof}
The ``if part'' and final part is a direct consequence of proposition \ref{prop. Hp equivalent quotient}. \\

On the other hand, for the ``only if'' part, note that the proof of theorem \ref{teo. second main} gives an sp-subset $Z\subsetneq \Spec(R)$ and decomposition $\Spec(R)\setminus Z=\underset{0 \leq k \leq t}{\bigcupdot}{\tilde{V}_k}$, where the $\tilde{V}_k$ are stable under specialization and generalization in $\Spec(R)\setminus Z,$ and all $\tilde{V}_k$ are nonempty except perhaps $\tilde{V}_0$, which is empty exactly when $\phi(i)=\Spec(R)$ for some $i\in \mathbb{Z}$. The existence of a unique minimal prime ideal of $R$ implies that $t=1$ and, with the terminology of theorem \ref{teo. second main}, that there is a unique integer $m=m_1$ in its assertion 2. Moreover, we have $\tilde{V}_0=\emptyset$. Then the associated filtration by supports satisfy that $\phi(i)=\Spec(R)$, for $i \leq m$, and $\phi(i)=Z$, for  $i>m$. Then the result follows from proposition \ref{prop. Hp equivalent quotient}.
\end{proof}

\vspace{0.3 cm}

After last theorem and its corollaries, the following question is pertinent.

\begin{question}
Let $R$ be a commutative Noetherian ring and $Z\subset\Spec (R)$ be an sp-subset such that $\frac{R\Mode}{\mathcal{T}_Z}$ is a module category. Does it come from a perfect localization?. In other words, is the canonical embedding $\frac{R\Mode}{\mathcal{T}_Z} \monic R_Z\Mode$ an equivalence of categories?
\end{question}

\begin{remark}
Note that if the answer to last question is affirmative, then assertion 2 of theorem \ref{teo. second main} could be rewritten as follows:

\begin{enumerate}

\item[2$^{'}$)] There are a possible empty sp-subset $Z\subsetneq\Spec (R)$, a family $\{e_1,\dots,e_t\}$ of nonzero orthogonal idempotents of $R_Z$ and integers $m_1<m_2< \dots <m_t$ such that $\mathcal{U}$ consists of the complexes $U\in\mathcal{D}(R)$ such that $R_Z\otimes_R^{\mathbf{L}}U\in \underset{1\leq k\leq t}{\oplus}\mathcal{D}^{\leq m_k}(R_Ze_k)$. 
\end{enumerate}
\end{remark}

\chapter{Open questions}
We present a list of questions which appear in a natural way from the results obtained for this work  and for which we do not have an answer.

\begin{enumerate}
\item[1)] Given a Grothendieck category $\G$ and a torsion pair in $\G$, $\te=(\T,\F)$, such that $\F$ is closed under taking direct limits, is $\Ht$ a Grothendieck category?;

\item[2)] Let $R$ be a ring and $V$ be a 1-tilting $R$-module such that $\Ker(\Hom_{R}(V,?))$ is closed under taking direct limits in $R$-Mod. Is $V$ equivalent to a classical 1-tilting module?

\item[3)]  Let $\G$ be a locally finitely presented Grothendieck category and $Q$ be a 1-cotilting object. Is $\F=\Cogen(Q)$ a generating class of $\G$?.

\item[4)] Let $\D$ be a triangulated category which has coproducts and let $(\mathcal{U},\mathcal{U}^{\perp}[1])$ be a compactly generated t-structure in $\D$. Is $\mathcal{H}=\mathcal{U}\cap \mathcal{U}^{\perp}[1]$ an AB5 abelian category?
Is it so when $\D=\D(R)$, where $R$ is a commutative Noetherian ring?. 

\item[5)] Given a Grothendieck category $\G$ and a t-structure in $\D(\G)$, $(\mathcal{U},\mathcal{U}^{\perp}[1])$, such that its associated heart $\mathcal{H}$ is an AB5 abelian category. Do the homology functors $H^{m}:\mathcal{H} \flecha \G$ preserve direct limits, for all $m\in $ \Z?. 

\item[6)] Let $R$ be a commutative Noetherian ring and $Z\subset\Spec (R)$ be an sp-subset such that $\frac{R\Mode}{\mathcal{T}_Z}$ is a module category. Does it come from a perfect localization?.

\item[7)] Given a recollement of triangulated categories 
$$\xymatrix{\D^{'} \ar@<0ex>[rr]  &&  \D  \ar@<0ex>[rr] \ar@<1ex>[ll] \ar@<-1 ex>[ll]  && \D^{''} \ar@<1ex>[ll] \ar@<-1 ex>[ll] }$$ 
and t-structures in $\D^{'}$ and $\D^{''}$ whose hearts are module categories, is the heart of the t-structure induced in $\D$ a module or a Grothendieck category?.
\end{enumerate}

\backmatter

\printindex

\end{document}